\documentclass[a4paper]{amsart}
\usepackage{amsmath}
\usepackage{graphicx}
\usepackage{amsfonts}
\usepackage{amssymb}
\providecommand{\U}[1]{\protect\rule{.1in}{.1in}}
\providecommand{\U}[1]{\protect\rule{.1in}{.1in}}
\newtheorem{theorem}{Theorem}
\theoremstyle{plain}

\newtheorem{corollary}[theorem]{Corollary}

\newtheorem{definition}[theorem]{Definition}
\newtheorem{example}[theorem]{Example}

\newtheorem{lemma}[theorem]{Lemma}
\newtheorem{notation}[theorem]{Notation}

\newtheorem{proposition}[theorem]{Proposition}
\newtheorem{remark}[theorem]{Remark}

\numberwithin{equation}{section}
\numberwithin{theorem}{section}

\begin{document}

\title{Topological radicals and Frattini theory of Banach Lie algebras}

\author{Edward Kissin}

\address{%
 STORM, London Metropolitan University\\
166-220 Holloway Road, London N7 8DB\\ 
Great Britain}

\email{e.kissin@londonmet.ac.uk}

\author{Victor S. Shulman}

\address{%
Department of Mathematics\\ 
Vologda State Technical University, Vologda\\ 
Russia}

\email{shulman.victor80@gmail.com}

\author{Yurii V. Turovskii}

\address{%
Institute of Mathematics and Mechanics\\
National Academy of Sciences of Azerbaijan\\ 
9 F. Agayev Street, Baku AZ1141\\  
Azerbaijan} 

\email{yuri.turovskii@gmail.com}



\begin{abstract}
In this paper we develop the theory of topological radicals of Banach Lie
algebras and apply it to the study of the structure of Banach Lie algebras
with sufficiently many closed Lie subalgebras of finite codimensions, that is,
the intersection of all these subalgebras is zero. The first part is devoted
to the radical theory of Banach Lie algebras; the second develops some
technique of construction of preradicals via sub\-space-multi\-fun\-ctions and
analyses the cor\-responding radicals, and the third part contains the Frattini
theory of in\-fin\-ite-dim\-en\-sional Banach Lie algebras. It is shown that the
multi\-fun\-ctions of closed Lie subalgebras of finite codimension (closed Lie
ideals of finite codimension, closed maximal Lie subalgebras of finite
codimension, closed maximal Lie ideals of finite codimension) produce
different preradicals, and that these preradicals generate the same radical,
\textit{the Frattini radical}. The main attention is given to structural
properties of Frat\-tini-semi\-simple Banach Lie algebras and, in particular, to a
novel infinite-dim\-en\-sional phenomenon associated with the strong Frattini
preradical introduced in this paper. A new constructive description of
Frat\-tini-free Banach Lie algebras is obtained.
\end{abstract}
\maketitle

%
%
%
%
%
%
%
%
%
%
%
%
%
%

\section{\label{1}Introduction}

A finite-dimensional Lie algebra is called \textit{Frattini-free}
(respectively, \textit{Jacobson-free}) if the set of all its\textbf{ }maximal
Lie subalgebras (respectively, ideals) has zero intersection. These conditions
are very significant.\textbf{ }Note that the similar\textbf{ }condition for
all\textbf{ }left (right) maximal ideals of a\textbf{ }unital associative
algebra means its semisimplicity --- a notion of crucial importance in the
structure theory of algebras. Finite-dimensional Frattini-free and
Jacobson-free Lie algebras were studied in \cite{Ba, Sc, M, S, S2, T} and
their theory\textbf{ }forms an important basic part of the theory of identical
relations of Lie algebras (see \cite{B}).

One of the main obstacles in transferring this theory to infinite-dimen\-si\-onal
Lie algebras is the fact\textbf{ }that, in the contrast to associative
algebras (see\textbf{ }\cite{L}), an infinite-dimensional Lie algebra
with\textbf{ }a maximal Lie subalgebra of finite codimension may have no Lie
ideals of finite codimension \cite{A1, A2}. Recently the authors proved in
\cite{KST1, KST2}\textbf{ }that a Banach Lie algebra with a maximal Lie
subalgebra of finite codimension always has a Lie ideal of finite codimension
(for Banach Lie algebras with Lie subalgebras of codimension 1 this was proved
in \cite{K}). This result provides a powerful tool for the study of
infinite-dimensional Frattini-free and Jacobson-free Banach Lie algebras.

The appropriate framework for this study is the theory of ideal maps on the
class of Banach Lie algebras --- the theory of radicals. The radical theory
approach to the\ Frattini theory is new, as it is new to the theory of Banach
Lie algebras in general. It opens novel and interesting perspectives for
further investigation of Banach Lie algebras and is a rich source of
intriguing and stimulating problems. Thus in the present paper we pursue two
interconnected aims: to develop the theory of topological radicals of Banach
Lie algebras and to apply this theory for the study of the structure of Banach
Lie algebras that have rich families of closed subalgebras of finite codimension.

The notion of the solvable radical --- the map that associates each Lie
algebra $\mathcal{L}$ with its maximal solvable Lie ideal $\mathrm{rad}\left(
\mathcal{L}\right)  $ --- lies at the core of the classical theory of
finite-dimensional Lie algebras. Another important map of this kind is the
``nil radical''\ which maps a Lie algebra into its largest nilpotent ideal. In
numerous other situations it is often useful and enlightening to construct
specific ``radical-like''\ maps that send Lie algebras into their Lie ideals
and have some special structure properties.

The intensive study of such maps for associative algebras was extremely
fruitful and produced an important branch of modern algebra --- the general
theory of radicals (see \cite{Di,Sz}). A topological counterpart of this
theory --- the theory of topological radicals of associative normed algebras
--- was initiated by Dixon in \cite{D}. He proposed a radical theory approach
to the study of the existence of topologically irreducible representations of
Banach algebras. Stimulated by Dixon's work, Read constructed in \cite{R2} his
famous example of a quasinilpotent operator on a Banach space that has no
non-trivial closed invariant subspaces. In \cite{ST-0,ST-I,ST-II,ST-III} the
second and third authors further developed the theory of topological radicals
of associative normed algebras and related this theory to many important
problems in Banach algebra theory and operator theory, such as the existence
of non-trivial ideals, radicality of tensor products, joint spectral radius,
invariant subspaces, spectral theory of multiplication operators etc.

The paper is divided into three parts. The first part -- Sections 2-5 -- is
devoted to the radical theory of Banach Lie algebras. Various new examples of
radicals and more general ideal maps for Banach Lie algebras are presented in
the second part -- Section 6. The third part -- Sections 7-9 -- contains the
Frattini theory of infinite-dimensional Banach Lie algebras.

A complex Lie algebra $\mathcal{L}$ with Lie bracket $[\cdot,\cdot]$ is a
Banach Lie algebra, if it is a Banach space in some norm $\left\|
\cdot\right\|  $ and there is a \textit{multiplication constant }%
$t_{\mathcal{L}}>0$ such that
\[
\left\|  \lbrack a,b]\right\|  \leq t_{\mathcal{L}}\left\|  a\right\|
\left\|  b\right\|  \text{ for all }a,b\in\mathcal{L}.
\]
For example, all Banach algebras are Banach Lie algebras with respect to the
Lie bracket $[a,b]=ab-ba$. In particular, all closed Lie subalgebras of the
algebra $\mathcal{B}(X)$ of all bounded operators on a Banach space $X$ are
Banach Lie algebras. Since bilinear maps on finite-dimensional spaces are
continuous, all complex finite-dimensional Lie algebras (with arbitrary norms)
can be considered as Banach Lie algebras.

Denote by $\mathfrak{L}$ the class of all Banach Lie algebras. We consider the
category $\overline{\mathbf{L}}$ of Banach Lie algebras with $\mathrm{Ob}%
\left(  \overline{\mathbf{L}}\right)  =\mathfrak{L,}$ assuming that morphisms
of $\overline{\mathbf{L}}$ are bounded homomorphisms with dense image, and the
subcategory $\mathbf{L}^{\text{f}}$ of $\overline{\mathbf{L}}$ with
$\mathrm{Ob}\left(  \mathbf{L}^{\text{f}}\right)  =\mathfrak{L}^{\mathrm{f}}$
--- the set of all finite-dimensional Lie algebras. It is sometimes reasonable
to consider the subcategory $\mathbf{L}$ of $\overline{\mathbf{L}}$ with
$\mathrm{Ob}\left(  \mathbf{L}\right)  =\mathrm{Ob}\left(  \overline
{\mathbf{L}}\right)  =\mathfrak{L}$ and bounded epimorphisms as morphisms, but
in this paper we will be mainly working in the category $\overline{\mathbf{L}%
}$.

A map $R$: $\mathfrak{L}\rightarrow\mathfrak{L}$ is\emph{ }a
\textit{preradical} in $\overline{\mathbf{L}}$ (in $\mathbf{L}$) if
$R(\mathcal{L)}$ is a closed Lie ideal of $\mathcal{L}$, for each
$\mathcal{L}\in\mathfrak{L}$, and%
\[
f(R\left(  \mathcal{L}\right)  )\subseteq R\left(  \mathcal{M}\right)  \text{
for each morphism\emph{ }}f\text{: }\mathcal{L}\longrightarrow\mathcal{M}%
\text{ in }\overline{\mathbf{L}}\text{ }(\text{in }\mathbf{L}).
\]
The study of any preradical $R$ leads naturally to the singling out two
subclasses of\textbf{ }$\mathfrak{L}$: the class $\mathbf{Sem}\left(
R\right)  $ of $R$-semisimple Lie algebras and the class $\mathbf{Rad}\left(
R\right)  $ of $R$-radical Lie algebras:%
\begin{align*}
\mathbf{Sem}(R)=\{\mathcal{L}\in\mathfrak{L}\text{: }R\left(  \mathcal{L}%
\right)  =\{0\}\}\text{ and } \\
\mathbf{Rad}(R)=\{\mathcal{L}\in\mathfrak
{L}\text{: }R\left(  \mathcal{L}\right)  =\mathcal{L}\}.
\end{align*}
A preradical $R$ is a radical if it behaves well on ideals and quotients. In
particular, $R(\mathcal{L})\in$ $\mathbf{Rad}(R)$ and $\mathcal{L}%
/R(\mathcal{L})\in\mathbf{Sem}(R)$. Thus the radical theory approach reduces
various problems concerning Lie algebras to the corresponding problems
concerning separately semisimple and radical algebras. For many radicals
constructed in this paper, the structure of Lie algebras in these classes is
far from trivial and the study of their structure is interesting and important
in many respects.

Section \ref{2} contains some basic definitions and preliminary results of the
theory of Banach Lie algebras. In particular, a considerable attention is
devoted to the notion of a Lie subideal that plays an important role
throughout the paper. In Section \ref{3} we introduce main notions of the
radical theory, consider special classes of preradicals and establish some of
their properties important for what follows.

Many naturally arising and important preradicals (for example, the classical
nil-radical) are not radicals. It is often helpful, using some ''improvement''
procedures, to construct from them other preradicals with certain additional
properties and, in particular, radicals associated with the initial
preradicals.\textbf{ }In Section \ref{section4} we examine these procedures.
They are the Banach Lie algebraic versions of the procedures employed by Dixon
for Banach associative algebras\textbf{ }which, in turn, are counterparts of
the Baer procedures for radicals of rings. They produce radicals that are
either the largest out of all radicals smaller than the original preradicals,
or the smallest out of all radicals larger than the original ones. We
extensively use the results and constructions of this section in the further sections.

A collection $\Gamma=\{\Gamma_{\mathcal{L}}\}_{\mathcal{L}\in\mathfrak{L}}$ of
families $\Gamma_{\mathcal{L}}$ of closed subspaces of Lie algebras
$\mathcal{L}\in\mathfrak{L}$ is called a subspace-multifunction. The
use\textbf{ }of subspace-multifunctions is new in the radical theory and gives
rise to many important preradicals on $\overline{\mathbf{L}}$. In Section
\ref{section5} we study the link between subspace-multifunctions and the
preradicals they generate.

In Section 6 we consider various subspace-multifunctions $\Gamma
=\{\Gamma_{\mathcal{L}}\}_{\mathcal{L}\in\mathfrak{L}}$ that\textbf{ }consist
of finite-dimensional Lie subalgebras and\textbf{ }of commutative Lie ideals
of $\mathcal{L}.$ We study the preradicals they generate and the corresponding
radicals obtained via the methods discussed in Section \ref{section4}. We show
that although the preradicals generated by these subspace-multifunctions are
different, the corresponding radicals often coincide and their restrictions to
$\mathbf{L}^{\text{f}}$ coincide with the classical radical ''rad''. Using and
improving the\textbf{ }ideas of Vasilescu (see \cite{V}), we introduce a new
radical that extends ''rad''\ to infinite-dimensional Lie algebras.

Aiming to investigate in Section 8 chains of Lie subalgebras and ideals of
Banach Lie algebras, we introduce and study in Section 6 a purely geometric
notion of a \textit{lower finite-gap} chain $C$ of closed subspaces of a
Banach space $X$. This means that each subspace $Y$ in $C$ contains another
subspace $Z$ from $C$ such that $Y/Z$ is finite-dimensional. In our paper such
chains appear as chains of Lie ideals, that is, the chains of subspaces
invariant for a family of operators.\textbf{ }Note that the concept of lower
finite-gap chains of subspaces invariant for families of operators is novel
and interesting in itself.

In Section \ref{7} we consider our main subject\textbf{:} the
subspace-multi\-fun\-c\-tions
\[
\mathfrak{S}=\{\mathfrak{S}_{\mathcal{L}}\}_{\mathcal{L}\in\mathfrak{L}}\text{
and }\mathfrak{S}^{\max}=\{\mathfrak{S}_{\mathcal{L}}^{\max}\}_{\mathcal{L}%
\in\mathfrak{L}},
\]
where families $\mathfrak{S}_{\mathcal{L}}$ and $\mathfrak{S}_{\mathcal{L}%
}^{\max}$ consist, respectively, of all closed proper and closed
\textit{maximal} proper Lie subalgebras of \textit{finite codimension } in
$\mathcal{L}$; and the subspace-multifunctions%
\[
\mathfrak{J}=\{\mathfrak{J}_{\mathcal{L}}\}_{\mathcal{L}\in\mathfrak{L}}\text{
and }\mathfrak{J}^{\max}=\{\mathfrak{J}_{\mathcal{L}}^{\max}\}_{\mathcal{L}%
\in\mathfrak{L}},
\]
where families $\mathfrak{J}_{\mathcal{L}}$ and $\mathfrak{J}_{\mathcal{L}%
}^{\max}$ consist, respectively, of all closed proper and closed
\textit{maximal} proper Lie ideals of \textit{finite codimension } in
$\mathcal{L}.$ The corresponding preradicals $P_{\mathfrak{S}},$
$P_{\mathfrak{S}^{\text{max}}},$ $P_{\mathfrak{J}}$ and $P_{\mathfrak
{J}^{\text{max}}}$ are defined by%
\begin{align*}
P_{\mathfrak{S}}(\mathcal{L})  &  =\cap_{L\in\mathfrak{S}_{\mathcal{L}}%
}L,\text{ \ }P_{\mathfrak{S}^{\text{max}}}(\mathcal{L})=\cap_{L\in\mathfrak
{S}_{\mathcal{L}}^{\text{max}}}L,\\
\text{\ }P_{\mathfrak{J}}(\mathcal{L})  &  =\cap_{L\in\mathfrak{J}%
_{\mathcal{L}}}L\text{ \ and }P_{\mathfrak{J}^{\text{max}}}(\mathcal{L}%
)=\cap_{L\in\mathfrak{J}_{\mathcal{L}}^{\text{max}}}L.
\end{align*}

The study of the above preradicals is based on the main result of \cite{KST2}
which states that if $\mathcal{L}_{0}$ is a maximal closed Lie subalgebra of
finite codimension in a Banach Lie algebra $\mathcal{L}$, then $\mathcal{L}%
_{0}$ contains a closed Lie ideal of finite codimension. Using it, we prove
that the radicals generated by the preradicals $P_{\mathfrak{S}}$,
$P_{\mathfrak{S}^{\text{max}}}$, $P_{\mathfrak{J}}$ and $P_{\mathfrak
{J}^{\text{max}}}$ coincide. The obtained radical is denoted by $\mathcal{F}$
and we call it\textbf{ }\textit{the Frattini radical}. It plays a central role
in the radical theory developed in this paper. We establish that the classes
of the radical Lie algebras corresponding to the above preradicals and to the
Frattini radical $\mathcal{F}$ coincide, while the classes of their semisimple
Lie algebras satisfy the inclusions%
\[
\mathbf{Sem}(P_{\mathfrak{J}^{\mathrm{\max}}})\subset\mathbf{Sem}%
(P_{\mathfrak{S}^{\mathrm{\max}}})\subset\mathbf{Sem}(P_{\mathfrak{J}}%
)\subset\mathbf{Sem}(P_{\mathfrak{S}})\subset\mathbf{Sem}(\mathcal{F})
\]
and all these inclusions are proper.

We prove that the Frattini radical is not hereditary and calculate it for
certain Banach Lie algebras. For instance, it is shown that the compact
operators in the continuous nest algebra form an $\mathcal{F}$-radical Banach
Lie algebra (Example \ref{E7.2}) and that\textbf{ }the commutator ideal of
a\textbf{ }simple infinite-dimensional associative Banach algebra with zero
center is an $\mathcal{F}$-radical Banach Lie algebra (Proposition \ref{E7}).

In Section \ref{8} we establish that each Banach Lie algebra $\mathcal{L}%
\in\mathbf{Sem}(P_{\mathfrak{J}})$ has a maximal lower finite-gap chain of
closed Lie ideals between $\{0\}$ and $\mathcal{L}$. We characterize
$\mathcal{F}$-semisimple Lie algebras in terms of lower finite-gap chains of
Lie subalgebras: a Banach Lie algebra $\mathcal{L}$ is $\mathcal{F}%
$-semisimple if and only if it has a lower finite-gap chain of closed Lie
subalgebras between $\{0\}$ and $\mathcal{L}$.

Making use of lower finite-gap chains of Lie ideals in Banach Lie algebras, we
introduce\textbf{ }another important preradical on $\overline{\mathbf{L}}$ ---
the\textit{ strong Frattini preradical} $\mathcal{F}_{s}$. It should be noted
that this preradical only appears as a phenomenon in the infinite-dimensional
Banach Lie algebras.\textbf{ }We show that $\mathcal{F}_{s}(\mathcal{F}%
_{s}(\mathcal{L}))=\mathcal{F(L)}$ and that $\mathcal{F}_{s}(\mathcal{L)}%
/\mathcal{F(L)}$ is commutative for each Banach Lie algebra $\mathcal{L}$. A
Banach Lie algebra $\mathcal{L}$ is $\mathcal{F}_{s}$-semisimple if and only
if it has a lower finite-gap chain of closed Lie ideals between $\{0\}$ and
$\mathcal{L}.$ Moreover, each closed Lie subalgebra of a $\mathcal{F}_{s}%
$-semisimple Lie algebra is also $\mathcal{F}_{s}$-semisimple.

Section 9 is devoted to the study of Frattini-free Banach Lie algebras --- the
Lie algebras satisfying the condition%
\[
P_{\mathfrak{S}^{\text{max}}}(\mathcal{L})=\cap_{L\in\mathfrak{S}%
_{\mathcal{L}}^{\text{max}}}L=\{0\},\text{ that is, }\mathcal{L}%
\in\mathbf{Sem}(P_{\mathfrak{S}^{\mathrm{\max}}}).
\]
In \cite{K} the first author considered Frattini-free Banach Lie algebras all
of whose maximal Lie subalgebras have codimension $1$. In this paper we
consider the general case and prove that each Frattini-free Banach Lie algebra
has the largest closed solvable Lie ideal $S$ and that this ideal has
solvability index 2, that is, $[S,S]$ is commutative. Introducing an important
notion of a finite-dimensional subsimple Lie algebra (Definition \ref{D9.1}),
we obtain a new structural description of Frattini-free Lie algebras as
subdirect products of families of finite-dimensional subsimple Lie algebras.

This structural description is very useful even for finite-dim\-en\-sional Lie
algebras. It enables us to obtain a new transparent description of each
finite-dimensional Frattini-free Lie algebra $\mathcal{L}$ as the direct sum
of at most three summands --- a semisimple Lie algebra, a commutative algebra
and a semidirect product $L\oplus^{\text{id}}X,$ where $L$ is a decomposable
Lie algebra of operators on a finite-dimensional linear space $X$ (in
Jacobson-free Lie algebras the third summand is absent). This, in turn, gives
us the description of finite-dimensional Frattini-free Lie algebras obtained
by Stitzinger \cite{S} and Towers \cite{T}. Using Marshall's results (see
\cite{M}) about the Frattini and Jacobson ideals of Lie algebras, we obtain an
inequality that relates the Frattini and Jacobson indices\textbf{
}$r_{P_{\mathfrak{S}^{\max}}}^{\circ}\left(  \mathcal{L}\right)  $ and
$r_{P_{\mathfrak{J}^{\max}}}^{\circ}\left(  \mathcal{L}\right)  $ of
$\mathcal{L}$ to the solvability index $i_{s}(\mathcal{N}_{\mathcal{L}})$ of
the nil-radical $\mathcal{N}_{\mathcal{L}}$ of $\mathcal{L}$:%
\[
i_{s}(\mathcal{N}_{\mathcal{L}})\leq r_{P_{\mathfrak{S}^{\max}}}^{\circ
}\left(  \mathcal{L}\right)  \leq r_{P_{\mathfrak{J}^{\max}}}^{\circ}\left(
\mathcal{L}\right)  \leq i_{s}(\mathcal{N}_{\mathcal{L}})+1.
\]
This allows us to partition the set of all finite-dimensional Lie algebras
into the subclasses of Lie algebras determined by the integer value of the
Frattini index.

\textbf{Acknowledgment.} We are indebted to Victor Lomonosov for a helpful discussion.

\section{\label{2}Characteristic Lie ideals and subideals of Banach\ Lie\ algebras}

Let $\mathcal{L}$ be a Banach Lie algebra. A subspace $L$ of $\mathcal{L}$ is
a \textit{Lie subalgebra }(\textit{ideal}) if $[a,b]\in L,$ for each $a,b\in
L$ (respectively, $a\in L,$ $b\in\mathcal{L}).$ A linear map $\delta$ on
$\mathcal{L}$ is a \textit{Lie derivation} if%
\begin{equation}
\delta([a,b])=[\delta(a),b]+[a,\delta(b)]\text{ for }a,b\in\mathcal{L}.
\label{1.1}%
\end{equation}
Each $a\in\mathcal{L}$ defines a bounded Lie derivation $\operatorname*{ad}%
\left(  a\right)  $ on $\mathcal{L}$: $\operatorname*{ad}\left(  a\right)  x=[a,x].$

Denote by $\mathfrak{D}\left(  \mathcal{L}\right)  $ the set of all bounded
Lie derivations on $\mathcal{L}$. It is a closed Lie subalgebra of the algebra
$\mathcal{B}\left(  \mathcal{L}\right)  $ of all bounded operators on
$\mathcal{L}$ and ad$\left(  \mathcal{L}\right)  =\left\{  \text{ad}\left(
a\right)  :a\in\mathcal{L}\right\}  $ is a Lie ideal of $\mathfrak{D}\left(
\mathcal{L}\right)  $, as $[\delta,$ad$\left(  a\right)  ]=\mathrm{ad}\left(
\delta(a)\right)  .$ If $J$ is a Lie ideal of $\mathcal{L,}$ we denote by
$\delta|_{J}$ the restriction of $\delta$ to $J$, and ad$\left(
\mathcal{L}\right)  |_{J}=\left\{  \text{ad}(a)|_{J}\text{: }a\in
\mathcal{L}\right\}  $.

A \textit{ }Lie ideal of $\mathcal{L}$ is called \textit{characteristic} if it
is invariant for all $\delta\in\mathfrak{D}\left(  \mathcal{L}\right)  .$

\begin{notation}
\emph{We write} $J\vartriangleleft\mathcal{L}$ \emph{if} $J$ \emph{is a closed
Lie ideal of a Banach Lie algebra} $\mathcal{L,}$ \emph{and }%
$J\vartriangleleft^{\emph{ch}}\mathcal{L}$ \emph{if} $J$ \emph{is a
characteristic closed Lie ideal of} $\mathcal{L}$.
\end{notation}

The center of $\mathcal{L}$ is a characteristic Lie ideal$\mathcal{.}$ If
$\mathcal{L}$ is commutative then $\{0\}$ and $\mathcal{L}$ are the only
characteristic Lie ideals of $\mathcal{L}$. Indeed, each closed subspace of
$\mathcal{L}$ is a Lie ideal, each bounded operator on $\mathcal{L}$ is a
derivation and only $\{0\}$ and $\mathcal{L}$ are invariant for $\mathcal{B}%
(\mathcal{L}).$

The following lemma shows that subspaces of $\mathcal{L}$ invariant for all
bounded Lie isomorphisms are characteristic ideals.

\begin{lemma}
\label{L1}Let $J$ be a closed linear subspace of a Banach Lie algebra
$\mathcal{L}$ invariant for all bounded Lie isomorphisms of $\mathcal{L}$.
Then $J$ is a characteristic Lie ideal of $\mathcal{L}$.
\end{lemma}

\begin{proof}
For each $\delta\in\mathfrak{D}\left(  \mathcal{L}\right)  ,$
\[
\exp(t\delta)=\sum_{i=0}^{\infty}\frac{t^{n}\delta^{n}}{n!},\text{ for }%
t\in\mathbb{R},
\]
is a one-parameter group of bounded Lie automorphisms of $\mathcal{L}$:
\[
\exp(t\delta)([a,b])=[\exp(t\delta)(a),\exp(t\delta)(b)],
\]
for all $a,b\in\mathcal{L}.$ Hence $\exp(t\delta)(J)\subseteq J$. Since
$
\delta(a)=\lim_{t\rightarrow0}(\exp(t\delta)(a)-a)/t,$ \text{for each }%
$a\in\mathcal{L},
$
$J$ is invariant for $\delta,$ so it is a characteristic Lie ideal of
$\mathcal{L}$.
\end{proof}

Clearly, the intersection and the closed linear span of a family of
characteristic Lie ideals are characteristic Lie ideals.

\begin{lemma}
\label{essl}Let $\mathcal{L}$ be a Banach Lie algebra\emph{, }let
$J\vartriangleleft^{\emph{ch}}\mathcal{L}$ and $q:\mathcal{L}\longrightarrow
\mathcal{L}/J$ be the quotient map. If $I\vartriangleleft^{\emph{ch}%
}\mathcal{L}/J$ then $q^{-1}(I)\vartriangleleft^{\emph{ch}}\mathcal{L}$.
\end{lemma}

\begin{proof}
As $J$ is a characteristic Lie ideal, $\delta(J)\subseteq J,$ for each\textbf{
}$\delta\in\mathfrak{D}\left(  \mathcal{L}\right)  .$ Hence the quotient map
$\delta^{q}$: $q(x)\rightarrow q(\delta(x))$ on $\mathcal{L}/J$ is, clearly, a
derivation of $\mathcal{L}/J$. Since $I\vartriangleleft^{\text{ch}}%
\mathcal{L}/J,$ we have $\delta^{q}(I)\subseteq I$. This means that
$\delta(q^{-1}(I))\subseteq q^{-1}(I),$ so that $q^{-1}(I)$ is a
characteristic Lie ideal of $\mathcal{L}$.
\end{proof}

If $I\!\vartriangleleft J\vartriangleleft\mathcal{L}$ then $I$ is not
necessarily a Lie ideal of $\mathcal{L.}$ For example, each subspace $I$ of a
commutative ideal $J$ of a Lie algebra $\mathcal{L}$ is not necessarily a Lie
ideal of $\mathcal{L}$ (e.g. subspaces of a Banach space $X$ in the semidirect
product $\mathcal{L}=\mathcal{B}(X)\oplus^{\text{id}}X$ (see (\ref{sem})) are
not Lie ideals of $\mathcal{L)}$.

Statements (i) and (ii) in the following lemma are related to Lemma 0.4
\cite{St}, and (iii) belongs to the mathematical folklore; for the sake of
completeness we present their proofs.

\begin{lemma}
\label{L3.1}\emph{(i) }If $I\vartriangleleft^{\emph{ch}}J\vartriangleleft
\mathcal{L}$ then $I\vartriangleleft\mathcal{L.}$

\begin{itemize}
\item [$\mathrm{(ii)}$]\emph{ }If $I\vartriangleleft^{\emph{ch}}%
J\vartriangleleft^{\emph{ch}}\mathcal{L}$ then $I\vartriangleleft^{\emph{ch}%
}\mathcal{L}$.

\item[$\mathrm{(iii)}$] \emph{ }If $J\vartriangleleft\mathcal{L}$ and
$J=\overline{[J,J]}$ then $J\vartriangleleft^{\emph{ch}}\mathcal{L}$.
\end{itemize}
\end{lemma}

\begin{proof}
(i) As $J\vartriangleleft\mathcal{L},$ ad$\left(  \mathcal{L}\right)  |_{J}$
is a Lie subalgebra of $\mathfrak{D}\left(  J\right)  $. Hence $I$ is
invariant for ad$\left(  \mathcal{L}\right)  |_{J}.$ Thus $I$ is a Lie ideal
of $\mathcal{L}$.

(ii) We have $\delta(J)\subseteq J$ and $\delta|_{J}\in\mathfrak{D}\left(
J\right)  $, for all $\delta\in\mathfrak{D}\left(  \mathcal{L}\right)  $, and
$\Delta(I)\subseteq I$ for all $\Delta\in\mathfrak{D}\left(  J\right)  $.
Hence $\delta(I)\subseteq I$ for all $\delta\in\mathfrak{D}\left(
\mathcal{L}\right)  $. Thus $I$ is a characteristic Lie ideal of $\mathcal{L}$.

(iii) By (\ref{1.1}), for each $\delta\in\mathfrak{D}\left(  \mathcal{L}%
\right)  $, we have $\delta([J,J])\subseteq\lbrack\delta(J),J]+[J,\delta
(J)]\subseteq\lbrack J,\mathcal{L}]\subseteq J.$ As $\delta$ is bounded,
$\delta(J)=\delta(\overline{[J,J]})\subseteq J.$ Hence $J$ is a characteristic
Lie ideal of $\mathcal{L}$.
\end{proof}

The existence of Lie ideals and characteristic Lie ideals of finite
codimension was studied in \cite{KST1,KST2}. We will often use the following
result obtained in \cite{KST2}.

\begin{theorem}
\label{KST1}\emph{\cite{KST2} }Let a Banach Lie algebra $\mathcal{L}$ have a
closed proper Lie subalgebra $\mathcal{L}_{0}$ of finite codimension$.$ Then
$\mathcal{L}$ has a closed proper Lie ideal of finite codimension. In addition$,$

\begin{itemize}
\item [$\mathrm{(i)}$]If $\mathcal{L}_{0}$ is maximal$,$ then $\mathcal{L}%
_{0}$ contains a closed Lie ideal of $\mathcal{L}$ of finite codimension.

\item[$\mathrm{(ii)}$] If $\mathcal{L}$ is non-commutative$,$ it has a proper
closed characteristic Lie ideal of finite codimension.
\end{itemize}
\end{theorem}

\begin{corollary}
\label{C3.1}Let $\mathcal{L}$ be a Banach Lie algebra and $J$ be a
non-commutative infinite-dimensional closed Lie ideal of $\mathcal{L}$. If $J$
has a proper closed Lie subalgebra of finite codimension\emph{,} then $J$
contains a closed Lie ideal $I$ of $\mathcal{L}$ that has non-zero finite
codimension in $J$.

If\emph{, }in addition\emph{,} $J$ is a characteristic Lie ideal of
$\mathcal{L}$, then $I$ is also a characteristic Lie ideal.
\end{corollary}

\begin{proof}
By Theorem \ref{KST1}(ii), $J$ has a proper closed characteristic Lie ideal
$I$ of finite codimension. By Lemma \ref{L3.1}, $I$ is a Lie ideal of
$\mathcal{L}$; if $J$ is characteristic then $I$ is also characteristic.
\end{proof}

\begin{definition}
\label{subideal}A Lie subalgebra $I$ of a Banach Lie algebra $\mathcal{L}$ is
called a \emph{Lie subideal }$($more precisely $n$\emph{-subideal}$)$,
if there are Lie subalgebras $J_{1},$...$,J_{n}$ of $\mathcal{L}$ such that
$J_{0}:=I\subseteq J_{1}\subseteq\cdots\subseteq J_{n}=\mathcal{L}$ and
each $J_{i}$ is a Lie ideal of $J_{i+1}$. We write
$I\vartriangleleft\!\!\!\vartriangleleft\mathcal{L}$ if $I$ is closed. In
this case all $J_{i}$ can be chosen closed \emph{(}otherwise, replace all
$J_{i}$ by their closures\emph{)}.
\end{definition}

In some important cases Lie subideals are automatically ideals. Recall that a
finite-dimensional Lie algebra is \textit{semisimple}, if it has no non-zero
commutative Lie ideals.

\begin{lemma}
\label{sem-sub}Let $L\!\vartriangleleft\!\!\!\vartriangleleft\mathcal{L.}$ If
$L$ is a finite-dimensional semisimple Lie algebra$\mathcal{,}$ then it is a
Lie ideal of $\mathcal{L}$.
\end{lemma}

\begin{proof}
Let $L=J_{0}\vartriangleleft J_{1}\vartriangleleft\cdots\vartriangleleft
J_{n}=\mathcal{L.}$ Since $L$ is semisimple, then it is well known that
$[L,L]=L.$ Hence, by Lemma \ref{L3.1}(iii), $L\vartriangleleft^{\text{ch}%
}J_{1}.$ Therefore, by Lemma \ref{L3.1}(i), $L$ is a Lie ideal of $J_{2}$.
Repeating the argument, we obtain that $L$ is a Lie ideal of $\mathcal{L}$.
\end{proof}

\begin{corollary}
\label{E5.1}Each Lie subideal of a finite-dimensional semisimple Lie algebra
is a Lie ideal.
\end{corollary}

\begin{proof}
Let $L=J_{0}\vartriangleleft J_{1}\vartriangleleft\cdots\vartriangleleft
J_{n}=\mathcal{L.}$ As $\mathcal{L}$ is semisimple, each Lie ideal of
$\mathcal{L}$ is a semisimple Lie algebra. Hence $L$ is semisimple. By Lemma
\ref{sem-sub}, it is a Lie ideal of $\mathcal{L}$.
\end{proof}

\section{Preradicals\label{3}}

\subsection{Basic properties}

Recall that $\mathfrak{L}$ denotes the class of all Banach Lie algebras and
that the symbol $J\vartriangleleft^{\text{ch}}\mathcal{L}$ means that $J$ is a
closed characteristic ideal of $\mathcal{L}$.

Now we will define a notion which plays the central role in this paper.

\begin{definition}
\label{D3.2}A map $R$ on $\mathfrak{L}$ that sends each $\mathcal{L}%
\in\mathfrak{L}$ into a closed Lie ideal $R\left(  \mathcal{L}\right)  $ of
$\mathcal{L}$ is a topological \textit{preradical} in $\mathbf{L}$ $($in
$\overline{\mathbf{L}})$ if$,$ for each morphism\emph{ }$f:$ $\mathcal{L}%
\longrightarrow\mathcal{M}$ in $\mathbf{L}$ $($in $\overline{\mathbf{L}}),$ we
have%
\begin{equation}
f(R\left(  \mathcal{L}\right)  )\subseteq R\left(  \mathcal{M}\right)  .
\label{4.0}%
\end{equation}
\end{definition}

\begin{remark}
\emph{We will omit the word ``topological''\ in all notions of the radical
theory, because we do not consider here the radical theory in the purely
algebraic setting.}
\end{remark}

For example, the map $R$: $\mathcal{L}\longmapsto\overline{[\mathcal{L,L]}},$ for
all $\mathcal{L}\in\mathfrak{L},$ is a preradical.

If $R$ is a preradical then it follows from (\ref{4.0}) that%
\begin{align}
\text{if }f\text{}  &  \text{: }\mathcal{L}\longrightarrow\mathcal{M}\text{ is
a bounded Lie isomorphism, then}\nonumber\\
\text{ }f\text{}  &  \text{: }R\left(  \mathcal{L}\right)  \longrightarrow
R\left(  \mathcal{M}\right)  \text{ is also a bounded Lie isomorphism.}
\label{3.1}%
\end{align}

\begin{corollary}
\label{C1}Let $I\vartriangleleft\mathcal{L}\in\mathfrak{L}$ and $q:\mathcal{L}%
\longrightarrow\mathcal{L}/I$ be the quotient map. For each preradical $R,$

\begin{itemize}
\item [$\mathrm{(i)}$]$R\left(  \mathcal{L}\right)  \vartriangleleft
^{\emph{ch}}\mathcal{L}$.

\item[$\mathrm{(ii)}$] $R\left(  I\right)  \vartriangleleft\mathcal{L}$ and
$q^{-1}\left(  R\left(  \mathcal{L}/I\right)  \right)  \vartriangleleft
\mathcal{L.}$

\item[$\mathrm{(iii)}$] If $I\vartriangleleft^{\emph{ch}}\mathcal{L}$ then
$R\left(  I\right)  \vartriangleleft^{\emph{ch}}\mathcal{L}$ and
$q^{-1}\left(  R\left(  \mathcal{L}/I\right)  \right)  \vartriangleleft
^{\emph{ch}}\mathcal{L}$.
\end{itemize}
\end{corollary}

\begin{proof}
Part (i) follows from Lemma \ref{L1} and (\ref{3.1}).

(ii) By (i), $R(I)$ is a characteristic Lie ideal of $I.$ Hence,\textbf{ }by
Lemma \ref{L3.1}(i), $R\left(  I\right)  \vartriangleleft\mathcal{L}$. As
$R(\mathcal{L}/I)\vartriangleleft\mathcal{L}/I,$ we have $q^{-1}\left(
R\left(  \mathcal{L}/I\right)  \right)  \vartriangleleft\mathcal{L}$.

(iii)\textbf{ }Let $I\vartriangleleft^{\text{ch}}\mathcal{L}$. Then, by (i),
$R(I)\vartriangleleft^{\text{ch}}I.$ Hence, by Lemma \ref{L3.1}(ii),
$R(I)\vartriangleleft^{\text{ch}}\mathcal{L}$. By (i), $R\left(
\mathcal{L}/I\right)  \vartriangleleft^{\text{ch}}\mathcal{L}/I$. Hence, by
Lemma \ref{essl}, 
\[
q^{-1}\left(  R\left(  \mathcal{L}/I\right)  \right)
\vartriangleleft^{\text{ch}}\mathcal{L}.
\]
\end{proof}

We are interested in preradicals with some additional algebraic properties:
$R$ is called%
\begin{align}
\text{\textit{lower stable }if }R\left(  R\left(  \mathcal{L}\right)
\right)   &  =R\left(  \mathcal{L}\right)  \text{ for all }\mathcal{L}%
\in\mathfrak{L;}\label{4.1'}\\
\text{\textit{upper stable }if }R\left(  \mathcal{L}/R\left(  \mathcal{L}%
\right)  \right)   &  =\{0\}\text{ for all }\mathcal{L}\in\mathfrak
{L;}\label{4.2'}\\
\text{\textit{balanced} if}\mathit{\ }R\left(  I\right)   &  \subseteq
R\left(  \mathcal{L}\right)  \text{ for all }I\vartriangleleft\mathcal{L}%
\in\mathfrak{\mathfrak{L};}\label{4.3'}\\
\text{\textit{hereditary } if}\mathit{\ }R\left(  I\right)   &  =I\cap
R\left(  \mathcal{L}\right)  \text{ for all }I\vartriangleleft\mathcal{L}%
\in\mathfrak{\mathfrak{L}.} \label{4.4'}%
\end{align}

\begin{definition}
\label{D3.1}A preradical is called

\begin{itemize}
\item [$\mathrm{(i)}$]an\emph{ under radical}\textit{ }if it is lower stable
and balanced.

\item[(ii)] an \emph{over radical}\textit{ }if it is upper stable and balanced.

\item[$\mathrm{(iii)}$] a \textit{\emph{radical} }if it is lower stable, upper
stable and balanced.
\end{itemize}
\end{definition}

For example, the maps $R_{0}$: $\mathcal{L}\longmapsto\{0\}$ and $R_{1}$:
$\mathcal{L}\longmapsto\mathcal{L}$, for all $\mathcal{L}\in\mathfrak{L,}$ are radicals.

\begin{remark}
\emph{The statement ``}$I\vartriangleleft\mathcal{L}$ \emph{implies} $R\left(
I\right)  \vartriangleleft\mathcal{L}$\emph{''}\ \emph{proved in Corollary
\ref{C1}(ii) is not generally true for associative algebras. So it was
included as a separate condition in the definition of the topological radical
in \cite{D}}.
\end{remark}

Let $R$ be a preradical. A Banach Lie algebra $\mathcal{L}$ is called%
\begin{equation}
1)\text{ }R\text{-\textit{semisimple }if }R\left(  \mathcal{L}\right)
=\{0\},\text{ }2)\text{ }R\text{-\textit{radical }if }R\left(  \mathcal{L}%
\right)  =\mathcal{L}. \label{1sr}%
\end{equation}
Set\textbf{ Sem}$\left(  R\right)  =\{\mathcal{L}\in\mathfrak{L}$: $R\left(
\mathcal{L}\right)  =\{0\}\}$ and \textbf{Rad}$\left(  R\right)
=\{\mathcal{L}\in\mathfrak{L}$: $R\left(  \mathcal{L}\right)  =\mathcal{L}\}$.

\begin{lemma}
\label{L-sem}Let $R$ be a preradical\emph{, }let $I\vartriangleleft
\mathcal{L}$ and let $q:$ $\mathcal{L}\longrightarrow\mathcal{L}/I$ be the
quotient map.

\begin{itemize}
\item [$\mathrm{(i)}$]If $\mathcal{L}\in\mathbf{Rad}\left(  R\right)  $ then
$q(\mathcal{L})\in\mathbf{Rad}\left(  R\right)  .$

\item[$\mathrm{(ii)}$] If $q(\mathcal{L})\in\mathbf{Sem}\left(  R\right)  $
then $R(\mathcal{L})\subseteq I.$

\item[$\mathrm{(iii)}$] Let $R$ be balanced. If $\mathcal{L}\in\mathbf{Sem}%
\left(  R\right)  $ then $I\in\mathbf{Sem}\left(  R\right)$.

\item[$\mathrm{(iv)}$] Let $R$ be balanced and upper stable. If $I$
and $q(\mathcal{L})$ belong to $\mathbf{Rad}\left(
R\right)  $ then $\mathcal{L}\in\mathbf{Rad}\left(  R\right).$

\item[$\mathrm{(v)}$] Let $R$ be balanced and lower stable. If $I$
and $q(\mathcal{L})$ belong to $\mathbf{Sem}\left(
R\right)  $ then $\mathcal{L}\in\mathbf{Sem}\left(  R\right).$
\end{itemize}
\end{lemma}

\begin{proof}
(i) As $R(\mathcal{L})=\mathcal{L}$, we have $q(\mathcal{L})=q\left(  R\left(
\mathcal{L}\right)  \right)  \overset{(\ref{4.0})}{\subseteq}R\left(
q(\mathcal{L})\right)  \subseteq q(\mathcal{L}).$ Hence $q(\mathcal{L}%
)=R\left(  q(\mathcal{L})\right)  $.

(ii) We have $q\left(  R\left(  \mathcal{L}\right)  \right)  \overset
{(\ref{4.0})}{\subseteq}R\left(  q(\mathcal{L})\right)  =\{0\}.$ Hence
$R(\mathcal{L})\subseteq I.$

(iii) If $I\vartriangleleft\mathcal{L}$ then $R\left(  I\right)  \subseteq
R\left(  \mathcal{L}\right)  =\{0\}$.

(iv) As $R$ is balanced and $I\in\mathbf{Rad}\left(  R\right)  $, we
have\textbf{ }$I=R(I)\subseteq R\left(  \mathcal{L}\right)  $. Hence there is
a quotient map $p$: $\mathcal{L}/I\rightarrow\mathcal{L}/R\left(
\mathcal{L}\right)  $. As $R$ is upper stable and $\mathcal{L}/I\in
\mathbf{Rad}\left(  R\right)  $,%
\[
\mathcal{L}/R(\mathcal{L})=p(\mathcal{L}/I)=p(R(\mathcal{L}/I))\subseteq
R(p(\mathcal{L}/I))=R(\mathcal{L}/R(L))=\{0\}.
\]
Thus $\mathcal{L}=R(\mathcal{L}).$

(v) It follows from (ii) that $R(\mathcal{L})\subseteq I$. Then $R(\mathcal{L}%
)\vartriangleleft I$. As $R$ is balanced, $R\left(  R(\mathcal{L})\right)
\subseteq R\left(  I\right)  =\{0\}$. As $R$ is lower stable, $R(\mathcal{L}%
)=R\left(  R(\mathcal{L})\right)  =\{0\}$.
\end{proof}

In particular, it follows from Lemma \ref{L-sem}(iv) and (v) that if $R$ is a
radical, then both classes \textbf{Sem}$\left(  R\right)  $ and \textbf{Rad}%
$\left(  R\right)  $ are closed under extensions.

There is a natural order in the class of all preradicals. If $R$ and $T$ are
preradicals, we write%
\begin{equation}
T\leq R,\text{ if }T\left(  \mathcal{L}\right)  \subseteq R\left(
\mathcal{L}\right)  \text{ for all }\mathcal{L}\in\mathfrak{L}. \label{4}%
\end{equation}
We write $T<R$, if $T\leq R$ and there is a Banach Lie algebra $\mathcal{L}$
such that $T\left(  \mathcal{L}\right)  \neq R\left(  \mathcal{L}\right)  $.

If $T\leq R$ then \textbf{Sem}$\left(  R\right)  \subseteq\mathbf{Sem}\left(
T\right)  $ and $\mathbf{Rad}\left(  T\right)  \subseteq\mathbf{Rad}\left(
R\right)  $. Conversely, the following result shows that in many cases the
order is determined by these inclusions.

\begin{proposition}
\label{P3.7}Let $T,R$ be preradicals.

\begin{itemize}
\item [$\mathrm{(i)}$]If $T$ is lower stable and $R$ is balanced then
$\mathbf{Rad}\left(  T\right)  \subseteq\mathbf{Rad}\left(  R\right)  $
implies $T\leq R$.

\item[$\mathrm{(ii)}$] If $T$ and $R$ are under radicals then $\mathbf{Rad}%
\left(  T\right)  =\mathbf{Rad}\left(  R\right)  $ if and only if $T=R$.

\item[$\mathrm{(iii)}$] If $R$ is upper stable then $\mathbf{Sem}\left(
R\right)  \subseteq\mathbf{Sem}\left(  T\right)  $ implies $T\leq R$.

\item[$\mathrm{(iv)}$] If $T$ and $R$ are upper stable then $\mathbf{Sem}%
\left(  T\right)  =\mathbf{Sem}\left(  R\right)  $ if and only if $T=R$.

\item[\textrm{(v)}] Let $T\leq R,$ $T$ be balanced and $I\vartriangleleft
\mathcal{L}.$ If $T(I)=I$ and $R(\mathcal{L}/I)=\{0\}$ then $T(\mathcal{L}%
)=R(\mathcal{L})=I.$
\end{itemize}
\end{proposition}

\begin{proof}
(i) As $T$ is lower stable, $T(\mathcal{L})\in\mathbf{Rad}\left(  T\right)  $
for each $\mathcal{L}\in\mathfrak{L}$. Hence $T(\mathcal{L})\in\mathbf{Rad}%
\left(  R\right)  .$ Then $T\left(  \mathcal{L}\right)  =R\left(  T\left(
\mathcal{L}\right)  \right)  $. Since $R$ is balanced and $T(\mathcal{L}%
)\vartriangleleft\mathcal{L}$, we have\textbf{ }$T(\mathcal{L})=R\left(
T\left(  \mathcal{L}\right)  \right)  \subseteq R(\mathcal{L})$.

(iii) As $R$ is upper stable, $\mathcal{L}/R(\mathcal{L})\in\mathbf{Sem}%
\left(  R\right)  $ for each $\mathcal{L}\in\mathfrak{L.}$ Hence
$\mathcal{L}/R(\mathcal{L})\in\mathbf{Sem}\left(  T\right)  $. By Lemma
\ref{L-sem}(ii), $T(\mathcal{L})\subseteq R(\mathcal{L})$. Part (iii) is proved.

Part (ii) follows from (i), and (iv) from (iii).

(v) As $R(\mathcal{L}/I)=\{0\},$ we have from Lemma \ref{L-sem}(ii) that
$R(\mathcal{L})\subseteq I.$ As $T$ is balanced,%
\[
I=T(I)\subseteq T(\mathcal{L})\subseteq R(\mathcal{L})\subseteq I.
\]
\end{proof}

\begin{corollary}
\label{Crad}\emph{(i) }If $R$ is a radical then $\mathcal{L}/R(\mathcal{L})\in
\mathbf{Sem}\left(R\right)$ and $R(\mathcal{L})\in
\mathbf{Rad}\left(R\right)$   for each $\mathcal{L}\in\mathfrak{L.}$
Moreover, $R(\mathcal{L})$ contains each $R$-radical Lie ideal of
$\mathcal{L}.$

\emph{(ii) }Let $T$ and $R$ be radicals. Then%
\[
T=R\,\,\,\,\Longleftrightarrow\,\,\,\,\mathbf{Rad}\left(  T\right)
=\mathbf{Rad}\left(  R\right)  \,\,\,\,\Longleftrightarrow\,\,\,\,\mathbf{Sem}%
\left(  T\right)  =\mathbf{Sem}\left(  R\right)  .
\]
\end{corollary}

\begin{proof}
We only need to prove that $R(\mathcal{L})$ contains each $R$-radical Lie
ideal $I$ of $\mathcal{L.}$ Indeed, as $R$ is balanced, $I=R(I)\subseteq
R(\mathcal{L)}$.
\end{proof}

\begin{definition}
\label{D3.3}Let $R$ be a preradical. A closed Lie ideal $I$ of a Banach Lie
algebra $\mathcal{L}$ is called $R$\textit{-absorbing }if $\mathcal{L}/I$ is
$R$-semisimple. \emph{Abs}$_{R}\left(  \mathcal{L}\right)  $ denotes the set
of all $R$-absorbing ideals of $\mathcal{L}$.
\end{definition}

The following useful result was proved in \cite[Theorem 2.11]{ST-I} for
radicals in normed associative algebras. We will just check that the proof
also works for Banach Lie algebras.

\begin{theorem}
\label{primgen}Let $R$ be a preradical and $\mathcal{L}$ be a Banach Lie
algebra. Then
\begin{enumerate}
 \item[(i)]
the intersection of any family of $R$-absorbing Lie ideals of
$\mathcal{L}$ is $R$-absorbing;
\item[(ii)] 
each $R$-absorbing Lie ideal of $\mathcal{L}$ contains
$R(\mathcal{L)};$
\item[(iii)] 
 if $R$ is an upper stable then $R\left(\mathcal{L}\right)$ is
the smallest $R$-absorbing ideal of $\mathcal{L}$.
\end{enumerate}
\end{theorem}

\begin{proof}
(i) Let $\left\{  J_{\lambda}\right\}  $ be a family of $R$-absorbing ideals
of $\mathcal{L}$ and $J=\cap J_{\lambda}.$ Since $J\subseteq J_{\lambda}$,
there is a bounded epimorphism $p_{\lambda}:\mathcal{L}/J\longrightarrow
\mathcal{L}/J_{\lambda}$ with $q_{\lambda}=p_{\lambda}q$, where $q_{\lambda
}:\mathcal{L}\longrightarrow\mathcal{L}/J_{\lambda}$ and $q:\mathcal{L}%
\longrightarrow\mathcal{L}/J$ are quotient maps. Therefore
\[
p_{\lambda}\left(  R\left(  \mathcal{L}/J\right)  \right)  \overset
{(\ref{4.0})}{\subseteq}R\left(  \mathcal{L}/J_{\lambda}\right)  =\{0\},
\]
so that $R\left(  \mathcal{L}/J\right)  \subseteq J_{\lambda}/J$ for every
$\lambda$. Then $q^{-1}\left(  R\left(  \mathcal{L}/J\right)  \right)
\subseteq\cap J_{\lambda}=$ $J$, whence $R\left(  \mathcal{L}/J\right)
=\{0\}$.

Part (ii) follows from Lemma \ref{L-sem}(ii).

(iii) If $R$ is upper stable then, by (\ref{4.2'}), $R\left(  \mathcal{L}%
\right)  \in\mathrm{Abs}_{R}\left(  \mathcal{L}\right)  $. So $R\left(
\mathcal{L}\right)  $ is the smallest $R$-absorbing ideal of $\mathcal{L}$.
\end{proof}

Note that in general not every ideal containing $R(\mathcal{L})$ is $R$-absorbing.

\subsection{Preradicals of direct and semidirect products}

Many examples below will be based on the following well known construction
(see \cite[Sec 1.8]{Bo}).

Let $L_{1},$ $L_{0}$ be Banach Lie algebras and $\varphi$ be a bounded Lie
homomorphism from $L_{1}$ into $\mathfrak{D}\left(  L_{0}\right)  $. Endowing
their direct Banach space sum $L_{1}\dotplus L_{0}$ with Lie multiplication
given by%
\begin{equation}
\left[  (a;x),(b;y)\right]  =\left(  [a,b];\varphi\left(  a\right)
y-\varphi\left(  b\right)  x+\left[  x,y\right]  \right)  , \label{fsemi}%
\end{equation}
for $a,b\in L_{1},$ $x,y\in L_{0},$ we get the \textit{semidirect product}
$\mathcal{L}=L_{1}\oplus^{\varphi}L_{0}.$ It is a Lie algebra. Moreover, it is
a Banach Lie algebra with norm $\left\|  (a;x)\right\|  =\max\left\{  \left\|
a\right\|  ,\left\|  x\right\|  \right\}  $ and the multiplication constant
$t_{\mathcal{L}}=\max\left\{  t_{L_{1}},2\left\|  \varphi\right\|  +t_{L_{0}%
}\right\}  .$ Identify $\{0\}\oplus^{\varphi}L_{0}$ and $L_{0}.$ Then
$L_{0}\vartriangleleft\mathcal{L}$ and $\mathcal{L}/L_{0}$ is isomorphic to
$L_{1}$.

If $\varphi=0,$ we obtain the direct product $L_{1}\oplus L_{0}.$

If $L_{1}$ is a Lie subalgebra of $\mathcal{B}\left(  L_{0}\right)  $ then we
take $\varphi=$ id and write $L_{1}\oplus^{\text{id}}L_{0}$.

Let $L_{0}$ be commutative. Denote $X=L_{0}.$ Then $X$ is a Banach space and
$\mathfrak{D}\left(  X\right)  =\mathcal{B}\left(  X\right)  $, as (\ref{1.1})
holds for all $x,y\in X$ and $T\in\mathcal{B}(X).$ Let us identify $L_{1}$
with the Lie subalgebra $\varphi(L_{1})$ of $\mathcal{B}(X)$ and write $ax$
instead of $\varphi(a)x,$ for $a\in L_{1}$ and $x\in X.$ Then the above
construction gives us the semidirect product $\mathcal{L}=L_{1}\oplus
^{\text{id}}X$ with binary operation%
\begin{equation}
\lbrack(a;x),(b;y)]=([a,b];ay-bx)\text{ for }a,b\in L_{1}\text{ and }x,y\in X.
\label{sem}%
\end{equation}
Let $M$ be a closed Lie subalgebra of $L_{1}$ and $Y$ be a closed subspace of
$X$ invariant for all operators in $M$. Then $M\oplus^{\text{id}}Y$ can be
identified with the closed subalgebra of $\mathcal{L}$ consisting of all pairs
$(a;x)$ with $a\in M$ and $x\in Y$.

Consider now the behavior of the semidirect product with respect to preradicals.

\begin{proposition}
\label{semi}Let $\mathcal{L}=L_{1}\oplus^{\varphi}L_{0}$ and let $R$ be a
preradical. Then

\begin{itemize}
\item [$\mathrm{(i)}$]$R\left(  \mathcal{L}\right)  \subseteq R\left(
L_{1}\right)  \oplus^{\varphi}L_{0}$.

\item[$\mathrm{(ii)}$] Let $R$ be balanced$.$ Then $R\left(  R\left(
L_{1}\right)  \oplus^{\varphi}L_{0}\right)  \subseteq R\left(  \mathcal{L}%
\right)  $ and

\begin{itemize}
\item [$1)$]if $\varphi=0,$ so that $\mathcal{L}=L_{1}\oplus L_{0},$ then
$R(\mathcal{L})=R(L_{1})\oplus R(L_{0}).$

\item[$2)$] if $R$ is upper stable and $L_{0},L_{1}\in\mathbf{Rad}(R),$ then
$\mathcal{L}\in\mathbf{Rad}(R).$

\item[$3)$] if $L_{1}\in\mathbf{Sem}(R)$ then $R^{2}\left(  \mathcal{L}%
\right)  \subseteq R\left(  L_{0}\right)  \subseteq R\left(  \mathcal{L}%
\right)  \subseteq L_{0}$. If $R$ is also lower stable then $R\left(
\mathcal{L}\right)  =R\left(  L_{0}\right)  $.
\end{itemize}
\end{itemize}
\end{proposition}

\begin{proof}
(i) The map $f$: $\mathcal{L}\longrightarrow L_{1}$ defined by $f\left(
\left(  a;x\right)  \right)  =a,$ for all $\left(  a;x\right)  \in
\mathcal{L},$ is a homomorphism from $\mathcal{L}$ onto $L_{1}$. As $R$ is a
preradical, $f\left(  R\left(  \mathcal{L}\right)  \right)  \subseteq R\left(
L_{1}\right)  $. Thus $R\left(  \mathcal{L}\right)  \subseteq R\left(
L_{1}\right)  \oplus^{\varphi}L_{0}$.

(ii) Let $R$ be balanced. As $R\left(  L_{1}\right)  \oplus^{\varphi}L_{0}$ is
a closed Lie ideal of $\mathcal{L}$, we have $R\left(  R\left(  L_{1}\right)
\oplus^{\varphi}L_{0}\right)  \subseteq R\left(  \mathcal{L}\right)  $.

1) If $\varphi=0$ then, by (i), $R(\mathcal{L})\subseteq R(L_{1})\oplus L_{0}$
and $R(\mathcal{L})\subseteq L_{1}\oplus R(L_{0}).$ Hence $R(\mathcal{L}%
)\subseteq R(L_{1})\oplus R(L_{0}).$ As $L_{1}$ and $L_{0}$ are closed Lie
ideals of $\mathcal{L}$, we have $R(L_{1})\subseteq R(\mathcal{L)}$ and
$R(L_{0})\subseteq R(\mathcal{L)}$. Hence $R(\mathcal{L})=R(L_{1})\oplus
R(L_{0}).$

Part 2) follows from Lemma \ref{L-sem}(iv).

3) As $R$ is balanced and $R\left(  L_{1}\right)  =0$, (i) implies $R\left(
L_{0}\right)  \subseteq R\left(  \mathcal{L}\right)  \subseteq L_{0}$. As
$R\left(  \mathcal{L}\right)  \vartriangleleft L_{0}$ and $R$ is balanced, we
have $R^{2}\left(  \mathcal{L}\right)  \subseteq R\left(  L_{0}\right)  $. If,
in addition, $R$ is lower stable then $R^{2}\left(  \mathcal{L}\right)
=R\left(  \mathcal{L}\right)  \mathcal{\ }$implies $R\left(  \mathcal{L}%
\right)  =R\left(  L_{0}\right)  $.
\end{proof}

In particular, if $R$ is a radical then a semidirect product of $R$-radical
algebras is $R$-radical.

We will define now the direct product of an arbitrary family of Banach Lie algebras.

\begin{definition}
\label{D3-dir}Let $\left\{  \mathcal{L}_{\lambda}\right\}  _{\lambda\in
\Lambda}$ be Banach Lie algebras with the multiplication constants
$t_{\lambda}$ satisfying $t_{\Lambda}=\sup\{t_{\lambda}\}<\infty.$ The Banach
Lie algebra%
\begin{align}
\mathcal{L}  &  =\oplus_{\Lambda}\mathcal{L}_{\lambda}=\{a=\left(  a_{\lambda
}\right)  _{\lambda\in\Lambda}\text{\emph{:} }a_{\lambda}\in\mathcal{L}%
_{\lambda}\text{ and }\nonumber\\
\left\|  a\right\|   &  =\sup\{\left\|  a_{\lambda}\right\|  _{\mathcal{L}%
_{\lambda}}\text{\emph{:} }\lambda\in\Lambda\}<\infty\} \label{e3.1}%
\end{align}
with coordinate-wise operations and the multiplication constant $t_{\Lambda}$
is called the \emph{normed direct product.}

Identify each $\mathcal{L}_{\lambda}$ with $\left\{  \left(  a_{\mu}\right)
_{\mu\in\Lambda}\in\oplus_{\Lambda}\mathcal{L}_{\mu}\text{\emph{:} }a_{\mu
}=0\text{ for }\mu\neq\lambda\right\}  .$ The closed Lie ideal $\widehat
{\mathcal{L}}=\widehat{\oplus}_{\Lambda}\mathcal{L}_{\lambda}$ of
$\mathcal{L}$ generated by all Lie ideals $\mathcal{L}_{\lambda}$ is called
\emph{the }$c_{0}$\emph{-direct product}.
\end{definition}

If $\Lambda=\mathbb{N}$ then $\widehat{\mathcal{L}}=\widehat{\oplus
}_{\mathbb{N}}\mathcal{L}_{n}=\{\left(  a_{n}\right)  _{n\in\mathbb{N}}%
\in\mathcal{L}$: $\left\|  a_{n}\right\|  _{\mathcal{L}_{n}}\rightarrow0$ as
$n\rightarrow\infty\}.$

\begin{proposition}
\label{P3.1n}Let $\mathcal{L}=\oplus_{\Lambda}\mathcal{L}_{\lambda}$ and
$\widehat{\mathcal{L}}=\widehat{\oplus}_{\Lambda}\mathcal{L}_{\lambda}.$ If
$R$ is a balanced preradical then%
\[
R\left(  \widehat{\mathcal{L}}\right)  =\widehat{\oplus}_{\Lambda}R\left(
\mathcal{L}_{\lambda}\right)  \subseteq R\left(  \mathcal{L}\right)
\subseteq\oplus_{\Lambda}R\left(  \mathcal{L}_{\lambda}\right)  .
\]
In particular$,$ $\mathcal{L}\in\mathbf{Sem}(R)$ if and only if all
$\mathcal{L}_{\lambda}\in\mathbf{Sem}(R).$
\end{proposition}

\begin{proof}
Let $\mathcal{N}_{\mu}=\left\{  \left(  a_{\lambda}\right)  _{\lambda
\in\Lambda}\in\oplus_{\Lambda}\mathcal{L}_{\lambda}:a_{\mu}=0\right\}  $. Then
$\mathcal{L}=\mathcal{N}_{\mu}\oplus\mathcal{L}_{\mu}$ for each $\mu\in
\Lambda.$ By Proposition \ref{semi}(ii) 1), $R\left(  \mathcal{L}\right)
=R(\mathcal{N}_{\mu})\oplus R(\mathcal{L}_{\mu}).$ Hence
\[
R\left(  \mathcal{L}\right)  =\cap_{\mu\in\Lambda}(R(\mathcal{N}_{\mu})\oplus
R(\mathcal{L}_{\mu}))\subseteq\cap_{\mu\in\Lambda}(\mathcal{N}_{\mu}\oplus
R(\mathcal{L}_{\mu}))=\oplus_{\Lambda}R\left(  \mathcal{L}_{\lambda}\right)
.
\]

As all $\mathcal{L}_{\lambda}\vartriangleleft\widehat{\mathcal{L}%
}\vartriangleleft\mathcal{L}$ and $R$ is balanced, all $R\left(
\mathcal{L}_{\lambda}\right)  \subseteq R\left(  \widehat{\mathcal{L}}\right)
\subseteq R\left(  \mathcal{L}\right)  .$ Hence%
\begin{equation}
\widehat{\oplus}_{\Lambda}R\left(  \mathcal{L}_{\lambda}\right)  \subseteq
R\left(  \widehat{\mathcal{L}}\right)  \subseteq R\left(  \mathcal{L}\right)
\subseteq\oplus_{\Lambda}R\left(  \mathcal{L}_{\lambda}\right)  . \label{4.9}%
\end{equation}
Since $R\left(  \widehat{\mathcal{L}}\right)  \subseteq\widehat{\mathcal{L}}$
and $R\left(  \widehat{\mathcal{L}}\right)  \subseteq\oplus_{\Lambda}R\left(
\mathcal{L}_{\lambda}\right)  ,$ it follows that 
\[
R\left(  \widehat
{\mathcal{L}}\right)  \subseteq(\oplus_{\Lambda}R(\mathcal{L}_{\lambda}%
))\cap\widehat{\mathcal{L}}=\widehat{\oplus}_{\Lambda}R\left(  \mathcal{L}%
_{\lambda}\right)  .
\]
 Hence, by (\ref{4.9}), $R\left(  \widehat{\mathcal{L}%
}\right)  =\widehat{\oplus}_{\Lambda}R\left(  \mathcal{L}_{\lambda}\right)  .$
Together with (\ref{4.9}) this gives us the complete proof.
\end{proof}

\section{\label{section4}Construction of radicals from preradicals}

In this section we consider various\textbf{ }ways to improve preradicals, that
is, to construct from them new preradicals with additional better properties
(in particular, radicals). First we consider some operations on families of
closed subspaces.

Let $G$ be a family of closed subspaces of a Banach space $X.$ Denote by
$\sum_{Y\in G}Y$ the linear subspace of $X$ that consists of all finite sums
of elements from all $Y\in G.$ Set%
\begin{align}
\mathfrak{p}\left(  G\right)   &  =X,\text{ if }G=\varnothing,\text{ and
}\mathfrak{p}\left(  G\right)  =\bigcap\limits_{Y\in G}Y,\text{ if }%
G\neq\varnothing,\label{F0}\\
\mathfrak{s}\left(  G\right)   &  =\{0\},\text{ if }G=\varnothing,\text{ and
}\mathfrak{s}\left(  G\right)  =\overline{\sum_{Y\in G}Y},\text{ if }%
G\neq\varnothing. \label{F0'}%
\end{align}
Let $f$ be a continuous linear map from $X$ into a Banach space $Z.$ Then%
\[
f(G):=\{\overline{f(Y)}:Y\in G\}.
\]
is a family of closed subspaces in $Z.$ As $f(\sum_{Y\in G}Y)=\sum_{Y\in
G}f(Y)$ and $f$ is continuous,%
\begin{align}
f(\mathfrak{s}(G))  &  =f\left(  \overline{\sum_{Y\in G}Y}\right)
\subseteq\overline{\sum_{Y\in G}f(Y)}\subseteq\overline{\sum_{Y\in G}%
\overline{f(Y)}}=\mathfrak{s}(f(G)),\label{F1}\\
f(\mathfrak{p}(G))  &  =f(\bigcap\limits_{Y\in G}Y)\subseteq\bigcap
\limits_{Y\in G}f(Y)\subseteq\bigcap\limits_{Y\in G}\overline{f(Y)}%
=\mathfrak{p}(f(G)). \label{F2}%
\end{align}

\subsection{$R$-superposition series}

We shall now develop a Lie algebraic version of the Dixon's constructions of
radicals\textbf{ }(see \cite{D}) (in pure algebra they are known as Baer procedures).

Let $R$ be a preradical. For $\mathcal{L}\in\mathfrak{L,}$ set $R^{0}\left(
\mathcal{L}\right)  =\mathcal{L}$, $R^{1}\left(  \mathcal{L}\right)  =R\left(
\mathcal{L}\right)  ,$%
\begin{align}
R^{\alpha+1}\left(  \mathcal{L}\right)   &  =R\left(  R^{\alpha}\left(
\mathcal{L}\right)  \right)  ,\text{ for an ordinal }\alpha\text{ }\nonumber\\
\text{and }R^{\alpha}\left(  \mathcal{L}\right)   &  =\underset{\alpha
^{\prime}<\alpha}{\cap}R^{\alpha^{\prime}}\left(  \mathcal{L}\right)  ,\text{
for a limit ordinal }\alpha. \label{4.o}%
\end{align}
By Corollary \ref{C1}, this is a decreasing transfinite chain of
characteristic Lie ideals of $\mathcal{L}$. It stabilizes at some ordinal
$\beta$: $R^{\beta+1}\left(  \mathcal{L}\right)  =R^{\beta}\left(
\mathcal{L}\right)  $, where $\beta$ is bounded by an ordinal that depends on
cardinality of $\mathcal{L}$. Denote the smallest such $\beta$ by
$r_{R}^{\circ}\left(  \mathcal{L}\right)  $ and, for all $\mathcal{L}%
\in\mathfrak{L},$ set%
\begin{equation}
R^{\circ}\left(  \mathcal{L}\right)  =R^{r_{R}^{\circ}(\mathcal{L}%
)}(\mathcal{L)},\text{ so that }R(R^{\circ}(\mathcal{L}))=R^{\circ
}(\mathcal{L}). \label{r1}%
\end{equation}

\begin{lemma}
\label{cin}Let $R$ and $T$ be preradicals. If at least one of them is balanced
and $R\leq T,$ then $R^{\alpha}\leq T^{\alpha}$ for every $\alpha,$ and
$R^{\mathbf{\circ}}\leq T^{\mathbf{\circ}}$.
\end{lemma}

\begin{proof}
Follows by induction. Indeed, let $\mathcal{L}$ be a Banach Lie algebra and
$R^{\alpha}\leq T^{\alpha}$ for some $\alpha$. Since $R^{\alpha}\left(
\mathcal{L}\right)  \vartriangleleft T^{\alpha}\left(  \mathcal{L}\right)  $,
it follows that $R\left(  R^{\alpha}\left(  \mathcal{L}\right)  \right)
\subseteq R\left(  T^{\alpha}\left(  \mathcal{L}\right)  \right)  \subseteq
T^{\alpha+1}\left(  \mathcal{L}\right)  $ if $R$ is balanced, and
$R^{\alpha+1}\left(  \mathcal{L}\right)  \subseteq T\left(  R^{\alpha}\left(
\mathcal{L}\right)  \right)  \subseteq T\left(  T^{\alpha}\left(
\mathcal{L}\right)  \right)  $ if $T$ is balanced.

If $R^{\alpha^{\prime}}\left(  \mathcal{L}\right)  \subseteq T^{\alpha
^{\prime}}\left(  \mathcal{L}\right)  $ for all $\alpha^{\prime}<\alpha$, then
$R^{\alpha}\left(  \mathcal{L}\right)  \subseteq T^{\alpha}\left(
\mathcal{L}\right)  $ follows from (\ref{4.o}). Taking $\alpha=\max\left\{
r_{R}^{\circ}(\mathcal{L}),r_{T}^{\circ}(\mathcal{L})\right\}  $, we obtain
that $R^{\mathbf{\circ}}\left(  \mathcal{L}\right)  \subseteq T^{\mathbf{\circ
}}\left(  \mathcal{L}\right)  $ for every Banach Lie algebra $\mathcal{L}$.
\end{proof}

\begin{theorem}
\label{T4.1}Let $R$ be a balanced preradical. Then

\begin{itemize}
\item [$\mathrm{(i)}$]$R^{\alpha}$ is a balanced preradical for each ordinal
$\alpha.$

\item[$\mathrm{(ii)}$] $R^{\circ}$ is an under radical$,$ $\mathbf{Rad}%
(R)=\mathbf{Rad}(R^{\circ})$ and $\mathbf{Sem}(R)\subseteq\mathbf{Sem}%
(R^{\circ}).$ Moreover$,$ $R^{\circ}$ is the largest under radical smaller
than or equal to $R.$ If $R$ is lower stable then $R^{\circ}=R.$

\item[$\mathrm{(iii)}$] If $\mathcal{L}=\oplus_{\Lambda}\mathcal{L}_{\lambda}$
is the normed direct product of $\left\{  \mathcal{L}_{\lambda}\right\}
_{\Lambda}$ then 
\[
r_{R}^{\circ}\left(  \mathcal{L}\right)  \leq\max_{\Lambda
}r_{R}^{\circ}\left(  \mathcal{L}_{\lambda}\right).
\]
\end{itemize}
\end{theorem}

\begin{proof}
(i) Let $R^{\alpha}$ be a balanced preradical for some $\alpha.$ Let us show
that $R^{\alpha+1}$ is a balanced preradical. We have $f\left(  R^{\alpha
}\left(  \mathcal{L}\right)  \right)  \subseteq R^{\alpha}\left(
\mathcal{M}\right)  $ for each morphism $f$: $\mathcal{L}\longrightarrow
\mathcal{M}$. Since $f$ is a homomorphism, and since $R^{\alpha}\left(
\mathcal{L}\right)  \vartriangleleft\mathcal{L}$ and $f(\mathcal{L)}$ is dense
in $\mathcal{M}$, we have\textbf{ }$\overline{f(R^{\alpha}(\mathcal{L)}%
)}\vartriangleleft\mathcal{M}$. Hence $\overline{f(R^{\alpha}\left(
\mathcal{L}\right)  )}\vartriangleleft R^{\alpha}\left(  \mathcal{M}\right)  $
and%
\begin{align*}
f(R^{\alpha+1}\left(  \mathcal{L}\right)  )  &  =f(R\left(  R^{\alpha}\left(
\mathcal{L}\right)  \right)  )\subseteq R\left(  \overline{f\left(  R^{\alpha
}\left(  \mathcal{L}\right)  \right)  }\right) \\
&  \subseteq R(R^{\alpha}(\mathcal{M}))=R^{\alpha+1}\left(  \mathcal{M}%
\right)  .
\end{align*}
Thus $R^{\alpha+1}$ is a preradical$.$ As $R^{\alpha}$ is balanced,
$R^{\alpha}(I)\subseteq R^{\alpha}\left(  \mathcal{L}\right)  $ if
$I\vartriangleleft\mathcal{L}$. By Corollary \ref{C1}(ii), $R^{\alpha}(I)$ is
a Lie ideal of $\mathcal{L}$. Hence $R^{\alpha}(I)\vartriangleleft R^{\alpha
}\left(  \mathcal{L}\right)  $. Since $R$ is balanced, it follows that
$R^{\alpha+1}(I)=R(R^{\alpha}(I))\subseteq R(R^{\alpha}\left(  \mathcal{L}%
\right)  )=R^{\alpha+1}(\mathcal{L}).$ Thus $R^{\alpha+1}$ is balanced.

Let $\alpha$ be a limit ordinal and $R^{\alpha^{\prime}}$, $\alpha^{\prime
}<\alpha,$ be balanced preradicals$.$ For $I\vartriangleleft\mathcal{L}$,
$R^{\alpha}(I)=\underset{\alpha^{\prime}<\alpha}{\cap}R^{\alpha^{\prime}%
}\left(  I\right)  \subseteq\underset{\alpha^{\prime}<\alpha}{\cap}%
R^{\alpha^{\prime}}\left(  \mathcal{L}\right)  =R^{\alpha}(\mathcal{L})$ and,
by (\ref{F2}), for each morphism $f$: $\mathcal{L}\longrightarrow\mathcal{M}$,%
\begin{align*}
f(R^{\alpha}\left(  \mathcal{L}\right)  )  &  =f\left(  \underset
{\alpha^{\prime}<\alpha}{\cap}R^{\alpha^{\prime}}\left(  \mathcal{L}\right)
\right)  \subseteq\underset{\alpha^{\prime}<\alpha}{\cap}f(R^{\alpha^{\prime}%
}\left(  \mathcal{L}\right)  )\\
&  \subseteq\underset{\alpha^{\prime}<\alpha}{\cap}R^{\alpha^{\prime}}\left(
\mathcal{M}\right)  =R^{\alpha}\left(  \mathcal{M}\right)  .
\end{align*}
Thus $R^{\alpha}$ are balanced preradicals for all $\alpha.$

(ii) From (i) and from the definition of $R^{\circ}$ we have that $R^{\circ}$
is a balanced preradical. As $\{R^{\alpha}(\mathcal{L)}\}$ is decreasing,
$R^{\circ}(\mathcal{L})\subseteq R\left(  \mathcal{L}\right)  $ for all
$\mathcal{L}\in\mathfrak{L,}$ so that $R^{\circ}\leq R.$ From this and from
the definition of $R^{\circ}$ it follows that $R(\mathcal{L})=\mathcal{L}%
\Longleftrightarrow R^{\circ}(\mathcal{L})=\mathcal{L}.$ Thus $\mathbf{Rad}%
(R)=\mathbf{Rad}(R^{\circ}).$ If $R(\mathcal{L})=\{0\},$ it follows from the
construction that $R^{\circ}(\mathcal{L})=\{0\}.$ Hence $\mathbf{Sem}%
(R)\subseteq\mathbf{Sem}(R^{\circ}).$

By (\ref{r1}), $R^{\circ}\left(  \mathcal{L}\right)  \in\mathbf{Rad}%
(R)=\mathbf{Rad}(R^{\circ}).$ Thus $R^{\circ}$ is lower stable. Hence
$R^{\circ}$ is an under radical.

If $R$ is lower stable, then $R$ is an under radical. As $\mathbf{Rad}%
(R)=\mathbf{Rad}(R^{\circ}),$ it follows from Proposition \ref{P3.7}(ii) that
$R=R^{\circ}.$

Let $T$ be an under radical and $T\leq R$. If $\mathcal{L}\in\mathbf{Rad}%
\left(  T\right)  $ then $\mathcal{L}=T(\mathcal{L})\subseteq R(\mathcal{L}%
)\subseteq\mathcal{L}$. Hence $\mathcal{L}\in$ $\mathbf{Rad}\left(  R\right)
=\mathbf{Rad}\left(  R^{\circ}\right)  .$ Thus $\mathbf{Rad}\left(  T\right)
\subseteq$ $\mathbf{Rad}\left(  R^{\circ}\right)  .$ By Proposition
\ref{P3.7}(i), $T\leq R^{\circ}$. Part (ii) is proved. Part (iii) follows from
Proposition \ref{P3.1n}.
\end{proof}

\begin{proposition}
\label{C-order}Let $R$ be a balanced preradical and let $I\vartriangleleft
\mathcal{L}$.

\begin{itemize}
\item [$\mathrm{(i)}$]If $\mathcal{L}\in\mathbf{Sem}(R^{\circ})$ then
$I\in\mathbf{Sem}(R^{\circ})$ and $r_{R}^{\circ}\left(  I\right)  \leq
r_{R}^{\circ}\left(  \mathcal{L}\right)  $.

\item[$\mathrm{(ii)}$] If $I$ and $\mathcal{L}/I$ belong to $\mathbf{Sem}%
(R^{\circ})$ then $\mathcal{L}\in\mathbf{Sem}(R^{\circ})$ and
\[ 
r_{R}^{\circ
}\left(  \mathcal{L}\right)  \leq r_{R}^{\circ}\left(  \mathcal{L}/I\right)
+r_{R}^{\circ}\left(  I\right).
\]
\end{itemize}
\end{proposition}

\begin{proof}
The first assertions in (i) and (ii) follow from (iii) and (v) of Lemma
\ref{L-sem}, respectively.

(i) As $R$ is balanced then, by Theorem \ref{T4.1}(i), $R^{\alpha}$ is
balanced for each ordinal $\alpha.$ Let $\beta=r_{R}^{\circ}\left(
\mathcal{L}\right)  $. Then $R^{\beta}(I)\subseteq R^{\beta}(\mathcal{L}%
)=R^{\circ}(\mathcal{L})=\{0\}.$ Hence $r_{R}^{\circ}\left(  I\right)
\leq\beta.$

(ii) Let $q$: $\mathcal{L}\rightarrow\mathcal{L}/I$ be the quotient map,
$\gamma=r_{R}^{\circ}\left(  I\right)  $ and $\beta=r_{R}^{\circ}\left(
\mathcal{L}/I\right)  $. As
\[
q(R^{\beta}(\mathcal{L}))\overset{(\ref{4.0})}{\subseteq}R^{\beta
}(q(\mathcal{L}))=R^{\beta}(\mathcal{L}/I)=R^{\circ}(\mathcal{L}/I)=\{0\},
\]
we have $R^{\beta}(\mathcal{L})\subseteq I.$ Hence $R^{\gamma}(R^{\beta
}\left(  \mathcal{L}\right)  )\subseteq R^{\gamma}(I)=R^{\circ}(I)=\{0\}.$
Thus $r_{R}^{\circ}\left(  \mathcal{L}\right)  \leq\beta+\gamma$.\medskip
\end{proof}

Note that the order of ordinal summands in Proposition \ref{C-order}(ii) is
essential, since, generally speaking, $R^{\beta+\gamma}(\mathcal{L}%
)=R^{\gamma}(R^{\beta}\left(  \mathcal{L}\right)  )\neq R^{\beta}(R^{\gamma
}\left(  \mathcal{L}\right)  )=R^{\gamma+\beta}(\mathcal{L}),$ so that
$\beta+\gamma\neq\gamma+\beta.$

\subsection{$R$-convolution series}

For each preradical $R$, denote by $q_{R}$ the quotient morphism on
$\mathfrak{L}$: $q_{R}$: $\mathcal{L}\longrightarrow\mathcal{L}/R(\mathcal{L)}%
$ for all $\mathcal{L}\in\mathfrak{L.}$ Define a product $R\ast T$ of
preradicals $R,T$ on $\mathfrak{L}$ by the formula%
\begin{equation}
(R\ast T)(\mathcal{L})=q_{T}^{-1}(R(q_{T}(\mathcal{L})))\text{ for each
}\mathcal{L}\in\mathfrak{L.} \label{conv}%
\end{equation}

\begin{proposition}
\label{P1}Let $R,T$ be preradicals. Then
\begin{enumerate}
 \item[(i)] 
$R\ast T$ is a preradical and $T\leq R\ast T.$
\item[(ii)] 
If $R$ is balanced then $R\ast T$ is balanced.
\item[(iii)] 
If\emph{ }$R$ is lower stable then $R\ast T$ is lower stable.
\item[(iv)]
 If $S$ is another preradical and $S\leq T,$ then $R\ast S\leq
R\ast T$ and $S\ast R\leq T\ast R.$
\end{enumerate}
\end{proposition}

\begin{proof}
(i) By the definition, for each $\mathcal{L}\in\mathfrak{L},$ we have that
$R(q_{T}(\mathcal{L}))$ is a closed Lie ideal of $q_{T}(\mathcal{L}).$ As
$q_{T}$ is a bounded epimorphism, $q_{T}^{-1}(R(q_{T}(\mathcal{L})))$ is a
closed Lie ideal of $\mathcal{L.}$

Let $f$: $\mathcal{L}\longrightarrow\mathcal{M}$ be a morphism in
$\overline{\mathbf{L}}.$ Set $q=q_{T}|\mathcal{L}$ and $q_{1}=q_{T}%
|\mathcal{M.}$ For each $x\in\mathcal{L,}$ set $h(x)=q_{1}(f(x)).$ Then $h$ is
a bounded homomorphism from $\mathcal{L}$ into $\mathcal{M/}T(\mathcal{M)}$
with dense image. As $T$ is a preradical, $f(T(\mathcal{L}))\subseteq
T(\mathcal{M).}$ Therefore, for each $a\in T(\mathcal{L),}$%
$
h(x+a)=q_{1}(f(x)+f(a))=q_{1}(f(x)).
$
Thus $h$ generates a bounded homomorphism $\widetilde{h}$: $q(\mathcal{L}%
)=\mathcal{L}/T(\mathcal{L})\longrightarrow\mathcal{M/}T(\mathcal{M}%
)=q_{1}(\mathcal{M)}$ with dense image and $\widetilde{h}q=q_{1}f.$ Then 
$(R\ast T)(\mathcal{L})=q_{T}^{-1}(R(q_{T}(\mathcal{L})))=q^{-1}%
(R(q(\mathcal{L}))),$ so that%
\[
q_{1}f((R\ast T)(\mathcal{L}))=\widetilde{h}q(q^{-1}(R(q(\mathcal{L}%
))))=\widetilde{h}(R(q(\mathcal{L})))\subseteq R(q_{1}(\mathcal{M})).
\]
Therefore $f((R\ast T)(\mathcal{L}))=q_{1}^{-1}(R(q_{1}(\mathcal{M})))=(R\ast
T)(\mathcal{M}).$ Thus $R\ast T$ is a preradical.

Clearly, $T(\mathcal{L})\subseteq q_{T}^{-1}(R(q_{T}(\mathcal{L})))=(R\ast
T)(\mathcal{L}),$ for each $\mathcal{L}\in\mathfrak{L}$, so that $T\leq R\ast T.$

(ii) For $I\vartriangleleft\mathcal{L}\in\mathfrak{L,}$ we have $q_{T}%
(I)\vartriangleleft q_{T}(\mathcal{L).}$ If $R$ is balanced, $R(q_{T}%
(I))\subseteq R(q_{T}(\mathcal{L)})\mathcal{.}$ Hence%
\[
(R\ast T)(I)=q_{T}^{-1}(R(q_{T}(I)))\subseteq q_{T}^{-1}(R(q_{T}%
(\mathcal{L})))=(R\ast T)(\mathcal{L}).
\]
Thus the preradical $R\ast T$ is balanced.

(iii) If $R$ is lower stable, $R(R(\mathcal{L}))=R(\mathcal{L)}$ for all
$\mathcal{L}\in\mathfrak{L.}$ Then $R\ast T$ is lower stable, as%
\begin{align*}
(R\ast T)((R\ast T)(\mathcal{L}))  &  =q_{T}^{-1}(R(q_{T}(q_{T}^{-1}%
(R(q_{T}(\mathcal{L}))))))\\
&  =q_{T}^{-1}(R(R(q_{T}(\mathcal{L}))))\\
&  =q_{T}^{-1}(R(q_{T}(\mathcal{L})))=(R\ast T)(\mathcal{L}).
\end{align*}

(iv) Let $S\leq T$ and $\mathcal{L}\in\mathfrak{L.}$ Then $S(\mathcal{L}%
)\subseteq T(\mathcal{L).}$ Hence there exists a quotient homomorphism $p$:
$q_{S}(\mathcal{L})=\mathcal{L}/S(\mathcal{L})\longrightarrow\mathcal{L/}%
T(\mathcal{L})=q_{T}(\mathcal{L),}$ such that $q_{T}=pq_{S}.$ Therefore%
\begin{align*}
q_{T}((R\ast S)(\mathcal{L}))  &  =pq_{S}(q_{S}^{-1}(R(q_{S}(\mathcal{L}%
))))=p(R(q_{S}(\mathcal{L})))\\
&  \subseteq R(pq_{S}(\mathcal{L}))=R(q_{T}(\mathcal{L})).
\end{align*}
Thus $(R\ast S)(\mathcal{L})\subseteq q_{T}^{-1}(R(q_{T}(\mathcal{L})))=(R\ast
T)(\mathcal{L}).$ Hence $R\ast S\leq R\ast T.$

As $S\leq T,$ we have $S(q_{R}(\mathcal{L}))\subseteq T(q_{R}(\mathcal{L})).$
Therefore
\[
(S\ast R)(\mathcal{L})=q_{R}^{-1}(S(q_{R}(\mathcal{L})))\subseteq q_{R}%
^{-1}(T(q_{R}(\mathcal{L})))=(T\ast R)(\mathcal{L}).
\]
Thus $S\ast R\leq T\ast R$.
\end{proof}

For each preradical $R$, we will define now an upper stable preradical
$R^{\ast}$ in the following way. For $\mathcal{L}\in\mathfrak{L,}$ set
$R^{\left(  0\right)  }\left(  \mathcal{L}\right)  =\{0\},$ $R^{\left(
1\right)  }\left(  \mathcal{L}\right)  =R\left(  \mathcal{L}\right)  ,$%
\begin{equation}
R^{\left(  \alpha+1\right)  }\left(  \mathcal{L}\right)  =(R\ast R^{\left(
\alpha\right)  })\left(  \mathcal{L}\right)  ,\text{ for an ordinal }\alpha,
\label{4.6}%
\end{equation}
By Proposition \ref{P1}(i), we have $R^{\left(  \alpha\right)  }\left(
\mathcal{L}\right)  \subseteq R^{\left(  \alpha+1\right)  }\left(
\mathcal{L}\right)  .$ Hence we can define%
\begin{equation}
R^{\left(  \alpha\right)  }\left(  \mathcal{L}\right)  =\overline
{\underset{\alpha^{\prime}<\alpha}{\cup}R^{\left(  \alpha^{\prime}\right)
}\left(  \mathcal{L}\right)  },\text{ for a limit ordinal }\alpha.
\label{4.6'}%
\end{equation}
$\{R^{(\alpha)}\left(  \mathcal{L}\right)  \}$ is an increasing transfinite
chain. By Corollary \ref{C1}, it consists of characteristic Lie ideals of
$\mathcal{L}$. As all $\alpha$ are bounded by an ordinal that depends on
cardinality of $\mathcal{L,}$ the chain stabilizes at some ordinal $\beta$:
$R^{\left(  \beta+1\right)  }\left(  \mathcal{L}\right)  =R^{\left(
\beta\right)  }\left(  \mathcal{L}\right)  .$ Denote the smallest such $\beta$
by $r_{R}^{\ast}(\mathcal{L)}$ and set%
\begin{equation}
R^{\ast}(\mathcal{L})=R^{r_{R}^{\ast}(\mathcal{L)}}(\mathcal{L)},\text{ so
that }R\ast R^{\ast}=R^{\ast}. \label{r2}%
\end{equation}

\begin{theorem}
\label{super}\emph{(i) }Let $R$ be a preradical. Then $R^{\ast}$ is an upper
stable preradical$,$%
\[
R\leq R^{\ast},\text{ \ }\mathbf{Sem}(R)=\mathbf{Sem}(R^{\ast})\text{ and
}\mathbf{Rad}(R)\subseteq\mathbf{Rad}(R^{\ast}).
\]
If $R$ is balanced$,$ then $R^{\ast}$ is an over radical$.$ Moreover$,$
$R^{\ast}$ is the smallest over radical larger than or equal to $R.$ If $R$ is
upper stable then $R^{\ast}=R.$

\emph{(ii) }Let $R$ and $T$ be preradicals. If $R\leq T,$ then $R^{\left(
\alpha\right)  }\leq T^{\left(  \alpha\right)  }$ for each $\alpha,$ and
$R^{\mathbf{\ast}}\leq T^{\mathbf{\ast}}.$
\end{theorem}

\begin{proof}
Let $R^{\left(  \alpha\right)  }$ be a preradical for some $\alpha.$ By
Proposition \ref{P1}(i), $R^{\left(  \alpha+1\right)  }=R\ast R^{\left(
\alpha\right)  }$ is a preradical for an ordinal $\alpha.$ Let $\alpha$ be a
limit ordinal and let $R^{(\alpha^{\prime})}$, for $\alpha^{\prime}<\alpha,$
be preradicals$.$ For each morphism $f$: $\mathcal{L}\longrightarrow
\mathcal{M}$ it follows from (\ref{F1}) that%
\begin{align*}
f(R^{(\alpha)}\left(  \mathcal{L}\right)  )  &  =f\left(  \overline
{\underset{\alpha^{\prime}<\alpha}{\cup}R^{\left(  \alpha^{\prime}\right)
}\left(  \mathcal{L}\right)  }\right)  \subseteq\overline{\underset
{\alpha^{\prime}<\alpha}{\cup}f(R^{\left(  \alpha^{\prime}\right)  }\left(
\mathcal{L}\right)  )}\\
&  \subseteq\overline{\underset{\alpha^{\prime}<\alpha}{\cup}R^{\left(
\alpha^{\prime}\right)  }\left(  \mathcal{M}\right)  }=R^{\alpha}\left(
\mathcal{M}\right)  .
\end{align*}
Thus $R^{(\alpha)}$ are preradicals for all $\alpha,$ so that $R^{\ast}$ is a preradical.

By (\ref{4.6}), (\ref{4.6'}) and Proposition \ref{P1}(i), we have $R\leq
R^{\ast}.$ Hence $R^{\ast}(\mathcal{L})=\{0\}$ implies $R(\mathcal{L})=\{0\}.$
If $R(\mathcal{L})=\{0\},$ it follows from (\ref{conv}) -- (\ref{4.6'}) that
$R^{\ast}(\mathcal{L})=\{0\}.$ Thus $\mathbf{Sem}(R)=\mathbf{Sem}(R^{\ast}).$

If $R(\mathcal{L})=\mathcal{L},$ it follows that all $R^{(\alpha)}%
(\mathcal{L})=\mathcal{L}$, so that $R^{\ast}(\mathcal{L})=\mathcal{L}.$ Hence
$\mathbf{Rad}(R)\subseteq\mathbf{Rad}(R^{\ast}).$

Set $q=q_{R^{\ast}}.$ As $R^{\ast}=R\ast R^{\ast}$ (see (\ref{r2})), we have
from (\ref{conv}) that $R^{\ast}\left(  \mathcal{L}\right)  =(R\ast R^{\ast
})(\mathcal{L})=q^{-1}(R(q(\mathcal{L}))).$ Hence $q(R^{\ast}\left(
\mathcal{L}\right)  )=R(q(\mathcal{L})).$ As $q$: $\mathcal{L}\longrightarrow
\mathcal{L}/R^{\ast}(\mathcal{L)},$ we have $q(R^{\ast}\left(  \mathcal{L}%
\right)  )=\{0\}.$ Hence $R(q(\mathcal{L}))=0.$ Thus $q(\mathcal{L})$ is
$R$-semisimple, so that $q(\mathcal{L})$ is $R^{\ast}$-semisimple by the above
argument. Hence $R^{\ast}(\mathcal{L}/R^{\ast}\left(  \mathcal{L}\right)
)=0$, whence $R^{\ast}$ is upper stable.

By (\ref{r2}), $R^{\ast}=R\ast R^{\ast}.$ Hence, if $R$ is balanced, it
follows from Proposition \ref{P1}(ii) that $R^{\ast}$ is balanced. Thus
$R^{\ast}$ is an over radical.

Let $T$ be an over radical such that $R\leq T$. If $\mathcal{L}\in
\mathbf{Sem}\left(  T\right)  $ then $R(\mathcal{L})\subseteq T(\mathcal{L}%
)=\{0\}$. Hence $\mathcal{L}\in\mathbf{Sem}\left(  R\right)  =\mathbf{Sem}%
\left(  R^{\ast}\right)  .$ Thus $\mathbf{Sem}\left(  T\right)  \subseteq$
$\mathbf{Sem}\left(  R^{\ast}\right)  .$ By Proposition \ref{P3.7}(iii),
$R^{\ast}\leq T$.

Let $R$ be upper stable. As $R^{\ast}$ is upper stable and $\mathbf{Sem}%
(R)=\mathbf{Sem}(R^{\ast}),$ it follows from Proposition \ref{P3.7}(iv) that
$R=R^{\ast}.$ Part (i) is proved.

Part (ii) follows by induction. Indeed, let $R^{\left(  \alpha\right)  }\leq
T^{\left(  \alpha\right)  }$ for some $\alpha$. As $R\leq T,$ we have from
Proposition \ref{P1}(iv) $R^{(\alpha+1)}=R\ast R^{(\alpha)}\leq R\ast
T^{\left(  \alpha\right)  }\leq T\ast T^{\left(  \alpha\right)  }%
=T^{(\alpha+1)}.$ Let $\mathcal{L}$ be a Banach Lie algebra. If $R^{\left(
\alpha^{\prime}\right)  }\left(  \mathcal{L}\right)  \subseteq T^{\left(
\alpha^{\prime}\right)  }\left(  \mathcal{L}\right)  $ for all $\alpha
^{\prime}<\alpha$, then $R^{\left(  \alpha\right)  }\left(  \mathcal{L}%
\right)  \subseteq T^{\left(  \alpha\right)  }\left(  \mathcal{L}\right)  $
follows from (\ref{4.6'}). Taking $\alpha=\max\left\{  r_{R}^{\ast
}(\mathcal{L}),r_{T}^{\ast}(\mathcal{L})\right\}  $, we have that
$R^{\mathbf{\ast}}\left(  \mathcal{L}\right)  \subseteq T^{\mathbf{\ast}%
}\left(  \mathcal{L}\right)  $ for each Banach Lie algebra $\mathcal{L}$.
\end{proof}

The following theorem gives sufficient conditions for $R^{\circ}$ and
$R^{\ast}$ to be radicals. It is similar to the result proved in \cite{D} for
the category of associative normed algebras.

\begin{theorem}
\label{under-over}\emph{(i) }If $R$ is an under radical\emph{ }then $R^{\ast}$
is the smallest radical larger than or equal to $R$.

\begin{itemize}
\item [$\mathrm{(ii)}$]If $R$ is an over radical then $R^{\circ}$ is the
largest radical smaller than or equal to $R$.
\end{itemize}
\end{theorem}

\begin{proof}
(i) By (\ref{r2}), $R^{\ast}=R\ast R^{\ast}.$ As $R$ is lower stable,
Proposition \ref{P1}(iii) implies that $R^{\ast}$ is lower stable. Hence, by
Theorem \ref{super}(i), $R^{\ast}$ is a radical, $R\leq R^{\ast}$ and
$R^{\ast}$ is the smallest over radical larger than or equal to $R.$ Hence
$R^{\ast}$ is the smallest radical larger than or equal to $R.$

(ii) Let $R^{\alpha}$ be upper stable for some $\alpha$: $R^{\alpha}\left(
\mathcal{L}/R^{\alpha}\left(  \mathcal{L}\right)  \right)  =\{0\}$ for all
$\mathcal{L}\in\mathfrak{L}$. As $R^{\alpha+1}\left(  \mathcal{L}\right)
\subseteq R^{\alpha}\left(  \mathcal{L}\right)  $, there is the quotient map
$q:\mathcal{L}/R^{\alpha+1}\left(  \mathcal{L}\right)  \longrightarrow
\mathcal{L}/R^{\alpha}\left(  \mathcal{L}\right)  $. Since $R^{\alpha}$ is a
preradical, it follows that $q(R^{\alpha}\left(  \mathcal{L}/R^{\alpha
+1}\left(  \mathcal{L}\right)  \right)  )\subseteq R^{\alpha}\left(
\mathcal{L}/R^{\alpha}\left(  \mathcal{L}\right)  \right)  =\{0\}$. So 
$R^{\alpha}\left(  \mathcal{L}/R^{\alpha+1}\left(  \mathcal{L}\right)
\right)  \subseteq\ker\left(  q\right)  =R^{\alpha}\left(  \mathcal{L}\right)
/R^{\alpha+1}\left(  \mathcal{L}\right)  .$ As $R$ is upper stable,%
\[
R(R^{\alpha}\left(  \mathcal{L}\right)  /R^{\alpha+1}\left(  \mathcal{L}%
\right)  )=R(R^{\alpha}\left(  \mathcal{L}\right)  /R(R^{\alpha}\left(
\mathcal{L}\right)  ))=\{0\}.
\]
Therefore, as $R$ is balanced,%
\begin{align*}
R^{\alpha+1}\left(  \mathcal{L}/R^{\alpha+1}\left(  \mathcal{L}\right)
\right)   &  =R\left(  R^{\alpha}\left(  \mathcal{L}/R^{\alpha+1}\left(
\mathcal{L}\right)  \right)  \right) \\
&  \subseteq R\left(  R^{\alpha}\left(  \mathcal{L}\right)  /R^{\alpha
+1}\left(  \mathcal{L}\right)  \right)  =\{0\}.
\end{align*}
Thus $R^{\alpha+1}(\mathcal{L)}$ is upper stable$.$

Let $\alpha$ be a limit ordinal. For all $\alpha^{\prime}<\alpha$,
$R^{\alpha^{\prime}}\left(  \mathcal{L}/R^{\alpha^{\prime}}\left(
\mathcal{L}\right)  \right)  =\{0\}$ and $R^{\alpha}\left(  \mathcal{L}%
\right)  \subseteq R^{\alpha^{\prime}}\left(  \mathcal{L}\right)  $ for each
$\mathcal{L}\in\mathfrak{L}$. Let $q$ be the quotient map $q:\mathcal{L}%
/R^{\alpha}\left(  \mathcal{L}\right)  \longrightarrow\mathcal{L}%
/R^{\alpha^{\prime}}\left(  \mathcal{L}\right)  $. Since $R^{\alpha^{\prime}}$
is a preradical,
\[
q(R^{\alpha^{\prime}}\left(  \mathcal{L}/R^{\alpha}\left(  \mathcal{L}\right)
\right)  )\subseteq R^{\alpha^{\prime}}\left(  q(\mathcal{L}/R^{\alpha}\left(
\mathcal{L}\right)  )\right)  =R^{\alpha^{\prime}}\left(  \mathcal{L}%
/R^{\alpha^{\prime}}\left(  \mathcal{L}\right)  \right)  =\{0\}.
\]
Hence $R^{\alpha^{\prime}}\left(  \mathcal{L}/R^{\alpha}\left(  \mathcal{L}%
\right)  \right)  \subseteq\ker\left(  q\right)  =R^{\alpha^{\prime}}\left(
\mathcal{L}\right)  /R^{\alpha}\left(  \mathcal{L}\right)  .$ Therefore, as
$R^{\alpha}\left(  \mathcal{L}\right)  =\underset{\alpha^{\prime}<\alpha}%
{\cap}R^{\alpha^{\prime}}\left(  \mathcal{L}\right)  ,$ we have%
\begin{align*}
R^{\alpha}(\mathcal{L}/R^{\alpha}\left(  \mathcal{L}\right)  )  &
=\underset{\alpha^{\prime}<\alpha}{\cap}R^{\alpha^{\prime}}\left(
\mathcal{L}/R^{\alpha}\left(  \mathcal{L}\right)  \right) \\
&  \subseteq\underset{\alpha^{\prime}<\alpha}{\cap}\left\{  R^{\alpha^{\prime
}}\left(  \mathcal{L}\right)  /R^{\alpha}\left(  \mathcal{L}\right)  \right\}
=\{0\}.
\end{align*}
Thus $R^{\alpha}$ is upper stable for all $\alpha,$ so that $R^{\circ}$ is
upper stable. Hence, by Theorem \ref{T4.1}, $R^{\circ}$ is a radical,
$R^{\circ}\leq R$ and $R^{\circ}$ is the largest under radical smaller than or
equal to $R.$
\end{proof}

\subsection{Construction of under radicals by subideals}

Let $\mathcal{L}\in\mathfrak{L}.$ Recall that a closed Lie subalgebra $I$ of
$\mathcal{L}$ is a Lie subideal ($I\!\vartriangleleft\!\!\!\vartriangleleft
\mathcal{L})$,\textit{ }if there is a chain of closed Lie subalgebras $J_{0}%
,$...$,J_{n}$ of $\mathcal{L}$ such that $I=J_{0}\vartriangleleft
J_{1}\vartriangleleft\cdots\vartriangleleft J_{n}=\mathcal{L}$. Let $R$ be a
preradical. Set $($see $(\ref{F0'}))$%
\begin{align}
\mathrm{Sub}\left(  \mathcal{L},R\right)   &  =\{I\!\vartriangleleft
\!\!\!\vartriangleleft\mathcal{L}:R\left(  I\right)  =I\}\text{ and
}\nonumber\\
R^{\mathbf{s}}  &  :\mathcal{L}\longmapsto\mathfrak{s}\left(
\mathrm{Sub}\left(  \mathcal{L},R\right)  \right)  . \label{4.4}%
\end{align}
The subideals in Sub($\mathcal{L},R)$ are called $R$-radical. Clearly,
$R^{\mathbf{s}}\left(  \mathcal{L}\right)  $ is a closed subspace of
$\mathcal{L}$.

\begin{lemma}
\label{sin}Let $R$ and $T$ be preradicals. If $R\leq T$ then $R^{\mathbf{s}%
}\leq T^{\mathbf{s}}$.
\end{lemma}

\begin{proof}
Let $I\in\mathrm{Sub}\left(  \mathcal{L},R\right)  $. As $\mathbf{Rad}\left(
R\right)  \subseteq\mathbf{Rad}\left(  T\right)  ,$ we have $I\in
\mathrm{Sub}\left(  \mathcal{L},T\right)  $. Hence $R^{\mathbf{s}}\left(
\mathcal{L}\right)  \subseteq T^{\mathbf{s}}\left(  \mathcal{L}\right)  $.
Thus $R^{\mathbf{s}}\leq T^{\mathbf{s}}$.
\end{proof}

\begin{theorem}
\label{T2}Let $R$ be a preradical in $\overline{\mathbf{L}}$. Then

\begin{itemize}
\item [$\mathrm{(i)}$]$R^{\mathbf{s}}$ is a balanced$,$ lower stable
preradical$,$ so that $R^{\mathbf{s}}$ is an under radical in $\overline
{\mathbf{L}}.$

\item[$\mathrm{(ii)}$] If $R$ is balanced, then $R^{\mathbf{s}}\leq R$ $($see
$(\ref{4}))\mathfrak{\ }$and $R^{\mathbf{s}}$ is the largest under radical
smaller than or equal to $R$.

\item[$\mathrm{(iii)}$] If $R$ is lower stable$,$ then $R\leq R^{\mathbf{s}}$
and $R^{\mathbf{s}}$ is the smallest under radical larger than or equal to $R$.

\item[$\mathrm{(iv)}$] If $R$ is an under radical then $R=R^{\mathbf{s}%
}\mathfrak{.}$
\end{itemize}
\end{theorem}

\begin{proof}
(i) Let $f$: $\mathcal{L}\longrightarrow\mathcal{M}$ be a morphism in
$\overline{\mathbf{L}},$ $I\in\mathrm{Sub}\left(  \mathcal{L},R\right)  $ and
$I=J_{0}\vartriangleleft\cdots\vartriangleleft J_{n}=\mathcal{L}$ for some
closed Lie algebras $J_{i}$ in $\mathcal{L}$. Then $\overline{f(I)}%
=\overline{f(J_{0})}\vartriangleleft\cdots\vartriangleleft\overline{f(J_{n}%
)}=\mathcal{M}$ and all $\overline{f(J_{i})}$ are closed Lie algebras in
$\mathcal{M}$. Hence $\overline{f(I)}\!\vartriangleleft\!\!\!\vartriangleleft
\mathcal{M}$. As $I=R(I)\mathfrak{,}$ we have from (\ref{4.0}) that
$f(I)=f(R(I))\subseteq R\left(  \overline{f(I)}\right)  \subseteq
\overline{f(I)}.$ As $R\left(  \overline{f(I)}\right)  $ is closed, $R\left(
\overline{f(I)}\right)  =\overline{f(I)}.$ Thus $\overline{f(I)}%
\in\mathrm{Sub}\left(  \mathcal{M},R\right)  $ and%
\[
f\left(  \sum_{I\in\mathrm{Sub}\left(  \mathcal{L},R\right)  }I\right)
=\sum_{I\in\mathrm{Sub}\left(  \mathcal{L},R\right)  }f(I)\subseteq\sum
_{J\in\mathrm{Sub}\left(  \mathcal{M},R\right)  }J,
\]
so that $f(R^{\mathbf{s}}\left(  \mathcal{L}\right)  )\subseteq R^{\mathbf{s}%
}\left(  \mathcal{M}\right)  $. Therefore $R^{\mathbf{s}}$ is a preradical
(see (\ref{4.0})).

Let $K\vartriangleleft\mathcal{L}$. If $I\in\mathrm{Sub}\left(  K,R\right)  $
then $I=J_{0}\vartriangleleft\cdots\vartriangleleft J_{n}=K$ for some closed
Lie algebras $J_{i}$, whence $I\in\mathrm{Sub}\left(  \mathcal{L},R\right)  $.
Thus $\mathrm{Sub}\left(  K,R\right)  \subseteq\mathrm{Sub}\left(
\mathcal{L},R\right)  $. Hence $R^{\mathbf{s}}$ is balanced (see
(\ref{4.3'})), since by (\ref{4.4}),%
\[
R^{\mathbf{s}}\left(  K\right)  =\overline{\sum_{I\in\mathrm{Sub}\left(
K,R\right)  }I}\subseteq\overline{\sum_{I\in\mathrm{Sub}\left(  \mathcal{L}%
,R\right)  }I}=R^{\mathbf{s}}\left(  \mathcal{L}\right)  .
\]

Set $K=R^{\mathbf{s}}\left(  \mathcal{L}\right)  $. If $I\in\mathrm{Sub}%
\left(  \mathcal{L},R\right)  $ then $I=J_{0}\vartriangleleft\cdots
\vartriangleleft J_{n}=\mathcal{L}$. By (\ref{4.4}), $I\subseteq K.$ Hence
$I=(J_{0}\cap K)\vartriangleleft\cdots\vartriangleleft(J_{n}\cap K)=K,$ so
that $I\in\mathrm{Sub}\left(  K,R\right)  .$ Thus $\mathrm{Sub}\left(
\mathcal{L},R\right)  =\mathrm{Sub}\left(  K,R\right)  $. Hence, by
(\ref{4.4}), $R^{\mathbf{s}}\left(  R^{\mathbf{s}}\left(  \mathcal{L}\right)
\right)  =R^{\mathbf{s}}\left(  K\right)  =R^{\mathbf{s}}\left(
\mathcal{L}\right)  .$ Thus (see (\ref{4.1'})) $R^{\mathbf{s}}$ is lower stable.

(iv) Let $R$ be balanced$\mathfrak{.}$ If $I\in\mathrm{Sub}\left(
\mathcal{L},R\right)  $ then $I=J_{0}\vartriangleleft\cdots\vartriangleleft
J_{n}=\mathcal{L}$ and $I=R(I)=R(J_{0})\subseteq...\subseteq R(J_{n})=R\left(
\mathcal{L}\right)  $. Hence, as $R\left(  \mathcal{L}\right)  $ is closed, it
follows from (\ref{4.4}) that $R^{\mathbf{s}}\left(  \mathcal{L}\right)
\subseteq R\left(  \mathcal{L}\right)  .$Thus $R^{\mathbf{s}}\leq R.$ This
also proves the first statement of (ii).

Let $R$ be lower stable. Then $R(R(\mathcal{L}))=R(\mathcal{L})$ for all
$\mathcal{L}\in\mathfrak{L,}$ so that $R(\mathcal{L})\in\mathrm{Sub}\left(
\mathcal{L},R\right)  .$ Hence, by (\ref{4.4}), $R(\mathcal{L})\subseteq
R^{\mathbf{s}}(\mathcal{L})$. Thus $R\leq R^{\mathbf{s}}.$ This also proves
the first statement of (iii).

So if $R$ is balanced and lower stable then $R=R^{\mathbf{s}}$ that proves (iv).

Let us finish the proofs of (ii) and (iii).

(ii) Let $R$ be balanced, let $Q$ be an under radical and $Q\leq R$. If $I$ is
a $Q$-radical subideal of $\mathcal{L}$, then $I$ is an $R$-radical subideal
of $\mathcal{L}$. Hence $Q^{\mathbf{s}}\leq R^{\mathbf{s}}$. It follows from
(iv) that $Q=Q^{\mathbf{s}}$. Therefore $R^{\mathbf{s}}$ is the largest under
radical smaller than or equal to $R$.

(iii) Let $R$ lower stable, let $T$ be an under radical and $R\leq T$. If
$I\in\mathrm{Sub}\left(  \mathcal{L},R\right)  $ then $I=R(I)\subseteq
T(I)\subseteq I.$ Hence $I=T(I),$ so that $I\in\mathrm{Sub}\left(
\mathcal{L},T\right)  $. Thus $\mathrm{Sub}\left(  \mathcal{L},R\right)
\subseteq\mathrm{Sub}\left(  \mathcal{L},T\right)  $. Using (\ref{4.4}), we
have $R^{\mathbf{s}}\left(  \mathcal{L}\right)  \subseteq T^{\mathbf{s}%
}\left(  \mathcal{L}\right)  $ for each $\mathcal{L}\in\mathfrak{L}$. By (iv),
$T=T^{\mathbf{s}}$, whence $R^{\mathbf{s}}\leq T$. Therefore $R^{\mathbf{s}}$
is the smallest under radical larger than or equal to $R$.
\end{proof}

Theorem \ref{T2}(i) and (iv) yield that $R^{\mathbf{ss}}=R^{\mathbf{s}}$ for
each preradical $R$.

\begin{corollary}
\label{C4.9}Let $R$ be a preradical in $\overline{\mathbf{L}}$.

\begin{itemize}
\item [$\mathrm{(i)}$]If $R$ is lower stable then $R\leq(R^{\mathbf{s}}%
)^{\ast}$ and $(R^{\mathbf{s}})^{\ast}$ is the smallest radical larger than or
equal to $R$.

\item[$\mathrm{(ii)}$] If $R$ is balanced then $R^{\mathbf{s}}=R^{\circ}\leq
R\leq R^{\ast},$ $(R^{\circ})^{\ast}$ and $(R^{\ast})^{\circ}$ are radicals$,$
and $(R^{\circ})^{\ast}\leq(R^{\ast})^{\circ}$.
\end{itemize}
\end{corollary}

\begin{proof}
(i) If $R$ is lower stable then $R\leq R^{\mathbf{s}}$ and $R^{\mathbf{s}}$ is
an under radical by Theorem \ref{T2}(iii), and $R^{\mathbf{s\ast}}$ is a
radical by Theorem \ref{under-over}(i). By definition, $R^{\mathbf{s}}\leq
R^{\mathbf{s\ast}}$. So $R\leq R^{\mathbf{s\ast}}$.

Let $T$ be a radical and $R\leq T$. Then $R^{\mathbf{s}}\leq T^{\mathbf{s}}$
by Lemma \ref{sin}. As $T$ is an under radical, $T^{\mathbf{s}}=T$ by Theorem
\ref{T2}(iv). Therefore $R^{\mathbf{s}}\leq T$. By Theorem \ref{super}(ii),
$(R^{\mathbf{s}})^{\ast}\leq T^{\ast}$. As $T$ is upper stable, $T^{\ast}=T$
by Theorem \ref{super}(i). Thus $(R^{\mathbf{s}})^{\ast}$ is the smallest
radical larger than or equal to $R$.

(ii) Let $R$ be balanced. Then $R^{\mathbf{s}}=R^{\circ}$ by Theorems
\ref{T4.1}(ii) and \ref{T2}(ii).

It follows from Theorems \ref{T4.1}, \ref{super}, \ref{under-over} that
$(R^{\circ})^{\ast}$ and $(R^{\ast})^{\circ}$ are radicals. Further, by
Theorems \ref{T4.1} and \ref{super}, $R^{\circ}\leq R\leq R^{\ast}$. By Lemma
\ref{cin}, $(R^{\circ})^{\circ}\leq(R^{\ast})^{\circ}$. As $R^{\circ}$ is
lower stable, $R^{\circ}=(R^{\circ})^{\circ}$ by Theorem \ref{T4.1}. Hence
$R^{\circ}\leq(R^{\ast})^{\circ}$. Since, by Theorem \ref{under-over}(i),
$(R^{\circ})^{\ast}$ is a smallest radical larger than or equal to $R^{\circ}%
$, we have $(R^{\circ})^{\ast}\leq(R^{\ast})^{\circ}$.
\end{proof}

\section{\label{section5}Construction of preradicals from multifunctions}

Some important preradicals in $\overline{\mathbf{L}}$ and its subcategories
arise from subspace-multifunctions on $\mathfrak{L}$; we will now study this link.

Let $F,G$ be non-empty families of subspaces of $X$. We write%
\begin{align}
F\overrightarrow{\subset}G\text{ if, for each }Y  &  \in F,\text{ there is
}Z\in G\text{ such that }Y\subseteq Z;\nonumber\\
G\overleftarrow{\subset}F\text{ if, for each }Y  &  \in F,\text{ there is
}Z\in G\text{ such that }Z\subseteq Y. \label{KK}%
\end{align}
We assume that $\varnothing\overrightarrow{\subset}G$ and $G\overleftarrow
{\subset}\varnothing.$ By (\ref{KK}), if $F\subseteq G$ then $F\overrightarrow
{\subset}G$ and $G\overleftarrow{\subset}F.$ It follows from (\ref{F0}),
(\ref{F0'}) and (\ref{KK}) that%
\begin{equation}
\text{if }F\overrightarrow{\subset}G\text{ then }\mathfrak{s}\left(  F\right)
\subseteq\mathfrak{s}\left(  G\right)  ;\text{ \ if }G\overleftarrow{\subset
}F\text{ then }\mathfrak{p}\left(  G\right)  \subseteq\mathfrak{p}\left(
F\right)  . \label{LP0}%
\end{equation}
If $G\neq\varnothing$ then $\{\{0\}\}\overleftarrow{\subset}G\overrightarrow
{\subset}\left\{  X\right\}  .$ For one-element families $F=\{Y\}$ and
$G=\{Z\},$ both relations coincide with inclusion: if $\left\{  Y\right\}
\overrightarrow{\subset}\{Z\},$ or $\{Y\}\overleftarrow{\subset}\left\{
Z\right\}  ,$ then $Y\subseteq Z$.

\begin{definition}
If$,$ for each $\mathcal{L}\in\mathfrak{L,}$ a family $\Gamma_{\mathcal{L}}$
of closed subspaces $($Lie algebras$,$ Lie ideals$)$ of $\mathcal{L}$ is
given$,$ we say that $\Gamma=\{\Gamma_{\mathcal{L}}\}$ is a \textit{subspace}
$($\textit{Lie algebra}$,$ \textit{Lie ideal}$)$\textit{-multifunc}tion on
$\mathfrak{L.}$
\end{definition}

Let $\Gamma=\{\Gamma_{\mathcal{L}}\}$ be a subspace-multifunction. Making use
of (\ref{F0}), set%
\begin{equation}
P_{\Gamma}\left(  \mathcal{L}\right)  =\mathfrak{p}(\Gamma_{\mathcal{L}%
})\text{ and }S_{\Gamma}\left(  \mathcal{L}\right)  =\mathfrak{s}%
(\Gamma_{\mathcal{L}}),\text{ for each }\mathcal{L}\in\mathfrak{L.}
\label{4.r}%
\end{equation}
If, for example, $\Gamma_{\mathcal{L}}$ is a singleton $\left\{  \phi\left(
\mathcal{L}\right)  \right\}  $ for each $\mathcal{L}\in\mathfrak{L,}$ then
$S_{\Gamma}\left(  \mathcal{L}\right)  =P_{\Gamma}\left(  \mathcal{L}\right)
=\phi\left(  \mathcal{L}\right)  $.

\begin{definition}
Let\emph{ }$\Gamma$ be a subspace-multifunction on $\mathfrak{L.}$%
\emph{\ }If$,$ for each morphism\emph{ }$f:\mathcal{L}\longrightarrow
\mathcal{M},$ the family $f\left(  \Gamma_{\mathcal{L}}\right)  =\{\overline
{f(Y)}:$ $Y\in\Gamma_{\mathcal{L}}\}$ of closed subspaces of $\mathcal{M}$ satisfies

\begin{itemize}
\item [$\mathrm{(i)}$]$f\left(  \Gamma_{\mathcal{L}}\right)  \,\overrightarrow
{\subset}\Gamma_{\mathcal{M}}$ then the multifunction $\Gamma$ is called\emph{ direct;}

\item[$\mathrm{(ii)}$] $f\left(  \Gamma_{\mathcal{L}}\right)  \subseteq
\Gamma_{\mathcal{M}}$ then the multifunction $\Gamma$ is called\emph{ strictly
direct}$;$

\item[$\mathrm{(iii)}$] $f\left(  \Gamma_{\mathcal{L}}\right)
\,\overleftarrow{\subset}\Gamma_{\mathcal{M}}$ then the multifunction $\Gamma$
is called\emph{ inverse}.
\end{itemize}
\end{definition}

\begin{proposition}
\label{L2}\emph{(i) }If $\Gamma$ is a direct multifunction then $S_{\Gamma}$
is a preradical in $\overline{\mathbf{L}}.$

\begin{itemize}
\item [$\mathrm{(ii)}$]If $\Gamma$ is an inverse multifunction then
$P_{\Gamma}$ is a preradical in $\overline{\mathbf{L}}.$
\end{itemize}
\end{proposition}

\begin{proof}
If $\Gamma$ is direct then $f\left(  \Gamma_{\mathcal{L}}\right)
\,\overrightarrow{\subset}\Gamma_{\mathcal{M}}$ for each morphism $f$:
$\mathcal{L}\longrightarrow\mathcal{M}$. Therefore%
\[
f\left(  S_{\Gamma}\left(  \mathcal{L}\right)  \right)  \overset{(\ref{4.r}%
)}{=}f(\mathfrak{s}(\Gamma_{\mathcal{L}}))\overset{(\ref{F1})}{\subseteq
}\mathfrak{s}(f(\Gamma_{\mathcal{L}}))\overset{(\ref{LP0})}{\subseteq
}\mathfrak{s}(\Gamma_{\mathcal{M}})\overset{(\ref{4.r})}{=}S_{\Gamma}\left(
\mathcal{M}\right)  .
\]
If $\Gamma$ is inverse then $f\left(  \Gamma_{\mathcal{L}}\right)
\,\overleftarrow{\subset}\Gamma_{\mathcal{M}}$ for every morphism $f$:
$\mathcal{L}\longrightarrow\mathcal{M}$. Therefore%
\[
f\left(  P_{\Gamma}\left(  \mathcal{L}\right)  \right)  \overset{(\ref{4.r}%
)}{=}f(\mathfrak{p}(\Gamma_{\mathcal{L}}))\overset{(\ref{F2})}{\subseteq
}\mathfrak{p}(f(\Gamma_{\mathcal{L}}))\overset{(\ref{LP0})}{\subseteq
}\mathfrak{p}(\Gamma_{\mathcal{M}})\overset{(\ref{4.r})}{=}P_{\Gamma}\left(
\mathcal{M}\right)  .
\]
Take $\mathcal{M}=\mathcal{L.}$ Considering inner automorphisms $f=\exp
{t(\mathrm{ad}}\left(  {a}\right)  {)}$ for $a\in\mathcal{L,}$ $t\in
\mathbb{C},$ we get that $P_{\Gamma}\left(  \mathcal{L}\right)  $ and
$S_{\Gamma}\left(  \mathcal{L}\right)  $ are ideals of $\mathcal{L}$. Thus we
have from (\ref{4.0}) that $S_{\Gamma}$ and $P_{\Gamma}$ are preradicals.
\end{proof}

If $\Gamma$ is a direct subspace-multifunction on $\mathfrak{L}\mathbf{,}$ set
$I_{\mathcal{L}}=S_{\Gamma}\left(  \mathcal{L}\right)  .$ If $\Gamma$ is an
inverse\emph{ }subspace-multifunction on $\mathfrak{L}\mathbf{,}$ set
$I_{\mathcal{L}}=P_{\Gamma}\left(  \mathcal{L}\right)  $. For
$J\vartriangleleft\mathcal{L,}$ set%
\begin{equation}
\Gamma_{\mathcal{L}}\cap J=\{L\cap J:L\in\Gamma_{\mathcal{L}}\}. \label{4.z}%
\end{equation}

\begin{definition}
\label{D4.7}A direct \emph{(}respectively, inverse) subspace-multifunction
$\Gamma$ is called

\begin{itemize}
\item [$\mathrm{(i)}$]\emph{balanced} if $\Gamma_{J}\overrightarrow{\subset
}\Gamma_{\mathcal{L}}$ \emph{(}respectively$,$ $\Gamma_{J}\overleftarrow
{\subset}\Gamma_{\mathcal{L}})$ for all $J\vartriangleleft\mathcal{L}%
\in\mathfrak{L;}$

\item[$\mathrm{(ii)}$] \emph{lower stable} if $\Gamma_{\mathcal{L}%
}\overrightarrow{\subset}\Gamma_{I_{\mathcal{L}}}$ \emph{(}respectively$,$
$\Gamma_{\mathcal{L}}\overleftarrow{\subset}\Gamma_{I_{\mathcal{L}}})$ for all
$\mathcal{L}\in\mathfrak{L;}$

\item[$\mathrm{(iii)}$] \emph{upper stable} if $\Gamma_{\mathcal{L}%
/I_{\mathcal{L}}}=\{\{0\}\}$ \emph{(}respectively\emph{, }$\mathfrak{p}%
(\Gamma_{\mathcal{L}/I_{\mathcal{L}}})=\{0\})$ for all $\mathcal{L}%
\in\mathfrak{L}$.
\end{itemize}
\end{definition}

\begin{theorem}
\label{Cmain}Let $\Gamma$ be a direct $($respectively, inverse$)$
subspace-multi\-fun\-ction on $\mathfrak{L.}$

\begin{itemize}
\item [$\mathrm{(i)}$]If $\Gamma$ is lower stable then $S_{\Gamma}$
$($respectively, $P_{\Gamma})$ is a lower stable preradical.

\item[$\mathrm{(ii)}$] If $\Gamma$ is balanced then $S_{\Gamma}$
$($respectively, $P_{\Gamma})$ is a balanced preradical.

\item[$\mathrm{(iii)}$] If $\Gamma$ is upper stable then $S_{\Gamma}$
$($respectively, $P_{\Gamma})$ is an upper stable preradical.
\end{itemize}
\end{theorem}

\begin{proof}
By Proposition \ref{L2}, $S_{\Gamma}$ is a preradical if $\Gamma$ is direct,
and $P_{\Gamma}$ is a preradical if $\Gamma$ is inverse.

(i) Let $\Gamma$ be lower stable. If $\Gamma$ is direct, $\Gamma_{\mathcal{L}%
}\overrightarrow{\subset}\Gamma_{I_{\mathcal{L}}}$ for all $\mathcal{L}%
\in\mathfrak{L}$, where $I_{\mathcal{L}}=S_{\Gamma}(\mathcal{L)}$. Hence%
\[
S_{\Gamma}(\mathcal{L})\overset{(\ref{4.r})}{=}\mathfrak{s}(\Gamma
_{\mathcal{L}})\overset{(\ref{LP0})}{\subseteq}\mathfrak{s}(\Gamma
_{I_{\mathcal{L}}})\overset{(\ref{4.r})}{=}S_{\Gamma}(I_{\mathcal{L}%
})=S_{\Gamma}(S_{\Gamma}(\mathcal{L)}).
\]
If $\Gamma$ is inverse then $\Gamma_{\mathcal{L}}\overleftarrow{\subset}%
\Gamma_{I_{\mathcal{L}}}$ for all $\mathcal{L}\in\mathfrak{L}$, where
$I_{\mathcal{L}}=P_{\Gamma}(\mathcal{L)}$. Hence%
\[
P_{\Gamma}(\mathcal{L})\overset{(\ref{4.r})}{=}\mathfrak{p}(\Gamma
_{\mathcal{L}})\overset{(\ref{LP0})}{\subseteq}\mathfrak{p}(\Gamma
_{I_{\mathcal{L}}})\overset{(\ref{4.r})}{=}P_{\Gamma}(I_{\mathcal{L}%
})=P_{\Gamma}(P_{\Gamma}(\mathcal{L)}).
\]
Thus (see (\ref{4.1'})) $S_{\Gamma}$ and $P_{\Gamma}$ are lower stable.

(ii) Let $\Gamma$ be balanced. If $\Gamma$ is direct then $\Gamma
_{J}\overrightarrow{\subset}\Gamma_{\mathcal{L}}$\ for all $J\vartriangleleft
\mathcal{L}\in\mathfrak{L}$. Hence $S_{\Gamma}(J)\overset{(\ref{4.r})}%
{=}\mathfrak{s}(\Gamma_{J})\overset{(\ref{LP0})}{\subseteq}\mathfrak{s}%
(\Gamma_{\mathcal{L}})\overset{(\ref{4.r})}{=}S_{\Gamma}(\mathcal{L}).$

If $\Gamma$ is inverse, $\Gamma_{J}\overleftarrow{\subset}\Gamma_{\mathcal{L}%
}$\ for all $J\vartriangleleft\mathcal{L}\in\mathfrak{L}$. Hence $P_{\Gamma
}(J)\overset{(\ref{4.r})}{=}\mathfrak{p}(\Gamma_{J})\overset{(\ref{LP0}%
)}{\subseteq}\mathfrak{p}(\Gamma_{\mathcal{L}})\overset{(\ref{4.r})}%
{=}P_{\Gamma}(\mathcal{L}).$ Thus $S_{\Gamma}$ and $P_{\Gamma}$ are balanced
(see (\ref{4.3'})).

(iii) Let $\Gamma$ be upper stable. If $\Gamma$ is direct then $\Gamma
_{\mathcal{L}/I_{\mathcal{L}}}=\{\{0\}\}$ for all $\mathcal{L}\in\mathfrak
{L,}$ where $I_{\mathcal{L}}=S_{\Gamma}\left(  \mathcal{L}\right)  .$ Hence
$S_{\Gamma}(\mathcal{L}/S_{\Gamma}\left(  \mathcal{L}\right)  )=S_{\Gamma
}(\mathcal{L}/I_{\mathcal{L}})\overset{(\ref{4.r})}{=}\mathfrak{s}%
(\Gamma_{\mathcal{L}/I_{\mathcal{L}}})=\{0\}.$

If $\Gamma$ is inverse then $p(\Gamma_{\mathcal{L}/I_{\mathcal{L}}})=\{0\}$
for all $\mathcal{L}\in\mathfrak{L,}$ where $I_{\mathcal{L}}=P_{\Gamma}\left(
\mathcal{L}\right)  .$ Hence $P_{\Gamma}(\mathcal{L}/P_{\Gamma}\left(
\mathcal{L}\right)  )=P_{\Gamma}(\mathcal{L}/I_{\mathcal{L}})\overset
{(\ref{4.r})}{=}\mathfrak{p}(\Gamma_{\mathcal{L/}I_{\mathcal{L}}})=\{0\}.$
Thus (see (\ref{4.2'})) $S_{\Gamma}$ and $P_{\Gamma}$ are upper stable.
\end{proof}

Now we characterize $S_{\Gamma}^{\ast}$-radical and $P_{\Gamma}^{\circ}%
$-semisimple Lie algebras via multifunctions $\Gamma$.

\begin{theorem}
\label{dir-quot}Let $\Gamma$ be a subspace-multifunction on $\mathfrak{L}$.

\begin{itemize}
\item [$\mathrm{(i)}$]If $\Gamma$ is direct$,$ then the following are
equivalent for each $\mathcal{L}\in\mathfrak{L}$.

\begin{itemize}
\item [\textrm{1)}]$\Gamma_{\mathcal{L}/I}$ is non-empty and $\Gamma
_{\mathcal{L}/I}\neq\{\{0\}\}$ for each $I\vartriangleleft\mathcal{L},$
$I\neq\mathcal{L}$.

\item[\textrm{2)}] $\Gamma_{\mathcal{L}/I}$ is non-empty and $\Gamma
_{\mathcal{L}/I}\neq\{\{0\}\}$ for each $I\vartriangleleft^{\emph{ch}%
}\mathcal{L},$ $I\neq\mathcal{L}$.

\item[\textrm{3)}] $\mathcal{L}$ is $S_{\Gamma}^{\ast}$-radical$.$
\end{itemize}

\item[$\mathrm{(ii)}$] Let $\Gamma$ be inverse and the preradical $P_{\Gamma}$
balanced. The following are equivalent\emph{,} for $\mathcal{L}\in\mathfrak
{L}$.

\begin{itemize}
\item [\textrm{1)}]$\Gamma_{I}$ is a non-empty family of proper subspaces for
each $\{0\}\neq I\vartriangleleft\mathcal{L}$.

\item[\textrm{2)}] $\Gamma_{I}$ is a non-empty family of proper subspaces for
each $\{0\}\neq I\vartriangleleft^{\emph{ch}}\mathcal{L}$.

\item[\textrm{3)}] $\mathcal{L}$ is $P_{\Gamma}^{\circ}$-semisimple$.$
\end{itemize}
\end{itemize}
\end{theorem}

\begin{proof}
(i) 1) $\Longrightarrow$ 2) is evident.

2) $\Longrightarrow$ 3). Set $R=S_{\Gamma}$. By Proposition \ref{L2}(i) and
Theorem \ref{super}, $R^{\ast}$ is an upper stable preradical and $R\leq
R^{\ast}.$ Set $I=R^{\ast}\left(  \mathcal{L}\right)  .$ Then $I\mathcal{\ }%
$is a characteristic Lie ideal of $\mathcal{L}$ and $R(\mathcal{L}/I)\subseteq
R^{\ast}(\mathcal{L}/I).$ If $\mathcal{L}$ is not $R^{\ast}$-radical then
$I\neq\mathcal{L}.$ As $\Gamma_{\mathcal{L}/I}$ is non-empty and
$\Gamma_{\mathcal{L}/I}\neq\{\{0\}\},$ we have $R(\mathcal{L}/I)=S_{\Gamma
}(\mathcal{L}/I)=\mathfrak{s}(\Gamma_{\mathcal{L}/I})\neq\{0\}.$ Hence
$\{0\}\neq R(\mathcal{L}/I)\subseteq R^{\ast}(\mathcal{L}/I)=R^{\ast}\left(
\mathcal{L}/R\left(  \mathcal{L}\right)  \right)  \overset{(\ref{4.2'})}%
{=}\{0\}$. This contradiction implies that $R^{\ast}(\mathcal{L}%
)=\mathcal{L}.$ Thus (see (\ref{1sr})) $\mathcal{L}$ is $R$-radical.

3) $\Longrightarrow$ 1). Let $\mathcal{L}$ be $R^{\ast}$-radical and let
$I\vartriangleleft\mathcal{L}$, $I\neq\mathcal{L}$. By Lemma \ref{L-sem}(i),
$R^{\ast}(\mathcal{L}/I)=\mathcal{L}/I\neq\{0\}$. Hence it follows from
Theorem \ref{super} that $R(\mathcal{L}/I)=S_{\Gamma}\left(  \mathcal{L}%
/I\right)  \neq0$. This is only possible when $\Gamma_{\mathcal{L}/I}$ is a
non-empty family with non-zero subspaces.

(ii) 1) $\Longrightarrow$ 2) is evident. 2) $\Longrightarrow$ 3). Set
$R=P_{\Gamma}.$ By Proposition \ref{L2}(ii) and Theorem \ref{T4.1}, $R^{\circ
}$ is an under radical. Let $\mathcal{L}$ be not $R^{\circ}$-semisimple. By
Lemma \ref{L1}, $I=R^{\circ}\left(  \mathcal{L}\right)  \neq\{0\}$ is a
characteristic Lie ideal of $\mathcal{L}$. Hence $\Gamma_{I}$ is a non-empty
family of proper subspaces of $I.$ Then $R\left(  R^{\circ}\left(
\mathcal{L}\right)  \right)  =P_{\Gamma}\left(  I\right)  =\mathfrak{p}%
(\Gamma_{I})\subsetneqq I$. By Theorem \ref{T4.1}, $R^{\circ}\leq R.$
Therefore, as $R^{\circ}$ is lower stable, $R^{\circ}\left(  \mathcal{L}%
\right)  \overset{(\ref{4.1'})}{=}R^{\circ}(R^{\circ}\left(  \mathcal{L}%
\right)  )\subseteq R\left(  R^{\circ}\left(  \mathcal{L}\right)  \right)
\subsetneqq R^{\circ}\left(  \mathcal{L}\right)  $, a contradiction. Thus
$\mathcal{L}$ is $R^{\circ}$-semisimple.

3) $\Longrightarrow$ 1). Let $\mathcal{L}$ be $R^{\circ}$-semisimple (that is,
$R^{\circ}(\mathcal{L})=\{0\})$ and let $\{0\}\neq I\vartriangleleft
\mathcal{L}.$ As $R^{\circ}$ is balanced, we have from Lemma \ref{L-sem}(ii)
that $R^{\circ}\left(  I\right)  =\{0\}$. If $R\left(  I\right)  =I,$ it
follows from (\ref{4.o}) that $R^{\circ}(I)=I,$ a contradiction. Thus
$R(I)=P_{\Gamma}\left(  I\right)  =\mathfrak{p}(\Gamma_{I})\subsetneqq I$.
This is only possible when $\Gamma_{I}$ is a non-empty family of proper
subspaces of $I$.
\end{proof}

Let $\Gamma$ be a Lie subalgebra-multifunction on $\mathfrak{L,}$ that is,
each family $\Gamma_{\mathcal{L}},$ $\mathcal{L}\in\mathfrak{L,}$ consists of
closed Lie subalgebras of $\mathcal{L}$. If $R$ is a preradical, then
$R\left(  \Gamma\right)  $ is also a Lie subalgebra-multifunction on
$\mathfrak{L}$, where each $R(\Gamma)_{\mathcal{L}}=R(\Gamma_{\mathcal{L}%
})=\{R(L)$: $L\in\Gamma_{\mathcal{L}}\}$ is a family of closed Lie subalgebras
of $\mathcal{L}$.

\begin{proposition}
\label{strictly}Let $\Gamma$ be a Lie subalgebra-multifunction on
$\mathfrak{L}$ and $R$ be a preradical on $\overline{\mathbf{L}}$. If $\Gamma$
is strictly direct then the multifunction $R\left(  \Gamma\right)  $ is
direct. If $R,$ in addition$,$ is balanced and $\Gamma_{J}\subseteq
\Gamma_{\mathcal{L}}\cap J$ for all $J\vartriangleleft\mathcal{L}\in
\mathfrak{L}$ $($see $(\ref{4.z})),$ then $R\left(  \Gamma\right)  $ is balanced.
\end{proposition}

\begin{proof}
Let $f$: $\mathcal{L}\longrightarrow\mathcal{M}$ be a homomorphism with dense
range. As $\Gamma$ is strictly direct, $M:=\overline{f\left(  L\right)  }%
\in\Gamma_{\mathcal{M}},$ for each $L\in\Gamma_{\mathcal{L}},$ and $f|_{L}$:
$L\longrightarrow M$ is a homomorphism with dense range. As $R$ is a
preradical on $\overline{\mathbf{L}}$, we have $f\left(  R\left(  L\right)
\right)  \subseteq R\left(  M\right)  $. Hence $f\left(  R\left(
\Gamma_{\mathcal{L}}\right)  \right)  \overrightarrow{\subset}R\left(
\Gamma_{\mathcal{M}}\right)  $.

Let $J\vartriangleleft\mathcal{L}$. Then, for each $L\in\Gamma_{\mathcal{L}},$
we have $(L\cap J)\vartriangleleft L.$ If $R$ is balanced then $R(L\cap
J)\subseteq R\left(  L\right)  $. Hence if $\Gamma_{J}\subseteq\Gamma
_{\mathcal{L}}\cap J$ then $R\left(  \Gamma_{J}\right)  \overrightarrow
{\subset}R\left(  \Gamma_{\mathcal{L}}\right)  $.
\end{proof}

It follows from Proposition \ref{L2} and Theorem \ref{Cmain}(ii) that in the
conditions of Proposition \ref{strictly} $S_{R\left(  \Gamma\right)  }$ is a
preradical and a balanced preradical, respectively.

\section{Examples of multifunctions and radicals\label{6}}

In the first subsection we consider some preliminary results about chains of
closed subspaces which we will later apply to describe examples of
multifunctions and radicals.

\subsection{Finite-gap families of subspaces\label{5e1}}

In the following lemma we gather several elementary results on subspaces of
finite codimension in a normed space.

\begin{lemma}
\label{Closed}Let $Z$ be a subspace of a normed space $X$ and let $Y$ be a
closed subspace of finite codimension in $X$. Then the subspace $Y+Z$ is
closed in $X,$ $Y\cap Z$ is a closed subspace of finite codimension in $Z$ and
$\dim\left(  Z/\left(  Y\cap Z\right)  \right)  =\dim\left(  \left(
Y+Z\right)  /Y\right)  $.
\end{lemma}

\begin{proof}
Let $q$: $X\longrightarrow X/Y$ be the quotient map. As $X/Y$ is finite
dimensional, $q\left(  Y+Z\right)  $ is closed in $X/Y$, whence $Y+Z=q^{-1}%
\left(  q\left(  Y+Z\right)  \right)  $ is closed in $X.$ It is clear that
$Y\cap Z$ is closed in $Z$. As $\left(  Y+Z\right)  /Y$ and $Z/\left(  Y\cap
Z\right)  $ are isomorphic in the pure algebraic sense, their dimensions
coincide, whence $Y\cap Z$ has finite codimension in $Z$.
\end{proof}

From now on $X$ denotes a Banach space and $G$ a family of closed subspaces of
$X$. For $Y,Z\in G$ with $Y\subsetneqq Z$, the set $[Y,Z]_{G}=\{W\in G$:
$Y\subseteq W\subseteq Z\}$ is called an \textit{interval} of $G$. If
$[Y,Z]_{G}=\left\{  Y,Z\right\}  $, the pair $\left(  Y,Z\right)  $ is called
a \textit{gap}. Recall (see (\ref{F0})) that
\begin{align*}
\mathfrak{p}\left(  G\right)   &  =X\text{ and }\mathfrak{s}\left(  G\right)
=\{0\},\text{ if }G=\varnothing;\text{ }\\
\text{otherwise }\mathfrak{p}\left(  G\right)   &  =\bigcap\limits_{Y\in
G}Y\text{ and }\mathfrak{s}\left(  G\right)  =\overline{\sum_{Y\in G}Y}.
\end{align*}

For a family $G$ of closed subspaces of $X\mathcal{,}$ define its\textbf{
}$\mathfrak{p}$-\textit{completion} and $\mathfrak{s}$-\textit{completion }as
follows:%
\begin{align*}
G^{\mathfrak{p}}  &  =\left\{  \mathfrak{p}\left(  G^{\prime}\right)  \text{:
}\varnothing\neq G^{\prime}\subseteq G\right\}  \cup\left\{  \mathfrak
{s}\left(  G\right)  \right\}  \text{ and }\\
G^{\mathfrak{s}}  &  =\left\{  \mathfrak{s}\left(  G^{\prime}\right)  \text{:
}\varnothing\neq G^{\prime}\subseteq G\right\}  \cup\left\{  \mathfrak
{p}\left(  G\right)  \right\}  .
\end{align*}
We add $\mathfrak{s}(G)$ to $G^{\mathfrak{p}}$ and $\mathfrak{p}(G)$ to
$G^{\mathfrak{s}}$ for technical convenience. We say that

1) $G$ is $\mathfrak{p}$-\textit{complete} if $G=G^{\mathfrak{p}};$

2) $G$ is $\mathfrak{s}$-\textit{complete} if $G=G^{\mathfrak{s}};$

3) $G$ is \textit{complete} if it is $\mathfrak{p}$-complete and $\mathfrak
{s}$-complete. Clearly, $G\subseteq G^{\mathfrak{p}}\cap G^{\mathfrak{s}}.$

\begin{definition}
\label{D2.1}A family $G$ of closed subspaces of $X$ is called

\begin{itemize}
\item [$\mathrm{(i)}$]a \emph{lower finite-gap family} if$,$\emph{ }for each
$Z\neq\mathfrak{p}(G)$ in $G,$\emph{ }there is $Y\in G$ such that $Y\subset Z$
and $0<\dim(Z/Y)<\infty.$

\item[$\mathrm{(ii)}$] an \emph{upper finite-gap family} if$,$\emph{ }for each
$Z\neq\mathfrak{s}(G)$ in $G,$ there is $Y\in G$ such that $Z\subset Y$
and\emph{ }$0<\dim(Y/Z)<\infty.$
\end{itemize}
\end{definition}

Note that a lower finite-gap family may have infinite gaps. Let $X$ be a
Hilbert space with a basis $\{e_{n}\}_{n=1}^{\infty}.$ The family $G$ of
subspaces $L_{k}=\{e_{n}\}_{n=k}^{\infty},$ $M_{k}=\{e_{2n-1}\}_{n=k}^{\infty
},$ for all $1\leq k<\infty,$ and $\{0\}$ is a $\mathfrak{p}$-complete, lower
finite-gap family. However, $[M_{1},L_{1}]_{G}$ is an infinite gap.

The next lemma provides us with numerous examples of lower and upper
finite-gap families.

\begin{lemma}
\label{L2.4}Let $G$ be a family of closed subspaces of $X$.

\begin{itemize}
\item [$\mathrm{(i)}$]If $G$ consists of subspaces of finite codimension then
$G^{\mathfrak{p}}$ is a lower finite-gap family.

\item[$\mathrm{(ii)}$] If $G$ consists of subspaces of finite dimension then
$G^{\mathfrak{s}}$ is an upper finite-gap family.
\end{itemize}
\end{lemma}

\begin{proof}
(i) Let $\mathfrak{p}\left(  G\right)  \neq Z\in G^{\mathfrak{p}}.$ Then
$Z=\mathfrak{p}\left(  G^{\prime}\right)  $ for some $\varnothing\neq
G^{\prime}\subsetneqq G$ . Set $G^{\prime\prime}=G\diagdown G^{\prime}$. Then
\[
\mathfrak{p}\left(  G\right)  =\mathfrak{p}\left(  G^{\prime}\right)
\cap\mathfrak{p}\left(  G^{\prime\prime}\right)  =Z\cap\mathfrak{p}\left(
G^{\prime\prime}\right)  =\cap_{Y\in G^{\prime\prime}}(Z\cap Y).
\]
As $Z\neq\mathfrak{p}\left(  G\right)  $, there is a subspace $Y\in
G^{\prime\prime}$ such that $Z\cap Y\neq Z$. Then $Z\cap Y\in G^{\mathfrak{p}%
}$ and, by Lemma \ref{Closed}, $Z\cap Y$ has finite codimension in $Z.$ Thus
$0<\dim\left(  Z/\left(  Z\cap Y\right)  \right)  <\infty$.

(ii) Let $\mathfrak{s}\left(  G\right)  \neq Z\in G^{\mathfrak{s}}.$ Then
$Z=\mathfrak{s}\left(  G^{\prime}\right)  $ for some $\varnothing\neq
G^{\prime}\subsetneqq G$. Set $G^{\prime\prime}=G\diagdown G^{\prime}$. Then
\[
\mathfrak{s}\left(  G\right)  =\mathrm{span}\left(  \mathfrak{s}\left(
G^{\prime}\right)  +\cup_{Y\in G^{\prime\prime}}Y\right)  =\mathrm{span}%
\left(  \cup_{Y\in G^{\prime\prime}}\left(  Z+Y\right)  \right)  .
\]
Since $Z\neq\mathfrak{s}\left(  G\right)  $, there is a subspace $Y\in
G^{\prime\prime}$ such that $Z+Y\neq Z$. By Lemma \ref{Closed}, $Z+Y$ is
closed. Hence $Z+Y\in G^{\mathfrak{s}}$ and $0<\dim\left(  \left(  Z+Y\right)
/Z\right)  <\infty$.
\end{proof}

\begin{remark}
\label{R6.1}\emph{In this paper we consider }$\mathfrak{p}$\emph{-complete
lower finite-gap families. Similar results hold for }$\mathfrak{s}%
$\emph{-complete upper finite-gap families.}
\end{remark}

A subfamily $C$ of $G$ is a \textit{chain} if every two subspaces in $C$ are
comparable, that is, the order defined by inclusion is linear on $C$. A chain
$C$ is \textit{maximal }if $G$ has no other larger chain.

\begin{lemma}
\label{L2.2} Let $G$ be a $\mathfrak{p}$-complete family of closed subspaces
in $X$. Then

\begin{itemize}
\item [$\mathrm{(i)}$]For each chain $C_{0}$ in $G,$ there is a maximal
$\mathfrak{p}$-complete chain $C_{m}$ in $G$ containing $C_{0}$ with
$\mathfrak{p}(C_{m})=\mathfrak{p}(C_{0})$ and $\mathfrak{s}(C_{m}%
)=\mathfrak{s}(C_{0})$.

\item[$\mathrm{(ii)}$] Let $C$ be a $\mathfrak{p}$-complete, lower
finite-gap chain. Then

\begin{itemize}
\item [$a)$]$C$ is complete and is a complete strictly decreasing transfinite
sequence of closed subspaces\emph{;}

\item[$b)$] each chain in $[\mathfrak{p}(C),\mathfrak{s}(C)]_{G}$ larger than
$C$ is also a lower finite-gap chain\emph{;}

\item[$c)$] the interval $[\mathfrak{p}(C),\mathfrak{s}(C)]_{G}$ has a maximal,
 complete lower finite-gap chain containing $C.$
\end{itemize}

\item[\textrm{(iii)}] Let $C$ be a lower finite-gap chain. If it is maximal in 
$[\mathfrak{p}(C),\mathfrak{s}(C)]_{G}$ then $C$ is complete.

\item[$\mathrm{(iv)}$] Let $C_{0}$ be a $\mathfrak{p}$-complete$,$ lower
finite-gap chain with $\mathfrak{s}(C_{0})=\mathfrak{s}(G).$ Then $G$ has a
maximal$,$ lower finite-gap chain $C$ containing $C_{0}.$ If
$G$ is a lower finite-gap family$,$ then $\mathfrak{p}(C)=\mathfrak{p}(G)$.
\end{itemize}
\end{lemma}

\begin{proof}
(i) As $G$ is $\mathfrak{p}$-complete, $C_{0}^{\mathfrak{p}}$ is a
$\mathfrak{p}$-complete chain in $G,$ $\mathfrak{p}(C_{0}^{\mathfrak{p}%
})=\mathfrak{p}(C_{0})$ and $\mathfrak{s}(C_{0}^{\mathfrak{p}})=\mathfrak
{s}(C_{0}).$ The set $\mathcal{G}$ of all $\mathfrak{p}$-complete chains $C$
in $G$ containing $C_{0}^{\mathfrak{p}},$ with $\mathfrak{p}(C)=\mathfrak
{p}(C_{0})$ and $\mathfrak{s}(C)=\mathfrak{s}(C_{0}),$ is partially ordered by
inclusion. Let $\{C_{\lambda}\}_{\lambda\in\Lambda}$ be a linearly ordered
subset of $\mathcal{G.}$ Then $C^{\prime}=\left(  \cup_{\lambda\in\Lambda
}C_{\lambda}\right)  ^{\mathfrak{p}}\in\mathcal{G}$. Hence $\mathcal{G}$ is
inductive. By Zorn's Lemma, $\mathcal{G}$ has a maximal element $C_{m}$.

(ii) a) We only need to show that $C$ is $\mathfrak{s}$-complete. Let
$C_{0}\neq\varnothing$ be a subset of $C.$ If $C_{0}=\left\{  \mathfrak
{p}\left(  C_{0}\right)  \right\}  $ then $\mathfrak{s}(C_{0})=\mathfrak
{p}(C_{0})\in C_{0}.$ If $C_{0}\neq\left\{  \mathfrak{p}\left(  C_{0}\right)
\right\}  ,$ set $C_{1}=\{Y\in C$: $Z\subseteq Y$ for all $Z\in C_{0}\}$. As
$\mathfrak{s}\left(  C\right)  \in C_{1}$, $C_{1}$ is not empty. It follows
that $Z\subseteq\mathfrak{p}\left(  C_{1}\right)  \in C_{1}$ for each $Z\in
C_{0}$.

Assume that $\mathfrak{p}\left(  C_{1}\right)  \notin C_{0}.$ As $C$ is a
lower finite-gap chain, there is $Y_{0}\in C$ such that $[Y_{0},\mathfrak
{p}(C_{1})]_{C}$ is a gap. Hence $Y_{0}\notin C_{1}$ and $Y_{0}\subsetneqq
Z_{0}$ for some $Z_{0}\in C_{0}$, otherwise $Y_{0}\in C_{1}.$ Thus
$Y_{0}\subsetneqq Z_{0}\subsetneqq\mathfrak{p}\left(  C_{1}\right)  ,$ so that
$[Y_{0},\mathfrak{p}(C_{1})]_{C}$ is not a gap. This contradiction shows that
$\mathfrak{p}\left(  C_{1}\right)  \in C_{0}.$ Hence $\mathfrak{s}\left(
C_{0}\right)  =\mathfrak{p}\left(  C_{1}\right)  .$

We also proved that $C$ is completely ordered by $\supseteq$. So it is
anti-isomorphic to an interval $\lfloor0,\beta\rfloor$ of transfinite numbers.
Thus subspaces in $C$ are indexed by transfinite numbers and $C$ is a strictly
decreasing transfinite sequence $\left\{  Y_{\alpha}\right\}  _{\alpha
\leq\beta}$ of closed subspaces (`strictly' means that $Y_{\alpha^{\prime}%
}\neq Y_{\alpha}$ if $\alpha^{\prime}\neq\alpha$).

b) Let $C\subset C_{1}\subseteq\lbrack\mathfrak{p}(C),\mathfrak{s}(C)]_{G}.$
Let $Y\in C_{1}\diagdown C.$ Then $\mathfrak{p}(C)\subsetneqq Y\subsetneqq
\mathfrak{s}(C).$ The chain $C^{\prime}=\{Z\in C$: $Y\subseteq Z\}$ is not
empty, as $\mathfrak{s}(C)\in C^{\prime},$ and $\mathfrak{p}(C^{\prime})\in
C^{\prime},$ as $C$ is $\mathfrak{p}$-complete. As $C$ is a lower finite-gap
chain and $\mathfrak{p}(C^{\prime})\neq\mathfrak{p}(C)$, there is $Y_{0}\in C$
such that $[Y_{0},\mathfrak{p}(C^{\prime})]_{C}$ is a finite gap. Hence
$Y_{0}\notin C^{\prime},$ so that $Y_{0}\subsetneqq Y\subsetneqq\mathfrak
{p}(C^{\prime}).$ Thus $0<\dim Y/Y_{0}<\infty.$ This implies that $C_{1}$ is a
lower finite-gap chain.

c) By (i), $[\mathfrak{p}(C),\mathfrak{s}(C)]_{G}$ has a maximal $\mathfrak
{p}$-complete chain containing $C.$ By (ii) a) and b), it is a complete, lower
finite-gap chain.

(iii) As $G$ is $\mathfrak{p}$-complete, $C^{\mathfrak{p}}$ is a chain in
$[\mathfrak{p}(C),\mathfrak{s}(C)]_{G}$ larger than $C.$ As $C$ is maximal,
$C=C^{\mathfrak{p}}.$ By (ii) a), $C$ is complete.

(iv) As $G$ is $\mathfrak{p}$-complete, $\mathfrak{p}(G),\mathfrak{s}(G)\in
G$. Consider the set $\mathcal{G}$ of all lower finite-gap chains $C$ in $G$
containing $C_{0}$ and maximal in the interval $[\mathfrak{p}(C),\mathfrak
{s}(G)]_{G}.$ By (ii) c)$,$ $\mathcal{G}$ is not empty. It is partially
ordered by inclusion. Let $\{C_{\lambda}\}_{\lambda\in\Lambda}$ be a linearly
ordered subset of $\mathcal{G}.$ By (iii), $\mathfrak{p}(C_{\lambda})\in
C_{\lambda}.$ Set $Y_{\lambda}=\mathfrak{p}(C_{\lambda}),$ $Y_{\Lambda
}=\mathfrak{p}\{Y_{\lambda}$: $\lambda\in\Lambda\}$ and%
\[
C_{\Lambda}=\left(  \underset{\lambda\in\Lambda}{\cup}C_{\lambda}\right)  \cup
Y_{\Lambda}.
\]
Then $C_{\Lambda}$ is a chain, $C_{0}\in C_{\Lambda}$ and $\mathfrak
{p}(C_{\Lambda})=Y_{\Lambda}\in C_{\Lambda}.$ Each $V\in C_{\Lambda},$ $V\neq
Y_{\Lambda},$ lies in some $C_{\lambda}.$ If $Y_{\lambda}\neq V\in C_{\lambda
},$ there is $W\in C_{\lambda}$ such that $W\subset V$ and $0<\dim\left(
V/W\right)  <\infty.$ If $V=Y_{\lambda}\neq Y_{\Lambda},$ there is $\mu
\in\Lambda$ such that $V\in C_{\lambda}\subsetneqq C_{\mu}$ and $V\neq Y_{\mu
}.$ Hence, as in the previous case, there is $W\in C_{\mu}$ such that
$W\subset V$ and $0<\dim\left(  V/W\right)  <\infty.$ Thus $C_{\Lambda}$ is a
lower finite-gap chain.

Let us show that $C_{\Lambda}$ is a maximal chain in $[Y_{\Lambda}%
,\mathfrak{s}(G)]_{G}.$ If not, then there is $V\notin C_{\Lambda},$
$V\in\lbrack Y_{\Lambda},\mathfrak{s}(G)]_{G}$ such that $C_{\Lambda}\cup V$
is a chain$.$ As $Y_{\Lambda}\subsetneqq V,$ there is $\lambda\in\Lambda$ such
that $Y_{\lambda}\subsetneqq V.$ Then $C_{\lambda}\cup V$ is a chain in
$[Y_{\lambda},\mathfrak{s}(G)]_{G}$ larger than $C_{\lambda}.$ This
contradiction shows that the chain $C_{\Lambda}$ is maximal in $[Y_{\Lambda
},\mathfrak{s}(G)]_{G}.$

Clearly, $C_{\Lambda}\leq C^{\prime}$ for each majorant $C^{\prime}$ of
$\{C_{\lambda}\}_{\lambda\in\Lambda}$ in $\mathcal{G}.$ Thus each linearly
ordered subset of $\mathcal{G}$ has a supremum in $\mathcal{G}.$ By Zorn's
Lemma, $\mathcal{G}$ has a maximal element $C$ which is a lower finite-gap
chain containing $C_{0}$ and maximal in the interval $[\mathfrak
{p}(C),\mathfrak{s}(G)]_{G}.$

Let $G$ be a lower finite-gap family. If $\mathfrak{p}\left(  C\right)
\neq\mathfrak{p}(G)$ then, as $G$\emph{ }is a lower finite-gap family, there
is $W\in G$ such that $W\subsetneqq\mathfrak{p}(C)$ and $\dim(\mathfrak
{p}(C)/W)<\infty.$ Choose a finite maximal chain $C_{1}$ in $[W,\mathfrak
{p}(C)]_{G}.$ Then the chain $C\cup C_{1}$ belongs to $\mathcal{G}$ and larger
than $C$ --- a contradiction. Thus $\mathfrak{p}\left(  C\right)
=\mathfrak{p}(G).$
\end{proof}

\begin{theorem}
\label{T2.2}A $\mathfrak{p}$-complete family $G$ of closed subspaces of $X$ is
a lower finite-gap family if and only if there is a maximal\emph{,} lower
finite-gap chain $C$ in $G$ with $\mathfrak{p(}C)=\mathfrak{p}(G)$ and
$\mathfrak{s}(C)=\mathfrak{s}(G)$.
\end{theorem}

\begin{proof}
$\Longrightarrow$ follows from Lemma \ref{L2.2}.

$\Longleftarrow$ Let a chain $C$ satisfy the conditions of the theorem and let
$\mathfrak{p}\left(  G\right)  \neq Z\in G$. We need to show that there exists
$Y_{1}\subset Z$ such that $0<\dim(Z/Y_{1})<\infty.$ It is evident if $Z\in
C.$ Let $Z\notin C.$ The chain $C_{1}=\left\{  Y\in C\text{: }Z\subseteq
Y\right\}  $ is not empty because at least $\mathfrak{s}\left(  G\right)  \in
C_{1}$. Clearly, $Z\subseteq\mathfrak{p}\left(  C_{1}\right)  $. As
$\mathfrak{p}\left(  C_{1}\right)  \neq\mathfrak{p}(G)=\mathfrak{p}(C)$, there
is a subspace $Y$ in $C$ with $0<\dim(\mathfrak{p}\left(  C_{1}\right)
/Y)<\infty$. Hence, by Lemma \ref{Closed}, $Y_{1}=Y\cap Z$ has finite
codimension in $Z$. Since $Y\notin C_{1}$, this codimension is non-zero. As
$G$ is $\mathfrak{p}$-complete, $Y_{1}=\mathfrak{p}(\{Y,Z\})\in G$. Thus $G$
is a lower finite-gap family.
\end{proof}

We will show now that complete, lower finite-gap families of subspaces of $X$
induce complete, lower finite-gap families of subspaces on closed subspaces of $X.$

\begin{corollary}
\label{pi}Let $G$ be a $\mathfrak{p}$-complete\emph{,} lower finite-gap family
of closed subspaces of $X.$ Then
$G\cap W:=\left\{  Y\cap W:Y\in G\right\}  $ is a $\mathfrak{p}$-complete,
lower finite-gap family, for every closed subspace $W$ of $X$.
\end{corollary}

\begin{proof}
We have $\mathfrak{p}(G\cap W)=\mathfrak{p}(G)\cap W$. If $Y\cap W=p(G)\cap W$
for all $Y\in G$, then $G\cap W$ consists of one element and our corollary holds.

Let $Z\in G$ be such that
\[ 
Z\cap W\neq\mathfrak{p}(G)\cap W.
\]
Set
$G_{1}=\left\{  T\in G\text{: }T\cap W=Z\cap W\right\}  $. Then $\mathfrak
{p}(G_{1})\in G$ and $\mathfrak{p}\left(  G_{1}\right)  \cap W=Z\cap
W\neq\mathfrak{p}(G)\cap W$. Hence $\mathfrak{p}\left(  G\right)
\subsetneqq\mathfrak{p}(G_{1})\in G_{1}.$ As $G$ is a lower finite-gap family,
there is $Y\in G$ such that $Y\subset\mathfrak{p}\left(  G_{1}\right)  $ and
$0<\dim(\mathfrak{p}\left(  G_{1}\right)  /Y)<\infty$. Replacing in Lemma
\ref{Closed} $X$ by $\mathfrak{p}\left(  G_{1}\right)  $ and $Z$ by
$\mathfrak{p}\left(  G_{1}\right)  \cap W$, we get that $Y\cap\left(
\mathfrak{p}\left(  G_{1}\right)  \cap W\right)  =Y\cap W$ has finite
codimension in $\mathfrak{p}\left(  G_{1}\right)  \cap W=Z\cap W$. As $Y\in
G\diagdown G_{1}$, we have $Y\cap W\subsetneqq Z\cap W$. Thus $0<\dim\left(
Z\cap W\right)  /\left(  Y\cap W\right)  <\infty$. This means that $G\cap W$
is a lower finite-gap family.

As the family $G$ is $\mathfrak{p}$-complete, it follows that the family
$G\cap W$ is also $\mathfrak{p}$-complete.
\end{proof}

Note that if $G$ is a lower finite-gap family of closed subspaces,
$G^{\mathfrak{p}}$ is not necessarily a lower finite-gap family. For example,
the family $G$ of all subspaces of finite dimension in $X$ is a lower
finite-gap family and $G^{\mathfrak{p}}=G\cup X$ is not a lower finite-gap family.

\begin{proposition}
\label{ui}Let $G=\cup_{\lambda\in\Lambda}G_{\lambda}$ where each $G_{\lambda}$
is a $\mathfrak{p}$-complete$,$ lower finite-gap family of closed subspaces of
$X$. If $X\in G_{\lambda},$ for all $\lambda\in\Lambda,$ then $G^{\mathfrak
{p}}$ is a lower finite-gap family.
\end{proposition}

\begin{proof}
Let $Y\in G^{\mathfrak{p}}$ and $\mathfrak{p}\left(  G\right)  \subsetneqq Y$.
Then $Y=\mathfrak{p}\left(  G^{\prime}\right)  $ for some $\varnothing\neq
G^{\prime}\subsetneqq G.$ Set $\Gamma_{\lambda}=G_{\lambda}\diagdown
(G^{\prime}\cap G_{\lambda})$ for each $\lambda\in\Lambda.$ Then $G\diagdown
G^{\prime}=\cup_{\lambda}\Gamma_{\lambda}$ and%
\[
\mathfrak{p}\left(  G\right)  =\mathfrak{p}\left(  G\diagdown G^{\prime
}\right)  \cap\mathfrak{p}\left(  G^{\prime}\right)  =(\cap_{\lambda\in
\Lambda}\mathfrak{p}\left(  \Gamma_{\lambda}\right)  )\cap Y=\cap_{\lambda
\in\Lambda}(\mathfrak{p}\left(  \Gamma_{\lambda}\right)  \cap Y).
\]
Hence $\mathfrak{p}\left(  \Gamma_{\lambda}\right)  \cap Y\neq Y,$ for some
$\lambda.$ Thus
\[
\mathfrak{p}\left(  G_{\lambda}\cap Y\right)  =\mathfrak{p}(G_{\lambda})\cap
Y=\mathfrak{p}(\Gamma_{\lambda})\cap\left(  \mathfrak{p}(G^{\prime}\cap
G_{\lambda})\cap Y\right)  =\mathfrak{p}(\Gamma_{\lambda})\cap Y\neq Y.
\]
By Corollary \ref{pi}, $G_{\lambda}\cap Y$ is a lower finite-gap family and,
as $X\in G_{\lambda},$ we have $Y=X\cap Y\in G_{\lambda}\cap Y.$ Hence there
is $Z\in G_{\lambda}$ such that $0<\dim\left(  Y/(Z\cap Y)\right)  <\infty$.
As $Z\cap Y\in G^{\mathfrak{p}},$ it follows that $G^{\mathfrak{p}}$ is a
lower finite-gap family.
\end{proof}

We need now the following auxiliary result.

\begin{lemma}
\label{nest} Let $\{X_{\lambda}:\lambda\in\Lambda\}$ be a family of closed
subspaces of a separable Banach space $X$.

\begin{itemize}
\item [$\mathrm{(i)}$]If $\cap_{\lambda\in\Lambda}X_{\lambda}=\{0\}$ then
there is a sequence $\{\lambda_{n}\in\Lambda:n\in\mathbb{N}\}$ such that
$\cap_{n=1}^{\infty}X_{\lambda_{n}}=\{0\}$.

\item[$\mathrm{(ii)}$] If $\overline{\sum_{\lambda\in\Lambda}X_{\lambda}}=X$
then there is a sequence $\{\lambda_{n}\in\Lambda:n\in\mathbb{N}\}$ such that
$\overline{\sum_{n=1}^{\infty}X_{\lambda_{n}}}=X$.
\end{itemize}
\end{lemma}

\begin{proof}
(i) For each $\lambda\in\Lambda$, set $W_{\lambda}=\{f\in X^{\ast}$: $\Vert
f\Vert\leq1,f(x)=0$ for all $x\in X_{\lambda}\}$. Then $W=\cup_{\lambda
\in\Lambda}W_{\lambda}$ is a subset of the unit ball $\mathbf{B}$ of $X^{\ast
}.$ As $X$ is separable, $\mathbf{B}$\textbf{ }is a separable metric space in
the weak* topology (see \cite[Section 4.1.7]{Sch}). Hence $W$ has a weak*
dense sequence $\{f_{n}$: $n\in\mathbb{N}\}$. It follows that
\[
\cap_{n}\ker(f_{n})=\cap_{f\in W}\ker\left(  f\right)  =\cap_{\lambda
}X_{\lambda}=\{0\}.
\]
Choosing an index $\lambda_{n}$ such that $f_{n}\in W_{\lambda_{n}}$ for each
$n$, we get $\cap_{n=1}^{\infty}X_{\lambda_{n}}\subseteq\cap_{n}\ker
(f_{n})=\{0\}$.

(ii) Set $E=\cup_{\lambda\in\Lambda}X_{\lambda}$. Then $X$ is the closed
linear span of $E$. As $E$ is separable, it has a dense sequence $\{x_{n}$:
$n\in\mathbb{N}\}$. Choosing $\lambda_{n}$ such that $x_{n}\in X_{\lambda_{n}%
}$ for each $n$, we get $\overline{\sum_{n=1}^{\infty}X_{\lambda_{n}}}=X$.
\end{proof}

\begin{theorem}
\label{T2.5}Let $G$ be a family of closed subspaces of finite codimension in
$X$. Then there is a maximal$,$ lower finite-gap chain $C$ of subspaces in
$G^{\mathfrak{p}}$ with $\mathfrak{p}(C)=\mathfrak{p}(G)$ and $\mathfrak
{s}(C)=\mathfrak{s}(G).$

Suppose that the quotient space $\mathfrak{s}\left(  G\right)  /\mathfrak
{p}\left(  G\right)  $ is separable and infinite-dimensional\emph{.} Then
there are subspaces $\left\{  V_{k}\right\}  _{k=1}^{\infty}$ in $G$ such that
their finite intersections $Y_{n}=\cap_{k=1}^{n}V_{k}$ together with
$\mathfrak{s}\left(  G\right)  $ form a decreasing complete\emph{, }lower
finite-gap chain between $\mathfrak{s}\left(  G\right)  $ and $\mathfrak
{p}\left(  G\right)  =\cap_{n=1}^{\infty}Y_{n}.$
\end{theorem}

\begin{proof}
The existence\textbf{ }of the chain\textbf{ }$C$ follows from Lemmas
\ref{L2.4} and \ref{L2.2}.

Let $\mathfrak{s}\left(  G\right)  /\mathfrak{p}\left(  G\right)  $ be
separable and infinite-dimensional. First assume that $\mathfrak{s}\left(
G\right)  =X$ and $\mathfrak{p}\left(  G\right)  =\{0\}$. By Lemma \ref{nest},
there is a sequence $\left\{  W_{i}\right\}  $ in $G$ such that $\cap
_{i=1}^{\infty}W_{i}=\{0\}$. Taking a subsequence, if necessary, one can
assume that $W_{i_{1}}\neq W_{i_{1}}\cap W_{i_{2}}\neq\ldots\neq\cap_{k=1}%
^{n}W_{i_{k}}\neq\ldots$ and $\cap_{k=1}^{\infty}W_{i_{k}}=\cap_{i=1}^{\infty
}W_{i}$. Setting $V_{k}=W_{i_{k}}$ for every $k$, we get the required sequence.

The general case can be reduced to the above one if we take $\mathfrak
{s}\left(  G\right)  /\mathfrak{p}\left(  G\right)  $ instead of $X$. By the
above, we can find a required sequence $\left\{  V_{i}^{\prime}\right\}  $ for
the family $G^{\prime}=\{V/\mathfrak{p}\left(  G\right)  $: $V\in G\}$. Taking
the preimages $V_{i}$ of $V_{i}^{\prime}$ in $\mathfrak{s}\left(  G\right)  $,
we obtain the required sequence.
\end{proof}

Let $G$ be a $\mathfrak{p}$-complete family of closed subspaces of $X$ and
$\mathfrak{s}(G)=X.$ By Lemma \ref{L2.2}, $G$ has a maximal, lower finite-gap
chain $C$ with $\mathfrak{s}(C)=X.$ Let $G_{\text{f}}$ be the subset of $G$
that consists of all $Y\in G$ such that there is a $\mathfrak{p}$-complete,
lower finite-gap chain $C_{Y}$ with $\mathfrak{s}(C_{Y})=X$ and $\mathfrak
{p}(C_{Y})=Y.$ Set%
\begin{equation}
\Delta_{G}=\mathfrak{p}(G_{\text{f}}). \label{6.d}%
\end{equation}

\begin{theorem}
\label{T6.x}\emph{(i)\ }The subset $G_{\emph{f}}$ of $G$ is a lower finite-gap family\emph{.}
\begin{enumerate}
 \item[(ii)]
 $\Delta_{G}=\mathfrak{p}(C)$ for each maximal\emph{,} lower
finite-gap chain $C$ in $G$ with $\mathfrak{s}(C)=X.$
\end{enumerate}
\end{theorem}

\begin{proof}
Let $C,C^{\prime}$ be maximal, lower finite-gap chains in $G$ with
$\mathfrak{s}(C)=\mathfrak{s}(C^{\prime})=X.$ Set $Y=\mathfrak{p}(C)$ and
$Z=\mathfrak{p}(C^{\prime})$. Then $Y,Z\in G.$ By Corollary \ref{pi}, $C\cap
Z=\{U\cap Z$: $U\in C\}$ is a $\mathfrak{p}$-complete, lower finite-gap chain
in $Z.$ Suppose that $Z\nsubseteq Y.$ Then $\mathfrak{s}(C\cap Z)=Z,$ as $X\in
C,$ and $\mathfrak{p}(C\cap Z)=Y\cap Z\neq Z.$ Hence $C^{\prime}\cup(C\cap Z)$
is a $\mathfrak{p}$-complete, lower finite-gap family larger than $C^{\prime}$
--- a contradiction. Therefore $Z\subseteq Y.$ Similarly, $Y\subseteq Z.$ Thus
$Y=Z.$ This immediately implies (ii).

If $Y\in G_{\text{f}}$ and $Y\neq\Delta_{G}$ then the chain $C_{Y}$ is not
maximal. By Lemma \ref{L2.2}, there is a maximal\emph{,} lower finite-gap
chain $C$ in $G$ with $\mathfrak{s}(C)=X$ that contains $C_{Y}.$ Hence there
is a subspace $Z\in C$ such that $Z\subsetneqq Y$ and $\dim Y/Z<\infty.$ As
$Z\cup C_{Y}$ is a $\mathfrak{p}$-complete$,$ lower finite-gap family, $Z\in
G_{\text{f}}.$ This proves (i).
\end{proof}

Note that $G$ may have many different maximal, lower finite-gap chains
starting at $X$. However, they all end at the same subspace $\Delta_{G}$.

Let $L$ be a Lie algebra of operators on a Banach space $X.$ The set Lat $L$
of all closed subspaces of $X$ invariant for all operators in $L$ is\emph{
}$\mathfrak{p}$-complete.\emph{ }Let Lat$_{\text{cf}}L=\{Y\in$ Lat $L:$ $Y$
has finite codimension in $X\}.$ Then Lemma \ref{L2.2} and Theorems \ref{T2.5}
and \ref{T6.x} yield

\begin{corollary}
\label{C6.5}\emph{(i) }There is a subspace $\Delta_{L}\in$ \emph{Lat} $L$ such
that $\mathfrak{p}(C)=\Delta_{L}$ for each maximal\emph{,} lower finite-gap
chain $C$ of invariant subspaces of $L$ with $s(C)=X,$ and $\Delta_{L}$ has no
invariant subspaces of finite codimension. If \emph{Lat} $L$ is a lower
finite-gap family then $\Delta_{L}=\{0\}.$
\begin{enumerate}
 \item[(ii)]
 If $X$ is separable then there is a sequence $\{Y_{n}%
\}_{n=0}^{\infty}$ of subspaces in \emph{Lat}$_{\text{\emph{cf}}}L$ such that
$...Y_{n+1}\subset Y_{n}\subset...\subset Y_{0}=X$ and $\cap_{n}%
Y_{n}=\mathfrak{p}($\emph{Lat}$_{\text{\emph{cf}}}L).$
\end{enumerate}
\end{corollary}

\begin{definition}
\label{modd}Let $G$ and $G^{\prime}$ be families of closed subspaces of $X$.
Then $G$ is called a \emph{lower finite-gap family modulo} $G^{\prime}$ if$,$
for each $Z$ in\emph{ }$G,$\emph{ }$Z\neq\mathfrak{p}(G\cup G^{\prime}%
),$\emph{ }there is $Y\in G\cup G^{\prime}$ such that $Y\subset Z$ and\emph{
}$0<\dim(Z/Y)<\infty.$
\end{definition}

Combining this definition and Definition \ref{D2.1}, we obtain

\begin{lemma}
\label{mod}Let $G$ and $G^{\prime}$ be families of closed subspaces of $X$. If
$G$ is a lower finite-gap family modulo $G^{\prime}$ and $G^{\prime}$ is a
lower finite-gap family$,$ then $G\cup G^{\prime}$ is a lower finite-gap family.
\end{lemma}

\subsection{Preradicals corresponding to finite-dimensional Lie subalgebra-multi\-fun\-ctions}

For a Banach Lie algebra $\mathcal{L},$ the sequences $\{\mathcal{L}%
^{[n+1]}\}$ and $\{\mathcal{L}_{[n+1]}\}$ of closed characteristic Lie ideals%
\begin{equation}
\mathcal{L}^{[1]}=\mathcal{L}_{[0]}=\mathcal{L},\text{ }\mathcal{L}%
^{[n+1]}=\overline{[\mathcal{L},\mathcal{L}^{[n]}]}\text{ and }\mathcal{L}%
_{[n+1]}=\overline{[\mathcal{L}_{[n]},\mathcal{L}_{[n]}]}, \label{5.3}%
\end{equation}
for $n\in\mathbb{N,}$ decrease$;$ $\mathcal{L}$ is \textit{nilpotent}, if
$\mathcal{L}^{[n]}=\{0\}$ for some $n,$ and \textit{solvable} if
$\mathcal{L}_{[n]}=\{0\}$ for some $n.$

Denote by $\mathfrak{L}^{\mathrm{f}}$ the class of all finite-dimensional Lie
algebras and by $\mathbf{L}^{\mathrm{f}}$ the subcategory of $\mathbf{L}$ of
all such algebras. As in (\ref{4.1'})---(\ref{4.3'}) and Definition
\ref{D3.1}, we define lower stable, upper stable and balanced preradicals$,$
under radicals, over radicals and radicals on $\mathbf{L}^{\text{f}}.$

For $\mathcal{L}\in\mathfrak{L}^{\mathrm{f}},$ denote by $\mathrm{rad}\left(
\mathcal{L}\right)  $ its maximal solvable Lie ideal$.$ The map rad:
$\mathcal{L}\in\mathfrak{L}^{\mathrm{f}}\mapsto\mathrm{rad}\left(
\mathcal{L}\right)  $ is a radical in $\mathbf{L}^{\mathrm{f}}.$ A Lie algebra
$\mathcal{L}$ is called \textit{semisimple} if it is rad-semisimple$;$
$\mathcal{L}$ is semisimple if and only if it is a direct sum of simple
algebras. We preserve this terminology when dealing with finite-dimensional
subalgebras of Banach Lie algebras.

Each $\mathcal{L}\in\mathfrak{L}^{\mathrm{f}}$ is the semidirect product
(Levi-Maltsev decomposition) of a semisimple Lie subalgebra $N_{\mathcal{L}}$
(uniquely defined up to an inner automorphism) and the largest solvable Lie
ideal rad($\mathcal{L)}$%
\begin{equation}
\mathcal{L}=N_{\mathcal{L}}\oplus^{\text{ad%
$\vert$%
}_{\text{rad}(\mathcal{L)}}}\operatorname*{rad}\left(  \mathcal{L}\right)  .
\label{5.2}%
\end{equation}

Recall that if $\Gamma$ is a Lie subalgebra-multifunction or
ideal-multi\-fun\-ction on $\mathfrak{L}$ then, for each $\mathcal{L}%
\in\mathfrak{L},$ $\Gamma_{\mathcal{L}}$ is a family of
closed Lie subalgebras (ideals) of $\mathcal{L.}$ In the rest of this
subsection we will consider the following four Lie
subalgebra-multifunctions $\Gamma$ on $\mathfrak
{L}$:

\begin{itemize}
\item [$\mathrm{1)}$]\textrm{A}$^{\text{\textrm{sem}}}:$ each family
\textrm{A}$_{\mathcal{L}}^{\text{\textrm{sem}}}$ consists of all
finite-dimensional semi\-simple Lie subalgebras of $\mathcal{L;}$

\item[$\mathrm{2)}$] \textrm{I}$^{\text{\textrm{sem}}}:$ each family
\textrm{I}$_{\mathcal{L}}^{\text{\textrm{sem}}}$ consists of all
finite-dimensional semisimple Lie ideals of $\mathcal{L}$;

\item[$\mathrm{3)}$] \textrm{I}$^{\text{\textrm{sol}}}:$ each family
\textrm{I}$_{\mathcal{L}}^{\text{\textrm{sol}}}$ consists of all
finite-dimensional solvable Lie ideals of $\mathcal{L}$;

\item[$\mathrm{4)}$] \textrm{I}$^{\text{\textrm{fin}}}:$ each family
\textrm{I}$_{\mathcal{L}}^{\text{\textrm{fin}}}$\textrm{ }consists of all
finite-dimensional Lie ideals of $\mathcal{L}$.\smallskip
\end{itemize}

We\textbf{ }will study the corresponding radicals and describe their
restrictions to\textbf{ }$\mathbf{L}^{\text{f}}.$

\begin{proposition}
\label{P6.1}\emph{(i) }The\emph{ }Lie subalgebra-multifunctions $\Gamma:$
\emph{A}$^{\text{\emph{sem}}},$\emph{ I}$^{\text{\emph{sem}}},$\emph{
I}$^{\text{\emph{sol}}},$ \emph{I}$^{\text{\emph{fin}}}$ on $\mathfrak{L}$ are
strictly direct and lower stable $($see Definition $\ref{D4.7}),$ so that the
corresponding maps $S_{\Gamma}$ are lower stable preradicals$.$
\begin{enumerate}
 \item[(ii)]
The multifunctions $\Gamma:$ \emph{A}$^{\text{\emph{sem}}}$\emph{
}and \emph{I}$^{\text{\emph{sem}}}$ are balanced $($see Definition
$\ref{D4.7}),$ so that the corresponding maps $S_{\Gamma}$ are under radicals
and $S_{\emph{I}^{\text{\emph{sem}}}}\leq S_{\text{\emph{A}}^{\text{\emph{sem}%
}}}.$
\end{enumerate}
\end{proposition}

\begin{proof}
(i) Let $f:\mathcal{L}\longrightarrow\mathcal{M}$ be a morphism in
$\overline{\mathbf{L}}$ and let $L$ be a finite-dimensional Lie subalgebra of
$\mathcal{L}$. Then $\overline{f(L)}=f(L)$ is a finite-dimensional Lie
subalgebra of $\mathcal{M}$. If $L$ is a Lie ideal of $\mathcal{L}$, then
$f(L)$ is a Lie ideal of $f\left(  \mathcal{L}\right)  $. As $f\left(
\mathcal{L}\right)  $ is dense in $\mathcal{M}$, $f(L)$ is a Lie ideal of
$\mathcal{M}$.

If $L$ is solvable, $f(L)$ is solvable. If $L$ is simple, $f\left(  L\right)
$ is either $\{0\}$ or simple. If $L$ is semisimple, it is a finite direct sum
of simple Lie algebras. Hence $f\left(  L\right)  $ is either $\{0\}$ or a
semisimple Lie subalgebra of $\mathcal{M}$. This shows that all multifunctions
are strictly direct. By Proposition \ref{L2}, all $S_{\Gamma}$ are preradicals.

Set $I_{\mathcal{L}}=S_{\Gamma}\left(  \mathcal{L}\right)  =\mathfrak
{s}(\Gamma_{\mathcal{L}})$. If $L\in\Gamma_{\mathcal{L}},$ it follows from
(\ref{F0'}) that $L\subseteq I_{\mathcal{L}}.$ Hence $L\in\Gamma
_{I_{\mathcal{L}}}$. Thus $\Gamma_{\mathcal{L}}\subseteq\Gamma_{I_{\mathcal{L}%
}},$ so that all multifunctions are lower stable (see Definition \ref{D4.7}).
By Theorem \ref{Cmain}(i), all preradicals $S_{\Gamma}$ are lower stable.

(ii) Let $I\vartriangleleft\mathcal{L}$. If $\Gamma=$ A$^{\text{sem}}$ then
$\Gamma_{I}\subseteq\Gamma_{\mathcal{L}}$. Let $\Gamma=$ I$^{\text{sem}}$ and
$J$ be a semisimple Lie ideal of $I$. Then $J=\left[  J,J\right]  $. By Lemma
\ref{L3.1}(iii), $J$ is a characteristic Lie ideal of $I$ and, by Lemma
\ref{L3.1}(i), $J$ is a semisimple Lie ideal of $\mathcal{L}$. Hence
$\Gamma_{I}\subseteq\Gamma_{\mathcal{L}}$. Thus the multifunctions
A$^{\text{sem}}$ and I$^{\text{sem}}$ are balanced. Hence, by Theorem
\ref{Cmain}(ii), all $S_{\Gamma}$ are balanced. Thus they are under radicals.
Clearly, $S_{\text{I}^{\text{sem}}}\leq S_{\text{A}^{\text{sem}}}$.
\end{proof}

Combining this with Theorems \ref{under-over} and \ref{T2} yields

\begin{corollary}
\label{C6.2}\emph{(i) }If $\Gamma=$ \emph{A}$^{\text{\emph{sem}}}$\emph{ }or
$\Gamma=$\emph{ I}$^{\text{\emph{sem}}}$ then the map $S_{\Gamma}^{\ast}$ is a radical.

\begin{itemize}
\item [$\mathrm{(ii)}$]\emph{ }If $\Gamma=$ \emph{I}$^{\text{\emph{fin}}}%
$\emph{ }or $\Gamma=$ \emph{I}$^{\text{\emph{sol}}}$ then the map $S_{\Gamma
}^{s}$ is an under radical and $(S_{\Gamma}^{s})^{\ast}$ is a radical.
\end{itemize}
\end{corollary}

\begin{corollary}
\label{C6.3}\emph{(i) }A Banach Lie algebra $\mathcal{L}$ is $S_{\emph{A}%
^{\text{\emph{sem}}}}^{\ast}$-radical (respectively, $S_{\emph{I}%
^{\text{\emph{sem}}}}^{\ast}$-radical$)$ if and only if$,$ for each closed
proper Lie ideal $I$ of $\mathcal{L}$, the quotient $\mathcal{L}/I$ contains a
finite-dimensional semisimple Lie subalgebra $($respect\-iv\-e\-ly, ideal$)$.
\begin{enumerate}
 \item[(ii)] 
If $\mathcal{L}$ is an $S_{\emph{A}^{\emph{sem}}}^{\ast}$-radical
or an $S_{\emph{I}^{\emph{sem}}}^{\ast}$-radical$,$ then $\mathcal{L}%
=\overline{\left[  \mathcal{L},\mathcal{L}\right]  }$.
\end{enumerate}
\end{corollary}

\begin{proof}
Part (i) follows from Theorem \ref{dir-quot}(i) and Corollary \ref{C6.2}(i).

(ii) If $\overline{[\mathcal{L},\mathcal{L}]}\neq\mathcal{L}$ then
$\mathcal{L}$ is not $S_{\text{A}^{\text{sem}}}^{\ast}$-radical (or
$S_{\text{I}^{\text{sem}}}^{\ast}$-radical) because the algebra $\mathcal{L}%
/\overline{[\mathcal{L},\mathcal{L}]}$ cannot contain semisimple subalgebras
(see (i)).
\end{proof}

Let $\mathcal{L}\in\mathfrak{L}^{\text{f}}$. Then $\mathcal{L}_{\left[
k+1\right]  }=\mathcal{L}_{\left[  k\right]  }$ for some $k.$ Set%
\begin{equation}
\mathfrak{P}\left(  \mathcal{L}\right)  =\cap\mathcal{L}_{\left[  k\right]  }.
\label{6.2}%
\end{equation}
The next theorem describes the restriction of the preradical $S_{\text{A}%
^{\text{sem}}}$ to $\mathbf{L}^{\text{f}}$.

\begin{theorem}
\label{Levi}\emph{(i) }The restriction of $S_{\emph{A}^{\text{\emph{sem}}}}$
to $\mathbf{L}^{\text{\emph{f}}}$ is a radical$.$

\begin{itemize}
\item [$\mathrm{(ii)}$]
For each $\mathcal{L}\in\mathfrak{L}^{\mathrm{f}},$
$S_{\emph{A}^{\text{\emph{sem}}}}\left(  \mathcal{L}\right)  =\mathfrak
{P}\left(  \mathcal{L}\right)  $ and it is the smallest characteristic Lie
ideal of $\mathcal{L}$ that contains all Levi subalgebras $N_{\mathcal{L}}$
\emph{(see }$(\ref{5.2}))$.
\item[$\mathrm{(iii)}$] 
A Lie algebra $\mathcal{L}\in\mathfrak{L}^{\mathrm{f}%
}$ is $S_{\emph{A}^{\text{\emph{sem}}}}$-semisimple if and only if
$\mathcal{L}$ is solvable.
\end{itemize}
\end{theorem}

\begin{proof}
(i) Set $\Gamma=$ A$^{\text{sem}}$ and $I_{\mathcal{L}}=S_{\Gamma}\left(
\mathcal{L}\right)  =\mathfrak{s}(\Gamma_{\mathcal{L}}).$ By Proposition
\ref{P6.1}, $S_{\Gamma}$ is an under radical. If $S_{\Gamma}|\mathbf{L}%
^{\emph{f}}$ is not a radical, $\mathfrak{s}(\Gamma_{\mathcal{L}%
/I_{\mathcal{L}}})=S_{\Gamma}(\mathcal{L}/I_{\mathcal{L}})\neq\{0\}$ for some
$\mathcal{L}\in\mathfrak{L}^{\text{f}}.$ Hence $\mathcal{L}/I_{\mathcal{L}}$
contains a semisimple Lie subalgebra $M\neq\{0\}$. Let $q$: $\mathcal{L}%
\longrightarrow\mathcal{L}/I_{\mathcal{L}}$ be the quotient map and
$L=q^{-1}(M)$. Let $L=N_{L}\dotplus\operatorname*{rad}\left(  L\right)  $ be
the Levi-Maltsev decomposition, where $N_{L}$ is a semisimple Lie subalgebra
of $\mathcal{L}$. As $N_{L}\subseteq I_{\mathcal{L}}$, we have that
$M=q\left(  L\right)  =q\left(  \operatorname*{rad}\left(  L\right)  \right)
$ is solvable, a contradiction.

(ii) Let $J$ be the minimal characteristic Lie ideal of $\mathcal{L}$ that
contains some Levi subalgebra $N_{\mathcal{L}}$. Since $N_{\mathcal{L}}%
\in\Gamma_{\mathcal{L}},$ we have $N_{\mathcal{L}}\subseteq\mathfrak{s}%
(\Gamma_{\mathcal{L}})=I_{\mathcal{L}}.$ As $I_{\mathcal{L}}$ is a
characteristic Lie ideal, $J\subseteq I_{\mathcal{L}}$. Let $M$ be a
semisimple Lie subalgebra of $\mathcal{L}$. By \cite[Corollary 1.6.8.1]{Bo},
there is $x$ in $\mathcal{L}$ such that $\exp\left(  \operatorname*{ad}\left(
x\right)  \right)  \left(  M\right)  \subseteq N_{\mathcal{L}}$. As $J$ is a
characteristic Lie ideal of $\mathcal{L}$, $M\subseteq\exp\left(
\operatorname*{ad}\left(  -x\right)  \right)  \left(  N_{\mathcal{L}}\right)
\subseteq\exp\left(  \operatorname*{ad}\left(  -x\right)  \right)  J\subseteq
J$. Hence all semisimple Lie subalgebras of $\mathcal{L}$ lie in $J.$
Therefore $I_{\mathcal{L}}\subseteq J$. Thus $I_{\mathcal{L}}=J.$

As $N_{\mathcal{L}}$ is semisimple, $N_{\mathcal{L}}=\left[  N_{\mathcal{L}%
},N_{\mathcal{L}}\right]  ,$ whence $N_{\mathcal{L}}\subseteq\mathcal{L}%
_{\left[  2\right]  }.$ Similarly, $N_{\mathcal{L}}\subseteq\mathcal{L}%
_{\left[  n\right]  }$ for all $n$, so that $N_{\mathcal{L}}\subseteq
\mathfrak{P}\left(  \mathcal{L}\right)  $. As all $\mathcal{L}_{\left[
n\right]  }$ are characteristic Lie ideals of $\mathcal{L}$, $\mathfrak
{P}\left(  \mathcal{L}\right)  $ is a characteristic Lie ideal of
$\mathcal{L}$ that contains $N_{\mathcal{L}}.$ Hence $I_{\mathcal{L}}%
\subseteq\mathfrak{P}\left(  \mathcal{L}\right)  $. As $S_{\Gamma}$ is a
radical, $S_{\Gamma}(\mathcal{L}/I_{\mathcal{L}})=\{0\}.$ Hence, by the
definition of $S_{\Gamma},$ $\mathcal{L}/I_{\mathcal{L}}$ has no semisimple
Lie subalgebras. Thus $\mathcal{L}/I_{\mathcal{L}}$ is solvable. Therefore
$\mathfrak{P}\left(  \mathcal{L}\right)  =\mathcal{L}_{\left[  k\right]
}\subseteq I_{\mathcal{L}}$.

Part (iii) follows from the fact that $\{0\}=S_{\Gamma}\left(  \mathcal{L}%
\right)  =\mathcal{L}_{\left[  k\right]  }$ for some $k$.
\end{proof}

\begin{theorem}
\label{T5.2}\emph{(i) }For each $\mathcal{L}\in\mathfrak{L}^{\mathrm{f}},$
$S_{\emph{I}^{\text{\emph{sem}}}}\left(  \mathcal{L}\right)  $ is the largest
semisimple Lie ideal of $\mathcal{L}$.

\begin{itemize}
\item [$\mathrm{(ii)}$]The restriction of $S_{\emph{I}^{\text{\emph{sem}}}}$
to $\mathbf{L}^{\mathrm{f}}$ is a hereditary radical \emph{(}see
\emph{(\ref{4.4'})).}
\end{itemize}
\end{theorem}

\begin{proof}
(i) Set $\Gamma=I^{\text{sem}}.$ Let $\mathcal{L}\in\mathfrak{L}^{\mathrm{f}%
},$ let $I\vartriangleleft\mathcal{L}$ and $J\vartriangleleft\mathcal{L}$. If
$J$ is simple, then either $J=I$ or $J\cap I=\{0\}.$ If $J\subseteq I$ and $I$
is semisimple, then $I=I_{1}\dotplus...\dotplus I_{n},$ where all $\{I_{i}\}$
are simple Lie ideals of $I$, and $J$ coincides with one of $I_{i}.$ As all
$I_{i}=\left[  I_{i},I_{i}\right]  $, we have from Lemma \ref{L3.1}(iii) that
all $I_{i}$ are simple Lie ideals of $\mathcal{L}$. Let $\{I_{j}\}_{j=1}^{m}$
be the set of \textit{all }simple Lie ideals of $\mathcal{L}$. It follows from
the discussion above that $K=I_{1}\dotplus...\dotplus I_{m}$ is the largest
semisimple Lie ideal of $\mathcal{L}$ and it contains all semisimple Lie
ideals of $\mathcal{L}$. Hence $S_{\Gamma}\left(  \mathcal{L}\right)
=\mathfrak{s}(\Gamma_{\mathcal{L}})=K$.

(ii) As in Theorem \ref{Levi}(i), one can prove that $S_{\Gamma}%
|\mathbf{L}^{\mathrm{f}}$ is a radical. Let $I\vartriangleleft\mathcal{L}$.
Then $S_{\Gamma}\left(  I\right)  $ is the largest semisimple Lie ideal of
$I.$ As $S_{\Gamma}\left(  I\right)  $ is a characteristic Lie ideal of $I$,
by Lemma \ref{L3.1}(i), $S_{\Gamma}\left(  I\right)  $ is the largest
semisimple Lie ideal of $\mathcal{L}$ contained in $I.$ Therefore $S_{\Gamma
}\left(  I\right)  \subseteq S_{\Gamma}\left(  \mathcal{L}\right)  \cap I.$ As
$S_{\Gamma}\left(  \mathcal{L}\right)  \cap I$ is a Lie ideal of the
semisimple Lie algebra $S_{\Gamma}\left(  \mathcal{L}\right)  ,$ it is
semisimple. Hence $S_{\Gamma}\left(  \mathcal{L}\right)  \cap I$ is a
semisimple Lie ideal of $I.$ By (i), $S_{\Gamma}\left(  \mathcal{L}\right)
\cap I\subseteq S_{\Gamma}\left(  I\right)  .$ Thus $S_{\Gamma}\left(
\mathcal{L}\right)  \cap I=S_{\Gamma}\left(  I\right)  ,$ so that $S_{\Gamma}$
is hereditary on $\mathbf{L}^{\text{f}}$.
\end{proof}

To see an example that distinguishes the radicals\textbf{ }$S_{\text{A}%
^{\text{sem}}}|\mathbf{L}^{\text{f}}$\textbf{ }and $S_{\text{I}^{\text{sem}}}%
$\textbf{$|$}$\mathbf{L}^{\text{f}}$, consider the semidirect product
$\mathcal{L}=sl(X)\oplus^{\text{id}}X,$ where $X$ is a finite-dimensional
space and $sl(X)$ the Lie algebra of all operators on $X$ with zero trace. It
has no semisimple ideals, so that I$_{\mathcal{L}}^{\text{sem}}=\varnothing$
and $S_{\text{I}^{\text{sem}}}(\mathcal{L})=\{0\},$ while $S_{\text{A}%
^{\text{sem}}}(\mathcal{L})=\mathcal{L}$ because the only ideal that contains
the semisimple Levi subalgebra $sl(X)\oplus^{\text{id}}\{0\}$ is $\mathcal{L}$ itself.

We call the restriction of\emph{ }$S_{\text{A}^{\text{sem}}}$ to
$\mathbf{L}^{\text{f}}$ the \textit{Levi radical} and denote it by
$R_{\mathrm{Levi}}$:%
\[
R_{\mathrm{Levi}}=S_{\text{A}^{\text{sem}}}|\mathbf{L}^{\text{f}}.
\]
It is \textit{not hereditary}. Indeed, let $L\in\mathbf{L}^{\text{f}}$ be
semisimple and let $\pi$ be an irreducible representation of $L$ on a
finite-dimensional space $X$. Then $\mathcal{L}=L\oplus^{\pi}X$ (see
(\ref{fsemi})) is a Lie algebra and $I=\{0\}\oplus^{\pi}X$ is a Lie ideal of
$\mathcal{L}$. It is easy to see that $\mathfrak{P}(\mathcal{L})=\mathcal{L}$
and $\mathfrak{P}(I)=\{0\}.$ Hence, by Theorem \ref{Levi}(ii),
$R_{\mathrm{Levi}}\left(  \mathcal{L}\right)  =\mathcal{L}$ and
$R_{\mathrm{Levi}}\left(  I\right)  =\{0\}.$ Thus $R_{\mathrm{Levi}}\left(
I\right)  \neq R_{\mathrm{Levi}}\mathcal{(L})\cap I=I$ (see (\ref{4.4'})).

\subsection{Some extensions of classical radicals}

In this section we consider some Lie ideal-multifunctions $\Gamma$ on
$\mathfrak{L}$ related to commutative and solvable ideals.\textbf{ }Although
they generate different preradicals $S_{\Gamma}$, the preradicals $S_{\Gamma
}^{\mathbf{s}}$ corresponding to them (see (\ref{4.4})) often generate equal
radicals that extend the classical radical rad on $\mathbf{L}^{\text{f}}$.

We start with the multifunction I$^{\text{sol}}$ defined above and the
multifunction ``$\mathrm{Abel}$'':%
\[
\text{Abel }=\{\mathrm{Abel}_{\mathcal{L}}\}_{\mathcal{L}\in\mathfrak{L}},
\]
where Abel$_{\mathcal{L}}$ is the family of all commutative Lie ideals of
$\mathcal{L}.$ As in Proposition \ref{P6.1}(i), we have that $\mathrm{Abel}$
is a strictly direct and lower stable multifunction (see Definition
\ref{D4.7}), so that $S_{\mathrm{Abel}}$ is a lower stable preradical in
$\overline{\mathbf{L}}.$ Hence $S_{\mathrm{Abel}}^{\mathbf{s}}$ is an under
radical (Theorem \ref{T2}(i)) and $(S_{\mathrm{Abel}}^{\mathbf{s}})^{\ast}$ is
a radical (Corollary \ref{C4.9}(i)).

\begin{theorem}
\label{T7.2}\emph{(i) }$\mathbf{Sem}(S_{\text{\emph{I}}^{\text{\emph{sol}}}%
}^{\mathbf{s}})=\mathbf{Sem}(S_{\mathrm{Abel}}^{\mathbf{s}})$ and
$\mathcal{L}$ belongs to them if and only if $\mathcal{L}$ has no non-zero
commutative finite-dimensional Lie subideals.
\begin{enumerate}
 \item[(ii)]
 The map $K:$ $\mathcal{L}\mapsto K(\mathcal{L})=\mathrm{Centre}%
\left(  \mathcal{L}\right)  ,$ for each $\mathcal{L}\in\mathfrak{L,}$ is a
lower stable preradical and $(S_{\text{\emph{I}}^{\text{\emph{sol}}}%
}^{\mathbf{s}})^{\ast}=(S_{\text{\emph{Abel}}}^{\mathbf{s}})^{\ast
}=(K^{\mathbf{s}})^{\ast}.$
\item[(iii)] 
$(S_{\text{\emph{Abel}}}^{\mathbf{s}})^{\ast
}|\mathbf{L}^{\text{\emph{f}}}=\operatorname*{rad}.$
\end{enumerate}
\end{theorem}

\begin{proof}
(i) If $J$ is a Lie subideal of $I$ and $I$ is a Lie subideal of
$\mathcal{L,}$ then $J$ is a Lie subideal of $\mathcal{L.}$

By (\ref{4.4}), $\mathcal{L}\in$ $\mathbf{Sem}\left(  R^{\mathbf{s}}\right)  $
for a preradical $R$, if and only if Sub($\mathcal{L}$,$R\mathcal{)}%
=\{\{0\}\},$ that is,
\begin{equation}
I=\{0\}\text{ is the only Lie subideal of }\mathcal{L}\text{ satisfying
}I=R(I). \label{C}%
\end{equation}

Let $\Gamma=\{\Gamma_{\mathcal{L}}\}_{\mathcal{L}\in\mathfrak{L}}$ where each
family $\Gamma_{\mathcal{L}}=\{J$: $J\vartriangleleft\mathcal{L}$ and $J$ has
some property $\mathbf{T}\}\mathbf{.}$ Let $R=S_{\Gamma},$ so that
$R(\mathcal{L})\overset{(\ref{4.r})}{=}\mathfrak{s}(\Gamma_{\mathcal{L}%
})\mathfrak{.}$ Then (\ref{C}) is equivalent to the following condition:%
\begin{equation}
\mathcal{L}\text{ has no Lie subideals that have property }\mathbf{T}.
\label{T}%
\end{equation}
Indeed, let $(\ref{T})$ do not hold and let $I\neq\{0\}$ be a Lie subideal of
$\mathcal{L}$ that has property $\mathbf{T.}$ Then $I\in\Gamma_{I},$ so that
$\{0\}\neq I=\mathfrak{s}(\Gamma_{I}).$ Conversely, let $I\neq\{0\}$ be a Lie
subideal of $\mathcal{L}$ such that $I=\mathfrak{s}(\Gamma_{I}).$ Then
$\Gamma_{I}\neq\{\{0\}\},$ so that $I$ has a Lie ideal $J\neq\{0\}$ that has
property $\mathbf{T}.$ As $J$ is a Lie subideal of $\mathcal{L,}$ (\ref{T})
does not hold. Thus $\mathcal{L}\in\mathbf{Sem}\left(  S_{\mathrm{Abel}%
}^{\mathbf{s}}\right)  $ if and only if $\mathcal{L}$ has no non-zero Lie
subideals which are finite-dimensional and commutative. Similarly,
$\mathcal{L}\in\mathbf{Sem}\left(  S_{\text{I}^{\text{sol}}}^{\mathbf{s}%
}\right)  $ if and only if $\mathcal{L}$ has no non-zero Lie subideals which
are finite-dimensional and solvable.

Let us show that $\mathcal{L}$ has a finite-dimensional solvable Lie subideal
$Y\neq\{0\}$ if and only if it has a non-zero finite-dim\-en\-sio\-nal commutative
Lie subideal. As $Z=[Y,Y]$ is a nilpotent Lie ideal of $Y,$ it is a Lie
subideal of $\mathcal{L}$. If $Z=\{0\}$ then $Y$ is a commutative Lie subideal
of $\mathcal{L}$. If $Z\neq\{0\}$ then $Z$ has a non-zero center which is a
commutative Lie subideal of $\mathcal{L}$. The converse statement is obvious.
This implies $\mathbf{Sem}(S_{\text{I}^{\text{sol}}}^{\mathbf{s}%
})=\mathbf{Sem}(S_{\mathrm{Abel}}^{\mathbf{s}})$.

(ii) If $\mathcal{L}\in\mathfrak{L}$ then $[K(\mathcal{L}),\mathcal{L}%
]=\{0\}.$ If $f$: $\mathcal{L}\longrightarrow\mathcal{M}$ is a morphism in
$\overline{\mathbf{L}}$, then $[f\left(  K(\mathcal{L}\right)
),f(\mathcal{L)}]=f([K(\mathcal{L}),\mathcal{L}])=\{0\}$, whence
$\overline{f\left(  K(\mathcal{L}\right)  )}\subseteq K(\mathcal{M)}$. Thus
$K$ is a preradical. As $K(\mathcal{L})$ is commutative, $K\left(  K\left(
\mathcal{L}\right)  \right)  =K\left(  \mathcal{L}\right)  $, so that $K$ is
lower stable.

By (\ref{C}), $\mathcal{L}\in$ $\mathbf{Sem}\left(  K^{\mathbf{s}}\right)  $
if and only if $I=\{0\}$ is the only Lie subideal of $\mathcal{L}$ satisfying
$I=K(I).$ From the definition of $K$ we have that this is possible if and only
if $\{0\}$ is the only commutative Lie subideal of $\mathcal{L}$. It means, in
turn, that $\{0\}$ is the only finite-dimensional commutative Lie subideal of
$\mathcal{L}$. Hence, by (i), $\mathbf{Sem}\left(  S_{\text{I}^{\text{sol}}%
}^{\mathbf{s}}\right)  =\mathbf{Sem}\left(  S_{\mathrm{Abel}}^{\mathbf{s}%
}\right)  =\mathbf{Sem}\left(  K^{\mathbf{s}}\right)  .$ Applying Theorem
\ref{super}, we have $\mathbf{Sem}\left(  (S_{\text{I}^{\text{sol}}%
}^{\mathbf{s}})^{\ast}\right)  =\mathbf{Sem}\left(  (S_{\mathrm{Abel}%
}^{\mathbf{s}})^{\ast}\right)  =\mathbf{Sem}\left(  (K^{\mathbf{s}})^{\ast
}\right)  .$ As $(S_{\text{I}^{\text{sol}}}^{\mathbf{s}})^{\ast},$
$(S_{\mathrm{Abel}}^{\mathbf{s}})^{\ast}$ and $(K^{\mathbf{s}})^{\ast}$ are
radicals, we have from Corollary \ref{Crad} that $(S_{\text{I}^{\text{sol}}%
}^{\mathbf{s}})^{\ast}=(S_{\mathrm{Abel}}^{\mathbf{s}})^{\ast}=(K^{\mathbf{s}%
})^{\ast}.$

(iii) Let $\mathcal{L}\in\mathfrak{L}^{\mathrm{f}}$. Note that
rad($\mathcal{L})=\{0\}$ if and only if $\mathcal{L}$ has no non-zero
commutative Lie ideals. Indeed, if $\mathrm{rad}\left(  \mathcal{L}\right)
\neq\{0\}$ then $J=[\mathrm{rad}\left(  \mathcal{L}\right)  ,\mathrm{rad}%
\left(  \mathcal{L}\right)  ]$ is a nilpotent Lie ideal of $\mathcal{L}$. If
$J=\{0\}$ then rad($\mathcal{L}$) is a commutative Lie ideal of $\mathcal{L}$.
If $J\neq\{0\}$ then $J$ has a non-zero center which is a commutative Lie
ideal of $\mathcal{L}$. Conversely, let $\mathrm{rad}\left(  \mathcal{L}%
\right)  =\{0\}.$ Since $\mathrm{rad}\left(  \mathcal{L}\right)  $ is the
largest solvable Lie ideal, $\mathcal{L}$ has no non-zero commutative Lie ideals.

By the above argument and by Lemma \ref{E5.1}, $\mathcal{L}\in\mathbf{Sem}%
\left(  \mathrm{rad}\right)  $ if and only if $\mathcal{L}$ has no non-zero
commutative Lie subideals. Hence, by (i), $\mathbf{Sem}\left(  \mathrm{rad}%
\right)  =\mathbf{Sem}(S_{\text{Abel}}^{\mathbf{s}}|\mathbf{L}^{\text{f}}).$
From Theorem \ref{super} it follows that 
\[
\mathbf{Sem}\left(  \mathrm{rad}%
\right)  =\mathbf{Sem}(S_{\text{Abel}}^{\mathbf{s}}|\mathbf{L}^{\text{f}%
})=\mathbf{Sem}((S_{\text{Abel}}^{\mathbf{s}})^{\ast}|\mathbf{L}^{\text{f}}).
\]
Since $(S_{\text{Abel}}^{\mathbf{s}})^{\ast}$ and $\mathrm{rad}$ are radicals, we
have from Corollary \ref{Crad} that 
\[
(S_{\text{Abel}}^{\mathbf{s}})^{\ast
}|\mathbf{L}^{\text{f}}=\mathrm{rad}.
\]
\end{proof}

F. Vasilescu \cite{V} extended the notion of the classical solvable radical to
infinite-dimensional Lie algebras $\mathcal{L}$ in the following way. He
called a Lie ideal $J$ of $\mathcal{L}$ \textit{primitive} if
\begin{equation}
\lbrack A,A]\subseteq J\Longrightarrow A\subseteq J, \label{primV}%
\end{equation}
for any Lie ideal $A$ of $\mathcal{L}$. This is equivalent to the condition
that $\mathcal{L}/J$ has no abelian Lie ideals. If $\mathcal{L}$ is
finite-dimensional, this condition means that $\mathcal{L}/J$ is semisimple,
so that in our terms (see Definition \ref{D3.3}) $J$ is $\text{rad}%
$-absorbing. Denote by $R_{\mathcal{L}}$ the intersection of all primitive Lie
ideals of $\mathcal{L}$. It was proved in \cite{V} that $R_{\mathcal{L}%
}=\text{rad}(\mathcal{L})$ if $\mathcal{L}$ is finite-dimensional. Applying
this to Banach Lie algebras $\mathcal{L}$ and denoting by $R_{\mathcal{L}}$
the intersection of all \textit{closed} primitive Lie ideals of $\mathcal{L}$,
one obtains an upper stable preradical on $\mathbf{L}$. However, it is not
clear whether this preradical is balanced and lower stable. The main obstacle
is that we don't know whether a Banach Lie algebra, whose Lie ideal has a
non-zero commutative Lie ideal, has itself a non-zero commutative Lie ideal.

To avoid this difficulty let us change the definition of the radical as
follows. We call a closed Lie subideal (ideal) $J$ of a Banach Lie algebra
$\mathcal{L}$ \textit{primitive}, if the implication (\ref{primV}) holds for
any Lie \textit{subideal} $A$ of $\mathcal{L}$. Set
\begin{align}
P_{V}(\mathcal{L})  &  =\cap\text{ }\{J\text{: }J\text{ is a closed primitive
Lie ideal of }\mathcal{L}\},\nonumber\\
I_{V}(\mathcal{L})  &  =\cap\text{ }\{J\text{: }J\text{ is a closed primitive
Lie subideal of }\mathcal{L}\}. \label{6.6}%
\end{align}

\begin{lemma}
$P_{V}(\mathcal{L})$ is a characteristic primitive ideal and $P_{V}%
(\mathcal{L})=I_{V}(\mathcal{L})$.
\end{lemma}

\begin{proof}
$P_{V}(\mathcal{L})$ is a closed Lie ideal of $\mathcal{L.}$ If $A$ is a Lie
subideal of $\mathcal{L}$ and $[A,A]\subseteq P_{V}(\mathcal{L}),$ then
$[A,A]\subseteq J$ for each primitive ideal of $\mathcal{L}$. Hence
$A\subseteq J$, so that $A\subseteq P_{V}(\mathcal{L})$. Thus $P_{V}%
(\mathcal{L})$ is primitive. A similar argument shows that $I_{V}%
(\mathcal{L})$ is a closed primitive subideal of $\mathcal{L}$.

Clearly, $I_{V}(\mathcal{L})\subseteq P_{V}(\mathcal{L})$. Each bounded
isomorphism of $\mathcal{L}$ maps Lie subideals of $\mathcal{L}$ into Lie
subideals and primitive Lie subideals into primitive Lie subideals. Hence
$I_{V}(\mathcal{L})$ is invariant for all bounded isomorphisms of
$\mathcal{L.}$ Therefore, by Lemma \ref{L1}, $I_{V}(\mathcal{L})$ is a
characteristic Lie ideal of $\mathcal{L.}$ Thus $I_{V}(\mathcal{L})$ is a
closed primitive Lie ideal of $\mathcal{L.}$ By (\ref{6.6}), $P_{V}%
(\mathcal{L})\subseteq I_{V}(\mathcal{L}),$ so $P_{V}(\mathcal{L}%
)=I_{V}(\mathcal{L})$.
\end{proof}

In the same way as Vasilescu's radical coincides with $\mathrm{rad}$ in the category
$\mathbf{L}^{\text{f}}$ of all finite-dimensional Lie algebras, $P_{V}$ also
coincides with $\mathrm{rad}$ in $\mathbf{L}^{\text{f}}.$

\begin{lemma}
\label{Vit}$P_{V}(\mathcal{L})=$ \emph{rad}$(\mathcal{L)}$ if $\mathcal{L}$ is finite-dimensional.
\end{lemma}

\begin{proof}
Note first that $\text{rad}(\mathcal{L})$ is a primitive Lie ideal of
$\mathcal{L}$. Indeed, let $A$ be a subideal of $\mathcal{L}$ and let
$[A,A]\subseteq$ rad$(\mathcal{L})$. Then $[A,A]$ is a solvable Lie algebra,
whence $A$ is a solvable Lie algebra. Let $A\vartriangleleft A_{1}%
\vartriangleleft A_{2}\vartriangleleft...\vartriangleleft A_{n}%
\vartriangleleft\mathcal{L.}$ As rad is a radical, $A=\text{rad(}%
A)\subseteq\text{rad(}A_{1})\subseteq...\subseteq\text{rad(}\mathcal{L)}$.
Thus $\text{rad}(\mathcal{L})$ is a primitive Lie ideal of $\mathcal{L.}$
Hence, by (\ref{6.6}), $P_{V}(\mathcal{L})\subseteq$ rad$(\mathcal{L}).$

To prove the lemma, it remains to establish the converse inclusion. Let
$J=\text{rad}(\mathcal{L})$ and let $J_{[1]}\supseteq J_{[2]}\supseteq
...\supseteq J_{[n]}\subseteq J_{[n+1]}=0$ be the Lie ideals of $\mathcal{L}$
defined in (\ref{5.3}). As $[J_{[n]},J_{[n]}]=J_{[n+1]}=\{0\}\subseteq
P_{V}(\mathcal{L})$ and as $P_{V}(\mathcal{L})$ is primitive, we have
$J_{[n]}\subseteq P_{V}(\mathcal{L})$. Proceeding in this way, we obtain that
$J_{[n-1]}\subseteq P_{V}(\mathcal{L})$,..., and finally $J\subseteq
P_{V}(\mathcal{L})$. Thus $\text{rad}(\mathcal{L})=P_{V}(\mathcal{L})$.
\end{proof}

Our aim is to show that $P_{V}$ is an over radical on $\mathfrak{L}$ and that
the corresponding radical $P_{V}^{\circ}$ coincides with $(S_{\text{Abel}%
}^{\text{s}})^{\ast}$.

\begin{proposition}
\label{VasHer} $P_{V}$ is an over radical on $\mathfrak{L}$.
\end{proposition}

\begin{proof}
Let $f:\mathcal{L}\rightarrow\mathcal{M}$ be a continuous homomorphism of
Banach Lie algebras and $\overline{f(\mathcal{L})}=\mathcal{M}$. Then
$\overline{f(A)}$ is a closed Lie ideal of $\mathcal{M}$, for each Lie ideal
$A$ of $\mathcal{L}$, which implies that $\overline{f(A)}$ is a closed Lie
subideal of $\mathcal{M}$ if $A$ is a Lie subideal of $\mathcal{L}$.

Let $I$ be a closed primitive Lie ideal of $\mathcal{M.}$ Then $J:=f^{-1}(I)$
is a closed primitive Lie ideal of $\mathcal{L}$. Indeed, if $A$ is a Lie
subideal of $\mathcal{L}$ with $[A,A]\subseteq J$, then $[\overline
{f(A)},\overline{f(A)}]\subseteq I$. Hence $\overline{f(A)}\subseteq I,$ so
that $A\subseteq J$. By (\ref{6.6}), $P_{V}(\mathcal{L})\subseteq J$, so that
$f(P_{V}(\mathcal{L}))\subseteq I$. Thus $f(P_{V}(\mathcal{L}))$ is contained
in all closed primitive Lie ideals of $\mathcal{M}$, so that $f(P_{V}%
(\mathcal{L}))\subseteq P_{V}(\mathcal{M})$. Therefore $P_{V}$ is a preradical.

Let $I$ be a closed Lie ideal of $\mathcal{L}$. If $A$ is a Lie subideal of
$I$ and $[A,A]\subseteq I\cap P_{V}(\mathcal{L}),$ then $A\subseteq
P_{V}(\mathcal{L})$ because $P_{V}(\mathcal{L})$ is primitive. Thus
$A\subseteq I\cap P_{V}(\mathcal{L}),$ so that $I\cap P_{V}(\mathcal{L})$ is a
primitive ideal of $I$. This implies that $P_{V}(I)\subseteq I\cap
P_{V}(\mathcal{L})$. We proved that the preradical $P_{V}$ is balanced.

Let $q$: $\mathcal{L}\rightarrow\mathcal{L}/P_{V}(\mathcal{L})$ be the
quotient map. If $K$ is a commutative Lie subideal of $\mathcal{L}%
/P_{V}(\mathcal{L}),$ then $F:=q^{-1}(K)$ is a Lie subideal of $\mathcal{L}$
and $[F,F]\subseteq P_{V}(\mathcal{L})$. As $P_{V}(\mathcal{L})$ is primitive,
$F\subseteq P_{V}(\mathcal{L})$, so that $K=\{0\}$. Thus $\{0\}$ is a
primitive Lie ideal in $\mathcal{L}/P_{V}(\mathcal{L})$ whence $P_{V}%
(\mathcal{L}/P_{V}(\mathcal{L}))=\{0\}$. Thus $P_{V}$ is upper stable.
\end{proof}

\begin{theorem}
\label{Vas=Abel} The radicals $(S_{\text{\emph{Abel}}}^{\text{\emph{s}}%
})^{\ast}$ and $P_{V}^{\circ}$ coincide.
\end{theorem}

\begin{proof}
Since $P_{V}$ is an over radical, $P_{V}^{\circ}$ is a radical by Theorem
\ref{under-over}(ii). Therefore, by Corollary \ref{Crad}, it suffices to show
that $\mathbf{Sem}(P_{V}^{\circ})=\mathbf{Sem}(S_{\text{Abel}}^{\text{s}%
})^{\ast}$. By Theorem \ref{T7.2}, $\mathcal{L}\in\mathbf{Sem}(S_{\text{Abel}%
}^{\text{s}})^{\ast}$ if and only if $\mathcal{L}$ has no non-zero commutative
Lie subideals. If this condition is fulfilled, then $\{0\}$ is a primitive
ideal of $\mathcal{L}$, so that $P_{V}(\mathcal{L})=\{0\}$ and therefore
$P_{V}^{\circ}(\mathcal{L})=\{0\}$. We have to show the converse: if
$P_{V}^{\circ}(\mathcal{L})=\{0\}$ then $\mathcal{L}$ has no commutative subideals.

It follows from (\ref{primV}) that if $A$ is a commutative Lie subideal of
$\mathcal{L,}$ then $A$ is contained in each primitive Lie ideal of
$\mathcal{L}$, so that $A\subseteq P_{V}(\mathcal{L})$. Furthermore, $A$ is a
commutative Lie ideal of $P_{V}(\mathcal{L}).$ Hence, as above, $A\subseteq
P_{V}(P_{V}(\mathcal{L}))$. Arguing in this way, one easily shows that
$A\subseteq P_{V}^{\alpha}(\mathcal{L})$ for each ordinal $\alpha$, whence
$A\subseteq P_{V}^{\circ}(\mathcal{L})$. As $P_{V}^{\circ}(\mathcal{L}%
)=\{0\},$ we get $A=\{0\}$.
\end{proof}

Consider  the map $D$: $\mathcal{L}\longmapsto\overline{\left[  \mathcal{L}%
,\mathcal{L}\right]  }=\mathcal{L}_{[1]}.$ Then $D^{n}\left(  \mathcal{L}%
\right)  =\overline{[\mathcal{L}_{[n]},\mathcal{L}_{[n]}\mathcal{]}%
}=\mathcal{L}_{[n+1]}$ for each $n\in\mathbb{N}$. Defining $D^{\alpha}\left(
\mathcal{L}\right)  $ as in (\ref{4.o}) for each ordinal $\alpha$, we obtain
the $D$-superposition series $\left\{  D^{\alpha}\left(  \mathcal{L}\right)
\right\}  $ --- a transfinite analogue of the \textit{derived series} of
$\mathcal{L}$, and define $D^{\circ}$ as in (\ref{r1}). Set
\begin{equation}
\mathcal{D}=D^{\circ}. \label{6.1}%
\end{equation}

\begin{theorem}
\label{T7.4}\emph{(i) }$D$ is an over radical and $\mathcal{D}$ is a radical
in $\overline{\mathbf{L}}$.

\begin{itemize}
\item [$\mathrm{(ii)}$]For each $\mathcal{L}\in\mathfrak{L}$, $\mathcal{D}%
\left(  \mathcal{L}\right)  $ is the largest $\mathcal{D}$-radical Lie ideal
of $\mathcal{L;}$ in other words it is the largest of all Lie ideals $I$ of
$\mathcal{L}$ satisfying $I=\overline{\left[  I,I\right]  }.$

\item[$\mathrm{(iii)}$] The restriction\emph{ }of $\mathcal{D}$ to
$\mathbf{L}^{\mathrm{f}}$ coincides with $R_{\emph{Levi}}.$
\end{itemize}
\end{theorem}

\begin{proof}
(i) For every $\mathcal{L}\in\mathfrak{L}$, $\overline{f\left(  \mathcal{L}%
_{\left[  1\right]  }\right)  }=\mathcal{M}_{\left[  1\right]  }$ for each
morphism $f$: $\mathcal{L}\longrightarrow\mathcal{M}$ in $\overline
{\mathbf{L}}$ and $F(I)=I_{\left[  1\right]  }\subseteq\mathcal{L}_{\left[
1\right]  }=D\left(  \mathcal{L}\right)  $ for each $I\vartriangleleft
\mathcal{L}$. Also $D\left(  \mathcal{L}/D(\mathcal{L})\right)  =\left(
\mathcal{L}/\mathcal{L}_{\left[  1\right]  }\right)  _{\left[  1\right]  }=0.$
Hence $D$ is a balanced upper stable preradical. Thus $D$ is an over radical.
By Theorem \ref{under-over}(ii), $D^{\circ}=\mathcal{D}$ is a radical.

(ii) Since $\mathcal{D}$ is a radical, $\mathcal{D}\left(  \mathcal{D}\left(
\mathcal{L}\right)  \right)  =\mathcal{D}\left(  \mathcal{L}\right)  $. Thus
$\mathcal{D}\left(  \mathcal{L}\right)  $ is $\mathcal{D}$-radical. On the
other hand, the condition $I=\mathcal{D}(I)$ implies $I=\mathcal{D}\left(
I\right)  \subseteq\mathcal{D}\left(  \mathcal{L}\right)  $.

(iii) By Theorem \ref{Levi}(ii) and (\ref{6.2}), $R_{\text{Levi}}%
(\mathcal{L})=\mathfrak{P}(\mathcal{L})=\mathcal{D}(\mathcal{L})$ for each
$\mathcal{L}\in\mathbf{L}^{\text{f}}$.
\end{proof}

As $R_{\mathrm{Levi}}$ is not a hereditary radical, $\mathcal{D}$ is not
hereditary too.

Consider the Lie ideal-multifunction C $=$ $\{$C$_{\mathcal{L}}\}$ on
$\mathfrak{L}$ such that each family $\mathrm{C}_{\mathcal{L}}$ consists of
closed characteristic Lie ideals $\left\{  \mathcal{L}^{\left[  \alpha\right]
}\right\}  _{\alpha}$, where $\mathcal{L}^{\left[  1\right]  }=\mathcal{L}$,
$\mathcal{L}^{\left[  \alpha+1\right]  }=\overline{\left[  \mathcal{L}%
,\mathcal{L}^{\left[  \alpha\right]  }\right]  }$ for each ordinal $\alpha,$
and $\mathcal{L}^{\left[  \alpha\right]  }=\cap_{\alpha^{\prime}<\alpha
}\mathcal{L}^{[\alpha^{\prime}]}$ for limit ordinal $\alpha$. The series
$\left\{  \mathcal{L}^{\left[  \alpha\right]  }\right\}  _{\alpha}$ is a
transfinite analogue of the \textit{lower central series} of $\mathcal{L}$. As
in (\ref{4.r}), set $P_{\mathrm{C}}(\mathcal{L})=\mathfrak{p}($C$_{\mathcal{L}%
})=\cap_{\alpha}\mathcal{L}^{[\alpha]}.$

\begin{theorem}
$P_{\mathrm{C}}$ is an over radical and $P_{\mathrm{C}}^{\circ}=\mathcal{D}.$
\end{theorem}

\begin{proof}
By induction and by (\ref{F2}), $\overline{f\left(  \mathcal{L}^{\left[
\alpha\right]  }\right)  }\subseteq\mathcal{M}^{\left[  \alpha\right]  }$ for
every morphism $f:\mathcal{L}\longrightarrow\mathcal{M}$ in $\overline
{\mathbf{L}}.$ Hence, by (\ref{F2}), $P_{\text{C}}$ is a preradical. For each
$I\vartriangleleft\mathcal{L}$ and $\alpha,$ $I^{\left[  \alpha\right]
}\subseteq\mathcal{L}^{\left[  \alpha\right]  }.$ Hence $P_{\text{C}}$ is balanced.

Set $I=P_{\mathrm{C}}\left(  \mathcal{L}\right)  .$ By induction, it is easy
to see that $\left(  \mathcal{L}/I\right)  ^{\left[  \alpha\right]  }%
\subseteq\mathcal{L}^{\left[  \alpha\right]  }/I$, whence
\[
P_{\text{C}}(\mathcal{L}/I)=\cap_{\alpha}(\mathcal{L}/I)^{[\alpha]}%
\subseteq\cap_{\alpha}(\mathcal{L}^{[\alpha]}/I)=P_{\text{C}}(\mathcal{L}%
)/I=I/I=\{0\}.
\]
Thus $P_{\text{C}}$ is upper stable. Hence $P_{\mathrm{C}}$ is an over radical
and, by Theorem \ref{under-over}, $P_{\mathrm{C}}^{\circ}$ is a radical.

By Corollary \ref{Crad}(ii), to prove the equality $P_{\mathrm{C}}^{\circ
}=\mathcal{D}$ it suffices to show that $\mathbf{Rad}\left(  P_{\mathrm{C}%
}^{\circ}\right)  =\mathbf{Rad}\left(  \mathcal{D}\right)  $. As
$P_{\mathrm{C}}$ and $D$ are balanced, it follows from Theorem \ref{T4.1}(i)
that we only need to show that $\mathbf{Rad}\left(  P_{\mathrm{C}}\right)
=\mathbf{Rad}\left(  D\right)  $ which is obvious, as $\mathcal{L}$ belongs to
any of these classes if and only if $\mathcal{L}=\overline{\left[
\mathcal{L},\mathcal{L}\right]  }$.
\end{proof}

\section{\label{7}Frattini radical}

The Frattini theory of finite-dimensional Lie algebras $\mathcal{L}$ studies
the structure of maximal Lie ideals and maximal Lie subalgebras in
$\mathcal{L}$. To extend it to Banach Lie algebras, one can introduce
multifunctions S$_{\mathcal{L}}^{\max}$ and J$_{\mathcal{L}}^{\max}$, where
S$_{\mathcal{L}}^{\max}$ (respectively, J$_{\mathcal{L}}^{\max}),$ for each
$\mathcal{L}\in\mathfrak{L,}$ is the family of \textit{all} maximal proper
closed Lie subalgebras (respectively, Lie ideals) of $\mathcal{L}.$ It can be
shown that $P_{\text{S}^{\max}}$ and $P_{\text{J}^{\max}}$ are upper stable
preradicals on $\overline{\mathbf{L}}$. However, this approach encounters
serious obstacles and has not given, so far, any further interesting results.
For example, we do not know whether the preradicals $P_{\text{S}^{\max}}$ and
$P_{\text{J}^{\max}}$ are balanced.

As we will see further, a substantial theory can be developed if, instead of
all maximal Lie subalgebras (ideals), one considers maximal Lie subalgebras
and ideals \textit{of finite codimension}.

Let us consider the following four Lie subalgebra-multifunctions on
$\mathfrak{L}$: \smallskip

\begin{itemize}
\item [$\mathrm{1)}$]$\mathfrak{S}=\{\mathfrak{S}_{\mathcal{L}}\}_{\mathcal{L}%
\in\mathfrak{L}}$ where $\mathfrak{S}_{\mathcal{L}}$ consists of all
\textit{proper }closed \textit{ }Lie subalgebras of finite codimension in
$\mathcal{L;}$

\item[$\mathrm{2)}$] $\mathfrak{S}^{\text{max}}=\{\mathfrak{S}_{\mathcal{L}%
}^{\text{max}}\}_{\mathcal{L}\in\mathfrak{L}}$ where $\mathfrak{S}%
_{\mathcal{L}}^{\text{max}}$ consists of all \textit{maximal proper }closed
Lie subalgebras of finite codimension in $\mathcal{L;}$

\item[$\mathrm{3)}$] $\mathfrak{J}=\{\mathfrak{J}_{\mathcal{L}}\}_{\mathcal{L}%
\in\mathfrak{L}}$ where $\mathfrak{J}_{\mathcal{L}}$ consists of all
\textit{proper }closed \textit{ }Lie ideals of finite codimension in
$\mathcal{L}$;

\item[$\mathrm{4)}$] $\mathfrak{J}^{\text{max}}=\{\mathfrak{J}_{\mathcal{L}%
}^{\text{max}}\}_{\mathcal{L}\in\mathfrak{L}}$ where $\mathfrak{J}%
_{\mathcal{L}}^{\text{max}}$ consists of all \textit{maximal} \textit{proper
}closed Lie ideals of finite codimension in $\mathcal{L}$.
\end{itemize}

\begin{proposition}
\label{quotient}Let $\Gamma$ be any of the multifunctions $\mathfrak{S},$
$\mathfrak{S}^{\max},$ $\mathfrak{J,}$ $\mathfrak{J}^{\max}$ and let
$f:\mathcal{L}\rightarrow\mathcal{M}$ be a morphism in $\overline{\mathbf{L}%
}.$\emph{\ }If $K\in\Gamma_{\mathcal{M}}$ then 
\[
f^{-1}\left(  K\right)
:=f^{-1}\left(  K\cap f\left(  \mathcal{L}\right)  \right)  \in\Gamma
_{\mathcal{L}}.
\]
\end{proposition}

\begin{proof}
Since $K$ is a proper closed subspace of finite codimension in $\mathcal{M}$,
it follows that $F:=f^{-1}\left(  K\right)  $ is a closed proper subspace of
finite codimension in $\mathcal{L}$. Clearly, $F$ is a Lie subalgebra, it is a
Lie ideal if $K$ is a Lie ideal.

We claim that $F$ is maximal if $K$ is maximal. Indeed, let $L\subseteq
\mathcal{L}$ be a maximal proper closed Lie subalgebra (ideal) containing $F$.
Note that $L=f^{-1}(f(L))$ because $f^{-1}(0)\subseteq F\subseteq L$. By Lemma
\ref{Closed}, $f(L)+K$ is a closed Lie subalgebra (ideal) in $\mathcal{M}$. If
$K$ is maximal then either $f(L)\subseteq K$ or $f(L)+K=\mathcal{M}$. In the
first case $L\subseteq f^{-1}(K)=F$. As $L$ is maximal, $F=L$ and $F$ is
maximal. In the second case $f(\mathcal{L})\subseteq\mathcal{M}=f(L)+K$,
whence $\mathcal{L}\subseteq f^{-1}(f(L))+f^{-1}(K)=L+F=L$, a contradiction.
\end{proof}

\begin{remark}
\emph{If }$f$\emph{ is }onto\emph{ then, as above, }$L\in\Gamma_{\mathcal{L}}$
\emph{implies} $f\left(  L\right)  \in\Gamma_{\mathcal{M}}\cup\left\{
\mathcal{M}\right\}  .$ \emph{Thus}%
\begin{equation}
\Gamma_{\mathcal{M}}\subseteq\{f(L):L\in\Gamma_{\mathcal{L}}\}\subseteq
\left\{  \mathcal{M}\right\}  \cup\Gamma_{\mathcal{M}}. \label{8.1}%
\end{equation}
\end{remark}

Recall (see (\ref{F0}), (\ref{4.r})) that, for each family $\Gamma$ of Lie
subalgebras of a Banach Lie algebra $\mathcal{L}$,%
\begin{equation}
P_{\Gamma}(\mathcal{L})=\mathfrak{p}(\Gamma)=\cap_{L\in\Gamma}L,\text{ if
}\Gamma\neq\varnothing;\text{ and }P_{\Gamma}(\mathcal{L})=\mathcal{L}\text{,
if }\Gamma=\varnothing. \label{8.1'}%
\end{equation}

\begin{theorem}
\label{Cfam2}$P_{\mathfrak{S}},$ $P_{\mathfrak{S}^{\max}},$ $P_{\mathfrak{J}%
},$ $P_{\mathfrak{J}^{\max}}$ are upper stable preradicals in $\overline
{\mathbf{L}}$.
\end{theorem}

\begin{proof}
Let $\Gamma$ be any of the multifunctions $\mathfrak{S},$ $\mathfrak{S}^{\max
},$ $\mathfrak{J,}$ $\mathfrak{J}^{\max}$ and $f$: $\mathcal{L}\rightarrow
\mathcal{M}$ be a morphism in $\overline{\mathbf{L}}.$ If $\Gamma
_{\mathcal{M}}=\varnothing$ then $P_{\Gamma}(\mathcal{M})\overset
{(\ref{8.1'})}{=}\mathcal{M}$, so that $f(P_{\Gamma}(\mathcal{L}))\subseteq
P_{\Gamma}(\mathcal{M}).$

Let $\Gamma_{\mathcal{M}}\neq\varnothing.$ By Proposition \ref{quotient},
$\Gamma_{\mathcal{L}}\neq\varnothing$ and if $x\in P_{\Gamma}(\mathcal{L}%
)=\mathfrak{p}(\Gamma_{\mathcal{L}})=\underset{L\in\Gamma_{\mathcal{L}}}{\cap
}L,$ then
\[
f(x)\in\underset{M\in\Gamma_{\mathcal{M}}}{\cap}f\left(  f^{-1}\left(
M\right)  \right)  =\underset{M\in\Gamma_{\mathcal{M}}}{\cap}M=P_{\Gamma
}(\mathcal{M}).
\]
This implies $f(P_{\Gamma}(\mathcal{L}))\subseteq P_{\Gamma}(\mathcal{M})$, so
that $P_{\Gamma}$ is a preradical.

Let $I=P_{\Gamma}(\mathcal{L})$ and $q$: $\mathcal{L}\rightarrow\mathcal{L}/I$
be the quotient map. As $I\subseteq L,$ for each $L\in\Gamma_{\mathcal{L}},$
we have%
\begin{align*}
\{0\}  &  =q(P_{\Gamma}\left(  \mathcal{L}\right)  )=q\left(  \underset
{L\in\Gamma_{\mathcal{L}}}{\cap}L\right) \\
&  =\underset{L\in\Gamma_{\mathcal{L}}}{\cap}q(L)\overset{(\ref{8.1})}%
{=}\underset{M\in\Gamma_{\mathcal{L}/I}}{\cap}M=P_{\Gamma}(\mathcal{L}/I).
\end{align*}
Hence $P_{\Gamma}$ is upper stable.
\end{proof}

For a subset $N$ of $\mathcal{L}$, denote by $\mathrm{Alg}(N)$ the closed Lie
subalgebra and by \textrm{Id}$\left(  N\right)  $ the closed Lie ideal of
$\mathcal{L}$ generated by $N.$ Let $F$ be a family of proper closed Lie
ideals of $\mathcal{L}$. We say that $a\in\mathcal{L}$ is an \textit{ideal}
$F$\textit{-nongenerator} if \textrm{Id}$\left(  L\cup\left\{  a\right\}
\right)  \neq\mathcal{L}$ for all $L\in F$. If $F$ is a family of proper
closed Lie subalgebras of $\mathcal{L}$, then $a\in\mathcal{L}$ is an
$F$\textit{-nongenerator} if $\mathrm{Alg}\left(  L\cup\left\{  a\right\}
\right)  \neq\mathcal{L}$ for all $L\in F$.

\begin{lemma}
\label{Tnongen}Let $\mathcal{L}$ be a Banach Lie algebra. Then $P_{\mathfrak
{S}^{\max}}\left(  \mathcal{L}\right)  $ is the set of all $\mathfrak
{S}_{\mathcal{L}}$-nongenerators and $P_{\mathfrak{J}^{\max}}\left(
\mathcal{L}\right)  $ is the set of all ideal $\mathfrak{J}_{\mathcal{L}}$-nongenerators.
\end{lemma}

\begin{proof}
As each $L\in\mathfrak{S}_{\mathcal{L}}$ is contained in some $M\in
\mathfrak{S}_{\mathcal{L}}^{\max},$ the sets of $\mathfrak{S}_{\mathcal{L}}%
$-nongenerators and $\mathfrak{S}_{\mathcal{L}}^{\max}$-nongenerators
coincide. If $a\notin P_{\mathfrak{S}^{\max}}\left(  \mathcal{L}\right)  $
then $a\notin L$ for some $L\in\mathfrak{S}_{\mathcal{L}}^{\max}$. Therefore
$\mathrm{Alg}\left(  L\cup\left\{  a\right\}  \right)  =\mathcal{L}$ while
$L\neq\mathcal{L}$, so that $a$ is not a $\mathfrak{S}_{\mathcal{L}}%
$-nongenerator. Conversely, if $b\in P_{\mathfrak{S}^{\max}}\left(
\mathcal{L}\right)  $ then $\mathrm{Alg}\left(  L\cup\left\{  b\right\}
\right)  =L\neq\mathcal{L}$ for all $L\in\mathfrak{S}_{\mathcal{L}}^{\max}$.
Hence $b$ is a $\mathfrak{S}_{\mathcal{L}}^{\max}$-nongenerator.

The proof of the result for $P_{\mathfrak{J}^{\max}}\left(  \mathcal{L}%
\right)  $ is identical.
\end{proof}

For $\dim\left(  \mathcal{L}\right)  <\infty$ the above result was proved in
\cite[Lemma 2.3]{T} (see also \cite[1.7.2]{B}).

\begin{lemma}
\label{ideal} Let $J\vartriangleleft\mathcal{L}$ and let $I$ be a maximal
proper closed Lie ideal of finite codimension in $\mathcal{L}$. Then either
$J\subseteq I$ or $I\cap J$ is a maximal proper closed Lie ideal of finite
codimension in $J$.
\end{lemma}

\begin{proof}
Let $J\nsubseteq I.$ By Lemma \ref{Closed}, $\left(  I+J\right)
\vartriangleleft\mathcal{L}$. As $I\subsetneqq I+J$, it follows that
$I+J=\mathcal{L}$. Set $K=I\cap J$. Then $K\neq J,$ $K$ is a closed Lie ideal
of $\mathcal{L}$ and, by Lemma \ref{Closed}, $K$ has finite codimension in
$J$. If $K$ is not maximal Lie ideal of $J$, there is a maximal closed proper
Lie ideal $W$ of $J$ containing $K$. Then $\left[  I+W,\mathcal{L}\right]
=\left[  I+W,I+J\right]  \subseteq I+W$, whence $I+W$ is a closed Lie ideal of
$\mathcal{L}$. As $I$ is maximal, either $I+W=I$ or $I+W=\mathcal{L}$.

If $I+W=\mathcal{L}$ then $(I+W)\cap J=I\cap J+W=\mathcal{L}\cap J=J$ which is
impossible because $K=I\cap J\subsetneqq W$. Hence $I+W=I$, so that $W\subset
I$. Thus $W\subset I\cap J=K$, a contradiction.
\end{proof}

\begin{theorem}
\label{T4}The preradicals $P_{\mathfrak{S}},$ $P_{\mathfrak{S}^{\max}},$
$P_{\mathfrak{J}},$ $P_{\mathfrak{J}^{\max}}$ are balanced\emph{,} so that
they are over radicals.
\end{theorem}

\begin{proof}
Let $J\vartriangleleft\mathcal{L}$. As $P_{\mathfrak{S}^{\max}}$ is a
preradical, it follows from Corollary \ref{C1} that $P_{\mathfrak{S}^{\max}%
}\left(  J\right)  \vartriangleleft\mathcal{L}$. Assume that $P_{\mathfrak
{S}^{\max}}\left(  J\right)  \nsubseteq L$ for some $L\in\mathfrak
{S}_{\mathcal{L}}^{\max}$. Then $L+P_{\mathfrak{S}^{\max}}\left(  J\right)  $
is a Lie subalgebra of $\mathcal{L}$ larger than $L$ and, by Lemma
\ref{Closed}, it is closed. As $L$ is a maximal closed Lie subalgebra of
$\mathcal{L}$, $L+P_{\mathfrak{S}^{\max}}\left(  J\right)  =\mathcal{L}$.
Hence there is a finite-dimensional linear subspace $K$ of $P_{\mathfrak
{S}^{\max}}\left(  J\right)  $ such that $\mathcal{L}=L\dotplus K.$

Let $\{e_{i}\}_{i=1}^{m}$ be a basis in $K$. As $K\subseteq J$, we have
$J=(L\cap J)\dotplus K$. Then $M=L\cap J$ and all Lie algebras $M_{i}%
=\mathrm{Alg}\left(  M\cup\left\{  e_{1},\ldots,e_{i}\right\}  \right)  $ have
finite codimension in $J$, so that $M,M_{i}\in\mathfrak{S}_{J}.$ By Lemma
\ref{Tnongen}, all $e_{i}$ are $\mathfrak{S}_{J}$-nongenerators of $J.$ As
$\mathrm{Alg}\left(  M_{m-1}\cup\left\{  e_{m}\right\}  \right)  =J,$ we have
$M_{m-1}=J$. Repeating this $m-1$ times, we obtain that $M=J$. Hence
$J\subseteq L$, so that $P_{\mathfrak{S}^{\max}}\left(  J\right)  \subseteq
J\subseteq L.$ This contradiction shows that $P_{\mathfrak{S}^{\max}}\left(
J\right)  \subseteq L,$ for all $L\in\mathfrak{S}_{\mathcal{L}}^{\max},$ so
that $P_{\mathfrak{S}^{\max}}\left(  J\right)  \subseteq P_{\mathfrak{S}%
^{\max}}\left(  \mathcal{L}\right)  .$ Thus $P_{\mathfrak{S}^{\max}%
}\mathcal{\ }$is balanced.

By Lemma \ref{ideal}, $\mathfrak{J}_{J}^{\max}\cup\left\{  J\right\}
\overleftarrow{\subset}\{I\cap J:$ $I\in\mathfrak{J}_{\mathcal{L}}^{\max
}\}\overleftarrow{\subset}\mathfrak{J}_{\mathcal{L}}^{\max}$ (see (\ref{KK})).

For $\Gamma=\mathfrak{S}$ or $\mathfrak{J,}$ the condition $\Gamma_{J}%
\cup\left\{  J\right\}  \overleftarrow{\subset}\{I\cap J:$ $I\in
\Gamma_{\mathcal{L}}\}\overleftarrow{\subset}\Gamma_{\mathcal{L}}$ follows
from Lemma \ref{Closed}.

This implies (see (\ref{LP0})) that $P_{\Gamma}\left(  J\right)  =\mathfrak
{p}(\Gamma_{J})\subseteq\mathfrak{p}(\Gamma_{\mathcal{L}})=P_{\Gamma}\left(
\mathcal{L}\right)  $, for $\Gamma=\mathfrak{J}^{\max},$ $\mathfrak{S}$ and
$\mathfrak{J}$.
\end{proof}

Recall that $\mathcal{L}^{[2]}=\mathcal{L}_{[1]}=\overline{[\mathcal{L}%
,\mathcal{L}]}$. The following proposition extends to the infinite-dimensional
case some results due to Marshall \cite[Lemma, p. 420, and Theorem, p.
422.]{M}.

\begin{proposition}
\label{Cder}Let $\mathcal{L}$ be a Banach Lie algebra and $Z_{\mathcal{L}}$ be
its center. Then

\begin{itemize}
\item [$\mathrm{(i)}$]$P_{\mathfrak{J}^{\max}}\left(  \mathcal{L}\right)
\subseteq\mathcal{L}_{[1]}$ and $\mathcal{L}_{[1]}\cap Z_{\mathcal{L}%
}\subseteq P_{\mathfrak{S}^{\max}}\left(  \mathcal{L}\right)  $.

\item[$\mathrm{(ii)}$] If $\mathcal{L}$ is solvable then $\mathcal{L}%
_{[1]}=P_{\mathfrak{J}^{\max}}\left(  \mathcal{L}\right)  $.

\item[$\mathrm{(iii)}$] If\emph{ }$\mathcal{L}=I\oplus J$ then $P_{\mathfrak
{S}^{\max}}\left(  \mathcal{L}\right)  =P_{\mathfrak{S}^{\max}}\left(
I\right)  \oplus P_{\mathfrak{S}^{\max}}\left(  J\right)  $.
\end{itemize}
\end{proposition}

\begin{proof}
(i) Let $I=\mathcal{L}_{[1]}$. Then $\mathcal{L}/I$ is a commutative Banach
Lie algebra. As each closed subspace of codimension one is a maximal closed
Lie ideal of $\mathcal{L}/I$, we have $P_{\mathfrak{J}^{\max}}\left(
\mathcal{L}/I\right)  =\{0\}$. Hence, by Lemma \ref{L-sem}(ii), $P_{\mathfrak
{J}^{\mathrm{\max}}}^{\circ}\left(  \mathcal{L}\right)  \subseteq
I=\mathcal{L}_{[1]}$.

Set $J=P_{\mathfrak{S}^{\max}}\left(  \mathcal{L}\right)  .$ Let
$\mathcal{M}=\mathcal{L}/J,$ $Z_{\mathcal{M}}$ be its center and $q$:
$\mathcal{L}\rightarrow\mathcal{M}$ the quotient map. Then%
\[
q\left(  \mathcal{L}_{[1]}\cap Z_{\mathcal{L}}\right)  \subseteq
q(\mathcal{L}_{[1]})\cap q\left(  Z_{\mathcal{L}}\right)  \subseteq
\mathcal{M}_{[1]}\cap Z_{\mathcal{M}}%
\]
and it suffices to show that $\mathcal{M}_{[1]}\cap Z_{\mathcal{M}}=\{0\}.$
Assume, to the contrary, that $\mathcal{M}_{[1]}\cap Z_{\mathcal{M}}$ contains
$a\neq0$. By Theorem \ref{Cfam2}, $P_{\mathfrak{S}^{\max}}$ is upper stable.
Hence $P_{\mathfrak{S}^{\max}}\left(  \mathcal{M}\right)  =\{0\},$ so that
there is $M\in\mathfrak{S}_{\mathcal{M}}^{\max}$ such that $a\notin M$. As
$a\in Z_{\mathcal{M}}$, we have $\left[  a,\mathcal{M}\right]  =0$, so that
$M+\mathbb{C}a$ is a closed Lie subalgebra of finite codimension in
$\mathcal{M}$ larger than $M.$ As $M$ is maximal, $M+\mathbb{C}a=\mathcal{M}$
and $\mathcal{M}_{[1]}=\overline{\left[  M+\mathbb{C}a,M+\mathbb{C}a\right]
}=\overline{\left[  M,M\right]  }\subseteq M$. Hence $a\in M$, a
contradiction. Thus $\mathcal{M}_{[1]}\cap Z_{\mathcal{M}}=\{0\}.$

(ii) Let $\mathcal{L}$ be solvable. For each $J\in\mathfrak{J}_{\mathcal{L}%
}^{\max},$ the finite-dimensional Lie algebra $\mathcal{L/}J$ is solvable and
has no non-zero Lie ideals. Hence $\dim\left(  \mathcal{L/}J\right)  =1$, so
that $[\mathcal{L},\mathcal{L}]\subseteq J.$ Thus $\mathcal{L}_{[1]}%
=\overline{[\mathcal{L},\mathcal{L}]}\subseteq\underset{J\in\mathfrak
{J}_{\mathcal{L}}^{\max}}{\cap}J=P_{\mathfrak{J}^{\text{max}}}\left(
\mathcal{L}\right)  $. Using (i), we have $\mathcal{L}_{[1]}=P_{\mathfrak
{J}^{\max}}\left(  \mathcal{L}\right)  $.

(iii) follows from Theorem \ref{T4} and Proposition \ref{semi}(ii) 1).
\end{proof}

By Theorem \ref{under-over}, $P_{\mathfrak{S}}^{\circ}$, $P_{\mathfrak
{S}^{\max}}^{\circ}$, $P_{\mathfrak{J}}^{\circ}$, $P_{\mathfrak{J}^{\max}%
}^{\circ}$ are radicals. We will see now that they all coincide.

\begin{theorem}
\label{C6.6} \emph{(i) }For every Banach Lie algebra $\mathcal{L,}$%
\begin{equation}
P_{\mathfrak{S}}\left(  \mathcal{L}\right)  \subseteq P_{\mathfrak{J}}\left(
\mathcal{L}\right)  \subseteq P_{\mathfrak{S}^{\mathrm{\max}}}\left(
\mathcal{L}\right)  \subseteq P_{\mathfrak{J}^{\mathrm{\max}}}\left(
\mathcal{L}\right)  . \label{8.2}%
\end{equation}
Thus $P_{\mathfrak{S}}\leq P_{\mathfrak{J}}\leq P_{\mathfrak{S}^{\mathrm{\max
}}}\leq P_{\mathfrak{J}^{\mathrm{\max}}}$. If $\mathfrak
{S}_{\mathcal{L}}\neq\varnothing$ then $P_{\mathfrak{J}^{\mathrm{\max}}%
}\left(  \mathcal{L}\right)  \neq\mathcal{L}$.

\emph{(ii) }The radicals $P_{\mathfrak{S}}^{\circ},$ $P_{\mathfrak{J}}^{\circ
},$ $P_{\mathfrak{S}^{\mathrm{\max}}}^{\circ}$ and $P_{\mathfrak
{J}^{\mathrm{\max}}}^{\circ}$ coincide.

\emph{(iii) }$r_{P_{\mathfrak{S}}}^{\circ}\left(  \mathcal{L}\right)  \leq
r_{P_{\mathfrak{J}}}^{\circ}\left(  \mathcal{L}\right)  \leq r_{P_{\mathfrak
{S}^{\max}}}^{\circ}\left(  \mathcal{L}\right)  \leq r_{P_{\mathfrak{J}^{\max
}}}^{\circ}\left(  \mathcal{L}\right)  $ $($see $(\ref{r1})),$ for each
$\mathcal{L}\in\mathfrak{L.}$
\end{theorem}

\begin{proof}
We begin with the proof of the last statement in part (i). Suppose that
$\mathfrak{S}_{\mathcal{L}}\neq\varnothing$. It follows from Theorem
\ref{KST1} that $\mathfrak{J}_{\mathcal{L}}\neq\varnothing$, whence
$\mathfrak{J}_{\mathcal{L}}^{\max}\neq\varnothing.$ Hence $P_{\mathfrak
{J}^{\mathrm{\max}}}\left(  \mathcal{L}\right)  \neq\mathcal{L}$.

If $\mathfrak{S}_{\mathcal{L}}^{\text{max}}=\varnothing,$ then $\mathfrak
{J}_{\mathcal{L}}^{\max}=\mathfrak{J}_{\mathcal{L}}=\mathfrak{S}_{\mathcal{L}%
}=\varnothing,$ and it follows from (\ref{F0}) and (\ref{4.r}) that
$P_{\Gamma}(\mathcal{L})=\mathfrak{p}(\Gamma_{\mathcal{L}})=\mathcal{L}$,
where $\Gamma$ is any of these multifunctions. Hence in this case (\ref{8.2}) holds.

Let $\mathfrak{S}_{\mathcal{L}}^{\mathrm{\max}}\neq\varnothing.$ Then
$\mathfrak{S}_{\mathcal{L}}\neq\varnothing$ and, by the above, $P_{\mathfrak
{J}^{\mathrm{\max}}}\left(  \mathcal{L}\right)  \neq\mathcal{L}$. As
$\mathfrak{J}_{\mathcal{L}}\subseteq\mathfrak{S}_{\mathcal{L}}$, we have
$P_{\mathfrak{S}}\left(  \mathcal{L}\right)  =\mathfrak{p}(\mathfrak
{S}_{\mathcal{L}})\subseteq\mathfrak{p}(\mathfrak{J}_{\mathcal{L}%
})=P_{\mathfrak{J}}\left(  \mathcal{L}\right)  .$

By Theorem \ref{KST1}(i), each $M\in\mathfrak{S}_{\mathcal{L}}^{\mathrm{\max}%
}$ contains a closed Lie ideal $J$ of $\mathcal{L}$ of finite codimension.
This means that $\mathfrak{J}_{\mathcal{L}}\overleftarrow{\subset}\mathfrak
{S}_{\mathcal{L}}^{\mathrm{\max}}.$ Therefore, by (\ref{LP0}), $P_{\mathfrak
{J}}\left(  \mathcal{L}\right)  \subseteq P_{\mathfrak{S}^{\mathrm{\max}}%
}\left(  \mathcal{L}\right)  $.

To prove the last inclusion in (\ref{8.2}), set $I=P_{\mathfrak{S}%
^{\mathrm{\max}}}\left(  \mathcal{L}\right)  $. For each $J\in\mathfrak
{J}_{\mathcal{L}}^{\max},$ there is $M\in\mathfrak{S}_{\mathcal{L}%
}^{\mathrm{\max}}$ such that $J\subseteq M.$ Then $I\subseteq M$, so that
$I+J\subseteq M$ . Thus $I+J$ is a proper Lie ideal of $\mathcal{L}$ and, by
Lemma \ref{Closed}, it is closed. As $J\subseteq I+J$ and $J$ is a maximal
closed Lie ideal of $\mathcal{L}$, we have $J=I+J.$ Hence $I\subseteq J.$
Therefore $P_{\mathfrak{S}^{\mathrm{\max}}}\left(  \mathcal{L}\right)
=I\subseteq\underset{J\in\mathfrak{J}_{\mathcal{L}}^{\mathrm{\max}}}{\cap
}J=P_{\mathfrak{J}^{\mathrm{\max}}}\left(  \mathcal{L}\right)  .$ Part (i) is proved.

As $P_{\mathfrak{S}}\left(  \mathcal{L}\right)  ,P_{\mathfrak{J}}\left(
\mathcal{L}\right)  ,P_{\mathfrak{S}^{\mathrm{\max}}}\left(  \mathcal{L}%
\right)  ,P_{\mathfrak{J}^{\mathrm{\max}}}\left(  \mathcal{L}\right)  $ are
Lie ideals of $\mathcal{L}$ and the preradicals are balanced, it follows by
induction from (\ref{8.2}) that%
\begin{equation}
P_{\mathfrak{S}}^{\alpha}\left(  \mathcal{L}\right)  \subseteq P_{\mathfrak
{J}}^{\alpha}\left(  \mathcal{L}\right)  \subseteq P_{\mathfrak{S}%
^{\mathrm{\max}}}^{\alpha}\left(  \mathcal{L}\right)  \subseteq P_{\mathfrak
{J}^{\mathrm{\max}}}^{\alpha}\left(  \mathcal{L}\right)  ,\text{ for all
}\mathcal{L}\in\mathfrak{L} \label{f-alfa}%
\end{equation}
and all ordinal $\alpha$. This implies%
\begin{equation}
P_{\mathfrak{S}}^{\circ}\left(  \mathcal{L}\right)  \subseteq P_{\mathfrak{J}%
}^{\circ}\left(  \mathcal{L}\right)  \subseteq P_{\mathfrak{S}^{\mathrm{\max}%
}}^{\circ}\left(  \mathcal{L}\right)  \subseteq P_{\mathfrak{J}^{\mathrm{\max
}}}^{\circ}\left(  \mathcal{L}\right)  . \label{f-rad}%
\end{equation}
Set $J=P_{\mathfrak{S}}^{\circ}\left(  \mathcal{L}\right)  $. As
$P_{\mathfrak{S}}^{\circ}$ is a radical, it is upper stable. Hence
$\{0\}=P_{\mathfrak{S}}^{\circ}(\mathcal{L}/J)\subseteq P_{\mathfrak
{J}^{\mathrm{\max}}}^{\circ}\left(  \mathcal{L}/J\right)  .$

Set $I=P_{\mathfrak{J}^{\mathrm{\max}}}^{\circ}\left(  \mathcal{L}/J\right)  $
and assume that $I\neq\{0\}$. As $P_{\mathfrak{S}}^{\circ}$ is balanced and as
$I\vartriangleleft\mathcal{L}/J$ and $P_{\mathfrak{S}}^{\circ}(\mathcal{L}%
/J)=\{0\}$, we have $P_{\mathfrak{S}}^{\circ}\left(  I\right)  =\{0\}$. Hence,
by (\ref{4.o}), $P_{\mathfrak{S}}\left(  I\right)  \neq I$, whence the family
$\mathfrak{S}_{I}\neq\varnothing$. Hence, by (i), $P_{\mathfrak{J}%
^{\mathrm{\max}}}\left(  I\right)  \neq I$. Therefore $P_{\mathfrak
{J}^{\mathrm{\max}}}^{\circ}(I)\subseteq P_{\mathfrak{J}^{\mathrm{\max}}%
}\left(  I\right)  \neq I$. On the other hand, as $P_{\mathfrak{J}%
^{\mathrm{\max}}}^{\circ}$ is a radical, $P_{\mathfrak{J}^{\mathrm{\max}}%
}^{\circ}(I)=I.$ Thus $I=\{0\}$. Therefore it follows from Lemma
\ref{L-sem}(ii) that $P_{\mathfrak{J}^{\mathrm{\max}}}^{\circ}\left(
\mathcal{L}\right)  \subseteq J=P_{\mathfrak{S}}^{\circ}\left(  \mathcal{L}%
\right)  $. Taking into account (\ref{f-rad}), we complete the proof of (ii).

(iii) Denote temporarily by $P$ the common radical constructed in (ii). Let
$\beta=r_{P_{\mathfrak{J}^{\max}}}^{\circ}\left(  \mathcal{L}\right)  .$ It
follows from (\ref{r1}) and (\ref{f-alfa}) that $P(\mathcal{L})\subseteq
P_{\mathfrak{S}^{\mathrm{\max}}}^{\beta}\left(  \mathcal{L}\right)  \subseteq
P_{\mathfrak{J}^{\mathrm{\max}}}^{\beta}\left(  \mathcal{L}\right)
=P(\mathcal{L}).$ Hence $P(\mathcal{L})=P_{\mathfrak{S}^{\mathrm{\max}}%
}^{\beta}\left(  \mathcal{L}\right)  ,$ so that $r_{P_{\mathfrak{S}^{\max}}%
}^{\circ}\left(  \mathcal{L}\right)  \leq\beta=r_{P_{\mathfrak{J}^{\max}}%
}^{\circ}\left(  \mathcal{L}\right)  .$ The proof of other inequalities is identical.
\end{proof}

\begin{definition}
We denote by $\mathcal{F}$ the common radical in Theorem\emph{\ref{C6.6}(ii):}%
\[
\mathcal{F}=P_{\mathfrak{S}}^{\circ}=P_{\mathfrak{J}}^{\circ}=P_{\mathfrak
{S}^{\mathrm{\max}}}^{\circ}=P_{\mathfrak{J}^{\mathrm{\max}}}^{\circ},
\]
and call it the\emph{ Frattini radical.} $\mathcal{F}$\emph{-}radical Banach
Lie algebras are called also \emph{Frattini-radical}.
\end{definition}

As\textbf{ }$P_{\mathfrak{S}},$ $P_{\mathfrak{S}^{\max}},$ $P_{\mathfrak{J}},$
$P_{\mathfrak{J}^{\max}}$ are balanced preradicals\textbf{ (}Theorem
\ref{T4}), it follows from Theorem \ref{T4.1}(ii) that%
\begin{align}
\mathbf{Rad}(\mathcal{F})  &  =\mathbf{Rad}(P_{\mathfrak{S}})=\mathbf{Rad}%
(P_{\mathfrak{S}^{\max}})\nonumber\\
&  =\mathbf{Rad}(P_{\mathfrak{J}})=\mathbf{Rad}(P_{\mathfrak{J}^{\max}}).
\label{7.7}%
\end{align}
Thus (see (\ref{8.1'})) a Banach Lie algebra is $\mathcal{F}$-radical if and
only if it satisfies one of the following equivalent conditions:

\begin{itemize}
\item [$\mathrm{a)}$]it has no proper closed subalgebras of finite codimension;\smallskip

\item[$\mathrm{b)}$] it has no proper closed ideals of finite codimension;\smallskip

\item[$\mathrm{c)}$] it has no maximal proper closed subalgebras of finite codimension;\smallskip

\item[$\mathrm{d)}$] it has no maximal proper closed ideals of finite codimension.
\end{itemize}

\begin{corollary}
\label{Yura}A Banach Lie algebra $\mathcal{L}$ has closed Lie subalgebras of
finite codimension if and only if it has closed Lie ideals of finite codimension.
\end{corollary}

\begin{proof}
The condition that $\mathcal{L}$ has no closed Lie ideals of finite
codimension means $\mathfrak{J}_{\mathcal{L}}=\varnothing.$ Then%
\begin{align*}
\mathfrak{J}_{\mathcal{L}}  &  =\varnothing\overset{(\ref{8.1'})}%
{\Longleftrightarrow}P_{\mathfrak{J}}(\mathcal{L})=\mathcal{L}%
\Longleftrightarrow\mathcal{L}\in\mathbf{Rad}(P_{\mathfrak{J}})\overset
{(\ref{7.7})}{\Longleftrightarrow}\mathcal{L}\in\mathbf{Rad}(P_{\mathfrak{S}%
})\\
&  \Longleftrightarrow P_{\mathfrak{S}}(\mathcal{L})=\mathcal{L}%
\overset{(\ref{8.1'})}{\Longleftrightarrow}\mathfrak{S}_{\mathcal{L}%
}=\varnothing.
\end{align*}
Hence $\mathcal{L}$ has no closed Lie subalgebras of finite codimension.
\end{proof}

Murphy and Radjavi \cite{MR} proved that the algebra $C(H)$ of all compact
operators on a separable Hilbert space $H$ and Schatten ideals $C_{p}$ of
$B(H),$ for $p\geq2,$ have no proper closed Lie subalgebras of finite
codimension. In \cite{BKS} Bre\v{s}ar, Kissin and Shulman established that
simple Banach associative algebras with trivial center and without tracial
functionals (in particular, $C(H)$ and \textit{all} Schatten ideals $C_{p},$
$1<p<\infty)$ have no proper closed Lie ideals. From Corollary \ref{Yura} it
follows that they also have no proper closed Lie subalgebras of finite codimension.

The following result can be considered as an ''external'' application of the
radical technique.

\begin{corollary}
\label{Frat}\emph{(i)} Each Banach Lie algebra $\mathcal{L}$ has the largest
closed Lie ideal $\mathcal{F}(\mathcal{L})$ that satisfies one and\emph{,}
therefore\emph{,} all the above conditions $\mathrm{a})$--$\mathrm{d});$ this
ideal is characteristic.
\begin{enumerate}
 \item[(ii)]
 Let $I\vartriangleleft\mathcal{L}$. If $I$ and $\mathcal{L}/I$
satisfy conditions $\mathrm{a})$--$\mathrm{d})$ then the same is true for
$\mathcal{L}$.
\end{enumerate}
\end{corollary}

\begin{proof}
As $\mathcal{F}$ is a radical, $\mathcal{F}(\mathcal{L})$ is a characteristic
Lie ideal$\mathcal{.}$ As $\mathcal{F}(\mathcal{L})\in\mathbf{Rad}%
(\mathcal{F}),$ it satisfies conditions a)-d). The rest of (i) follows from
Corollary \ref{Crad}(i). Part (ii) from Lemma \ref{L-sem}(iv).
\end{proof}

In the next theorem we compare the Frattini radical $\mathcal{F}$ and the
radical $\mathcal{D}$ (see (\ref{6.1})).

\begin{theorem}
\label{free1}$\mathbf{Sem}\left(  \mathcal{D}\right)  \subsetneqq
\mathbf{Sem}(\mathcal{F}),$ $\ \mathbf{Rad}\left(  \mathcal{F}\right)
\subsetneqq\mathbf{Rad}(\mathcal{D})$ \ and $\ \mathcal{F}<\mathcal{D}$.
\end{theorem}

\begin{proof}
Recall (see Theorem \ref{T7.4}) that $D(\mathcal{L})=\mathcal{L}_{[1]}$ and
$\mathcal{D}\left(  \mathcal{L}\right)  =D^{\circ}(\mathcal{L})=\cap
\mathcal{L}_{[\alpha]}$ for $\mathcal{L}\in\mathfrak{L.}$ By Proposition
\ref{Cder}(i), $P_{\mathfrak{J}^{\text{max}}}\left(  \mathcal{L}\right)
\subseteq\mathcal{L}_{[1]}=D\left(  \mathcal{L}\right)  $. This implies
$P_{\mathfrak{J}^{\text{max}}}^{\alpha}\left(  \mathcal{L}\right)
\subseteq\mathcal{L}_{\left[  \alpha\right]  }=D^{\alpha}(\mathcal{L)}$ for
all $\alpha.$ Hence, by Theorem \ref{C6.6}(ii), $\mathcal{F}\leq\mathcal{D},$
so that $\mathbf{Sem}\left(  \mathcal{D}\right)  \subseteq$ $\mathbf{Sem}%
(\mathcal{F}).$

Each finite-dimensional semisimple Lie algebra $\mathcal{L}$ belongs to
$\mathbf{Sem}(\mathcal{F}),$ as $\{0\}$ is a Lie ideal of finite
codimension$.$ However, $D^{\circ}(\mathcal{L})=\mathcal{L}$, as $D\left(
\mathcal{L}\right)  =[\mathcal{L}$,$\mathcal{L}]=\mathcal{L}$. Hence
$\mathcal{L}\notin\mathbf{Sem}\left(  \mathcal{D}\right)  .$ Thus
$\mathbf{Sem}\left(  \mathcal{D}\right)  \neq$ $\mathbf{Sem}(\mathcal{F})$. By
Proposition \ref{P3.7} and Corollary \ref{Crad}, $\mathbf{Rad}\left(
\mathcal{F}\right)  \subsetneqq\mathbf{Rad}(\mathcal{D})$ and $\mathcal{F}%
<\mathcal{D}$.
\end{proof}

\begin{corollary}
\label{fred}Let $\mathcal{L}$ be a Banach Lie algebra. Then $\mathcal{L}%
/\mathcal{D}\left(  \mathcal{L}\right)  $ is $\mathcal{F}$-semi\-simple.
\end{corollary}

\begin{proof}
The Lie algebra $\mathcal{L}/\mathcal{D}\left(  \mathcal{L}\right)  $ is
$\mathcal{D}$-semisimple. Therefore it is $\mathcal{F}$-semisimple by Theorem
\ref{free1}.
\end{proof}

We will consider now some examples of $\mathcal{F}$-semisimple algebras.

\begin{example}
\label{E7.1}\emph{(i)} \textit{Each finite-dimensional Lie algebra is}
$\mathcal{F}$-\textit{semi\-sim\-ple}\emph{.} \emph{This follows from the fact
that }$\{0\}$\emph{ is a Lie ideal of finite codimension.}

\emph{(ii)} \textit{Each solvable Banach Lie algebra }$\mathcal{L}$
\textit{is} $\mathcal{F}$-\textit{semisimple. }\emph{This follows from Theorem
\ref{free1}. If }$\mathcal{L}$ \emph{is commutative then, in addition, we have
from (\ref{8.2}) and Proposition \ref{Cder}(i) that}%
\begin{equation}
\mathcal{F}(\mathcal{L})=P_{\mathfrak{S}}\left(  \mathcal{L}\right)
=P_{\mathfrak{J}}\left(  \mathcal{L}\right)  =P_{\mathfrak{S}^{\mathrm{\max}}%
}\left(  \mathcal{L}\right)  =P_{\mathfrak{J}^{\mathrm{\max}}}\left(
\mathcal{L}\right)  =\{0\}. \label{7.1}%
\end{equation}
\emph{(iii)} \textit{Let }$X$\textit{ be a Banach space\emph{,} }$L$\textit{
be a Lie subalgebra of }$B(X)$\textit{ and let }$\mathcal{L}=L\oplus
^{\text{\emph{id}}}X$ \emph{(}see \emph{(\ref{sem})).} \textit{If }$L$\textit{
is} $\mathcal{F}$\textit{-semisimple} \textit{then }$\mathcal{L}$ \textit{is}
$\mathcal{F}$-\textit{semisimple}. \emph{Indeed, by Proposition \ref{semi}(ii)
3) and the above example,} $\mathcal{F}(\mathcal{L})=\{0\}\oplus
^{\text{\emph{id}}}\mathcal{F(}X)=\{0\},$ \emph{so that} $\mathcal{L}$
\emph{is} $\mathcal{F}$-\emph{semisimple. }$\square$
\end{example}

Each infinite-dimensional, topologically simple Banach Lie algebra
$\mathcal{L}$ is $\mathcal{F}$-radical, since $\mathfrak{J}_{\mathcal{L}%
}=\varnothing$, so that (see (\ref{8.1'}) and (\ref{7.7})) $\mathcal{L}%
\in\mathbf{Rad}(P_{\mathfrak{J}})=\mathbf{Rad}(\mathcal{F})$. For example, the
Lie algebra of all nuclear operators with zero trace on a Hilbert space $H$
with respect to the usual Lie product $[a,b]=ab-ba$ is topologically simple
(see \cite[Theorem 5.8]{BKS}) and therefore $\mathcal{F}$-radical.

Examples of $\mathcal{F}$-radical Banach Lie algebras can be found also among
Lie algebras which are far from being simple.

\begin{example}
\label{E7.2}\textit{The Banach Lie algebra} $K_{\mathfrak{N}}$ \textit{of all
compact operators on a Hilbert space }$H$ \textit{that preserves a given
continuous nest} $\mathfrak{N}$ \textit{of subspaces in }$H$\textit{ is
}$\mathcal{F}$-\textit{radical}.

\emph{To see that }$K_{\mathfrak{N}}$\emph{ is }$\mathcal{F}$\emph{-radical,
we will show that it has no closed Lie ideals of finite codimension. Assume,
to the contrary, that }$J$\emph{ is such a Lie ideal. All operators in
}$K_{\mathfrak{N}}$\emph{ are quasinilpotent (see \cite{Ri}), so that all
operators ad}$\left(  a\right)  $\emph{ are quasinilpotent on }$K_{\mathfrak
{N}}$\emph{ and induce quasinilpotent operators on its quotients. Then
}$L=K_{\mathfrak{N}}/J$\emph{ is a nilpotent finite-dimensional Lie algebra.
Therefore }$[L,L]\neq L$\emph{, whence }$\overline{[K_{\mathfrak{N}%
},K_{\mathfrak{N}}]}\neq K_{\mathfrak{N}}$\emph{. On the other hand, each rank
one operator }$a=e\otimes f$\emph{ in }$K_{\mathfrak{N}}$\emph{ belongs to
}$\overline{[K_{\mathfrak{N}},K_{\mathfrak{N}}]}$\emph{. Indeed, by
\cite[Lemma 3.7]{Da}, there is a projection }$p$\emph{ on a subspace in
}$\mathfrak{N}$\emph{ with }$pe=e$\emph{ and }$pf=0$\emph{. Hence }%
$a=pa-ap$\emph{. Since projections on subspaces in }$\mathfrak{N}$\emph{
belong to the strong closure of }$K_{\mathfrak{N}}$\emph{ (\cite[Lemma
3.9]{Da}), }$a$\emph{ is the norm limit of a sequence }$b_{n}a-ab_{n}%
\in\lbrack K_{\mathfrak{N}},K_{\mathfrak{N}}]$\emph{. Thus }$\overline
{[K_{\mathfrak{N}},K_{\mathfrak{N}}]}$\emph{ contains all rank one operators.
It remains to note that rank one operators generate }$K_{\mathfrak{N}}$\emph{
by \cite[Corollary 3.12 and Proposition 3.8]{Da}, whence }$K_{\mathfrak{N}}%
=$\emph{ }$\overline{[K_{\mathfrak{N}},K_{\mathfrak{N}}]}$\emph{, a
contradiction. \ \ }$\square$
\end{example}

\begin{example}
The Frattini radical is not hereditary \emph{(}see \emph{(\ref{4.4'}%
)).}\textit{ }

\emph{Indeed, the Calkin algebra} $\mathcal{C}=\mathcal{B}(H)/\mathcal{K}(H)$
\emph{is simple. Considered as a Lie algebra with respect to the Lie product}
$[a,b]=ab-ba$\emph{, it has only one non-zero Lie ideal }$J=\mathbb{C}e,$
\emph{where} $e$ \emph{is the unit of }$\mathcal{C}$ \emph{(since}
$[\mathcal{C},\mathcal{C}]=\mathcal{C}$\emph{, this follows from Herstein's
\cite{H} description of Lie ideals in simple associative algebras). Hence}
$\mathcal{C}$ \emph{is} $\mathcal{F}$-\emph{radical:} $\mathcal{F}%
(\mathcal{C})=\mathcal{C}$. \emph{On the other hand,} $J$ \emph{is}
$\mathcal{F}$\emph{-semisimple, since it is finite-dimensional. Hence}
$\{0\}=\mathcal{F}(J)\neq J\cap\mathcal{F}(\mathcal{C})=J\cap\mathcal{C}=J$.
\ \ $\square$
\end{example}

\begin{proposition}
\label{E7}Let $A$ be a simple infinite-dimensional Banach associative algebra.
If its center $Z_{A}=\{0\}$ then $\mathcal{F}(A)=\overline{[A,A]}$.
\end{proposition}

\begin{proof}
As $A$ is a simple Banach algebra, each closed Lie ideal $J$ of $A$ either
contains the Lie ideal $C_{A}:=\overline{[A,A]},$ or is contained in $Z_{A}$
(see \cite{H} and \cite[Theorem 2.5]{BKS}). As $Z_{A}=\{0\},$ either
$C_{A}\subseteq J$ or $J=\{0\}.$ Hence, since $\dim A=\infty,$ we have that
each $J\in\mathfrak{J}_{A}$ contains $C_{A}.$ Thus $C_{A}\subseteq\cap
_{J\in\mathfrak{J}_{A}}J=P_{\mathfrak{J}}(A).$ On the other hand, each
subspace of finite codimension containing $C_{A}$ is a closed Lie ideal of
$A.$ As the intersection of such subspaces is $C_{A},$ we have $P_{\mathfrak
{J}}(A)\subseteq C_{A}$. Hence $P_{\mathfrak{J}}(A)=C_{A}.$ It remains to
prove that $\mathcal{F}(A)=P_{\mathfrak{J}}(A)$. For this we only have to show
that $C_{A}$ has no closed Lie ideals of finite codimension.

It is well known (see, for example, the proof of \cite[Proposition 2.4]{BKS})
that $[a,[a,x]]=0,$ for all $x\in A,$ implies $a\in Z_{A}.$ In our case this
can be written in the form
\begin{equation}
\lbrack a,[a,x]]=0,\text{ for all }x\in A,\text{ implies }a=0.
\label{centerzero}%
\end{equation}
Note that (\ref{centerzero}) implies that $A$ has no commutative Lie ideals.
Indeed, if $a\in I,$ where $I$ is a commutative Lie ideal, then $[a,[a,x]]=0$
for all $x\in A$, so that $a=0$.

It follows also from (\ref{centerzero}) that $\dim C_{A}=\infty.$ Indeed,
otherwise, as $\dim A=\infty,$ the map $x\in A\rightarrow\text{ad}(x)|_{C_{A}%
}$ has a non-trivial kernel. Thus there is a non-zero $a\in A$ such that
$[a,[y,z]]=0$, for all $y,z\in A$. This contradicts (\ref{centerzero}).

Thus we have that $C_{A}=P_{\mathfrak{J}}(A)$ is a non-commutative
characteristic Lie ideal of $A$ and $\dim C_{A}=\infty$. If $C_{A}$ has a
proper closed Lie ideal of finite codimension, it follows from Theorem
\ref{C3.1} that $C_{A}$ has a proper closed characteristic ideal $J$ of finite
codimension, so that $J\vartriangleleft^{\text{ch}}C_{A}\vartriangleleft
^{\text{ch}}A.$ Hence, by Lemma \ref{L3.1}, $J$ is a Lie ideal of $A.$ As
$J\subsetneqq C_{A},$ we have $J\subseteq Z_{A}=\{0\}.$ As $\dim C_{A}%
=\infty,$ $\{0\}$ is not a Lie ideal of finite codimension in $C_{A},$ a contradiction.
\end{proof}

It follows from (\ref{7.7}) that, for the preradicals $P_{\mathfrak{S}%
}$, $P_{\mathfrak{J}}$, $P_{\mathfrak{S}^{\max}}$, $P_{\mathfrak{J}^{\max}}$ and the
radical $\mathcal{F,}$ the classes of their radical Lie algebras coincide. We
will finish this section by showing that the classes of their semi\-simple Lie
algebras differ.

Recall that $\mathbf{Sem}(\mathcal{F})=\{\mathcal{L}\in\mathfrak{L}$:
$\mathcal{F}(\mathcal{L})=\{0\}\}$ is the class of all $\mathcal{F}%
$-semisimple Banach Lie algebras. Consider also the classes of semisimple
algebras for the preradicals $P_{\mathfrak{S}},P_{\mathfrak{J}},P_{\mathfrak
{S}^{\max}},P_{\mathfrak{J}^{\max}}$:%
\begin{align*}
\mathbf{Sem}(P_{\mathfrak{S}}) &  =\{\mathcal{L}\in\mathfrak{L}\text{: }%
\cap_{L\in\mathfrak{S}_{\mathcal{L}}}L=\{0\}\},\\
\mathbf{Sem}(P_{\mathfrak{J}}) &  =\{\mathcal{L}\in\mathfrak{L}\text{: }%
\cap_{L\in\mathfrak{J}_{\mathcal{L}}}L=\{0\}\},\\
\mathbf{Sem}(P_{\mathfrak{S}^{\max}}) &  =\{\mathcal{L}\in\mathfrak{L}\text{:
}\cap_{L\in\mathfrak{S}_{\mathcal{L}}^{\max}}L=\{0\}\},\\
\mathbf{Sem}(P_{\mathfrak{J}^{\max}}) &  =\{\mathcal{L}\in\mathfrak{L}\text{:
}\cap_{L\in\mathfrak{J}_{\mathcal{L}}^{\max}}L=\{0\}\}.
\end{align*}
By (\ref{8.2}),
\[
\mathbf{Sem}(P_{\mathfrak{J}^{\mathrm{\max}}})\subseteq\mathbf{Sem}%
(P_{\mathfrak{S}^{\mathrm{\max}}})\subseteq\mathbf{Sem}(P_{\mathfrak{J}%
})\subseteq\mathbf{Sem}(P_{\mathfrak{S}})\subseteq\mathbf{Sem}(\mathcal{F}).
\]
We will show that%
\begin{align}
\mathbf{Sem}(P_{\mathfrak{J}^{\mathrm{\max}}}) &  \subsetneqq\mathbf{Sem}%
(P_{\mathfrak{S}^{\mathrm{\max}}})\subsetneqq\mathbf{Sem}(P_{\mathfrak{J}%
})\nonumber\\
&  \subsetneqq\mathbf{Sem}(P_{\mathfrak{S}})\subsetneqq\mathbf{Sem}%
(\mathcal{F}).\label{9.2}%
\end{align}
Firstly, in the following theorem we will establish that $\mathbf{Sem}%
(P_{\mathfrak{S}})\neq\mathbf{Sem}(\mathcal{F}).$

Let $T$ be a bounded operator on a Banach space $X$ (an example of such
operator can be found in \cite{HL}) whose lattice of invariant subspaces
$\mathrm{Lat}(T)$ has the following properties:

\begin{itemize}
\item [$\mathrm{C}_{1})$]the subspaces of finite codimension in $\mathrm{Lat}%
(T)$ are linearly ordered by inclusion,

\item[$\mathrm{C}_{2}\mathrm{)}$] their intersection $X_{\omega}\neq\{0\}$.
\end{itemize}

\begin{lemma}
\label{cod}Let $T\in\mathcal{B}\left(  X\right)  $ and $\mathrm{Lat}(T)$
satisfy \emph{C}$_{1})$ and \emph{C}$_{2})$. If $p\neq0$ is a polynomial then

\begin{itemize}
\item [$\mathrm{(i)}$]the closure of the range of $p(T)$ has finite codimension\emph{;}

\item[$\mathrm{(ii)}$] each closed subspace $Y$ of finite codimension in $X$
which is invariant for $p(T)$\emph{,} contains a closed subspace of finite
codimension which is invariant for $T$.
\end{itemize}
\end{lemma}

\begin{proof}
(i) Suppose that $\mathrm{codim}(\overline{p(T)X})=\infty$. Since
$p(T)=(T-\lambda_{1}\mathbf{1})\cdots(T-\lambda_{n}\mathbf{1}),$ it follows
that $\mathrm{codim}(\overline{(T-\lambda_{k}\mathbf{1})X})=\infty$ for some
$k.$ All subspaces that contain $\overline{(T-\lambda_{k}\mathbf{1})X}$ are
invariant for $T-\lambda_{k}\mathbf{1}$ and hence for $T$. This contradicts
the assumption that finite-codimensional subspaces in $\mathrm{Lat}(T)$ are
linearly ordered.

(ii) The operator $p(T)$ induces an algebraic operator on $X/Y$ because
$\dim(X/Y)<\infty$. Thus there is a polynomial $q(t)$ such that
$q(p(T))X\subset Y$. By (i), $\mathrm{codim}(\overline{q(p(T))X})<\infty$.
Since $q(p(T))X$ is invariant for $T$, we are done.
\end{proof}

\begin{theorem}
Let $T\in\mathcal{B}\left(  X\right)  $ with $\mathrm{Lat}(T)$ satisfying
\emph{C}$_{1})$ and \emph{C}$_{2})$. Let $A$ be the Banach algebra of
operators generated by $T,$ and let $\mathcal{L}=A\oplus^{\mathrm{id}}X$
$($see $(\ref{sem}))$. Then the intersection of all subalgebras of finite
codimension in $\mathcal{L}$ is non-zero$,$ while the Frattini radical is
trivial$:$
\[
P_{\mathfrak{S}}(\mathcal{L})\neq\mathcal{F}(\mathcal{L})=\{0\}.
\]
\end{theorem}

\begin{proof}
If $K$ is a subalgebra of finite codimension in $\mathcal{L}$, then $B=\{a\in
A$: ($a;0)\in K\}$ has finite codimension in $A.$ (Indeed, if $a_{1}%
,...,a_{n}\in A$ are linearly independent modulo $B,$ then ($a_{1}%
;0),...,(a_{n};0)$ are linearly independent modulo $K$.) Hence, since
polynomials of $T$ form an infinite-dimensional subspace of $A$, it follows
that $B$ contains a non-zero polynomial $p(T)$.

Similarly, the subspace $M=\{x\in X$: $(0;x)\in K\}$ has finite codimension in
$X$. It is invariant for $B$ because, if $b\in B$, $x\in M,$ then ($0;x)\in K$
and ($b;0)\in K.$ Hence ($0;bx)=[(b;0),(0;x)]\in K$, so that $bx\in M$. In
particular, $M$ is invariant for $p(T)$. By Lemma \ref{cod}(ii), $M$ contains
a closed subspace of finite codimension invariant for $T$. By condition
C$_{2}),$ each such subspace contains the subspace $X_{\omega}$ invariant for
$T.$ Therefore $K$ contains the subspace $\{0\}\oplus^{\text{id}}X_{\omega}$.
Thus $\{0\}\oplus^{\text{id}}X_{\omega}\subseteq P_{\mathfrak{S}%
}(\mathcal{L).}$

On the other hand, if $A_{\alpha}$ is a subspace of finite codimension in $A$
and $X_{\beta}$ is a subspace of finite codimension in $X$ invariant for $T,$
then $A_{\alpha}\oplus^{\text{id}}X$\textbf{ }and $A\oplus^{\text{id}}%
X_{\beta}$ are subalgebras of finite codimension in $\mathcal{L}.$\textbf{ }As
$A$ is commutative,\textbf{ }$P_{\mathcal{S}}(A)=\cap A_{\alpha}=\{0\}$ (see
(\ref{7.1}))\textbf{. }Therefore\textbf{ }$P_{\mathfrak{S}}(\mathcal{L}%
)\subseteq\{0\}\oplus^{\text{id}}X_{\omega}$\textbf{. }Thus\textbf{
}$P_{\mathfrak{S}}(\mathcal{L})=\{0\}\oplus^{\text{id}}X_{\omega}$.

As $A$ is commutative, it follows from Example \ref{E7.1}(i) and (iii) that
$\mathcal{F}(\mathcal{L})=\{0\}$.
\end{proof}

Denote by Lid(\emph{$\mathcal{L})$ }the set of all closed Lie ideals of
$\mathcal{L}$\emph{. }We will now construct examples that prove the rest of
(\ref{9.2}).

\begin{example}
\label{E3}\emph{(i) }$\mathbf{Sem}(P_{\mathfrak{J}})\subsetneqq\mathbf{Sem}%
(P_{\mathfrak{S}})$ \emph{(\cite{KST1}).} \emph{Let }$X$ \emph{be a Banach
space and }$\dim X=\infty.$ \emph{Let} $M$ \emph{be a finite-dimensional Lie
subalgebra of $\mathcal{B}$(}$X)$ \emph{that has no non-trivial invariant
subspaces and let }$\mathcal{L}=M\oplus^{\emph{id}}X$ \emph{(see (\ref{sem}%
))}.\emph{ Let} $\mathfrak{Y}$ \emph{be the set of all closed subspaces of
codimension} $1$ \emph{in} $X.$ \emph{For} $Y\in\mathfrak{Y},$ $L_{Y}%
=\{0\}\oplus^{\emph{id}}Y$ \emph{is a closed Lie subalgebra of} $\mathfrak{L}$
\emph{and} \emph{codim}$(L_{Y})=$ $1+\dim(M)<\infty.$ \emph{Thus}
$\mathcal{L}\in\mathbf{Sem}(P_{\mathfrak{S}}),$ \emph{as }$P_{\mathfrak{S}%
}\left(  \mathcal{L}\right)  \subseteq\cap_{Y\in\mathfrak{Y}}\mathcal{L}%
_{Y}=\{0\}.$ \emph{Then}%
\[
\mathrm{Lid(}\mathcal{L})=\{0\}\cup\{J\oplus^{\emph{id}}X:J\in
\text{\emph{Lid(}}M)\}.
\]
\emph{Thus }$\{0\}\oplus^{\emph{id}}X$\emph{ is the smallest non-zero Lie
ideal of }$\mathcal{L.}$ \emph{As }$\{0\}$ \emph{is not a Lie ideal of finite
codimension in $\mathcal{L}$, we have }$P_{\mathfrak{J}}\left(  \mathcal{L}%
\right)  =\cap(J\oplus^{\emph{id}}X)=\{0\}\oplus^{\emph{id}}X,$ \emph{so that
}$\mathcal{L}\notin\mathbf{Sem}(P_{\mathfrak{J}}).$ \emph{Thus }%
$\mathbf{Sem}(P_{\mathfrak{J}})\subsetneqq\mathbf{Sem}\left(  P_{\mathfrak{S}%
}\right)  .$

\emph{In particular, let }$e$ \emph{be a bounded operator on }$X=l_{1}$
\emph{that has no non-trivial closed invariant subspaces }$($\emph{see
\cite{R}}$).$ \emph{Then the Banach Lie algebra }$\mathcal{L}=\mathbb{C}%
e\oplus^{\emph{id}}X$ \emph{belongs to }$\mathbf{Sem}(P_{\mathfrak{S}}%
)$.\emph{ As }$e$ \emph{has no closed invariant subspaces and dim(}%
$X)=\infty,$ $\{0\}\oplus^{\emph{id}}X$ \emph{is the only proper Lie ideal of
$\mathcal{L}$ of finite codimension, so that $\mathcal{L}$}$\notin
\mathbf{Sem}(P_{\mathfrak{J}})$. \smallskip

\emph{(ii) }$\mathbf{Sem}(P_{\mathfrak{S}^{\text{\emph{max}}}})\subsetneqq
\mathbf{Sem}(P_{\mathfrak{J}}).$ \emph{Let}
\[
\mathcal{L}=\left\{  a(x,y,z)=\left(
\begin{array}
[c]{ccc}%
0 & x & z\\
0 & 0 & y\\
0 & 0 & 0
\end{array}
\right)  :x,y,z\in\mathbb{C}\right\}
\]
\emph{be the Heisenberg} $3$\emph{-dimensional Lie algebra with
one-dimensional center} $Z=\left[  \mathcal{L},\mathcal{L}\right]
=\{a(0,0,z):z\in\mathbb{C}\}$. \emph{As the Lie ideal} $\{0\}$ \emph{has
finite codimension in $\mathcal{L}$,} \emph{we have }$\mathcal{L}%
\in\mathbf{Sem}(P_{\mathfrak{J}})$. \emph{Let} $M$ \emph{be a maximal Lie
subalgebra of} $\mathcal{L}$ \emph{that does not contain} $Z$. \emph{If}
$\dim\left(  M\right)  =1,$\emph{ then} $M+Z$ \emph{is a }$2$%
\emph{-dimensional subalgebra larger than}\ $M,$ \emph{a contradiction}.
\emph{Thus} $\dim\left(  M\right)  =2$. \emph{Hence} $\mathcal{L}=M+Z$
\emph{and} $\left[  M,M\right]  =\left[  M+Z,M+Z\right]  =\left[
\mathcal{L},\mathcal{L}\right]  =Z,$\emph{ so that }$Z\subseteq M$ --- \emph{a
contradiction}. \emph{Thus all maximal Lie subalgebras of} $\mathcal{L}$
\emph{contain} $Z$. \emph{Therefore }$Z\in\cap_{M\in\mathfrak{S}%
^{\text{\emph{max}}}}M,$\emph{ so that }$\mathcal{L}\notin\mathbf{Sem}%
(P_{\mathfrak
{S}^{\text{\emph{max}}}}).$ \emph{Thus }$\mathbf{Sem}(P_{\mathfrak
{S}^{\text{\emph{max}}}})\subsetneqq\mathbf{Sem}(P_{\mathfrak{J}}).\smallskip$

\emph{(iii) }$\mathbf{Sem}(P_{\mathfrak{J}^{\text{\emph{max}}}})\subsetneqq
\mathbf{Sem}(P_{\mathfrak{S}^{\text{\emph{max}}}}).$ \emph{Let} $\mathfrak{h}$
\emph{be the Lie algebra of all upper triangular matrices in }$\mathfrak{sl}%
$\emph{(}$2,\mathbb{C})$\emph{,} $\mathfrak{n}$ \emph{be the Lie subalgebra of
}$\mathfrak{h}$ \emph{of matrices with zero on the diagonal and }$\mathfrak
{d}$ \emph{be the Lie subalgebra of} $\mathfrak{h}$ \emph{of diagonal
matrices}. \emph{Then} $\mathfrak{n}$\emph{ is the only maximal Lie ideal of}
$\mathfrak{h}$\emph{,} $\mathfrak{n}$ \emph{and} $\mathfrak{d}$ \emph{are
maximal Lie subalgebras of }$\mathfrak{h}$\emph{,} \emph{so} \emph{that}
$\mathfrak{n}=P_{\mathfrak
{J}^{\text{\emph{max}}}}\left(  \mathfrak{h}\right)  \neq P_{\mathfrak
{S}^{\text{\emph{max}}}}\left(  \mathfrak{h}\right)  \subseteq$ $\mathfrak
{n}\cap\mathfrak{d}=0$.

\emph{To give an example of an infinite-dimensional algebra, we will use the
direct product. Let} $\mathfrak{h}_{i}=\mathfrak{h}$\emph{ for all }%
$i\in\mathbb{N}$. \emph{Then} $\widehat{\mathcal{L}}=\underset{i\in\mathbb{N}%
}{\widehat{\oplus}}\mathfrak{h}_{i}=\{a=\{a_{i}\}_{i\in\mathbb{N}}:$\emph{
all} $a_{i}\in\mathfrak{h}$ \emph{and }$\lim\left\|  a_{i}\right\|  =0\}$ ---
\emph{the} $c_{0}$-\emph{direct} \emph{product of all} $\mathfrak{h}_{i}$
\emph{(see (\ref{e3.1})) is a solvable Banach Lie algebra. For each }%
$i\in\mathbb{N},$\emph{ }$J_{i}=\{a\in\widehat{\mathcal{L}}$\emph{:} $a_{i}%
\in\mathfrak{n}\}$ \emph{is a maximal closed Lie ideal of codimension }$1$
\emph{in }$\widehat{\mathcal{L}}$ \emph{and }$M_{i}=\{a\in\widehat
{\mathcal{L}}$\emph{:} $a_{i}\in\mathfrak{d}\}$ \emph{is a maximal closed Lie
subalgebra in }$\widehat{\mathcal{L}}$\emph{ of codimension }$1.$ \emph{If
}$J\in\mathfrak{J}_{\widehat{\mathcal{L}}}^{\max}$ \emph{then }$[\mathfrak
{h}_{i},J]$ \emph{is either }$\{0\}$ \emph{or $\mathfrak{n}_{i}$ for each
}$i\in\mathbb{N}$.\emph{ If} $[\mathfrak{h}_{i},J]=\{0\}$ \emph{then}
$J+\mathfrak{n}_{i}$ \emph{is a closed proper Lie ideal of} $\widehat
{\mathcal{L}}$ \emph{larger than} $J.$ \emph{This contradiction shows that}
$[\mathfrak{h}_{i},J]=\mathfrak{n}_{i}$ \emph{for each} $i\in\mathbb{N}$.
\emph{Hence either} $\mathfrak{h}_{i}\subseteq J$ \emph{or} $\mathfrak{n}%
_{i}\subseteq J.$ \emph{As} $\widehat{\mathcal{L}}$ \emph{is the} $c_{0}%
$-\emph{direct product of all} $\mathfrak{h}_{i},$ \emph{we conclude that} $J$
\emph{coincides with one of} $J_{i}.$ \emph{Thus} $\mathfrak{J}_{\widehat
{\mathcal{L}}}^{\max}=\{J_{i}$\emph{:} $i\in\mathbb{N}\}.$ \emph{Therefore}
$\mathcal{L}\notin\mathbf{Sem}(P_{\mathfrak{J}^{\max}})$ \emph{and}
$\mathcal{L}\in\mathbf{Sem}(P_{\mathfrak{S}^{\max}}),$ \emph{as}%
\begin{align*}
P_{\mathfrak{J}^{\text{\emph{max}}}}(\mathcal{L}) & =\cap_{i\in\mathbb{N}}%
J_{i}\neq\{0\}\text{  \emph{and}  } \\
P_{\mathfrak{S}^{\text{\emph{max}}}%
}(\mathcal{L}) & =\left(  \cap_{i\in\mathbb{N}}J_{i}\right)  \cap(\cap
_{i\in\mathbb{N}}M_{i}) =\{0\}.
\end{align*}

\end{example}

\section{Frattini-semisimple Banach Lie algebras\label{8}}

As in the classical theory of finite-dimensional Lie algebras, the most
''tract\-able'' infinite-dimensional Banach Lie algebras are Frattini-semisimple
Lie algebras. In this section we show that they admit chains of closed Lie
subalgebras and even of Lie 2-step subideals decreasing to $\{0\}$ with
finite-dimensional quotients. We also consider a subclass of $\mathbf{Sem}%
$($\mathcal{F)}$ that consists of Banach Lie algebras that admit chains of
closed Lie ideals decreasing to $\{0\}$ with finite-dimensional quotients. We
call these algebras \textit{strongly Frattini-semisimple} and prove that they
can be equivalently defined in terms of the structure of the sets of their Lie
ideals, of their $\mathcal{F}$-absorbing Lie ideals and of their commutative
Lie ideals.

\subsection{Chains of Lie subalgebras and ideals in Banach Lie algebras}

We begin with a result which, in particular, shows that separable
$P_{\mathfrak{S}}$- and $P_{\mathfrak{J}}$-semisimple Banach Lie algebras are
characterized, respectively, by the existence of sequences of Lie subalgebras
and Lie ideals of finite codimension that decrease to $\{0\}$. Recall that
$P_{\mathfrak{S}}(\mathcal{L})\subseteq P_{\mathfrak{J}}(\mathcal{L).}$

\begin{proposition}
\label{T8.6}\emph{(i) }Each Banach Lie algebra $\mathcal{L}$ has
complete\emph{, }lower finite-gap chains of closed Lie ideals between
$\mathcal{L}$ and $P_{\mathfrak{J}}(\mathcal{L),}$ and of closed Lie
subalgebras between $\mathcal{L}$ and $P_{\mathfrak{S}}(\mathcal{L).}$
\begin{enumerate}
 \item[(ii)]
 Each separable Banach Lie algebra $\mathcal{L}$ has a sequence
$\{J_{n}\}_{n=1}^{\infty}$ of closed Lie ideals of finite codimension and a
sequence $\{L_{n}\}_{n=1}^{\infty}$ of closed Lie subalgebras of finite
codimension such that%
\[
J_{n+1}\subseteq J_{n}\text{ and }\cap_{n=0}^{\infty}J_{n}=P_{\mathfrak{J}%
}(\mathcal{L);}\text{\ }L_{n+1}\subseteq L_{n}\text{ and }\cap_{n=0}^{\infty
}L_{n}=P_{\mathfrak{S}}(\mathcal{L)}.
\]
\end{enumerate}
\end{proposition}

\begin{proof}
(i) The family $\mathfrak{J}_{\mathcal{L}}$ consists of all closed Lie ideals
of finite codimension in $\mathcal{L.}$ Hence its $\mathfrak{p}$-completion
$\mathfrak{J}_{\mathcal{L}}^{\mathfrak{p}}$ --- the set of intersections of
Lie ideals from all subfamilies of $\mathfrak{J}_{\mathcal{L}}$ --- consists
of Lie ideals of $\mathcal{L.}$ The family $\mathfrak{S}_{\mathcal{L}}$
consists of all closed Lie subalgebras of finite codimension in $\mathcal{L}$
and its $\mathfrak{p}$-completion $\mathfrak{S}_{\mathcal{L}}^{\mathfrak{p}}$
consists of Lie subalgebras of $\mathcal{L.}$ By Lemma \ref{L2.4},
$\mathfrak{J}_{\mathcal{L}}^{\mathfrak{p}}$ and $\mathfrak{S}_{\mathcal{L}%
}^{\mathfrak{p}}$ are lower finite-gap families. Thus (i) follows from Theorem
\ref{T2.2}.

Part (ii) follows from Theorem \ref{T2.5}.
\end{proof}

\begin{corollary}
\label{polup}A separable Banach Lie algebra is $P_{\mathfrak{S}}$-semisimple
if and only if it has a chain $\{L_{n}\}_{n=1}^{\infty}$ of closed Lie
subalgebras of finite codimension such that $L_{n+1}\subseteq L_{n}$ and
$\cap_{n=0}^{\infty}L_{n}=\{0\}.$

It is $P_{\mathfrak{J}}$-semisimple if and only if it has a chain
$\{J_{n}\}_{n=1}^{\infty}$ of closed Lie ideals of finite codimension such
that $J_{n+1}\subseteq J_{n}$ and $\cap_{n=0}^{\infty}J_{n}=\{0\}$.
\end{corollary}

\begin{proposition}
\label{comdiff}Let the quotient Lie algebra\emph{ }$\mathcal{L}/P_{\mathfrak
{S}}(\mathcal{L})$ have no characteristic commutative\emph{,}
infinite-dimensional Lie ideals\emph{. }Then $\mathcal{L}$ has a
maximal\emph{,} lower finite-gap chain of\emph{ }closed characteristic Lie
ideals between $\mathcal{L}$ and $P_{\mathfrak{S}}(\mathcal{L).}$
\end{proposition}

\begin{proof}
Recall that $P_{\mathfrak{S}}(\mathcal{L})=\cap_{L\in\mathfrak{S}%
_{\mathcal{L}}}L$ is a characteristic Lie ideal of $\mathcal{L.}$ Firstly
assume that $P_{\mathfrak{S}}(\mathcal{L})=\{0\}.$ Let $G$ be the family of
all closed characteristic Lie ideals of $\mathcal{L.}$ Then $\mathfrak
{p}(G)=\{0\}\in G$ and $\mathfrak{s}(G)=\mathcal{L}\in G.$ For each subfamily
$G^{\prime}$ of $G,$ the Lie ideal $\mathfrak{p}(G^{\prime})=\cap_{J\in
G^{\prime}}J$ of $\mathcal{L}$ is characteristic. Hence $G$ is $\mathfrak{p}$-complete.

Let $\{0\}\neq I\in G.$ If $\dim I<\infty$ then $\{0\}$ has finite codimension
in $I.$ Let $\dim I=\infty.$ As $\cap_{L\in\mathfrak{S}_{\mathcal{L}}%
}L=\{0\},$ there is $L\in\mathfrak{S}_{\mathcal{L}}$ that does not contain
$I.$ By Lemma \ref{Closed}, $I\cap L$ is a proper closed Lie subalgebra of
finite codimension in $I.$ By our assumption, $I$ is non-commutative. Hence,
by Theorem \ref{KST1}(ii), $I$ has a proper closed characteristic Lie ideal
$K$ of finite codimension$.$ Then, by Lemma \ref{L3.1}(ii), $K\vartriangleleft
^{\text{ch}}\mathcal{L}\mathcal{,}$ so that $K\in G$ and $0<\dim(I/K)<\infty.$
Hence $G$ is a $\mathfrak{p}$-complete lower finite-gap family. We have from
Lemma \ref{L2.2} that $G$ has a maximal, lower finite-gap chain\emph{ }$C$ of
closed characteristic Lie ideals such that $\mathfrak{p}(C)=\{0\}$ and
$\mathfrak{s}(C)=\mathcal{L.}$

Let now $P_{\mathfrak{S}}(\mathcal{L})\neq\{0\}$ and $\dim(\mathcal{L}%
/P_{\mathfrak{S}}(\mathcal{L}))=\infty.$ Set $\widetilde{\mathcal{L}%
}=\mathcal{L}/P_{\mathfrak{S}}(\mathcal{L).}$ As $P_{\mathfrak{S}}$ is upper
stable (see Theorem \ref{Cfam2}), we have from (\ref{4.2'}) that
$P_{\mathfrak{S}}(\widetilde{\mathcal{L}})=\{0\}.$ By our assumption,
$\widetilde{\mathcal{L}}$ has no infinite-dimensional commutative
characteristic Lie ideals. Hence, by the above, $\widetilde{\mathcal{L}}$ has
a maximal, lower finite-gap chain\emph{ }$\widetilde{C}=\{\widetilde
{I_{\lambda}}\}$ of closed characteristic Lie ideals such that $\mathfrak
{p}(\widetilde{C})=\{0\}$ and $\mathfrak{s}(\widetilde{C})=\widetilde
{\mathcal{L}}\mathcal{.}$ By Lemma \ref{essl}, the preimages $I_{\lambda}$ of
$\widetilde{I_{\lambda}}$ in $\mathcal{L}$ are closed characteristic Lie
ideals of $\mathcal{L.}$ Hence $C=\{I_{\lambda}\}$ is a maximal, lower
finite-gap chain\emph{ }of closed characteristic Lie ideals of $\mathcal{L}$
with $\mathfrak{p}(C)=P_{\mathfrak{S}}(\mathcal{L})$ and $\mathfrak
{s}(C)=\mathcal{L}$.
\end{proof}

We will prove in this subsection that $\mathcal{F}$-semisimple algebras are
characterized by the existence of complete lower finite-gap chains of closed
Lie subalgebras decreasing to $\{0\}$. Since the radical $\mathcal{F}$ is
generated by the preradical $P_{\mathfrak{J}}$, one can expect that the same
holds for chains of Lie ideals. However, this is not true in general; the
class of algebras for which this is true will be considered in the next
subsection. Nevertheless, we will see that symmetry can be partially recovered
if instead of ideals one works with 2-step subideals.

Recall (see Definition \ref{subideal}) that a Lie subalgebra $I$ of
$\mathcal{L}$ is a $2$\textit{-step Lie subideal} if it is a Lie ideal of some
Lie ideal of $\mathcal{L}$. We write $I\vartriangleleft_{2}\mathcal{L}$ if $I$
is closed. This matches the notation in Definition \ref{subideal} because a
closed $2$\textit{-}step Lie subideal of $\mathcal{L}$ is clearly a Lie ideal
of a closed Lie ideal of $\mathcal{L}$. Clearly, all Lie ideals are $2$-step
Lie subideals.

\begin{lemma}
\label{sub}Let $\mathcal{L}$ be a Banach Lie algebra. Then

\begin{itemize}
\item [$\mathrm{(i)}$]the sum of a Lie ideal of $\mathcal{L}$ and a $2$-step
Lie subideal of $\mathcal{L}$ is a $2$-step Lie subideal of $\mathcal{L;}$

\item[$\mathrm{(ii)}$] the intersection of a family of $2$-step Lie subideals
of $\mathcal{L}$ is a $2$-step Lie subideal of $\mathcal{L}$.
\end{itemize}
\end{lemma}

\begin{proof}
(i) If $I$ is a Lie ideal of some Lie ideal $K$ of $\mathcal{L}$ and
$J\vartriangleleft\mathcal{L,}$ then $I+J$ is a\textbf{ }Lie ideal of $K+J$
and $K+J$ is a Lie ideal of $\mathcal{L}$.

(ii) Let $I_{\alpha}$ be a Lie ideal of some Lie ideal $K_{\alpha}$ of
$\mathcal{L}$ for each $\alpha\in\Lambda$, where $\Lambda$ is an index set.
Then $\cap_{\alpha}K_{\alpha}$ is a Lie ideal of $\mathcal{L}$ and
$\cap_{\alpha}I_{\alpha}$ is a Lie ideal of $\cap_{\alpha}K_{\alpha}$.
\end{proof}

\begin{proposition}
\label{P8.2}For each Banach Lie algebra $\mathcal{L,}$ there is a
complete\emph{, }lower finite-gap chain of closed Lie $2$-step subideals of
$\mathcal{L}$ between $\mathcal{F}(\mathcal{L)}$ and $\mathcal{L.}$
\end{proposition}

\begin{proof}
Let $\left\{  P_{\mathfrak{J}}^{\alpha}\left(  \mathcal{L}\right)  \right\}
_{\alpha=0}^{\beta}$ be the $P_{\mathfrak{J}}$-superposition series of closed
Lie ideals of $\mathcal{L}$. Then $P_{\mathfrak{J}}^{\beta}\left(
\mathcal{L}\right)  =P_{\mathfrak{J}}^{\circ}\left(  \mathcal{L}\right)
=\mathcal{F(L})$. As $P_{\mathfrak{J}}^{\alpha+1}\left(  \mathcal{L}\right)
=P_{\mathfrak{J}}\left(  P_{\mathfrak{J}}^{\alpha}\left(  \mathcal{L}\right)
\right)  ,$ we have from Proposition \ref{T8.6} that there is a complete chain
$C_{\alpha}$ of closed Lie ideals of $P_{\mathfrak{J}}^{\alpha}\left(
\mathcal{L}\right)  $ such that it is a lower finite-gap chain, $\mathfrak
{s}\left(  C_{\alpha}\right)  =P_{\mathfrak{J}}^{\alpha}\left(  \mathcal{L}%
\right)  $ and $\mathfrak{p}\left(  C_{\alpha}\right)  =P_{\mathfrak{J}%
}^{\alpha+1}\left(  \mathcal{L}\right)  $. As $P_{\mathfrak{J}}^{\alpha
}\left(  \mathcal{L}\right)  \vartriangleleft\mathcal{L,}$ these Lie ideals
are Lie 2-step subideals of $\mathcal{L.}$ Therefore the chain $\left(
\cup_{\alpha=0}^{\beta}C_{\alpha}\right)  \cup\left(  \cup_{\alpha=0}^{\beta
}P_{\mathfrak{J}}^{\alpha}\left(  \mathcal{L}\right)  \right)  $ is a
complete, lower finite-gap chain between $\mathcal{F(L})$ and $\mathcal{L}$
that consists of closed $2$-step Lie subideals of $\mathcal{L}$.
\end{proof}

The following theorem describes $\mathcal{F}$-semisimple (Frattini-semi\-sim\-ple)
Lie algebras in terms of lower finite-gap chains of Lie subalgebras and 2-step ideals.

\begin{theorem}
\label{free} Let $\mathcal{L}$ be a Banach Lie algebra. Then the following
conditions are equivalent$:$

\begin{itemize}
\item [$\mathrm{(i)}$]$\mathcal{L}$ is $\mathcal{F}$-semisimple.

\item[$\mathrm{(ii)}$] $\mathcal{L}$ has a complete\emph{,} lower finite-gap
chain of closed Lie $2$-step subideals from $\{0\}$ to $\mathcal{L}$.

\item[$\mathrm{(iii)}$] $\mathcal{L}$ has a complete\emph{,} lower finite-gap
chain of closed Lie subalgebras from $\{0\}$ to $\mathcal{L}$.

\item[$\mathrm{(iv)}$] The set of all closed Lie $2$-step subideals of
$\mathcal{L}$ is a $\mathfrak{p}$-complete\emph{,} lower finite-gap family.

\item[$\mathrm{(v)}$] The set of all closed Lie subalgebras of $\mathcal{L}$
is a $\mathfrak{p}$-complete\emph{, }lower finite-gap family.
\end{itemize}
\end{theorem}

\begin{proof}
(i) $\Longrightarrow$ (ii) follows from Proposition \ref{P8.2}.

(ii) $\Longrightarrow$ (iii) is obvious. The set in (v) is $\mathfrak{p}%
$-complete. The set in (iv) contains $\mathcal{L,}$ so that it is
$\mathfrak{p}$-complete by Lemma \ref{sub}(ii). Hence (iv)
$\Longleftrightarrow$ (ii) and (v) $\Longleftrightarrow$ (iii) follow from
Theorem \ref{T2.2}.

(iii) $\Longrightarrow$ (i) By Lemma \ref{L2.2}, the chain $C=\left\{
M_{\alpha}\right\}  _{\alpha=0}^{\beta}$ of closed Lie subalgebras in (iii) is
a strictly decreasing transfinite chain with $M_{0}=\mathcal{L}$ and
$M_{\beta}=\{0\}$. Let $\mathcal{F(L)}\neq\{0\}$ and $\alpha_{0}$ be the first
ordinal such that $M_{\alpha_{0}}$ does not contain $\mathcal{F(L)}$. Clearly,
$\alpha_{0}$ is not a limit ordinal. Hence $\alpha_{0}=\delta+1$ and
$M_{{\alpha_{0}}}$ is a Lie subalgebra of finite codimension in $M_{\delta}$.
Thus $\mathcal{F}(M_{\delta})\subseteq P_{\mathfrak{S}}(M_{\delta})\subseteq
M_{\alpha_{0}}$. As $\mathcal{F(L)}$ is a Lie ideal of $M_{\delta}$ and
$\mathcal{F}$ is a radical, we have $\mathcal{F(L)}=\mathcal{F(F(L)}%
)\subseteq\mathcal{F}(M_{\delta})\subset M_{\alpha_{0}}$, a contradiction.
Thus all $M_{\alpha}$ contain $\mathcal{F(L)}$, whence\textbf{ }%
$\mathcal{F(L)}=\{0\}$.
\end{proof}

Recall (see Definition \ref{D3.3}) that a closed ideal $I$ of a Banach Lie
algebra $\mathcal{L}$ is $\mathcal{F}$\textit{-absorbing}
(\textit{Frattini-absorbing}) if $\mathcal{L}/I$ is $\mathcal{F}$-semisimple.
We saw in Theorem \ref{free1}(ii) that $\mathcal{D}\left(  \mathcal{L}\right)
$ is $\mathcal{F}$-absorbing. Denote the set of all $\mathcal{F}$-absorbing
ideals of $\mathcal{L}$ by Abs$_{\mathcal{F}}\left(  \mathcal{L}\right)  $.

\begin{lemma}
\label{prim}\emph{(i)} A closed Lie ideal $I$ of $\mathcal{L}$ is
$\mathcal{F}$-absorbing if and only if there is a complete\emph{,} lower
finite-gap chain of closed Lie subalgebras between $I$ and $\mathcal{L}$.
\begin{enumerate}
 \item[(ii)]
 Let $I\in\mathrm{Abs}_{\mathcal{F}}\left(  \mathcal{L}\right)
.$ If $J\vartriangleleft\mathcal{L,}$ $J\subseteq I$ and $\dim(I/J)<\infty,$
then $J\in\mathrm{Abs}_{\mathcal{F}}\left(  \mathcal{L}\right)  $.
\item[(iii)]
Each complete\emph{,} lower finite-gap chain $C$ of closed Lie
ideals of $\mathcal{L}$ with $\mathfrak{s}\left(  C\right)  =\mathcal{L}$
consists of $\mathcal{F}$-absorbing ideals of $\mathcal{L}$.
\item[(iv)]
The set $\mathrm{Abs}_{\mathcal{F}}\left(  \mathcal{L}\right)  $
is $\mathfrak{p}$-complete\emph{, }$\mathfrak{s}(\mathrm{Abs}_{\mathcal{F}%
}\left(  \mathcal{L}\right)  )=\mathcal{L}$ and $\mathfrak{p}\left(
\mathrm{Abs}_{\mathcal{F}}\left(  \mathcal{L}\right)  \right)  =\mathcal{F}%
\left(  \mathcal{L}\right)  .$
\item[(v)]
Let $\mathcal{M}$ be a closed Lie subalgebra of $\mathcal{L.}$ If
$I\in\mathrm{Abs}_{\mathcal{F}}\left(  \mathcal{L}\right)  $ then
$I\cap\mathcal{M}\in\mathrm{Abs}_{\mathcal{F}}\left(  \mathcal{M}\right)  .$
\end{enumerate}
\end{lemma}

\begin{proof}
(i) If $\mathcal{L}/I$ is $\mathcal{F}$-semisimple then, by Theorem
\ref{free}, there is a complete, lower finite-gap chain of closed Lie
subalgebras of $\mathcal{L}/I$ between $\{0\}$ and $\mathcal{L}/I$. Their
preimages in $\mathcal{L}$ form a complete, lower finite-gap chain of closed
Lie subalgebras of $\mathcal{L}$ between $I$ and $\mathcal{L}$.

Conversely, if $C=\{L_{\lambda}\}$ is such a chain$,$ then the quotients
$L_{\lambda}/I$ form a complete, lower finite-gap chain of closed Lie
subalgebras of $\mathcal{L}/I$ between $\{0\}$ and $\mathcal{L}/I$. Thus
$\mathcal{L}/I$ is $\mathcal{F}$-semisimple.

(ii) By (i), there is a complete\emph{,} lower finite-gap chain $C$ of closed
Lie subalgebras between $I$ and $\mathcal{L}$. Then $C^{\prime}=J\cup C$ is
the same type of chain between $J$ and $\mathcal{L.}$ By (i), $J$ is
$\mathcal{F}$-absorbing.

(iii) For $I\in C,$ $\{J\in C$: $I\subseteq J\}$ is a complete\emph{,} lower
finite-gap chain. By (i), $I$ is $\mathcal{F}$-absorbing.

(iv) As $\mathcal{L}\in$ Abs$_{\mathcal{F}}\left(  \mathcal{L}\right)  ,$ we
have $\mathfrak{s}($Abs$_{\mathcal{F}}\left(  \mathcal{L}\right)
)=\mathcal{L.}$ The rest follows from Theorem \ref{primgen}.

(v) If $I\in\mathrm{Abs}_{\mathcal{F}}\left(  \mathcal{L}\right)  $ then, by
(i$)$, $\mathcal{L}$ has a complete\emph{,} lower finite-gap chain $C$ of
closed Lie subalgebras between $I$ and $\mathcal{L}$. By Corollary \ref{pi},
$C_{\mathcal{M}}=\{J\cap\mathcal{M}$: $J\in C\}$ is a complete, lower
finite-gap chain of closed Lie subalgebras of $\mathcal{M}$ between
$I\cap\mathcal{M}$ and $\mathcal{M}$. Hence, by $($i), $I\cap\mathcal{M}%
\in\mathrm{Abs}_{\mathcal{F}}\left(  \mathcal{M}\right)  $.
\end{proof}

Let $G$ be the set of all closed Lie subalgebras of $\mathcal{L.}$ It is
$\mathfrak{p}$-complete. Comparing (\ref{6.d}), Theorem \ref{T6.x} and Lemma
\ref{prim}, we have that $G_{\text{f}}=$ Abs$_{\mathcal{F}}(\mathcal{L})$ and
$\Delta_{G}=\mathcal{F(L).}$ This and Lemma \ref{L2.2} yield

\begin{corollary}
\emph{(i) }$\mathcal{F(L)}=\mathfrak{p}(C)$ for each maximal\emph{,} lower
finite-gap chain $C$ of closed Lie subalgebras of $\mathcal{L}$ with
$\mathfrak{s}(C)=\mathcal{L}$.
\begin{enumerate}
 \item[(ii)]
 Each $\mathfrak{p}$-complete$,$ lower finite-gap chain $C$ of
closed Lie subalgebras of $\mathcal{L}$ with $\mathfrak{s}(C)=\mathcal{L}$
extends to a maximal\emph{,} lower finite-gap chain of closed Lie subalgebras.
\end{enumerate}
\end{corollary}

In general, for a radical $R,$ a subalgebra of an $R$-semisimple Lie algebra
is not necessarily $R$-semisimple. However, for the Frattini radical
$\mathcal{F,}$ the situation is much better.

\begin{corollary}
\label{free3}If $\mathcal{L}\in\mathbf{Sem}(\mathcal{F)}$ then each closed Lie
subalgebra $\mathcal{M}$ of $\mathcal{L}$ is $\mathcal{F}$-semisimple.

\begin{proof}
As $\{0\}\in\mathrm{Abs}_{\mathcal{F}}\left(  \mathcal{L}\right)  ,$ by Lemma
\ref{prim}(v), $\{0\}\in\mathrm{Abs}_{\mathcal{F}}\left(  \mathcal{M}\right)
$. Hence $\mathcal{M}\in\mathbf{Sem}(\mathcal{F).}$
\end{proof}
\end{corollary}

Now we consider the sets of $\mathcal{F}$-absorbing characteristic Lie ideals
in Banach Lie algebras.

\begin{theorem}
\label{T-prim}Let $\mathcal{L}$ be a Banach Lie algebra.

\begin{itemize}
\item [\textrm{(i)}]$\mathcal{L}$ has a maximal chain of $\mathcal{F}%
$-absorbing ideals between $\mathcal{F(L)}$ and $\mathcal{L.}$

\item[\textrm{(ii)}] $\mathcal{L}$ has a maximal lower finite-gap chain of
$\mathcal{F}$-absorbing Lie ideals between $P_{\mathfrak{J}}(\mathcal{L)}$ and
$\mathcal{L.}$

\item[\textrm{(iii)}] Let $R$ be one of the preradicals $P_{\mathfrak{S}},$
$P_{\mathfrak{J}},$ $P_{\mathfrak{S}^{\max}},$ $P_{\mathfrak{J}^{\max}}.$ Then

\begin{itemize}
\item [\textrm{a)}]for each $\mathcal{F}$-absorbing Lie ideal $I$ of
$\mathcal{L,}$ the Lie ideal $P_{R}(I)$ is also $\mathcal{F}$-absorbing$;$

\item[\textrm{b)}] the characteristic Lie ideals $P_{R}^{\alpha}\left(
\mathcal{L}\right)  $ in the $R$-superposition series are $\mathcal{F}$-absorbing.
\end{itemize}

\item[\textrm{(iv)}] Let $\mathcal{L}/P_{\mathfrak{S}}(\mathcal{L})$ have no
characteristic commutative, infinite-dim\-en\-sional Lie ideals$.$ Then
$\mathcal{L}$ has a maximal, lower finite-gap chain of
$\mathcal{F}$-absorbing characteristic Lie ideals between $P_{\mathfrak{S}%
}(\mathcal{L})$ and $\mathcal{L.}$
\end{itemize}
\end{theorem}

\begin{proof}
(i) follows from Lemmas \ref{L2.2} and \ref{prim}(iv).

(ii) follows from Proposition \ref{T8.6}(i) and Lemma \ref{prim}(iii).

(iii) a) By Lemma \ref{L3.1}(i), $P_{R}(I)$ is a Lie ideal of $\mathcal{L.}$
As $I$ is $\mathcal{F}$-absorbing, we have from Lemma \ref{prim}(i) that there
is a complete\emph{,} lower finite-gap chain $C_{I}$ of closed Lie subalgebras
between $I$ and $\mathcal{L}$. By Proposition \ref{T8.6}(i), there is a
complete lower finite-gap chain $C^{\prime}$ of closed Lie subalgebras between
$R(I\mathcal{)}$ and $I\mathcal{.}$ Hence $C=C_{I}\cup C^{\prime}$ is a
complete lower finite-gap chain of closed Lie subalgebras between
$R(I\mathcal{)}$ and $\mathcal{L.}$ Thus, by Lemma \ref{prim}(i),
$R(I\mathcal{)}$ is $\mathcal{F}$-absorbing.

(iii) b) follows by induction. Let $P_{R}^{\alpha}\left(  \mathcal{L}\right)
$ be $\mathcal{F}$-absorbing. By (i), $P_{R}^{\alpha+1}\left(  \mathcal{L}%
\right)  $ is also $\mathcal{F}$-absorbing. The case of a limit ordinal
$\alpha$ follows from Lemma \ref{prim}(iv).

(iv) follows from Proposition \ref{T8.6}(i) and Lemma \ref{prim}(iii).
\end{proof}

Denote, as in Example \ref{E3}, by Lid$(\mathcal{L)}$ the lattice of all
closed Lie ideals in $\mathcal{L.}$ Denote by

1) $\mathcal{A}_{\mathcal{L}}$ the set of all closed commutative Lie ideals of
$\mathcal{L;}$

2) $\mathcal{A}_{\mathcal{L}}^{\text{ch}}$ the set of all closed commutative
characteristic Lie ideals of $\mathcal{L};$

3) $\mathcal{A}_{\mathcal{L}}^{\mathrm{Abs}}:=\mathcal{A}_{\mathcal{L}}%
\cap\mathrm{Abs}_{\mathcal{F}}\left(  \mathcal{L}\right)  $ the set of all
$\mathcal{F}$-absorbing Lie ideals of $\mathcal{L}$ in $\mathcal{A}%
_{\mathcal{L}}$.

\begin{proposition}
\label{strongl}Let $\mathcal{L}\in\mathbf{Sem}(\mathcal{F})$. Then

\emph{(i) }the set\emph{ Lid}$(\mathcal{L})\diagdown\mathcal{A}_{\mathcal{L}}$
is a lower finite-gap family modulo $\mathcal{A}_{\mathcal{L}}$ \emph{(}see
Definition $\ref{modd}$\emph{);}

\emph{(ii) }the set of all infinite-dimensional non-commutative closed
characteristic Lie ideals of $\mathcal{L}$ is a lower finite-gap family modulo
$\mathcal{A}_{\mathcal{L}}^{\emph{ch}};$

\emph{(iii) }the set $\mathrm{Abs}_{\mathcal{F}}\left(  \mathcal{L}\right)
\diagdown\mathcal{A}_{\mathcal{L}}^{\mathrm{Abs}}$ is a lower finite-gap
family modulo $\mathcal{A}_{\mathcal{L}}^{\mathrm{Abs}}.$
\end{proposition}

\begin{proof}
Let $J$ be a non-commutative Lie ideal of $\mathcal{L}$ and $\dim J=\infty$.
By Theorem \ref{free}(v), $J$ has a proper subalgebra of finite codimension.
By Corollary \ref{C3.1}, it contains a closed Lie ideal $I$ of $\mathcal{L}$
that has non-zero finite codimension in $J$. Part (i) is proved.

If $J$ is characteristic then, by Corollary \ref{C3.1}, $I$ is also
characteristic. This proves (ii).

Let $J\in\mathrm{Abs}_{\mathcal{F}}\left(  \mathcal{L}\right)  \diagdown
\mathcal{A}_{\mathcal{L}}.$ Then, by (i), $\mathcal{L}$ has a closed Lie ideal
$I$ such that $I\subsetneqq J$ and $\dim(J/I)<\infty.$ By Lemma \ref{prim}%
(ii), 
\[
I\in\mathrm{Abs}_{\mathcal{F}}\left(  \mathcal{L}\right)
\subseteq(\mathrm{Abs}_{\mathcal{F}}\left(  \mathcal{L}\right)  \diagdown
\mathcal{A}_{\mathcal{L}}^{\mathrm{Abs}})\cup\mathcal{A}_{\mathcal{L}%
}^{\mathrm{Abs}}.
\]
\end{proof}

\subsection{Strongly Frattini-semisimple Banach Lie algebras}

Theorem \ref{free} gives us a satisfactory description of $\mathcal{F}%
$-semisimple Lie algebras in terms of lower finite-gap chains of Lie
subalgebras and 2-step subideals. These algebras may also have lower
finite-gap chains of Lie ideals. However, there are $\mathcal{F}$-semisimple
algebras where these chains do not stretch from $\mathcal{L}$ to $\{0\}$.

Indeed, the Lie algebra $\mathcal{L}=M\oplus^{\text{id}}X$ in Example
\ref{E3}(i) is $\mathcal{F}$-semisimple (see Example \ref{E7.1}(iii)) and its
Lie ideal $J=\{0\}\oplus^{\text{id}}X$ is infinite-dimensional, commutative
and contained in each non-zero Lie ideal of $\mathcal{L}.$ Hence if $C$ is a
maximal lower finite-gap chains of Lie ideals of $\mathcal{L}$ then
$\mathfrak{p}(C)=J.$ Thus $C$ does not continue to $\{0\}.$

\begin{definition}
A Banach Lie algebra $\mathcal{L}$ is strongly \textit{$\mathcal{F}%
$-semisimple} \emph{(}\textit{Frattini-semisimple}\emph{) } if there is a
complete\emph{,} lower finite-gap chain of closed Lie ideals of $\mathcal{L}$
between $\{0\}$ and $\mathcal{L}$.
\end{definition}

We will see later that each $\mathcal{F}$-semisimple Banach Lie algebra
$\mathcal{L}$ contains a characteristic commutative Lie ideal $J$ such that
$\mathcal{L}/J$ is strongly $\mathcal{F}$-semisimple. Therefore, for Lie
algebras without commutative Lie ideals, these two notions coincide. Thus the
presence of the commutative Lie ideal $J=\{0\}\oplus^{\text{id}}X$ in the
above example is not incidental.

The following result shows that one can define strongly $\mathcal{F}%
$-semisimple Lie algebras as algebras with complete, lower finite-gap chains
of $\mathcal{F}$-absorbing Lie ideals between $\{0\}$ and $\mathcal{L}$.

\begin{theorem}
\label{strong}Let $\mathcal{L}$ be a Banach Lie algebra. Then the following
conditions are equivalent\emph{.}

\begin{itemize}
\item [$\mathrm{(i)}$]$\mathcal{L}$ is strongly $\mathcal{F}$-semisimple.

\item[$\mathrm{(ii)}$] $\mathcal{L}$ has a complete\emph{,} lower finite-gap
chain of $\mathcal{F}$-absorbing ideals between $\{0\}$ and $\mathcal{L}$.

\item[$\mathrm{(iii)}$] The set \emph{Lid}$(\mathcal{L)}$ of all closed Lie
ideals of $\mathcal{L}$ is a lower finite-gap family.

\item[$\mathrm{(iv)}$] The set $\mathrm{Abs}_{\mathcal{F}}\left(
\mathcal{L}\right)  $ is a lower finite-gap family containing $\{0\}$.
\end{itemize}
\end{theorem}

\begin{proof}
The set Lid($\mathcal{L)}$ is $\mathfrak{p}$-complete. The set
Abs$_{\mathcal{F}}(\mathcal{L)}$ is $\mathfrak{p}$-complete by Lemma
\ref{prim}(iv); (i) $\Longrightarrow$ (ii) follows from Lemma \ref{prim}(iii);
(ii) $\Longrightarrow$ (i) is obvious; (iii) $\Longleftrightarrow$ (i) and
(iv) $\Longleftrightarrow$ (ii) follow from Theorem \ref{T2.2}.
\end{proof}

Strongly Frattini-semisimple algebras can be characterized in the class of
Frattini-semisimple algebras by the structure of the sets of their commutative ideals.

\begin{theorem}
\label{stron}Let $\mathcal{L}\in\mathfrak{L}$. Then the following conditions
are equivalent.

\begin{itemize}
\item [$\mathrm{(i)}$]$\mathcal{L}$ is strongly $\mathcal{F}$-semisimple.

\item[$\mathrm{(ii)}$] The set $\mathcal{A}_{\mathcal{L}}$ of all closed
commutative Lie ideals of $\mathcal{L}$ is a lower finite-gap family.

\item[$\mathrm{(iii)}$] The set $\mathcal{A}_{\mathcal{L}}^{\mathrm{Abs}%
}=\mathcal{A}_{\mathcal{L}}\cap\mathrm{Abs}_{\mathcal{F}}\left(
\mathcal{L}\right)  $ is a lower finite-gap family.

\item[$\mathrm{(iv)}$] $\mathcal{L}$ has a complete\emph{,} lower finite-gap
chain of closed Lie ideals between $\{0\}$ and $\mathfrak{s}\left(
\mathcal{A}_{\mathcal{L}}\right)  $.

\item[$\mathrm{(v)}$] $\mathcal{L}$ has a complete\emph{,} lower finite-gap
chain of closed ideals between $\{0\}$ and $\mathfrak{s}\left(  \mathcal{A}%
_{\mathcal{L}}^{\mathrm{Abs}}\right)  $.
\end{itemize}
\end{theorem}

\begin{proof}
(i) $\Longrightarrow$ (ii) follows from Theorem \ref{strong}(iii), and (ii)
$\Longrightarrow$ (iii) follows from Lemma \ref{prim}(ii).

(iii) $\Longrightarrow$ (i). It follows from Lemma \ref{mod} and Proposition
\ref{strongl} that $\mathrm{Abs}_{\mathcal{F}}\left(  \mathcal{L}\right)  $ is
a lower finite-gap family. By Theorem \ref{strong}, $\mathcal{L}$ is strongly
$\mathcal{F}$-semisimple.

(i) $\Longrightarrow$ (iv) and (v). By Theorem \ref{strong}(iii),
Lid($\mathcal{L)}$ is a $\mathfrak{p}$-complete, lower finite-gap family.
Hence, for any $J\in$ Lid($\mathcal{L),}$ the set Lid$_{J}$($\mathcal{L}%
)=\{I\in$ Lid($\mathcal{L)}$: $I\subseteq J\}$ is a $\mathfrak{p}$-complete,
lower finite-gap family. Thus (iv) and (v) follow from Lemma \ref{L2.2}.

(iv) $\Longrightarrow$ (ii). Lid$_{\mathfrak{s}\left(  \mathcal{A}%
_{\mathcal{L}}\right)  }$($\mathcal{L})$ is a $\mathfrak{p}$-complete family.
If the required chain exists then, by Theorem \ref{T2.2}, Lid$_{\mathfrak
{s}\left(  \mathcal{A}_{\mathcal{L}}\right)  }$($\mathcal{L})$ is a lower
finite-gap family. As $\mathcal{A}_{\mathcal{L}}\subseteq$ Lid$_{\mathfrak
{s}\left(  \mathcal{A}_{\mathcal{L}}\right)  }$($\mathcal{L}),$ we easily have
that $\mathcal{A}_{\mathcal{L}}$ is a lower finite-gap family.

(v) $\Longrightarrow$ (iii). Replacing $\mathfrak{s}\left(  \mathcal{A}%
_{\mathcal{L}}\right)  $ by $\mathfrak{s}\left(  \mathcal{A}_{\mathcal{L}%
}^{\mathrm{Abs}}\right)  $ in (iv) $\Longrightarrow$ (ii) and using Lemma
\ref{prim}(ii), we obtain that $\mathcal{A}_{\mathcal{L}}^{\mathrm{Abs}}$ is a
lower finite-gap family.
\end{proof}

\begin{corollary}
\label{C8.1c}Let $\mathcal{L}\in\mathbf{Sem}(\mathcal{F)}$. If the set
$\mathcal{A}_{\mathcal{L}}^{\emph{ch}}$ is a lower finite-gap family\emph{,
}then $\mathcal{L}$ has a maximal\emph{,} lower finite-gap chain of
characteristic Lie ideals between $\{0\}$ and $\mathcal{L.}$
\end{corollary}

\begin{proof}
If $\mathcal{A}_{\mathcal{L}}^{\text{ch}}$ is a lower finite-gap family, we
have from Lemma \ref{mod} and Proposition \ref{strongl} that the set
Lid$^{\text{ch}}(\mathcal{L)}$ of all closed characteristic Lie ideals of
$\mathcal{L}$ is a lower finite-gap family. As $\mathcal{L,}\{0\}\in$
Lid$^{\text{ch}}(\mathcal{L)}$ and the intersection of any subfamily of
characteristic Lie ideals is also a characteristic Lie ideal, Lid$^{\text{ch}%
}(\mathcal{L)}$ is $\mathfrak{p}$-complete. Applying Lemma \ref{L2.2}, we
complete the proof.
\end{proof}

Let $G$ be a family of closed subspaces in a Banach space $X.$ A subspace
$Y\in G$ is called\textit{ lower essential in }$G$ if the set%
\[
G_{-}(Y)=\{Z\in G:Z\subsetneqq Y\}\neq\varnothing\text{ and }\dim
(Y/Z)=\infty,
\]
for each $Z\in G_{-}(Y).$ Denote by $\mathrm{Ess}_{l}\left(  G\right)  $ the
set of all lower essential subspaces $Y$ in $G$.

\begin{corollary}
\label{stronc}A Banach Lie algebra $\mathcal{L}$ is strongly $\mathcal{F}%
$-semisimple if and only if $\mathcal{A}_{\mathcal{L}}^{\mathrm{Abs}}%
\cap\mathrm{Ess}_{l}\left(  \mathcal{A}_{\mathcal{L}}\right)  =\varnothing.$
\end{corollary}

\begin{proof}
If $\mathcal{A}_{\mathcal{L}}^{\mathrm{Abs}}\cap\mathrm{Ess}_{l}\left(
\mathcal{A}_{\mathcal{L}}\right)  \neq\{0\}$ then $\mathcal{A}_{\mathcal{L}}$
is not a lower finite-gap family. By Theorem \ref{stron}, $\mathcal{L}$ is not
strongly $\mathcal{F}$-semisimple. Conversely, if $\mathcal{A}_{\mathcal{L}%
}^{\mathrm{Abs}}\cap\mathrm{Ess}_{l}\left(  \mathcal{A}_{\mathcal{L}}\right)
=\varnothing$ then, for each $Y\in\mathcal{A}_{\mathcal{L}}^{\mathrm{Abs}},$
$Y\neq\mathfrak{p}(\mathcal{A}_{\mathcal{L}}),$ there is $Z\in\left(
\mathcal{A}_{\mathcal{L}}\right)  _{-}\left(  Y\right)  $ of finite
codimension in $Y$. By Lemma \ref{prim}(ii), $Z\in\mathcal{A}_{\mathcal{L}%
}^{\mathrm{Abs}}$ . Hence $\mathcal{A}_{\mathcal{L}}^{\mathrm{Abs}}$ is a
lower finite-gap family. By Theorem \ref{stron}, $\mathcal{L}$ is strongly
$\mathcal{F}$-semisimple.
\end{proof}

In two examples below $H$ is a Hilbert space with an orthonormal basis
$\{e_{i}\}_{i=1}^{\infty}$, and $H_{0}=\{0\}$ and $H_{n}=\sum_{i=1}^{n}%
\oplus\mathbb{C}e_{i}$, for $n>0$, are its finite-dimensional subspaces. In
the first example we consider a $\mathcal{D}$-semisimple (hence $\mathcal{F}%
$-semisimple) Banach Lie algebra $\mathcal{L}$ that has a commutative
$\mathcal{F}$-absorbing ideal in $\mathrm{Ess}_{l}\left(  \mathcal{A}%
_{\mathcal{L}}\right)  $, so it is not strongly $\mathcal{F}$-semisimple by
Corollary \ref{stronc}.

\begin{example}
\emph{Consider the nest }$G=H\cup\{H_{n}\}_{n=0}^{\infty}$\emph{ --- a
complete chain of subspaces from }$\{0\}$\emph{ to }$H$\emph{. Let }$M=$
\emph{Alg}$(G)$\emph{ be the algebra of all operators in }$B\left(  H\right)
$\emph{ leaving each subspace from }$G$\emph{ invariant. Then }$M$\emph{ has a
chain of closed two-sided ideals }$I_{n}=\{T\in M:T|_{H_{n}}=0\}$\emph{ that
have finite codimension in }$M$\emph{ and }$\cap_{n}I_{n}=\{0\}$\emph{.}

\emph{Let }$\mathcal{L}=M\oplus^{\emph{id}}H$ \emph{(see (\ref{sem}))}.
\emph{For each }$n,$ $J_{n}:=I_{n}\oplus^{\emph{id}}H$ \emph{is a closed Lie
ideal of finite codimension and }$K:=\{0\}\oplus^{\emph{id}}H=\cap_{n}J_{n}%
$\emph{ is the largest commutative closed Lie ideal of }$\mathcal{L}$.
\emph{Then} $D^{n}(\mathcal{L})\subseteq J_{n}.$ \emph{Hence} $D^{\infty
}(\mathcal{L})=\cap_{n}D^{n}(\mathcal{L})=K,$ \emph{so that} $\mathcal{D}%
(\mathcal{L})=D(K)=\{0\}.$ \emph{Thus} $\mathcal{L}$ \emph{is} $\mathcal{D}%
$\emph{-semisimple and, hence, $\mathcal{F}$-semisimple. Apart from }$K,$
\emph{only} $K_{n}:=\{0\}\oplus^{\emph{id}}H_{n},$\emph{ for }$n\in
\mathbb{N}\cup\left\{  0\right\}  ,$ \emph{are} \emph{other commutative closed
Lie ideals of} $\mathcal{L}$\emph{. Then }$K\in\mathrm{Ess}_{l}\left(
\mathcal{A}_{\mathcal{L}}\right)  ,$ \emph{as }$\dim K/K_{n}=\infty$\emph{ for
all }$n\mathcal{.}$\emph{ As $\mathcal{L}/K\approx M$ is $\mathcal{F}%
$-semisimple, }$K$ \emph{is a $\mathcal{F}$-absorbing ideal of} $\mathcal{L}$.
\emph{Thus} $K\in\mathcal{A}_{\mathcal{L}}^{\mathrm{Abs}}\cap\mathrm{Ess}%
_{l}\left(  \mathcal{A}_{\mathcal{L}}\right)  .$ \emph{By Corollary
\ref{stronc},} $\mathcal{L}$ \emph{is not strongly $\mathcal{F}$-semisimple.}
\end{example}

The algebra $\mathcal{L}$ in the next example is $\mathcal{D}$-radical and
strongly $\mathcal{F}$-semi\-sim\-p\-le.

\begin{example}
\emph{Modify the nest }$G$\emph{ in the example above as follows. Let
}$G=H\cup\{H_{2n}\}_{n=0}^{\infty}$\emph{. Let }$P_{n}$\emph{ be the
orthogonal projections on }$H_{2n}$\emph{ and }$Q_{n}=P_{n}-P_{n-1}$\emph{.
Let $\mathcal{L}$ be the Lie algebra of all compact operators }$T$\emph{
preserving }$G:$ $TP_{n}=P_{n}TP_{n},$\emph{ for all }$n$\emph{, and such that
Tr(}$Q_{n}TQ_{n})=0,$\emph{ for all }$n$\emph{. Let us check that }%
$\overline{[\mathcal{L},\mathcal{L}]}=\mathcal{L}$\emph{ whence }%
$\mathcal{D}(\mathcal{L})=\mathcal{L,}$ \emph{so that }$\mathcal{L}$\emph{ is
}$\mathcal{D}$\emph{-radical}.

\emph{For each} $n,$ \emph{set }$\mathcal{L}_{n}=\{T\in\mathcal{L:}$\emph{
}$T=P_{n}TP_{n}\}.$ \emph{For all }$T\in\mathcal{L,}$ \emph{we have }%
$TP_{n}\in\mathcal{L}_{n}$\emph{ and }$TP_{n}\rightarrow T.$\emph{ Hence
}$\cup_{n}\mathcal{L}_{n}$\emph{ is norm dense in }$\mathcal{L}$\emph{ and it
suffices to show that }$[\mathcal{L}_{n},\mathcal{L}_{n}]=\mathcal{L}_{n}%
$\emph{ for all }$n.$\emph{ Each }$T\in\mathcal{L}_{n}$\emph{ can be realized
as an upper triangular block-matrix }$T=(T_{ij})$\emph{ with entries }%
$T_{ij}=Q_{i}TQ_{j}$ \emph{in }$M_{2}(\mathbb{C})$\emph{ whose diagonal
entries }$T_{ii}$\emph{ belong to }$sl(2,\mathbb{C})$\emph{ and }$T_{ij}=0$
\emph{if }$i>n,$\emph{ or }$j>n.$

\emph{For }$k\leq m\leq n,$\emph{ the subspace} $\mathcal{L}_{n}%
^{km}=\{T=(T_{ij})\in\mathcal{L}_{n}:$ $T_{ij}=0$ \emph{if} $(i,j)\neq(k,m)\}$
\emph{of} $\mathcal{L}_{n}$ \emph{is isomorphic to }$M_{2}(\mathbb{C})$
\emph{if} $k\neq m,$ \emph{the Lie algebra }$\mathcal{L}_{n}^{kk}$ \emph{to}
$sl(2,\mathbb{C})$ \emph{and }$\mathcal{L}_{n}$ \emph{is the direct sum of
all} $\mathcal{L}_{n}^{km}.$ \emph{As} $[sl(2,\mathbb{C}),sl(2,\mathbb{C}%
)]=sl(2,\mathbb{C})$ \emph{and} $sl(2,\mathbb{C})M_{2}(\mathbb{C}%
)=M_{2}(\mathbb{C}),$ \emph{we have} \emph{that} $[\mathcal{L}_{n}%
^{kk},\mathcal{L}_{n}^{kk}]=\mathcal{L}_{n}^{kk}$ \emph{and} $[\mathcal{L}%
_{n}^{kk},\mathcal{L}_{n}^{km}]=\mathcal{L}_{n}^{kk}\mathcal{L}_{n}%
^{km}=\mathcal{L}_{n}^{km}.$ \emph{Thus} $[\mathcal{L}_{n},\mathcal{L}%
_{n}]=\mathcal{L}_{n},$ \emph{so that }$\mathcal{L}$\emph{ is }$\mathcal{D}%
$\emph{-radical}.

\emph{Setting }$I_{n}=\{T\in\mathcal{L}:T|_{H_{2n}}=0\},$\emph{ we see that
all }$I_{n}$\emph{ are closed ideals of finite codimension in $\mathcal{L,}$
}$I_{n+1}\subseteq I_{n}$\emph{ and }$\cap_{n=1}^{\infty}I_{n}=\{0\}$\emph{,
so that $\mathcal{L}$ is strongly $\mathcal{F}$-semisimple}.
\end{example}

A closed Lie ideal $I$ of $\mathcal{L}\in\mathfrak{L}$ is called
\textit{strongly} $\mathcal{F}$\textit{-absorbing} (\textit{strongly
Frattini-absorbing}) if $\mathcal{L}/I$ is strongly $\mathcal{F}$-semisimple.
Denote by $\mathrm{Abs}_{\mathcal{F}}^{s}\left(  \mathcal{L}\right)  $ the set
of all strongly $\mathcal{F}$-absorbing ideals of $\mathcal{L}$. Then
$\mathrm{Abs}_{\mathcal{F}}^{s}\left(  \mathcal{L}\right)  \subseteq
\mathrm{Abs}_{\mathcal{F}}\left(  \mathcal{L}\right)  ,$ for each
$\mathcal{L}\in\mathfrak{L.}$ Set%
\begin{equation}
\mathcal{F}_{s}\left(  \mathcal{L}\right)  =\mathfrak{p}\left(  \mathrm{Abs}%
_{\mathcal{F}}^{s}\left(  \mathcal{L}\right)  \right)  =\cap_{J\in
\mathrm{Abs}_{\mathcal{F}}^{s}\left(  \mathcal{L}\right)  }J. \label{8.3}%
\end{equation}
Then
\begin{equation}
\mathcal{F}\left(  \mathcal{L}\right)  \subseteq\mathcal{F}_{s}\left(
\mathcal{L}\right)  ,\text{ so that }\mathcal{F}\leq\mathcal{F}_{s}.
\label{8.4}%
\end{equation}
Clearly $\mathcal{F}_{s}\left(  \mathcal{L}\right)  =\{0\}$ if and only if
$\mathcal{L}$ is strongly $\mathcal{F}$-semisimple.

The following statement is similar to Lemma \ref{prim}.

\begin{lemma}
\label{sprim}Let $\mathcal{L}$ be a Banach Lie algebra$\mathfrak{.}$
\begin{enumerate}
 \item[(i)]
A closed Lie ideal $I$ of $\mathcal{L}$ is strongly $\mathcal{F}%
$-absorbing if and only if there is a complete\emph{,} lower finite-gap chain
of closed Lie ideals between $I$ and $\mathcal{L}$.
\item[(ii)] 
Let $I,J$ be closed Lie ideal of $\mathcal{L,}$ $J\subseteq I$ and
$\dim(I/J)<\infty.$ If $I$ is strongly $\mathcal{F}$-absorbing then $J$ is
strongly $\mathcal{F}$-absorbing.
\item[(iii)] 
Each complete\emph{,} lower finite-gap chain $C$ of closed Lie
ideals of $\mathcal{L}$ with $\mathfrak{s}\left(  C\right)  =\mathcal{L}$
consists of strongly $\mathcal{F}$-absorbing ideals of $\mathcal{L}$.
\item[(iv)] 
The set
$\mathrm{Abs}_{\mathcal{F}}^{s}\left(  \mathcal{L}\right)
$ is $\mathfrak{p}$-complete\emph{, }lower finite-gap family\emph{,}
$\mathfrak{s}(\mathrm{Abs}_{\mathcal{F}}^{s}\left(  \mathcal{L}\right)
)=\mathcal{L}$ and $\mathcal{F}_{s}\left(  \mathcal{L}\right)  $ is the
smallest strongly $\mathcal{F}$-absorbing Lie ideal of $\mathcal{L.}$
\item[(v)]
Let $\mathcal{M}$ be a closed Lie subalgebra of $\mathcal{L.}$ If
$I\in\mathrm{Abs}_{\mathcal{F}}^{s}\left(  \mathcal{L}\right)  $ then
$I\cap\mathcal{M}\in\mathrm{Abs}_{\mathcal{F}}^{s}\left(  \mathcal{M}\right)  .$
\item[(vi)]
If $\mathcal{L}$ is a commutative Banach Lie algebra then
$\mathcal{F}_{s}(\mathcal{L})=\{0\}.$
\end{enumerate}
\end{lemma}

\begin{proof}
Parts (i)-(iii), (v) can be proved in the same way as parts (i)-(iii), (v) in
Lemma \ref{prim}.

(iv) As $\mathcal{L}\in\mathrm{Abs}_{\mathcal{F}}^{s}\left(  \mathcal{L}%
\right)  $, we have $\mathfrak{s}\left(  \mathrm{Abs}_{\mathcal{F}}^{s}\left(
\mathcal{L}\right)  \right)  =\mathcal{L}$. Let $G=\{I_{\lambda}\}_{\lambda
\in\Lambda}$ be a subfamily in $\mathrm{Abs}_{\mathcal{F}}^{s}\left(
\mathcal{L}\right)  $. By (i), for each $I_{\lambda},$ there is a complete,
lower finite-gap chain $C_{\lambda}$ of closed Lie ideals of $\mathcal{L}$
between $I_{\lambda}$ and $\mathcal{L}$. By Proposition \ref{ui},
$X_{G}:=\left(  \cup_{\lambda}C_{\lambda}\right)  ^{\mathfrak{p}}$ is a lower
finite-gap family of closed Lie ideals of $\mathcal{L.}$ By Lemma \ref{L2.2},
$X_{G}$ has a complete, lower finite-gap chain $C$ of subspaces (i.e., closed
Lie ideals of $\mathcal{L}$) between $\mathfrak{p}\left(  X_{G}\right)  $ and
$\mathcal{L}$. By (i), $\mathfrak{p}\left(  X_{G}\right)  \in\mathrm{Abs}%
_{\mathcal{F}}^{s}\left(  \mathcal{L}\right)  .$ Also
\[
\mathfrak{p}\left(  X_{G}\right)  =\mathfrak{p}\left(  \left(  \cup_{\lambda
}C_{\lambda}\right)  ^{\mathfrak{p}}\right)  =\cap_{\lambda}\mathfrak
{p}\left(  C_{\lambda}\right)  =\cap_{\lambda}I_{\lambda}=\mathfrak{p}\left(
G\right)  .
\]
Thus $\mathfrak{p}\left(  G\right)  \in\mathrm{Abs}_{\mathcal{F}}^{s}\left(
\mathcal{L}\right)  $. Therefore $\mathrm{Abs}_{\mathcal{F}}^{s}\left(
\mathcal{L}\right)  $ is $\mathfrak{p}$-complete.

Take $G=\mathrm{Abs}_{\mathcal{F}}^{s}\left(  \mathcal{L}\right)  $ and let
$I\in\mathrm{Abs}_{\mathcal{F}}^{s}\left(  \mathcal{L}\right)  $. By the above
argument, $X_{G}$ is a lower finite-gap family and $\mathrm{Abs}_{\mathcal{F}%
}^{s}\left(  \mathcal{L}\right)  \subseteq X_{G}.$ Then there is $J\in X_{G}$
such that $J\subsetneqq I$ and $\dim(I/J)<\infty.$ By (ii), $J\in
\mathrm{Abs}_{\mathcal{F}}^{s}\left(  \mathcal{L}\right)  $. Hence
$\mathrm{Abs}_{\mathcal{F}}^{s}\left(  \mathcal{L}\right)  $ is a lower
finite-gap family.

(vi) If $\mathcal{L}$ is commutative then, by (ii), each subspace of
$\mathcal{L}$ of finite codimension is a strongly $\mathcal{F}$-absorbing Lie
ideal of $\mathcal{L}$. Hence, by (\ref{8.3}), $\mathcal{F}_{s}\left(
\mathcal{L}\right)  =\{0\}$.
\end{proof}

We will construct now some new examples of strongly $\mathcal{F}$-semi\-sim\-ple
Lie algebras as the normed direct products and the $c_{0}$-direct products of
strongly $\mathcal{F}$-semisimple Lie algebras. Let $\{\mathcal{L}_{\lambda
}\}_{\lambda\in\Lambda}$ be a family of Banach Lie algebras with a bounded set
of multiplication constants, let $\mathcal{L}=\oplus_{\Lambda}\mathcal{L}%
_{\lambda}$ and $\widehat{\mathcal{L}}=\widehat{\oplus}_{\Lambda}%
\mathcal{L}_{\lambda}$ (see (\ref{e3.1})). For $a=(a_{\lambda})_{\lambda
\in\Lambda}\in\mathcal{L}$, let $\psi_{\mu}(a)=a_{\mu}$, so $\psi_{\mu}$ is a
homomorphism from $\mathcal{L}$ to $\mathcal{L}_{\mu}$.

\begin{proposition}
\label{E2.4}\emph{(i)}
If all $\mathcal{L}_{\lambda}$ are strongly
$\mathcal{F}$-semisimple then $\mathcal{L}$ and $\widehat{\mathcal{L}}$ are
strongly $\mathcal{F}$-semisimple.
\begin{enumerate}
 \item[(ii)] 
 If all $\mathcal{L}_{\lambda}$ are finite-dimensional and
semisimple then
\begin{enumerate}
 \item[\textit{a)}] 
$\mathcal{L}$ has a maximal lower finite-gap chain of
characteristic Lie ideals from $\{0\}$ to $\mathcal{L}$;
  \item[\textit{b)}]
$\widehat{\mathcal{L}}$ also has such a chain and is $\mathcal{D}$-radical.
\end{enumerate}
\end{enumerate}
\end{proposition}

\begin{proof}
For each $\mu\in\Lambda$, set $\mathcal{N}_{\mu}=\psi_{\mu}^{-1}(0).$ Then
$\mathcal{L/N}_{\mu}$ is strongly $\mathcal{F}$-semi\-sim\-p\-le, as it is
isomorphic to $\mathcal{L}_{\mu}$. Hence $\mathcal{N}_{\mu}$ is a strongly
$\mathcal{F}$-absorbing Lie ideal. Therefore, by (\ref{8.3}), $\mathcal{F}%
_{s}(\mathcal{L})\subseteq\cap_{\mu\in\Lambda}\mathcal{N}_{\mu}=\{0\}.$ Part
(i) is proved.

If each $\mathcal{L}_{\lambda}$ is semisimple finite-dimensional, then
$\mathcal{L}$ has no non-zero commutative Lie ideals$.$ Hence the set
$\mathcal{A}_{\mathcal{L}}^{\text{ch}}=\mathcal{A}_{\mathcal{L}}=\{\{0\}\}$ is
a lower finite-gap family. By Corollary \ref{C8.1c}, $\mathcal{L}$ has the
required chain. The existence of this type of chains in $\widehat{\mathcal{L}%
}$ can be proved similarly. As $\mathcal{D}(\mathcal{L}_{\lambda}%
)=\mathcal{L}_{\lambda},$ for each $\lambda,$ we have from Proposition
\ref{P3.1n} that $\mathcal{D}(\widehat{\mathcal{L}})=\widehat{\oplus}%
_{\Lambda}\mathcal{D}(\mathcal{L}_{\lambda})=\widehat{\mathcal{L}}.$
\end{proof}

Let $G$ be the set of all closed Lie ideals of $\mathcal{L.}$ It is
$\mathfrak{p}$-complete. Comparing (\ref{6.d}), Theorem \ref{T6.x} and Lemma
\ref{prim}, we have that $G_{\text{f}}=$ Abs$_{\mathcal{F}}^{s}(\mathcal{L})$
and $\Delta_{G}=\mathcal{F}_{s}\mathcal{(L).}$ This and Lemma \ref{L2.2} yield

\begin{corollary}
\label{C8.1n}\emph{(i) }$\mathcal{F}_{s}\mathcal{(L)}=\mathfrak{p}(C)$ for
each maximal\emph{,} lower finite-gap chain $C$ of closed Lie ideals of
$\mathcal{L}$ with $\mathfrak{s}(C)=\mathcal{L}$.
\begin{enumerate}
 \item[(ii)] 
Each $\mathfrak{p}$-complete$,$ lower finite-gap chain $C$ of
closed Lie ideals of $\mathcal{L}$ with $\mathfrak{s}(C)=\mathcal{L}$ extends
to a maximal\emph{,} lower finite-gap chain of closed Lie ideals of
$\mathcal{L}$.
\end{enumerate}
\end{corollary}

Note that $\mathcal{F}_{s}\mathcal{(L)}$ may have closed Lie ideals of finite
codimension, but they are not Lie ideals of $\mathcal{L.}$ Thus all lower
finite-gap chains of closed Lie ideals end at $\mathcal{F}_{s}\mathcal{(L)}$
and can not be extended further.

\begin{corollary}
Each closed Lie subalgebra $\mathcal{M}$ of a strongly $\mathcal{F}%
$-semi\-sim\-ple algebra $\mathcal{L}$ is strongly $\mathcal{F}$-semisimple.
\end{corollary}

\begin{proof}
Since $\{0\}\in\mathrm{Abs}_{\mathcal{F}}^{s}\left(  \mathcal{L}\right)  ,$ it
follows from Lemma \ref{sprim}(v) that $\{0\}\in\mathrm{Abs}_{\mathcal{F}}%
^{s}\left(  \mathcal{M}\right)  .$ Thus $\mathcal{M}$ is a strongly
$\mathcal{F}$-semisimple$\mathcal{.}$
\end{proof}

\begin{theorem}
\label{ess1}$\mathcal{F}_{s}$ is an over radical in $\overline{\mathbf{L}}$
\emph{(}see Definition \emph{\ref{D3.1}).}
\end{theorem}

\begin{proof}
Let $f$: $\mathcal{L}\longrightarrow\mathcal{M}$ be a morphism in
$\overline{\mathbf{L}}$. By Lemma \ref{sprim}(i) and (iv), there exists a
complete, lower finite-gap chain $C$ of strongly $\mathcal{F}$-absorbing
ideals of $\mathcal{M}$ between $\mathcal{F}_{s}\left(  \mathcal{M}\right)  $
and $\mathcal{M}$. Then $C^{\prime}:=\{f^{-1}\left(  I\right)  $: $I\in C\}$
is a complete, lower finite-gap chain of closed Lie ideals between
$f^{-1}(\mathcal{F}_{s}\left(  \mathcal{M}\right)  )$ and $\mathcal{L}$. By
Lemma \ref{sprim}(iii), $C^{\prime}$ consists of strongly $\mathcal{F}%
$-absorbing ideals of $\mathcal{L}$. So $\mathcal{F}_{s}\left(  \mathcal{L}%
\right)  \subseteq f^{-1}(\mathcal{F}_{s}\left(  \mathcal{M}\right)  ),$ as
$\mathcal{F}_{s}\left(  \mathcal{L}\right)  $ is the smallest strongly
$\mathcal{F}$-absorbing ideal of $\mathcal{L}$ by Lemma \ref{sprim}(iv). Hence
$f\left(  \mathcal{F}_{s}\left(  \mathcal{L}\right)  \right)  \subseteq
\mathcal{F}_{s}\left(  \mathcal{M}\right)  $. This means that $\mathcal{F}%
_{s}$ is a preradical.

Let $J\vartriangleleft\mathcal{L}.$ By Lemma \ref{sprim}(v), $I\cap J$ is a
strongly $\mathcal{F}$-absorbing ideal of $J$, for each $I\in\mathrm{Abs}%
_{\mathcal{F}}^{s}\left(  \mathcal{L}\right)  .$ Thus $\{I\cap J$:
$I\in\mathrm{Abs}_{\mathcal{F}}^{s}\left(  \mathcal{L}\right)  \}\subseteq
\mathrm{Abs}_{\mathcal{F}}^{s}\left(  J\right)  $. Hence $\mathcal{F}_{s}$ is
balanced, as%
\begin{align*}
\mathcal{F}_{s}\left(  J\right)   &  =\mathfrak{p}(\mathrm{Abs}_{\mathcal{F}%
}^{s}\left(  J\right)  )\subseteq\mathfrak{p}(J\cap\mathrm{Abs}_{\mathcal{F}%
}^{s}\left(  \mathcal{L}\right)  )\\
&  =J\cap\mathfrak{p}(\mathrm{Abs}_{\mathcal{F}}^{s}\left(  \mathcal{L}%
\right)  )=J\cap\mathcal{F}_{s}\left(  \mathcal{L}\right)  \subseteq
\mathcal{F}_{s}\left(  \mathcal{L}\right)  .
\end{align*}

By Lemma \ref{sprim}(iv), $\mathcal{F}_{s}\left(  \mathcal{L}\right)
\in\mathrm{Abs}_{\mathcal{F}}^{s}\left(  \mathcal{L}\right)  .$ Therefore
$\mathcal{L}/\mathcal{F}_{s}\left(  \mathcal{L}\right)  $ is strongly
$\mathcal{F}$-semisimple. Thus $\{0\}\in\mathrm{Abs}_{\mathcal{F}}^{s}(\left(
\mathcal{L}/\mathcal{F}_{s}\left(  \mathcal{L}\right)  \right)  ,$ so that
$\mathcal{F}_{s}\left(  \mathcal{L}/\mathcal{F}_{s}\left(  \mathcal{L}\right)
\right)  =\{0\}$. Hence $\mathcal{F}_{s}$ is an over radical.
\end{proof}

Consider a Banach space $X$ as a commutative Lie algebra. Let\emph{ }$L$ be a
Banach Lie algebra and $\varphi$ be a bounded Lie homomorphism from $L$ into
$B(X)=\mathfrak{D}(X).$ Let $\mathcal{L}=L\oplus^{\varphi}X$ (see
(\ref{fsemi})) be the semidirect product. Set $M=\varphi(L).$ The set Lat $M$
of all closed subspaces of $X$ invariant for all operators in $M$ is\emph{
}$\mathfrak{p}$-complete. It follows from Corollary \ref{C6.5} that there is a
subspace $\Delta_{M}\in$ Lat $M$ such that $\mathfrak{p}(C)=\Delta_{M}$ for
each maximal, lower finite-gap chain $C$ of invariant subspaces of $M$ with
$\mathfrak{s}(C)=X;$ and $\Delta_{M}$ has no invariant subspaces of finite codimension.

\begin{proposition}
\label{P8.2n}\emph{(i) }If $L$ is strongly $\mathcal{F}$-semisimple then
$\mathcal{F}_{s}\left(  \mathcal{L}\right)  =\{0\}\oplus^{\text{\emph{id}}%
}\Delta_{M}.$
\begin{enumerate}
 \item[(ii)]
 If $\Delta_{M}\neq\{0\}$ $($e.g. $M$ has no non-trivial invariant
subspaces in $X),$ then%
\[
\mathcal{F}\left(  \mathcal{L}\right)  =\mathcal{F}_{s}\left(  \mathcal{F}%
_{s}\left(  \mathcal{L}\right)  \right)  =\{0\}\neq\mathcal{F}_{s}\left(
\mathcal{L}\right)  .
\]
\end{enumerate}
\end{proposition}

\begin{proof}
(i) By (\ref{fsemi}), any Lie ideal of $\mathcal{L}$ contained in
$\{0\}\oplus^{\varphi}X$ has form $J_{Z}=\{0\}\oplus^{\varphi}Z,$ where
$Z\subseteq X$ is invariant for $M,$ i.e., $Z\in$ Lat $M.$ By Proposition
\ref{semi}(i),\emph{ }$\mathcal{F}_{s}\left(  \mathcal{L}\right)
\subseteq\mathcal{F}_{s}(L\mathcal{)}\oplus^{\varphi}X=\{0\}\oplus^{\varphi
}X.$ Hence, since $\mathcal{F}_{s}\left(  \mathcal{L}\right)  $ is a Lie ideal
of $\mathcal{L,}$ we have $\mathcal{F}_{s}\left(  \mathcal{L}\right)
=J_{Y}=\{0\}\oplus^{\varphi}Y$ where $Y\in$ Lat $M\mathcal{.}$

As $\mathcal{F}_{s}(L)=\{0\}$, it follows from Corollary \ref{C8.1n} that
there is a maximal, lower finite-gap chain $C_{M}=\{I_{\lambda}\}$ of closed
Lie ideals of $L$ between $L$ and $\{0\}.$ Then $\widetilde{C}_{M}%
=\{I_{\lambda}\oplus^{\varphi}X\}$ is a maximal, lower finite-gap chain of
closed Lie ideals of $\mathcal{L}$ between $\mathcal{L}$ and $\{0\}\oplus
^{\varphi}X.$

Let $C_{\Delta}=\{L_{\mu}\}$ be a maximal, lower finite-gap chain of invariant
subspaces of $M$ with $\mathfrak{s}(C_{\Delta})=X.$ By Corollary \ref{C6.5},
$\mathfrak{p}(C_{\Delta})=\Delta_{M}.$ Hence $\widetilde{C}_{\Delta
}=\{\{0\}\oplus^{\varphi}L_{\mu}\}$ is a maximal, lower finite-gap chain of
Lie ideals of $\mathcal{L}$ in $\{0\}\oplus^{\varphi}X$ and $\mathfrak
{p}(\widetilde{C}_{\Delta})=\{0\}\oplus^{\varphi}\Delta_{M}.$ Therefore
$C=\widetilde{C}_{M}\cup\widetilde{C}_{\Delta}$ is a maximal, lower finite-gap
chain of Lie ideals of $\mathcal{L}$, $\mathfrak{p}(C)=\{0\}\oplus^{\varphi
}\Delta_{M}$ and $\mathfrak{s}(C)=\mathcal{L}.$ By Corollary \ref{C8.1n},
$\mathcal{F}_{s}(\mathcal{L})=\mathfrak{p}(C)=\{0\}\oplus^{\varphi}\Delta
_{M}=\{0\}\oplus^{\text{id}}\Delta_{M}.$

(ii)  $\mathcal{F}_{s}\left(  \mathcal{F}%
_{s}\left(  \mathcal{L}\right)  \right)  =\mathcal{F}_{s}\left(
\{0\}\oplus^{\text{id}}\Delta_{M}\right)  =\{0\}\neq\mathcal{F}_{s}\left(
\mathcal{L}\right)  $ by (i) and Lemma \ref{sprim}(vi). By Example \ref{E7.1}(iii),\ $\mathcal{F}\left(
\mathcal{L}\right)  =\{0\}$.
\end{proof}

It follows from Proposition \ref{P8.2n} that $\mathcal{F}_{s}\left(  \mathcal{F}_{s}\left(  L\oplus^{\varphi
}X\right)  \right)  =\mathcal{F}\left(  L\oplus^{\varphi}X\right)  $ and $\mathcal{F}_{s}$ is not a
radical. As the
following theorem shows, 
$\mathcal{F}_{s}\left(  \mathcal{F}%
_{s}\left(  \mathcal{L}\right)  \right)  =\mathcal{F}\left(  \mathcal{L}%
\right)  $ holds for all $\mathcal{L}\in\mathfrak{L.}$

\begin{theorem}
\label{ess}For each algebra $\mathcal{L}\in\mathfrak{L}\mathcal{,}$ the
quotient Lie algebra $\mathcal{F}_{s}\left(  \mathcal{L}\right)
/\mathcal{F}\left(  \mathcal{L}\right)  $ is commutative\emph{,}%
\begin{align}
\mathcal{F}_{s}\left(  \mathcal{L}/\mathcal{F}\left(  \mathcal{L}\right)
\right)   &  =\mathcal{F}_{s}\left(  \mathcal{L}\right)  /\mathcal{F}\left(
\mathcal{L}\right)  =\mathfrak{s}\left(  \mathrm{Ess}_{l}\left(
\mathcal{A}_{\mathcal{L}/\mathcal{F}\left(  \mathcal{L}\right)  }\right)
\right)  \text{ and}\nonumber\\
\mathcal{F}_{s}\left(  \mathcal{F}_{s}\left(  \mathcal{L}\right)  \right)   &
=\mathcal{F}\left(  \mathcal{L}\right)  . \label{8.7}%
\end{align}
\end{theorem}

\begin{proof}
Firstly assume that $\mathcal{L}$ is $\mathcal{F}$-semisimple. Then
$\mathcal{F(L})=\{0\}$ and we have to show that%
\begin{align}
\mathcal{F}_{s}\left(  \mathcal{L}\right)  \text{ is commutative},\text{
}\mathcal{F}_{s}\left(  \mathcal{L}\right)   &  =\mathfrak{s}\left(
\mathrm{Ess}_{l}\left(  \mathcal{A}_{\mathcal{L}}\right)  \right)  \text{ and
}\nonumber\\
\mathcal{F}_{s}\left(  \mathcal{F}_{s}\left(  \mathcal{L}\right)  \right)   &
=\{0\}. \label{8.6}%
\end{align}

Let $I\in\mathrm{Ess}_{l}\left(  \mathcal{A}_{\mathcal{L}}\right)  $. Then
$\mathcal{L}$ has no Lie ideals contained in $I$ that have finite, non-zero
codimension in $I$. By Lemma \ref{sprim}(iv), Abs$_{\mathcal{F}}^{s}\left(
\mathcal{L}\right)  $ is a lower finite-gap family of Lie ideals of
$\mathcal{L.}$ Hence, by Corollary \ref{pi}, $I\cap\mathrm{Abs}_{\mathcal{F}%
}^{s}\left(  \mathcal{L}\right)  :=\left\{  I\cap J\text{: }J\in
\mathrm{Abs}_{\mathcal{F}}^{s}\left(  \mathcal{L}\right)  \right\}  $ is also
a lower finite-gap family of Lie ideals of $\mathcal{L}$ and $I$ belongs to
it$,$ as $\mathcal{L}\in\mathrm{Abs}_{\mathcal{F}}^{s}\left(  \mathcal{L}%
\right)  $. Thus $I\cap\mathrm{Abs}_{\mathcal{F}}^{s}\left(  \mathcal{L}%
\right)  =\{I\}$, so that $I$ lies in each $J$ in $\mathrm{Abs}_{\mathcal{F}%
}^{s}\left(  \mathcal{L}\right)  .$ Hence $I\subseteq\mathfrak{p}\left(
\mathrm{Abs}_{\mathcal{F}}^{s}\left(  \mathcal{L}\right)  \right)
=\mathcal{F}_{s}\left(  \mathcal{L}\right)  $. As $I$ is arbitrary,
$\mathfrak{s}\left(  \mathrm{Ess}_{l}\left(  \mathcal{A}_{\mathcal{L}}\right)
\right)  \subseteq\mathcal{F}_{s}\left(  \mathcal{L}\right)  $. Set
$K=\mathfrak{s}\left(  \mathrm{Ess}_{l}\left(  \mathcal{A}_{\mathcal{L}%
}\right)  \right)  $.

Let a Lie ideal $I$ contain $K.$ If $I$ contains a Lie ideal $J$ of non-zero,
finite codimension in $I$, then $K\subseteq J$. Indeed, if $L\in
\mathrm{Ess}_{l}\left(  \mathcal{A}_{\mathcal{L}}\right)  $ then $L\subseteq
J$; otherwise, by Lemma \ref{Closed}, $L\cap J$ has non-zero, finite
codimension in $L$ which contradicts the fact that $L\in\mathrm{Ess}%
_{l}\left(  \mathcal{A}_{\mathcal{L}}\right)  $. Hence $K\subseteq J$.

Assume that $K\neq I.$ If $I\in\mathcal{A}_{\mathcal{L}}$ then $I$ contains a
Lie ideal $J\in\mathcal{A}_{\mathcal{L}}$ of non-zero, finite codimension in
$I$. By the above, $K\subseteq J$. Let $I$ be non-commutative, i.e., $I\in$
Lid$(\mathcal{L})\diagdown\mathcal{A}_{\mathcal{L}}.$ By Proposition
\ref{strongl}(i), Lid$(\mathcal{L})\diagdown\mathcal{A}_{\mathcal{L}}$ is a
lower finite-gap family modulo $\mathcal{A}_{\mathcal{L}}.$ Hence $I$ contains
a Lie ideal $J$ that has non-zero, finite codimension in $I.$ By the above,
$K\subseteq J$.

Thus the set $\{I$: $I\vartriangleleft\mathcal{L}$ and $K\subseteq I\}$ is a
$\mathfrak{p}$-complete, lower finite-gap family. By Lemma \ref{L2.2}, there
is a complete, lower finite-gap chain of closed Lie ideals between $K$ and
$\mathcal{L}$. Hence $K$ is strongly $\mathcal{F}$-absorbing by Lemma
\ref{sprim}(i). Therefore $\mathcal{F}_{s}\left(  \mathcal{L}\right)
\subseteq K$. Thus we have finally that $\mathcal{F}_{s}\left(  \mathcal{L}%
\right)  =K=\mathfrak{s}\left(  \mathrm{Ess}_{l}\left(  \mathcal{A}%
_{\mathcal{L}}\right)  \right)  $.

By Lemma \ref{sprim}(iv), $\mathcal{F}_{s}\left(  \mathcal{L}\right)  $ is the
smallest strongly $\mathcal{F}$-absorbing ideal of $\mathcal{L.}$ Hence, by
Lemma \ref{sprim}(ii),%
\begin{align}
&  \mathcal{F}_{s}\left(  \mathcal{L}\right)  \text{ contains no closed Lie
ideals of }\mathcal{L}\nonumber\\
&  \text{of non-zero finite codimension}. \label{8.5}%
\end{align}

Let $\dim\mathcal{F}_{s}(\mathcal{L})<\infty.$ As $\{0\}$ is a Lie ideal of
finite codimension in $\mathcal{F}_{s}\left(  \mathcal{L}\right)  $, we have
from (\ref{8.5}) that $\mathcal{F}_{s}\left(  \mathcal{L}\right)  =\{0\}$ and
(\ref{8.6}) holds.

Let $\dim\mathcal{F}_{s}(\mathcal{L})=\infty.$ If $\mathcal{F}_{s}\left(
\mathcal{L}\right)  $ is not commutative, it has a closed Lie subalgebra of
non-zero, finite codimension by Theorem \ref{free}(v). Hence, by Corollary
\ref{C3.1}, $\mathcal{F}_{s}\left(  \mathcal{L}\right)  $ contains a closed
Lie ideal of $\mathcal{L}$ of non-zero, finite codimension$.$ This contradicts
(\ref{8.5}) and shows that $\mathcal{F}_{s}\left(  \mathcal{L}\right)  $ is
commutative. By Lemma \ref{sprim}(vi), $\mathcal{F}_{s}\left(  \mathcal{F}%
_{s}\left(  \mathcal{L}\right)  \right)  =\{0\}$ and (\ref{8.6}) is proved.

Suppose now that $\mathcal{L}$ is not $\mathcal{F}$-semisimple. Let $q$:
$\mathcal{L}\rightarrow\mathcal{M:}=\mathcal{L}/\mathcal{F}\left(
\mathcal{L}\right)  .$ If $I\in\mathrm{Abs}_{\mathcal{F}}^{s}\left(
\mathcal{M}\right)  $ then $\mathcal{M}/I$ is strongly $\mathcal{F}%
$-semisimple. As $\mathcal{M}/I\approx\mathcal{L}/q^{-1}(I)$, we have
$q^{-1}(I)\in\mathrm{Abs}_{\mathcal{F}}^{s}\left(  \mathcal{L}\right)  .$

Conversely, let $J\in\mathrm{Abs}_{\mathcal{F}}^{s}\left(  \mathcal{L}\right)
.$ As $\mathrm{Abs}_{\mathcal{F}}^{s}\left(  \mathcal{L}\right)
\subseteq\mathrm{Abs}_{\mathcal{F}}\left(  \mathcal{L}\right)  $ and (see
Lemma \ref{prim}(iv)) $\mathcal{F}\left(  \mathcal{L}\right)  =\mathfrak
{p}(\mathrm{Abs}_{\mathcal{F}}\left(  \mathcal{L}\right)  ),$ we have
$\mathcal{F}\left(  \mathcal{L}\right)  \subseteq J.$ As $\mathcal{M}%
/q(J)\approx\mathcal{L}/J,$ we have $q(J)\in\mathrm{Abs}_{\mathcal{F}}%
^{s}\left(  \mathcal{M}\right)  .$ Thus, by (\ref{8.3}),%
\begin{align*}
\mathcal{F}_{s}\left(  \mathcal{M}\right)   &  =\underset{I\in\mathrm{Abs}%
_{\mathcal{F}}^{s}\left(  \mathcal{M}\right)  }{\cap}I=\underset
{J\in\mathrm{Abs}_{\mathcal{F}}^{s}\left(  \mathcal{L}\right)  }{\cap}q(J)\\
&  =q(\mathcal{F}_{s}\left(  \mathcal{L}\right)  )=\mathcal{F}_{s}\left(
\mathcal{L}\right)  /\mathcal{F}\left(  \mathcal{L}\right)  .
\end{align*}

The Lie algebra $\mathcal{M}$ is $\mathcal{F}$-semisimple, as $\mathcal{F}$ is
a radical. Hence it follows from (\ref{8.6}) that the Lie ideal $\mathcal{F}%
_{s}(\mathcal{M})=\mathcal{F}_{s}\left(  \mathcal{L}\right)  /\mathcal{F}%
\left(  \mathcal{L}\right)  $ is commutative, (\ref{8.7}) holds and
\begin{equation}
\mathcal{F}_{s}(\mathcal{F}_{s}\left(  \mathcal{L}\right)  /\mathcal{F}\left(
\mathcal{L}\right)  )=\mathcal{F}_{s}\left(  \mathcal{F}_{s}\left(
\mathcal{M}\right)  \right)  =\{0\}. \label{8.8}%
\end{equation}

Set $I=\mathcal{F(L)}$ and $L=\mathcal{F}_{s}\left(  \mathcal{L}\right)  .$
Then formula (\ref{8.8}) turns into $\mathcal{F}_{s}(L/I)=\{0\}.$ By
(\ref{8.4}), $\mathcal{F}\leq\mathcal{F}_{s},$ so that $I\vartriangleleft L.$
As $\mathcal{F}$ is a radical, $I=\mathcal{F(L})=\mathcal{F(}\mathcal{F(L)}%
)=\mathcal{F}(I).$ Hence, by Proposition \ref{P3.7}(v), $\mathcal{F(}%
\mathcal{L})=I=\mathcal{F}_{s}(L)=\mathcal{F}_{s}\mathcal{(F}_{s}%
(\mathcal{L)).}$ The proof is complete.
\end{proof}

Let $\mathcal{L}$ be an $\mathcal{F}$-semisimple Banach algebra and
$\mathcal{F}_{s}\left(  \mathcal{L}\right)  \neq\{0\}.$ By Theorem \ref{ess},
$\mathcal{F}_{s}\left(  \mathcal{L}\right)  $ is commutative. Let\emph{
}$L=\mathcal{L}/\mathcal{F}_{s}\left(  \mathcal{L}\right)  $ and\emph{ }$q$:
$\mathcal{L}\longrightarrow L$ be the quotient map. As \emph{$\mathcal{F}_{s}$
}is an over radical, $\mathcal{F}_{s}(L)=\{0\}.$\emph{ }Thus each
$\mathcal{F}$-semisimple Banach algebra $\mathcal{L}$ is an \textit{extension}%
\emph{ }of a commutative Banach Lie algebra $\mathcal{F}_{s}\left(
\mathcal{L}\right)  $ by an $\mathcal{F}_{s}$-semisimple algebra\emph{. }

Associate with $\mathcal{L}$ the semidirect product $\mathcal{N}%
=L\oplus^{\varphi}\mathcal{F}_{s}\left(  \mathcal{L}\right)  $ (see
(\ref{fsemi})) in the following way.\textbf{ }As $\mathcal{F}_{s}\left(
\mathcal{L}\right)  $ is commutative, the map $\varphi$ from $L$ into the Lie
algebra $\mathfrak{D}(\mathcal{F}_{s}(\mathcal{L)})=B(\mathcal{F}%
_{s}(\mathcal{L)})$ of all bounded derivations on $\mathcal{F}_{s}\left(
\mathcal{L}\right)  $ defined by $\varphi(q(a))=\delta_{a}|_{\mathcal{F}%
_{s}\left(  \mathcal{L}\right)  },$ where $\delta_{a}(x)=[a,x]$ for
$x\in\mathcal{F}_{s}\left(  \mathcal{L}\right)  \mathcal{,}$ is a correctly
defined Lie homomorphism$.$ Hence $\mathcal{N}$ is well defined. If
$\mathcal{L}$ has a closed Lie subalgebra topologically isomorphic to $L,$
then $\mathcal{L}$ is topologically isomorphic to $\mathcal{N}$.

By Proposition \ref{semi}(ii) 3), $\mathcal{N}$ is 
Frattini-semisimple.
Moreover, we have $\mathcal{F}_{s}\left(  \mathcal{N}\right)  =\{0\}\oplus
^{\varphi}\mathcal{F}_{s}\left(  \mathcal{L}\right)  .$ Indeed, let
$M=\varphi\left(  L\right)  .$ The lattice Lat $M$ of subspaces in
$\mathcal{F}_{s}\left(  \mathcal{L}\right)  $ invariant for $M$ coincides with
the lattice of Lie ideals of $\mathcal{L}$\emph{ }in $\mathcal{F}_{s}\left(
\mathcal{L}\right)  .$ By Corollary \ref{C8.1n}, $\mathcal{F}_{s}\left(
\mathcal{L}\right)  $ contains no Lie ideals of $\mathcal{L}$ of non-zero,
finite codimension. Hence Lat $M$ has no subspaces of finite codimension. Thus
(see Proposition \ref{P8.2n}) the subspace $\Delta_{M}=\mathcal{F}_{s}\left(
\mathcal{L}\right)  $ and $\mathcal{F}_{s}\left(  \mathcal{N}\right)
=\{0\}\oplus^{\varphi}\mathcal{F}_{s}\left(  \mathcal{L}\right)  .$

\section{The structure of Frattini-free Banach Lie algebras}

In this section we study a special subclass of Frattini-semisimple Lie
algebras that consists of Frattini-free Lie algebras. A Banach Lie algebra
$\mathcal{L}$ is called \textit{Frattini-free }if it has sufficiently many
maximal closed Lie subalgebras of finite codimension, that is,%
\[
P_{\mathfrak{S}^{\max}}\left(  \mathcal{L}\right)  =\cap_{L\in\mathfrak
{S}_{\mathcal{L}}^{\max}}L=\{0\},\text{ or }\mathcal{L}\in\mathbf{Sem}%
(P_{\mathfrak{S}^{\text{max}}}).
\]
We will use the term \textit{Jacobson-free} for Banach Lie algebras in
$\mathbf{Sem}(P_{\mathfrak{J}^{\text{max}}})$.

Marshall in \cite[p. 417]{M} proved that all \textit{simple }%
finite-dimensional\textit{ }Lie algebras are Frattini-free. In Theorem
\ref{T4.3} we give a full description of Frattini-free Lie algebras: they are
isomorphic to subdirect products of the normed direct products of
finite-dimensional subsimple Lie algebras.

\subsection{Subsimple algebras and submaximal ideals}

Frattini-free algebras need not have sufficiently many maximal Lie ideals (see
Example \ref{E3}(iii)). Instead they have sufficiently many
\textit{submaximal} ideals (see Theorem \ref{T4.3} below).

\begin{definition}
\label{D9.1}\emph{(i) }We call a finite-dimensional Lie algebra $\mathcal{L}$
\emph{subsimple }if either\emph{ }$\dim\mathcal{L}=1$ or it has a proper
maximal Lie subalgebra that contains no non-zero Lie ideals of\emph{
}$\mathcal{L}.$

\emph{(ii) }We call a closed Lie ideal $J$ of\emph{ }$\mathcal{L}\in
\mathfrak{L}$ \emph{submaximal}$,$ if\emph{ }$\mathcal{L}/J$ is a subsimple
Lie algebra$.$
\end{definition}

\begin{lemma}
\label{L9.1}Each subsimple Lie algebra $\mathcal{L}$ is Frattini-free. Its
center $Z=\{0\}$ if $\dim\mathcal{L}\geq2.$
\end{lemma}

\begin{proof}
Let $\dim\mathcal{L}\geq2$ and let a maximal Lie subalgebra $M$ of
$\mathcal{L}$ contain no non-zero Lie ideal\textbf{s} of $\mathcal{L}$. As
$P_{\mathfrak{S}^{\max}}\left(  \mathcal{L}\right)  =\cap_{L\in\mathfrak
{S}_{\mathcal{L}}^{\max}}L$ is a Lie ideal of $\mathcal{L}$ contained in $M$,
we have $P_{\mathfrak{S}^{\max}}\left(  \mathcal{L}\right)  =\{0\}$.

If $Z\neq\{0\}$ then $M\cap Z=\{0\},$ as $M\cap Z$ is a Lie ideal of
$\mathcal{L}$. Hence $Z\dotplus M=\mathcal{L}$ and $\dim\left(  Z\right)  =1$
by maximality of $M$. Thus $M$ itself is a Lie ideal of $\mathcal{L}$ -- a
contradiction. Hence $Z=\{0\}$.
\end{proof}

It follows from the definition that simple Lie algebras are subsimple. To
clarify the structure of subsimple Lie algebras, consider the following
classes of finite-dimensional Lie algebras:

\begin{itemize}
\item [$(\mathbf{I})$]\textit{Class} \textbf{(I) }\textit{consists of}\textbf{
}\textit{Lie algebras }$\mathcal{L}=N\oplus N,$ \textit{where} $N$ \textit{is
a simple Lie algebra}.

\item[$(\mathbf{II)}$] \textit{Class} \textbf{(II) }\textit{consists
of}\textbf{ }\textit{Lie algebras }$\mathcal{L}=N\oplus^{\mathrm{{id}}}X,$
\textit{where} $N$ \textit{is a Lie algebra of operators on a
finite-dimensional space} $X$ \textit{which has no non-trivial invariant
subspaces. }
\end{itemize}

\begin{lemma}
All Lie algebras in classes $(\mathbf{I})$ and $\mathbf{(II)}$ are subsimple\textit{.}
\end{lemma}

\begin{proof}
Let $\mathcal{L}=N\oplus N.$ Clearly, $N\oplus\{0\}$ and $\{0\}\oplus N$ are
the only non-trivial Lie ideals of $\mathcal{L.}$ Hence the Lie subalgebra
$M=\{a\oplus a$: $a\in N\}$ does not contain non-zero Lie ideals. To see that
it is maximal note that, for each $b=a_{1}\oplus a_{2}\notin M$, the
subalgebra $M^{\prime}$ generated by $\{b\}\cup M$ contains all elements of
the form $0\oplus\lbrack a,a_{2}-a_{1}]$ where $a\in N$. Since $N$ is simple,
$M^{\prime}$ contains $\{0\}\oplus N$, whence $M^{\prime}=\mathcal{L}$. Thus
$\mathcal{L}$ is subsimple.

Let now $\mathcal{L}=N\oplus^{\mathrm{{id}}}X.$ The Lie ideal $\{0\}\oplus
^{\mathrm{{id}}}X$ is contained in every Lie ideal of $\mathcal{L}$, so that
the Lie subalgebra $M=N\oplus^{\mathrm{{id}}}\{0\}$ contains no non-zero Lie
ideals of $\mathcal{L}$. If $M^{\prime}$ is a Lie subalgebra that contains $M$
then the subspace $Y=\{x\in X$: $0\oplus x\in M^{\prime}\}$ is invariant for
$N$. Hence either $Y=\{0\}$ and $M^{\prime}=M$, or $Y=X$ and $M^{\prime
}=\mathcal{L}$. Thus $M$ is maximal, whence $\mathcal{L}$ is subsimple.
\end{proof}

Now we will show that our list of subsimple Lie algebras is exhausting.

\begin{theorem}
\label{TM1}Let $\dim\mathcal{L}\geq2.$ Then $\mathcal{L}$ is subsimple if and
only if it is either simple\emph{, }or isomorphic to a Lie algebra from
classes $\mathbf{(I)}$ or $\mathbf{(II)}$. More precisely\emph{,} if
$\mathcal{L}$ is semisimple and not simple\emph{,} it is isomorphic to a Lie
algebra in the class $\mathbf{(I);}$ if $\mathcal{L}$ is neither simple\emph{,
}nor semisimple\emph{, }it is isomorphic to a Lie algebra in the class
$\mathbf{(II)}$.
\end{theorem}

\begin{proof}
We saw above that simple Lie algebras and Lie algebras from the classes
\textbf{(I) }and \textbf{(II) }are subsimple.

Conversely, let $\mathcal{L}$ be subsimple and let $M$ be a maximal Lie
subalgebra that does not contain non-zero Lie ideals of $\mathcal{L}$. Assume
firstly that $\mathcal{L}$ is not semisimple. Then $\mathcal{L}$ has a proper
non-zero minimal commutative Lie ideal $X$. Since $M$ is maximal,
$M+X=\mathcal{L}$. Let $I=\{a\in\mathcal{L}$: $[a,X]=0\}$. Then $I$ is a Lie
ideal of $\mathcal{L.}$ We also have that $M\cap I$ is a Lie ideal of
$\mathcal{L,}$ since%
\[
\left[  M\cap I,\mathcal{L}\right]  =\left[  M\cap I,M+X\right]  =\left[
M\cap I,M\right]  \subseteq M\cap I.
\]
Therefore $M\cap I=\{0\}$; in particular, $M\cap X=\{0\},$ so that the sum
$L=M+X$ is direct. Moreover, the equality $M\cap I=\{0\}$ shows that the map
$a\longmapsto\mathrm{{ad}}\left(  a\right)  |_{X}$ is injective on $M$. Set
$N=\mathrm{{ad}}\left(  M\right)  |_{X}.$ Then $\mathcal{L}$ is isomorphic to
$N\oplus^{\mathrm{{id}}}X$ and belongs to class \textbf{(II)}.

Assume that $\mathcal{L}$ is semisimple, but not simple. Then $\mathcal{L}%
=L_{1}\oplus\cdots\oplus L_{n}$ is the direct sum of simple Lie algebras
$L_{i}$ and $n\geq2.$ As each $L_{i}$ is a Lie ideal of $\mathcal{L},$ we have
$\mathcal{L}=M+L_{i}.$ Let $P_{j}$ be the natural projection $x_{1}%
\oplus\cdots\oplus x_{n}\longmapsto x_{j}$ from $\mathcal{L}$ onto $L_{j},$
where $x_{j}\in L_{j}.$ Then $P_{j}\left(  M\right)  =L_{j}$. Let $K_{j}=M\cap
L_{j}$. Then $K_{j}$ is a Lie ideal of $M$, whence $[K_{j},L_{j}]=[K_{j}%
,P_{j}\left(  M\right)  ]=[K_{j},M]\subseteq K_{j}.$ Thus $K_{j}$ is a Lie
ideal of $L_{j}.$ Hence $K_{j}$ is a Lie ideal of $M+L_{j}=\mathcal{L.}$ As
$M$ contains no non-zero Lie ideals of $\mathcal{L}$, we have $K_{j}=\{0\}$
for all $j.$

Fix $i$ and $j$ for $j\neq i.$ Then $L_{j}\subseteq M+L_{i}.$ Hence, for each
$x\in L_{j},$ there is $y\in L_{i}$ such that $x+y\in M.$ Combining this with
the fact that $M\cap L_{j}=\{0\}$ for all $j$, we have that there is an
injective Lie homomorphism $\varphi_{ij}$ from $L_{j}$ into $L_{i}$ such that
$\{x\oplus\varphi_{ij}(x)$: $x\in L_{j}\}\subseteq M$ (as each $\varphi
_{ij}(x)\in L_{i},$ we have that all such $x\oplus\varphi_{ij}(x)$ lie in
$L_{j}\oplus L_{i}).$ Exchanging $i$ and $j$, we have that $\varphi_{ij}$ is a
Lie isomorphism. Thus all $L_{j}$ are isomorphic.

If $n\geq3,$ set $\psi=\varphi_{21}$ and $\omega=\varphi_{31}.$ Then $x_{\psi
}:=x\oplus\psi(x)\in M$ and $x_{\omega}:=x\oplus\omega(x)\in M$ for every
$x\in L_{1}.$ Therefore $x_{\psi}-x_{\omega}=\psi(x)\oplus(-\omega(x))\in M$
for all $x\in L_{1}.$ Hence $[x_{\psi}-x_{\omega},y_{\psi}-y_{\omega}%
]=\psi([x,y])\oplus\omega([x,y]))\in M$ for all $x,y\in L_{1}.$ On the other
hand, $[x,y]_{\psi}-[x,y]_{\omega}=\psi([x,y])\oplus(-\omega([x,y]))\in M.$
Thus $\omega([x,y]))\in M$ for all $x,y\in L_{1}.$ Therefore $M\cap L_{3}%
\neq\{0\}.$ This contradiction shows that $n\leq2$ and $\mathcal{L}$ is
isomorphic to an algebra from class \textbf{(I)}.
\end{proof}

\begin{corollary}
\label{C7.1}If a subsimple Lie algebra $\mathcal{L}$ is solvable then
$\dim(\mathcal{L})\leq2.$ If it is nilpotent\emph{,} $\dim\mathcal{L}\leq1$.
\end{corollary}

\begin{proof}
Let $\dim(\mathcal{L})>1.$ As $\mathcal{L}$ is solvable, we have from Theorem
\ref{TM1} that $\mathcal{L}=N\oplus^{\text{id}}X$ and the operator Lie algebra
$N$ has no non-trivial invariant subspaces in $X$. Since $\mathcal{L}$ is
solvable, $N$ is also solvable and, by the Lie Theorem, $N$ always has a
one-dimensional invariant subspace. Thus $\dim X=1.$ As $N\subseteq B(X),$ we
have $\dim N=1,$ so that $\dim\mathcal{L}=2.$ If $\mathcal{L}$ is nilpotent
then, as $\dim\mathcal{L}=2,$ it is commutative which contradicts Lemma
\ref{L9.1}.
\end{proof}

Let $\mathcal{L}\in\mathfrak{L}$ and $\mathcal{M}\in\mathfrak{S}_{\mathcal{L}%
}^{\max}.$ If $\mathcal{M}$ is a Lie ideal, $\mathcal{L/M}$ has no proper Lie
subalgebras. Hence%
\begin{equation}
\dim\left(  \mathcal{L/M}\right)  =1. \label{9.1}%
\end{equation}
Denote by $\mathfrak{J}_{\mathcal{L}}^{\text{sm}}$ the set of all submaximal
Lie ideals of $\mathcal{L}.$ Then $\mathfrak{J}_{\mathcal{L}}^{\max}%
\subseteq\mathfrak{J}_{\mathcal{L}}^{\text{sm}}\subseteq\mathfrak
{J}_{\mathcal{L}}.$ The following result strengthens Theorem \ref{KST1} ---
the central result of \cite{KST2}.

\begin{proposition}
\label{P3.2}\emph{(i) }Every maximal closed subalgebra of finite codimension
in a Banach Lie algebra $\mathcal{L}$ contains a submaximal Lie ideal of
$\mathcal{L}.$
\begin{enumerate}
 \item[(ii)]
 Conversely\emph{,} for each submaximal Lie ideal $J$ of
$\mathcal{L},$ there is $\mathcal{M}\in\mathfrak{S}_{\mathcal{L}}^{\max}$ such
that $J$ is a maximal element in the set of all closed Lie ideals of
$\mathcal{L}$ contained in $\mathcal{M}$.
\end{enumerate}
\end{proposition}

\begin{proof}
(i) If dim$\left(  \mathcal{L}\right)  <\infty$ then each maximal Lie ideal of
$\mathcal{L}$ contained in every maximal subalgebra of $\mathcal{L}$ is submaximal.

Let dim$\left(  \mathcal{L}\right)  =\infty$ and $\mathcal{M}\in\mathfrak
{S}_{\mathcal{L}}^{\max}$. If $\mathcal{M}$ is a Lie ideal then, by
(\ref{9.1}), dim$\left(  \mathcal{L/M}\right)  =1.$ Thus $\mathcal{M}$ is a
submaximal Lie ideal. If $\mathcal{M}$ is not a Lie ideal of $\mathcal{L}$, it
follows from Theorem \ref{KST1}(i) that $\mathcal{M}$ contains a closed Lie
ideal of $\mathcal{L}$ of finite codimension. Hence $\mathcal{M}$ contains a
largest closed Lie ideal $J$ of $\mathcal{L}$ of finite codimension. Then
$\mathcal{L/}J$ is finite-dimensional and $\mathcal{M/}J$ is a maximal Lie
subalgebra of $\mathcal{L/}J$ that contains no non-zero Lie ideals of
$\mathcal{L/}J.$ Thus $\mathcal{L/}J$ is subsimple, so that $J$ is submaximal.

(ii) Let $J\in\mathfrak{J}_{\mathcal{L}}^{\text{sm}}.$ If $J\notin\mathfrak
{S}_{\mathcal{L}}^{\max}$ then, by (\ref{9.1}), $\dim\mathcal{L}/J\neq1$ and
there is a maximal proper Lie subalgebra $M$ of $\mathcal{L/}J$ that contains
no non-zero Lie ideals of\emph{ }$\mathcal{L}/J.$ Then the preimage
$\mathcal{M}$ of $M$ in $\mathcal{L}$ belongs to $\mathfrak{S}_{\mathcal{L}%
}^{\max}$ and $J$ is a maximal Lie ideal of $\mathcal{L}$ contained in
$\mathcal{M}$.
\end{proof}

We will show now that the Lie ideal-multifunctions $\mathfrak{J}^{\text{sm}}$
and $\mathfrak{S}^{\max}$ generate equal preradicals.

\begin{proposition}
\label{Lmaxli}$P_{\mathfrak{S}^{\max}}\left(  \mathcal{L}\right)
=P_{\mathfrak{J}_{\mathcal{L}}^{\text{\emph{sm}}}}\left(  \mathcal{L}\right)
$ for all Banach Lie algebras $\mathcal{L}$\emph{,} so that $P_{\mathfrak
{S}^{\max}}=P_{\mathfrak{J}_{\mathcal{L}}^{\text{\emph{sm}}}}$.
\end{proposition}

\begin{proof}
By Proposition \ref{P3.2}(i), $\mathfrak{J}_{\mathcal{L}}^{\text{sm}%
}\overleftarrow{\subset}\mathfrak{S}_{\mathcal{L}}^{\max}$ for all
$\mathcal{L}\in\mathfrak{L.}$ Hence, by (\ref{LP0}), $P_{\mathfrak
{J}^{\text{sm}}}\left(  \mathcal{L}\right)  \subseteq P_{\mathfrak{S}^{\max}%
}\left(  \mathcal{L}\right)  $.

On the other hand, let $J\in\mathfrak{J}_{\mathcal{L}}^{\text{sm}}.$ By
Proposition \ref{P3.2}(ii), there exists $\mathcal{M}\in\mathfrak
{S}_{\mathcal{L}}^{\text{max}}$ such that $J$ is maximal among closed Lie
ideals of $\mathcal{L}$ contained in $\mathcal{M}.$ As $J$ has finite
codimension in $\mathcal{L}$, we have from Lemma \ref{Closed} that
$J+P_{\mathfrak{S}^{\max}}\left(  \mathcal{L}\right)  $ is closed. Since
$P_{\mathfrak{S}^{\max}}\left(  \mathcal{L}\right)  \subseteq\mathcal{M},$
$J+P_{\mathfrak{S}^{\max}}\left(  \mathcal{L}\right)  $ is a closed Lie ideal
of $\mathcal{L}$ contained in $\mathcal{M}.$ As $J$ is a maximal such Lie
ideal in $\mathcal{M},$ $P_{\mathfrak{S}^{\max}}\left(  \mathcal{L}\right)
\subseteq J.$ Thus $P_{\mathfrak{S}^{\max}}\left(  \mathcal{L}\right)
\subseteq J$ for all $J\in\mathfrak{J}_{\mathcal{L}}^{\text{sm}}$. Hence
$P_{\mathfrak{S}^{\max}}\left(  \mathcal{L}\right)  \subseteq P_{\mathfrak
{J}_{\mathcal{L}}^{\text{sm}}}\left(  \mathcal{L}\right)  $.
\end{proof}

As a consequence of the above proposition, $\mathbf{Sem}(P_{\mathfrak
{S}^{\text{max}}})$ coincides with the class of all algebras with sufficiently
many submaximal ideals: the intersection of submaximal ideals equals zero.

\subsection{ Subdirect products}

To describe Frattini-free Lie algebras in a more constructive way we need the
following definition. Let $\{\mathcal{L}_{\lambda}\}_{\lambda\in\Lambda}$ be a
family of Banach Lie algebras with multiplication constants $t_{\lambda}$
satisfying $\sup\{t_{\lambda}\}<\infty$. Let $\mathcal{L}_{\Lambda}%
=\oplus_{\Lambda}\mathcal{L}_{\lambda}$ be their normed direct product (see
Subsection 3.2). For each $\mu\in\Lambda,$ denote by $\psi_{\mu}$ the
homomorphism from $\mathcal{L}_{\Lambda}$ onto $\mathcal{L}_{\mu}$:
\begin{equation}
\psi_{\mu}(\{a_{\lambda}\})=a_{\mu}. \label{9.9}%
\end{equation}

\begin{definition}
\label{D9.2}A Lie subalgebra $\mathcal{M}$ of $\mathcal{L}_{\Lambda}%
=\oplus_{\Lambda}\mathcal{L}_{\lambda}$ is called a \textit{subdirect} product
of the algebras $\{\mathcal{L}_{\lambda}\}_{\lambda\in\Lambda}$ if $\psi_{\mu
}(\mathcal{M})=\mathcal{L}_{\mu}$ for each $\mu\in\Lambda$.
\end{definition}

\begin{theorem}
\label{T4.3}A Banach Lie algebra $\mathcal{L}$ belongs to $\mathbf{Sem}%
(P_{\mathfrak{S}^{\max}})$ if and only if there is a bounded isomorphism
$\theta$ from $\mathcal{L}$ onto a subdirect product of some family of
subsimple Lie algebras$.$

The algebra $\mathcal{L}$ belongs to $\mathbf{Sem}(P_{\mathfrak{J}^{\max}})$
$($respectively\emph{,} to $\mathbf{Sem}(P_{\mathfrak{J}}))$ if and only if
there is a bounded isomorphism from $\mathcal{L}$ onto a subdirect product of
some family of simple or one-dimensional $($respectively\emph{,}
finite-dimensional$)$ Lie algebras.
\end{theorem}

\begin{proof}
We will only consider the first case when $\mathcal{L}\in\mathbf{Sem}%
(P_{\mathfrak
{S}^{\max}})$; the two remaining cases can be proved similarly.

Let $\theta$ and $\mathcal{L}_{\Lambda}$ exist. For each $\lambda\in\Lambda,$
$\psi_{\lambda}\circ\theta$ is a bounded homomorphism from $\mathcal{L}$ onto
$\mathcal{L}_{\lambda}.$ Then $J_{\lambda}:=\ker(\psi_{\lambda}\circ\theta)$
is a closed Lie ideal of $\mathcal{L}$ and $\mathcal{L}/J_{\lambda}$ is
isomorphic to $\mathcal{L}_{\lambda},$ so that $J_{\lambda}$ is submaximal.
Also $\cap_{\lambda\in\Lambda}J_{\lambda}=\{0\}$. By Lemma \ref{L9.1},
$P_{\mathfrak{S}^{\max}}(\mathcal{L}/J_{\lambda})=\{0\}.$ Therefore, by Lemma
\ref{L-sem}(ii), $P_{\mathfrak{S}^{\max}}\mathcal{(L})\subseteq\cap
_{\lambda\in\Lambda}J_{\lambda}=\{0\}.$ Thus $\mathcal{L}\in$ $\mathbf{Sem}%
(P_{\mathfrak{S}^{\text{max}}}).$

Conversely, let $\mathcal{L}\in$ $\mathbf{Sem}(P_{\mathfrak{S}^{\text{max}}%
}).$ By Proposition \ref{Lmaxli}, 
\[
\cap_{J\in\mathfrak{J}_{\mathcal{L}%
}^{\text{sm}}}J=P_{\mathfrak{J}^{\text{sm}}}\left(  \mathcal{L}\right)
=P_{\mathfrak{S}^{\text{max}}}\left(  \mathcal{L}\right)  =\{0\}.
\]
 Choose any
subset $\Lambda$ of $\mathfrak{J}_{\mathcal{L}}^{\text{sm}}$ such that
$\cap_{J\in\Lambda}J=\{0\}.$ For each $J\in\Lambda,$ the quotient Lie algebra
$\mathcal{L/}J$ is subsimple. Let $q_{J}$: $\mathcal{L}\longrightarrow
\mathcal{L}/J$ be the quotient map. Then, for $a,b\in\mathcal{L}$,%
\begin{align*}
\left\|  \lbrack q_{J}(a),q_{J}(b)]\right\|  _{\mathcal{L/}J} &  =\left\|
q_{J}([a,b])\right\|  _{\mathcal{L/}J}=\text{ }\underset{z\in J}{\inf}\left\|
[a,b]+z\right\|  _{\mathcal{L}}\\
&  \leq\underset{z,u\in J}{\inf}\left\|  [a+z,b+u]\right\|  _{\mathcal{L}}\\
&  \leq t\underset{z,u\in J}{\inf}\left\|  a+z\right\|  _{\mathcal{L}}\left\|
b+u\right\|  _{\mathcal{L}}\\
&  =t\left\|  q_{J}(a)\right\|  _{\mathcal{L/}J}\left\|  q_{J}\left(
b\right)  \right\|  _{\mathcal{L/}J}.
\end{align*}

Hence $t_{J}\leq t$ for all $J\in\Lambda.$ Thus $\sup\{t_{J}$: $J\in
\Lambda\}<\infty.$ Let $\mathcal{L}_{\Lambda}:=\oplus_{\Lambda}(\mathcal{L/}%
J)$ be the normed direct product. For each $a\in\mathcal{L}$, we have
$\{q_{J}(a)\}_{J\in\Lambda}\in\mathcal{L}_{\Lambda}$ and the map $\theta$:
$a\longrightarrow\{q_{J}(a)\}_{J\in\Lambda}$ is bounded homomorphism. If
$\theta(a)=0$ then $q_{J}(a)=0$ for all $J\in\Lambda,$ so that $a\in\cap
_{J\in\Lambda}J=\{0\}.$ Thus $\theta$ is an isomorphism. As $\psi_{J}%
(\theta(a))=q_{J}(a),$ we have $\psi_{J}(\theta(\mathcal{L}))=q_{J}%
(\mathcal{L})=\mathcal{L}_{J}.$ Hence the Lie algebra $\theta\left(
\mathcal{L}\right)  $ is a subdirect product of the algebras $\{\mathcal{L}%
_{\lambda}\}_{\lambda\in\Lambda}$.
\end{proof}

As an illustration, consider the Lie algebra $\mathcal{L}=\mathbb{C}%
U\oplus^{\text{id}}l^{2},$ where $U$ is the unilateral shift: $Ue_{n}=e_{n+1}$
and $(e_{n})_{n=1}^{\infty}$ is a basis in $l^{2}$. Then $\{0\}\oplus
^{\text{id}}l^{2}$ is a maximal Lie subalgebra of $\mathcal{L.}$ Let
$\mathbb{D}\subset\mathbb{C}$ be the open unit disk. For each $\lambda
\in\mathbb{D,}$ the vector $e_{\lambda}=(1,\lambda,\lambda^{2},...)$ belongs
to $l^{2}$ and the subspace $E_{\lambda}=\mathbb{C}e_{\lambda}$ is invariant
for the adjoint operator $U^{\ast}$: $U^{\ast}e_{\lambda}=\lambda e_{\lambda
}.$ Hence $E_{\lambda}^{\bot}$ is invariant for $U$ and has codimension 1 in
$l^{2}.$ Thus (see (\ref{sem})) $L_{\lambda}=\mathbb{C}U\oplus^{\text{id}%
}E_{\lambda}^{\bot}$ is a maximal Lie subalgebra of $\mathcal{L}$ of
codimension 1. The Lie algebra $\mathcal{L}$ is Frattini-free, since%
\[
P_{\mathfrak{S}^{\text{max}}}(\mathcal{L})\subseteq(\{0\}\oplus^{\text{id}%
}l^{2})\cap(\cap_{\lambda\in\mathbb{D}}L_{\lambda})=\{0\}.
\]

To map $\mathcal{L}$ onto a subdirect product of subsimple Lie algebras,
consider two-dimensional Lie algebras $\mathcal{L}_{\lambda}=\mathbb{C}%
U_{\lambda}\oplus^{\text{id}}E_{\lambda},$ $\lambda\in\mathbb{D},$ where
$U_{\lambda}=\overline{\lambda}1_{E_{\lambda}}.$ All $\mathcal{L}_{\lambda}$
are subsimple algebras of class (\textbf{II}). Set $\mathcal{M}=\oplus
_{\lambda\in\mathbb{D}}\mathcal{L}_{\lambda}$. Denote by $P_{\lambda}$ the
orthogonal projections in $l^{2}$ onto subspaces $E_{\lambda}.$ The map
$\theta$: $\mathcal{L}\rightarrow\mathcal{M}$ defined by the rule
\[
\theta(\alpha U\oplus^{\text{id}}x)=\oplus_{\lambda\in\mathbb{D}}(\alpha
U_{\lambda}\oplus^{\text{id}}P_{\lambda}x)\in\mathcal{M}%
\]
is a homomorphism, because $P_{\lambda}U=\overline{\lambda}P_{\lambda}$ for
each $\lambda\in\mathbb{D}$, since $U^{\ast}P_{\lambda}=\lambda P_{\lambda}$.
Furthermore $\theta$ is injective. Indeed, if $\theta(\alpha U\oplus
^{\text{id}}x)=0$ then $\alpha=0$ and $P_{\lambda}x=0$ for all $\lambda
\in\mathbb{D}$. Hence $(x,e_{\lambda})=\sum_{n}(x,e_{n})\lambda^{n}=0$, for
all $\lambda\in\mathbb{D}$. Thus all $(x,e_{n})=0$, so that $x=0$. Finally,
the projection of the image $\theta(\mathcal{L)}$ on each component
$\mathcal{L}_{\lambda}$ clearly coincides with $\mathcal{L}_{\lambda}$.
Therefore $\theta(\mathcal{L)}$ is a subdirect product of the Lie algebras
$\mathcal{L}_{\lambda}$.

Theorem \ref{T4.3} gives a fairly transparent description of the class of
finite-dimensional Frattini-free algebras as direct sums of simple ''model''
examples; we will return to this subject in the next subsection. Note that in
general the subdirect sums can be indecomposable even in the commutative case
(take any Banach space of bounded analytic functions).

We shall now consider the structure of some special types of Lie algebras from
$\mathbf{Sem}(P_{\mathfrak{S}^{\text{max}}}).$

\begin{theorem}
\label{T1}\emph{(i) }If a Lie algebra $\mathcal{L}\in\mathbf{Sem}%
(P_{\mathfrak{S}^{\max}})$ is solvable then $\mathcal{L}_{[2]}=\{0\}$.
\begin{enumerate}
 \item[(ii)]
 Let $\mathcal{L}\in\mathbf{Sem}(P_{\mathfrak{S}^{\max}}).$ Then
$\mathcal{L}$ is nilpotent if and only if $\mathcal{L}$ is commutative$.$
\item[(iii)]
Any solvable Lie algebra in $\mathbf{Sem}(P_{\mathfrak{J}^{\max}%
})$ is commutative.
\end{enumerate}
\end{theorem}

\begin{proof}
Let $\mathcal{L}\in$ $\mathbf{Sem}(P_{\mathfrak{S}^{\max}})$. By Proposition
\ref{P3.2}(ii), for each $J\in\mathfrak{J}_{\mathcal{L}}^{\text{sm}},$ either
$J\in\mathfrak{S}_{\mathcal{L}}^{\max}$ in which case $\dim\left(
\mathcal{L}/J\right)  =1$ by (\ref{9.1}), or there is $\mathcal{M}^{J}%
\in\mathfrak{S}_{\mathcal{L}}^{\max}$ such that $\mathcal{M}^{J}/J$ is a
maximal Lie subalgebra of $\mathcal{L}/J$ and it contains no non-zero Lie
ideals of $\mathcal{L}/J$.

(i) If $\mathcal{L}$ is solvable, $\mathcal{L}/J$ are solvable for all
$J\in\mathfrak{J}_{\mathcal{L}}^{\text{sm}}.$ By Corollary \ref{C7.1}, we have
$\dim\left(  \mathcal{L}/J\right)  \leq2,$ so that $(\mathcal{L}%
/J)_{[2]}=\{0\}.$ Hence $\mathcal{L}_{[2]}\subseteq J$ for all $J\in
\mathfrak{J}_{\mathcal{L}}^{\text{sm}}.$ Therefore, by Proposition
\ref{Lmaxli}, $\mathcal{L}_{[2]}\subseteq\cap_{J\in\mathfrak{J}_{\mathcal{L}%
}^{\text{sm}}}J=P_{\mathfrak{J}^{\text{sm}}}\left(  \mathcal{L}\right)
=P_{\mathfrak{S}^{\max}}\left(  \mathcal{L}\right)  =\{0\}.$

(ii) If $\mathcal{L}$ is nilpotent then $\mathcal{L}/J$ is nilpotent for each
$J\in\mathfrak{J}_{\mathcal{L}}^{\text{sm}}.$ By Corollary \ref{C7.1},
$\dim(\mathcal{L}/J)=1.$ Hence $\mathcal{L}^{[1]}\subseteq J$ for all
$J\in\mathfrak{J}_{\mathcal{L}}^{\text{sm}}.$ Thus, by Proposition
\ref{Lmaxli},
\[
\mathcal{L}^{[1]}\subseteq\underset{J\in\mathfrak{J}_{\mathcal{L}}^{\text{sm}%
}}{\cap}J=P_{\mathfrak{J}^{\text{sm}}}\left(  \mathcal{L}\right)
=P_{\mathfrak{S}^{\max}}\left(  \mathcal{L}\right)  =\{0\}.
\]

(iii) If $\mathcal{L}\in\mathbf{Sem}(P_{\mathfrak{J}^{\max}})$ is solvable
then, by Corollary \ref{Cder}(ii), $\{0\}=P_{\mathfrak{J}^{\max}}\left(
\mathcal{L}\right)  =\mathcal{L}_{[1]}.$
\end{proof}

The condition $\mathcal{L}_{[2]}=\{0\}$ is not sufficient for a Banach Lie
algebra $\mathcal{L}$ to belong to $\mathbf{Sem}(P_{\mathfrak{S}^{\text{max}}%
})$. Indeed, if $\mathcal{L}$ is the Heisenberg 3-dimensional Lie algebra$,$
as in Example \ref{E3}(ii), then $\mathcal{L}_{[2]}=\{0\}$ and $\mathcal{L}%
\notin\mathbf{Sem}(P_{\mathfrak{S}^{\text{max}}}).$

The following corollary shows that for Frattini-free algebras there is a
natural analogue of the classical solvable radical.

\begin{corollary}
\label{C9.3}Each $\mathcal{L}\in\mathbf{Sem}(P_{\mathfrak{S}^{\max}})$ has the
largest solvable \emph{(}commutative\emph{) }Lie ideal --- the closed solvable
\emph{(}commutative\emph{) }Lie\emph{ }ideal that contains all solvable
\emph{(}commutative\emph{) }Lie ideals of $\mathcal{L}$.
\end{corollary}

\begin{proof}
Let $\mathcal{L}\in$ $\mathbf{Sem}(P_{\mathfrak{S}^{\text{max}}}).$ The set
$\mathcal{E}$ of all closed solvable Lie ideals of $\mathcal{L}$ is partially
ordered by inclusion. If $I$ is a closed ideal of $\mathcal{L}$ then
$I\in\mathbf{Sem}(P_{\mathfrak{S}^{\text{max}}})$ by Lemma \ref{L-sem}(iii).
Therefore $I_{[2]}=\{0\}$ for each $I\in\mathcal{E,}$ by Theorem \ref{T1}(i).
Hence if $\{I_{\lambda}\}_{\lambda\in\Lambda}$ is a linearly ordered subset of
$\mathcal{E,}$ then the ideal $I=\cup_{\lambda\in\Lambda}I_{\lambda}$
satisfies $I_{[2]}=\{0\}$. Therefore its closure $\overline{I}$ also satisfies
$\overline{I}_{[2]}=\{0\},$ so that $\overline{I}\in\mathcal{E}$. By Zorn's
Lemma, $\mathcal{E}$ has a maximal element $J.$

Let $I\in\mathcal{E}.$ Then $I+J$ is a Lie ideal of $\mathcal{L}$, $J\subseteq
I+J$ and $(I+J)_{[1]}\subseteq I_{[1]}+J.$ Hence $(I+J)_{[2]}\subseteq
(I_{[1]}+J)_{[1]}\subseteq I_{[2]}+J=J.$ Thus $(I+J)_{[4]}\subseteq
J_{[2]}=\{0\},$ so that $\overline{I+J}\in\mathcal{E}$ and $J\subseteq
\overline{I+J}.$ As $J$ is maximal, $I\subseteq J,$ so that $J$ contains all
solvable Lie ideals of $\mathcal{L}$.

As above, the set $\mathcal{E}_{\text{c}}$ of all closed commutative Lie
ideals of $\mathcal{L}$ is partially ordered by inclusion and contains a
maximal element $K.$ Let $I\in\mathcal{E}_{\text{c}}.$ Then $I+K$ is a Lie
ideal of $\mathcal{L}$, $K\subseteq I+K$ and (see \ref{5.3})) $(I+K)^{[2]}%
\subseteq I\cap K.$ Hence $(I+K)^{[3]}\subseteq\lbrack I+K,I\cap K]=\{0\}.$
Hence $\overline{I+K}^{[3]}=\{0\},$ so that $\overline{I+K}$ is nilpotent. By
Theorem \ref{T1}(ii), $\overline{I+K}$ is commutative. Thus $\overline{I+K}%
\in\mathcal{E}_{\text{c}}$ and $K\subseteq\overline{I+K}.$ As $K$ is maximal,
$I\subseteq K,$ so that $K$ contains all commutative Lie ideals of
$\mathcal{L}$.
\end{proof}

\subsection{Frattini- and Jacobson-free finite-dimensional Lie algebras}

The general description of {P$_{\mathfrak{J}^{\text{max}}}$-semisimple and
P$_{\mathfrak{S}^{\text{max}}}$-semisimple Banach Lie algebras in terms of
semidirect products of subsimple algebras (Theorem \ref{T4.3}) enables one to
obtain sufficiently simple ''models'' for such algebras in the
finite-dimensional case. }

We say that a Lie algebra $L$ of operators on a finite-dimensional linear
space $X$ is \textit{decomposable} if $X$ decomposes into the direct sum of
minimal subspaces invariant for $L$: $X=X_{1}\dotplus...\dotplus X_{n}$ where
all $X_{k}$ are invariant for $L$ and the restriction of $L$ to each $X_{k}$
is irreducible. A representation of a Lie algebra will be called decomposable
if its image is decomposable.

\begin{lemma}
\label{L9.10}Let $\pi$ be a decomposable representation of a Lie algebra $L$
on a finite-dimensional space $X$ and let $\mathcal{L}=L\oplus^{\pi}X$ $($see
$(\ref{fsemi})).$ Then $P_{\mathfrak{S}^{\max}}(\mathcal{L})\subseteq(\ker
\pi\cap P_{\mathfrak{S}^{\max}}(L))\oplus^{\pi}\{0\}.$
\end{lemma}

\begin{proof}
We have $X=X_{1}\dotplus...\dotplus X_{n}$ where all $X_{k}$ are invariant for
$\pi$ and all restrictions $\pi|_{X_{k}}$ are irreducible. Then all
$M_{k}=L\oplus^{\pi}(X-X_{k})$ are maximal Lie subalgebras of $\mathcal{L,}$
so that the Lie ideal $P_{\mathfrak{S}^{\max}}(\mathcal{L})\subseteq\cap
_{k}M_{k}=L\oplus^{\pi}\{0\}.$ Let $(a,0)\in P_{\mathfrak{S}^{\max}%
}(\mathcal{L}).$ If $a\notin\ker\pi$ then $\pi(a)x\neq0$ for some $x\in X.$
Hence $[(a,0),(0,x)]=(0,\pi(a)x)\in P_{\mathfrak{S}^{\max}}(\mathcal{L})$ -- a
contradiction. Thus $P_{\mathfrak{S}^{\max}}(\mathcal{L})\subseteq\ker
\pi\oplus^{\pi}\{0\}.$ Using Proposition \ref{semi}(i), we conclude the proof.
\end{proof}

\begin{corollary}
\label{fdFf}A finite-dimensional Lie algebra $\mathcal{L}$ is Frattini-free if
and only if it is isomorphic to the direct sum of Lie algebras of the
following types\emph{:}
\begin{enumerate}
 \item[(i)]
 one-dimensional algebras\emph{;}
\item[(ii)]
simple Lie algebras\emph{;}
\item[(iii)]
Lie algebras $L\oplus^{\text{\emph{id}}}X,$ where $L$ is a
decomposable Lie algebra of operators on a linear space $X$.
\end{enumerate}
\end{corollary}

\begin{proof}
The subsimple Lie algebras in (i), (ii) are Frattini-free by Lemma \ref{L9.1}.
The Lie algebras $\mathcal{L}=L\oplus^{\text{id}}X$ in (iii) are also
Frattini-free: by Lemma \ref{L9.10}, $P_{\mathfrak{S}^{\max}}(\mathcal{L}%
)=\{0\}$ as $\ker($id$)=\{0\}.$

Conversely, let $\mathcal{L}$ be a Frattini-free Lie algebra. If it decomposes
in the direct sum of Lie ideals then, as the preradical $P_{\mathfrak
{S}^{\text{max}}}$ is balanced, each of them is Frattini-free. Hence we will
assume that $\mathcal{L}$ does not decompose in the direct sum of Lie ideals.

Theorem \ref{T4.3} implies that $\mathcal{L}$ can be identified with a
subdirect product of some set $\Lambda$ of subsimple algebras $\{\mathcal{L}%
_{\lambda}\}_{\lambda\in\Lambda}$. For each $\lambda\in\Lambda$, let
$\psi_{\lambda}$ be the homomorphism from $\oplus_{\Lambda}\mathcal{L}%
_{\lambda}$ onto $\mathcal{L}_{\lambda}$ (see (\ref{9.9})). We may assume that
$\Lambda$ is finite. Indeed, for each $\lambda\in\Lambda$, $N_{\lambda}%
:=\ker\psi_{\lambda}$ is a Lie ideal of $\mathcal{L}$ and $\cap_{\lambda
\in\Lambda}N_{\lambda}=\{0\}$. As $\dim\mathcal{L}<\infty,$ there is a finite
subfamily $\lambda_{1},...,\lambda_{n}$ of $\Lambda$ with
\begin{equation}
\cap_{i=1}^{n}N_{\lambda_{i}}=\{0\}. \label{9.8}%
\end{equation}
Choose the least possible $n$ in (\ref{9.8}). It follows that $\mathcal{L}$ is
isomorphic to a subdirect product of the direct product $\mathcal{M}%
=\oplus_{i=1}^{n}\mathcal{L}_{i},$ where $\mathcal{L}_{i}=\mathcal{L}%
_{\lambda_{i}}.$ Set $N_{i}=N_{\lambda_{i}}$ and $\psi_{i}=\psi_{\lambda_{i}}.$

Using the description of subsimple algebras in Theorem \ref{TM1}, we may
assume that each $\mathcal{L}_{i}$ is either one-dimensional or a simple Lie
algebra or isomorphic to $L_{i}\oplus^{\text{id}}X_{i}$, where $L_{i}$ is an
irreducible Lie algebra of operators on a linear finite-dimensional space
$X_{i}$.

If $n=1,$ the theorem is proved. Let $n>1$. Then $\cap_{i=2}^{n}N_{i}$ is a
Lie ideal of $\mathcal{L.}$ As $\mathcal{L}_{1}=\psi_{1}(\mathcal{L),}$ we
have that $J_{1}=\psi_{1}(\cap_{i=2}^{n}N_{i})$ is a Lie ideal of
$\mathcal{L}_{1}$. If $J_{1}=\{0\}$ then $\cap_{i=2}^{n}N_{i}\subseteq
N_{1}=\ker\psi_{1},\ $so\ that $\cap_{i=2}^{n}N_{i}=\{0\}\ $which contradicts
the fact that $n$ is the least in (\ref{9.8}).

If $J_{1}=\mathcal{L}_{1}$ then, for each $x\in\mathcal{L,}$ there is
$y_{x}\in\cap_{i=2}^{n}N_{i}$ such that $\psi_{1}(x)=\psi_{1}(y_{x}).$ Hence
$x=y_{x}+(x-y_{x})$ and $x-y_{x}\in\ker\psi_{1}=N_{1}.$ As $(\cap_{i=2}%
^{n}N_{i})\cap N_{1}=\{0\}$ by (\ref{9.8}), we have that $\mathcal{L}%
=(\cap_{i=2}^{n}N_{i})\oplus N_{1}$ is the direct sum of its Lie ideals. This
contradicts our assumption. Thus $\{0\}\neq J_{1}\neq\mathcal{L}_{1},$ so that
$\mathcal{L}_{1}=L_{1}\oplus^{\text{id}}X_{1}.$ As the Lie ideal
$\{0\}\oplus^{\text{id}}X_{1}$ is contained in each Lie ideal of
$\mathcal{L}_{1},$ it is contained in $J_{1}$ and, hence, in $\mathcal{L.}$

The similar argument shows that simple and one-dimensional summands are absent
in $\mathcal{M}$ and each $\mathcal{L}_{i}=L_{i}\oplus^{\text{id}}X_{i}$.
Moreover, $\mathcal{L}$ contains the Lie ideal $\{0\}\oplus^{\text{id}}X,$
where $X=\sum_{k=1}^{n}\dotplus X_{i}.$

Set $M=\oplus_{i=1}^{n}L_{i}.$ Clearly, $M$ can be considered as a Lie algebra
of operators on $X$, preserving each $X_{i}$ and irreducible on it, and
$\mathcal{M}=M\oplus^{\text{id}}X.$ As $\mathcal{L}\subseteq\mathcal{M}$ and
contains $\{0\}\oplus^{\text{id}}X,$ there is a Lie subalgebra $L$ of $M$ such
that $\mathcal{L}=L\oplus^{\text{id}}X.$ As $\mathcal{L}$ is a subdirect
product, $\psi_{i}(\mathcal{L})=\mathcal{L}_{i}=L_{i}\oplus^{\text{id}}X_{i}$
for each $i.$ As $\psi_{i}(\{0\}\oplus^{\text{id}}X)=\{0\}\oplus^{\text{id}%
}X_{i},$ we have $\psi_{i}(L\oplus^{\text{id}}\{0\})=L_{i}\oplus^{\text{id}%
}\{0\}.$ Thus $L|_{X_{i}}\approx L_{i}$ is irreducible on $X_{i},$ so that $L$
is decomposable.
\end{proof}

One can easily deduce from Corollary \ref{fdFf} the characterization of
finite-dimensional Frattini-free Lie algebras obtained by Stitzinger \cite{S}
and Towers \cite{T}. For this we will use the following well known result (see
for example \cite[Proposition 4.4.2.3]{Ch}).

\begin{lemma}
\label{reduct}Let $L$ be a decomposable Lie algebra of operators on a
finite-dimensional space $X=X_{1}\dotplus...\dotplus X_{n},$ where all $X_{i}$
are irreducible components. Let $Z_{L}$ be the center of $L.$ Then
\begin{enumerate}
 \item[(i)] 
 $a|_{X_{i}}=\lambda_{i}(a)\mathbf{1}_{X_{i}},$ for all $a\in Z_{L}$
and $i,$ where $\lambda_{i}$ are linear functionals on $Z_{L};$
\item[(ii)]
 $[L,L]$ is semisimple and $L=[L,L]\oplus Z_{L}$.
\end{enumerate}
\end{lemma}

In fact, for a finite-dimensional Lie algebra $L$ the conditions
$L=[L,L]\oplus Z_{L}$ and $[L,L]$ is semisimple in (ii) are equivalent
(\cite[Proposition 4.4.2.1]{Ch}); the Lie algebras satisfying these conditions
are called \textit{reductive}.

\begin{corollary}
\label{stitz-tow}\cite{S,T}A finite-dimensional Lie algebra $\mathcal{L}$ is
Frattini-free if and only if it is the direct space sum $\mathcal{L}=C\dotplus
S\dotplus J,$ where $J$ is a commutative Lie ideal of $\mathcal{L}$\emph{,}
$C$ is a commutative Lie subalgebra of $\mathcal{L}$ whose adjoint
representation on $J$ is decomposable and $S$ is a semisimple Lie subalgebra
of $\mathcal{L}$ such that $[C,S]=\{0\}$.
\end{corollary}

\begin{proof}
Let $\mathcal{L}$ be Frattini-free. Applying Corollary \ref{fdFf}, it suffices
to obtain the needed decomposition for each direct summand of $\mathcal{L}$.
For summands of type (i) and (ii) this is evident. For $\mathcal{L}%
=L\oplus^{\text{id}}X,$ where $L$ is decomposable, set $J=\{0\}\oplus
^{\text{id}}X$, $S=[L,L]\oplus^{\text{id}}\{0\}$, $C=Z_{L}\oplus^{\text{id}%
}\{0\}$ and apply Lemma \ref{reduct}.

Conversely, let $\mathcal{L}=C\dotplus S\dotplus J$ and $J,C,S$ have the
properties listed above. Then the Lie algebra $L=C\oplus S$ is reductive and
$C=Z_{L}$. Let $\pi=$ ad$|_{J}$ be the adjoint representation of $L$ on $J$.
By our assumptions, the restriction of $\pi$ to $Z_{L}$ is decomposable. It
follows that $\pi$ is decomposable (see \cite[Corollary 4.4.1.2]{Ch}). As $L$
is the direct sum of a semisimple and commutative Lie ideals, we have
$P_{\mathfrak{S}^{\max}}(L)=\{0\}$. Hence, by Lemma \ref{L9.10},
$P_{\mathfrak{S}^{\max}}(\mathcal{L})=\{0\}.$
\end{proof}

Recall that $\mathcal{L}\in\mathfrak{L}$ is Jacobson-free if $P_{\mathfrak
{J}^{\max}}(\mathcal{L})=\{0\}.$ Similar, but simpler arguments give us the
description of Jacobson-free algebras (for a different proof see the end of
the paper).

\begin{corollary}
\label{C10.2}A finite-dimensional Lie algebra $\mathcal{L}$ is Jacobson-free
if and only if $\mathcal{L}$ is the direct sum of a semisimple and a
commutative Lie algebras.
\end{corollary}

\subsection{Frattini and Jacobson indices of finite-dimensional Lie algebras}

In this section we study the class $\mathfrak{L}^{\mathrm{f}}$ of complex
finite-dimensional Lie algebras. As $\{0\}$ is a Lie ideal of finite
codimension in each $\mathcal{L}\in\mathfrak{L}^{\mathrm{f}},$ we have
$\mathcal{F}(\mathcal{L})=\{0\}$ and $\mathfrak{L}^{\mathrm{f}}\subseteq
\mathbf{Sem}(P_{\mathfrak{J}}).$

The Lie ideal $P_{\mathfrak{S}^{\max}}\left(  \mathcal{L}\right)  $ is called
the \textit{Frattini ideal} and $P_{\mathfrak{J}^{\max}}\left(  \mathcal{L}%
\right)  $ the \textit{Jacobson ideal} of $\mathcal{L}$ (in \cite{M} it was
called the \textit{Jacobson radical}). By Theorem \ref{C6.6}, $P_{\mathfrak
{S}^{\max}}\left(  \mathcal{L}\right)  \subseteq P_{\mathfrak{J}^{\max}%
}\left(  \mathcal{L}\right)  $. The ordinal numbers $r_{P_{\mathfrak{S}^{\max
}}}^{\circ}\left(  \mathcal{L}\right)  $ and $r_{P_{\mathfrak{J}^{\max}}%
}^{\circ}\left(  \mathcal{L}\right)  $ (see (\ref{r1})) belong to $\mathbb{N}$
and satisfy%
\[
\{0\}=\mathcal{F}(\mathcal{L})=P_{\mathfrak{S}^{\max}}^{\alpha}(\mathcal{L)}%
=P_{\mathfrak{J}^{\max}}^{\beta}(\mathcal{L)},
\]
where $\alpha=r_{P_{\mathfrak{S}^{\max}}}^{\circ}\left(  \mathcal{L}\right)
\mathbb{,}$ $\beta=r_{P_{\mathfrak{J}^{\max}}}^{\circ}\left(  \mathcal{L}%
\right)  \mathbb{.}$ They are called, respectively, the \textit{Frattini }(see
\cite[Definitions 4]{M}) and \textit{Jacobson indices} of $\mathcal{L.}$ By
Theorem \ref{C6.6}, $r_{P_{\mathfrak{S}^{\max}}}^{\circ}\left(  \mathcal{L}%
\right)  \leq r_{P_{\mathfrak{J}^{\max}}}^{\circ}\left(  \mathcal{L}\right)
<\infty.$

Denote by $\mathcal{N}_{\mathcal{L}}$ the nil-radical of $\mathcal{L}$ --- the
maximal nilpotent ideal of $\mathcal{L}$. Combining this with results of
\cite[p. 420 and 422]{M} and \cite[Theorem II.7.13]{J} yields%
\begin{equation}
P_{\mathfrak{S}^{\max}}\left(  \mathcal{L}\right)  \subseteq P_{\mathfrak
{J}^{\max}}\left(  \mathcal{L}\right)  =\mathcal{K}_{\mathcal{L}}%
\subseteq\mathcal{N}_{\mathcal{L}}\subseteq\operatorname*{rad}\left(
\mathcal{L}\right)  \text{,} \label{Fmarsh}%
\end{equation}
where $\mathcal{K}_{\mathcal{L}}=\left[  \mathcal{L},\operatorname*{rad}%
\left(  \mathcal{L}\right)  \right]  .$

For a solvable Lie algebra $\mathcal{L}$, the \textit{solvability index
}$i_{s}(\mathcal{L)}$ is the least $n$ such that $\mathcal{L}_{[n]}=0.$
Marshall \cite[p. 421]{M} established that $r_{P_{\mathfrak{S}^{\max}}}%
^{\circ}\left(  \mathcal{L}\right)  \leq i_{s}(\mathcal{N}_{\mathcal{L}})+1.$
Below we refine this result.

\begin{proposition}
\label{P9.1}\emph{(i) }If $\mathcal{L}$ is nilpotent$,$ $r_{P_{\mathfrak
{S}^{\max}}}^{\circ}\left(  \mathcal{L}\right)  =r_{P_{\mathfrak{J}^{\max}}%
}^{\circ}\left(  \mathcal{L}\right)  =i_{s}\left(  \mathcal{L}\right)  .$
\begin{itemize}
\item[$\mathrm{(ii)}$]If $\mathcal{L}$ is a finite-dimensional complex Lie
algebra$,$ then%
\begin{equation}
i_{s}(\mathcal{N}_{\mathcal{L}})\leq r_{P_{\mathfrak{S}^{\max}}}^{\circ
}\left(  \mathcal{L}\right)  \leq r_{P_{\mathfrak{J}^{\max}}}^{\circ}\left(
\mathcal{L}\right)  =i_{s}(\mathcal{K}_{\mathcal{L}})+1\leq i_{s}%
(\mathcal{N}_{\mathcal{L}})+1, \label{10}%
\end{equation}
so that $1\leq r_{P_{\mathfrak{S}^{\max}}}^{\circ}\left(  \mathcal{L}\right)
\leq r_{P_{\mathfrak{J}^{\max}}}^{\circ}\left(  \mathcal{L}\right)  \leq
r_{P_{\mathfrak{S}^{\max}}}^{\circ}\left(  \mathcal{L}\right)  +1.$
\end{itemize}
\end{proposition}

\begin{proof}
(i) If $\mathcal{L}$ is nilpotent then (see \cite[p. 420]{M}) every maximal
Lie subalgebra is a Lie ideal, so that $P_{\mathfrak{S}^{\max}}\left(
\mathcal{L}\right)  =P_{\mathfrak{J}^{\max}}\left(  \mathcal{L}\right)  $.
Hence, by (\ref{Fmarsh}), $P_{\mathfrak{S}^{\max}}\left(  \mathcal{L}\right)
=P_{\mathfrak{J}^{\max}}\left(  \mathcal{L}\right)  =\mathcal{K}_{\mathcal{L}%
}=\mathcal{L}_{[1]}.$ Thus%
\begin{equation}
P_{\mathfrak{S}^{\max}}^{k}\left(  \mathcal{L}\right)  =P_{\mathfrak{J}^{\max
}}^{k}\left(  \mathcal{L}\right)  =\mathcal{L}_{[k]}\text{ for each }k,
\label{10.1}%
\end{equation}
so that $r_{P_{\mathfrak{S}^{\max}}}^{\circ}\left(  \mathcal{L}\right)
=r_{P_{\mathfrak{J}^{\max}}}^{\circ}\left(  \mathcal{L}\right)  =i_{s}\left(
\mathcal{L}\right)  .$

(ii) By (\ref{Fmarsh}), $P_{\mathfrak{S}^{\max}}\left(  \mathcal{L}\right)  $
and $P_{\mathfrak{J}^{\max}}\left(  \mathcal{L}\right)  $ are nilpotent for
each $\mathcal{L}\in\mathfrak{L}^{\mathrm{f}}.$ Hence, by (\ref{10.1}),%
\begin{align*}
P_{\mathfrak{S}^{\max}}^{k}\left(  \mathcal{L}\right)   &  =P_{\mathfrak
{S}^{\max}}^{k-1}\left(  P_{\mathfrak{S}^{\max}}\left(  \mathcal{L}\right)
\right)  =P_{\mathfrak{S}^{\max}}\left(  \mathcal{L}\right)  _{[k-1]}%
\text{,}\\
P_{\mathfrak{J}^{\max}}^{k}\left(  \mathcal{L}\right)   &  =P_{\mathfrak
{J}^{\max}}^{k-1}\left(  P_{\mathfrak{J}^{\max}}\left(  \mathcal{L}\right)
\right)  =P_{\mathfrak{J}^{\max}}\left(  \mathcal{L}\right)  _{[k-1]}.
\end{align*}
Let $R$ be $P_{\mathfrak{S}^{\max}}$ or $P_{\mathfrak{J}^{\max}}.$ By
(\ref{r1}), $r_{R}^{\circ}\left(  \mathcal{L}\right)  $ is the least $n$ such
that $R^{n}\left(  \mathcal{L}\right)  =\{0\}.$ Thus%
\begin{align}
r_{P_{\mathfrak{S}^{\max}}}^{\circ}\left(  \mathcal{L}\right)   &
=i_{s}(P_{\mathfrak{S}^{\max}}\left(  \mathcal{L}\right)  )+1\text{ and
}\nonumber\\
r_{P_{\mathfrak{J}^{\max}}}^{\circ}\left(  \mathcal{L}\right)   &
=i_{s}(P_{\mathfrak{J}^{\max}}\left(  \mathcal{L}\right)  )+1. \label{e10.2}%
\end{align}
Hence, by (\ref{Fmarsh}) and (\ref{e10.2}),%
\begin{equation}
r_{P_{\mathfrak{J}^{\max}}}^{\circ}\left(  \mathcal{L}\right)  =i_{s}%
(P_{\mathfrak{J}^{\max}}\left(  \mathcal{L}\right)  )+1=i_{s}(\mathcal{K}%
_{\mathcal{L}})+1\leq i_{s}(\mathcal{N}_{\mathcal{L}})+1. \label{e10.3}%
\end{equation}

As $P_{\mathfrak{S}^{\max}}$ is balanced and $\mathcal{N}_{\mathcal{L}}$ is
nilpotent, we obtain 
\[
(\mathcal{N}_{\mathcal{L}})_{[1]}=P_{\mathfrak{S}^{\max
}}(\mathcal{N}_{\mathcal{L}})\subseteq P_{\mathfrak{S}^{\max}}\left(
\mathcal{L}\right)  
\]
 from (\ref{10.1}). Hence $(\mathcal{N}_{\mathcal{L}%
})_{[k+1]}\subseteq P_{\mathfrak{S}^{\max}}\left(  \mathcal{L}\right)
_{[k]},$ so that 
\[
i_{s}(\mathcal{N}_{\mathcal{L}})\leq i_{s}(P_{\mathfrak
{S}^{\max}}\left(  \mathcal{L}\right)  )+1.
\]
Combining this with (\ref{e10.2})
and (\ref{e10.3}) and taking into account that $r_{P_{\mathfrak{S}^{\max}}%
}^{\circ}\left(  \mathcal{L}\right)  \leq r_{P_{\mathfrak{J}^{\max}}}^{\circ
}\left(  \mathcal{L}\right)  $, we have (\ref{10}).
\end{proof}

It follows from Proposition \ref{P9.1}(ii) that $\mathfrak{L}^{\mathrm{f}}$
can be partitioned into three following classes:%
\begin{align*}
\mathfrak{L}^{\mathrm{f}}  &  =\mathrm{C}_{1}\cup\mathrm{C}_{2}\cup
\mathrm{C}_{3},\text{ where }\\
\mathrm{C}_{1}  &  =\{\mathcal{L}\in\mathfrak{L}^{\mathrm{f}}:r_{P_{\mathfrak
{S}^{\max}}}^{\circ}\left(  \mathcal{L}\right)  =r_{P_{\mathfrak{J}^{\max}}%
}^{\circ}\left(  \mathcal{L}\right)  =i_{s}(\mathcal{K}_{\mathcal{L}}%
)+1=i_{s}(\mathcal{N}_{\mathcal{L}})+1\};\\
\mathrm{C}_{2}  &  =\{\mathcal{L}\in\mathfrak{L}^{\mathrm{f}}:r_{P_{\mathfrak
{S}^{\max}}}^{\circ}\left(  \mathcal{L}\right)  =r_{P_{\mathfrak{J}^{\max}}%
}^{\circ}\left(  \mathcal{L}\right)  =i_{s}(\mathcal{K}_{\mathcal{L}}%
)+1=i_{s}(\mathcal{N}_{\mathcal{L}})\};\\
\mathrm{C}_{3}  &  =\{\mathcal{L}\in\mathfrak{L}^{\mathrm{f}}:r_{P_{\mathfrak
{S}^{\max}}}^{\circ}\left(  \mathcal{L}\right)  +1=r_{P_{\mathfrak{J}^{\max}}%
}^{\circ}\left(  \mathcal{L}\right) \\
&  =i_{s}(\mathcal{K}_{\mathcal{L}})+1=i_{s}(\mathcal{N}_{\mathcal{L}})+1\}.
\end{align*}

For each integer $n\geq1,$ set%
\begin{align*}
\mathfrak{L}_{\left(  n,n\right)  }^{\text{f}}  &  =\{\mathcal{L}\in
\mathfrak{L}^{\text{f}}\text{: }r_{P_{\mathfrak{S}^{\max}}}^{\circ}\left(
\mathcal{L}\right)  =r_{P_{\mathfrak{J}^{\max}}}^{\circ}\left(  \mathcal{L}%
\right)  =n\},\\
\mathfrak{L}_{\left(  n,n+1\right)  }^{\text{f}}  &  =\{\mathcal{L}%
\in\mathfrak{L}^{\text{f}}\text{: }r_{P_{\mathfrak{S}^{\max}}}^{\circ}\left(
\mathcal{L}\right)  =n\text{ and }r_{P_{\mathfrak{J}^{\max}}}^{\circ}\left(
\mathcal{L}\right)  =n+1\}.
\end{align*}
We have from Proposition \ref{P9.1}(i) that C$_{2}$ contains all nilpotent Lie
algebras and that $\mathfrak{L}_{(n,n)}^{\text{f}}\neq\varnothing$ for all
$n\geq1.$ We also obtain from Proposition \ref{P9.1} that
\begin{align*}
\mathrm{C}_{1}\cup\mathrm{C}_{2}  &  =\cup_{n\geq1}\mathfrak{L}_{(n,n)}%
^{\text{f}},\text{ }\quad\mathrm{C}_{3}=\cup_{n\geq1}\mathfrak{L}_{(n,n+1)}%
^{\text{f}},\text{ }\\
\mathrm{C}_{1}\cap\mathfrak{L}_{(1,1)}^{\text{f}}  &  =\{\mathcal{L}%
\in\mathfrak{L}^{\mathrm{f}}\text{: }\mathcal{L}\text{ is semisimple}\},\\
\mathrm{C}_{2}\cap\mathfrak{L}_{(1,1)}^{\text{f}}  &  =\{\mathcal{L}%
\in\mathfrak{L}^{\mathrm{f}}\text{: }\mathcal{L}=N_{\mathcal{L}}%
\oplus\mathrm{rad}\left(  \mathcal{L}\right)  ,\text{ }N_{\mathcal{L}}\text{
is semisimple }\\
& \phantom{= = and rad a} \text{and }\mathrm{rad}\left(  \mathcal{L}\right)    \neq\{0\}\text{ is
commutative}\}.
\end{align*}

Let us show that $\mathfrak{L}_{(n,n+1)}^{\text{f}}\neq\varnothing$ for all
$n\geq1.$ Consider the solvable Lie algebra $\mathcal{L}$ of all upper
triangular $n\times n$ matrices. Then $\mathcal{K}_{\mathcal{L}}%
=\mathcal{L}_{[1]}=\mathcal{N}_{\mathcal{L}}$ is the nilpotent Lie subalgebra
of $\mathcal{L}$ that consists of all matrices with zero on the diagonal. The
Lie subalgebras $\mathcal{L}_{kk}=\{a=(a_{ij})\in\mathcal{L}$: $a_{kk}=0\},$
$1\leq k\leq n,$ and $\mathcal{L}_{k,k+1}=\{a=(a_{ij})\in\mathcal{L}$:
$a_{k,k+1}=0\},$ $1\leq k\leq n-1,$ have codimension $1$ in $\mathcal{L}$, so
that they are maximal. Hence%
\[
P_{\mathfrak{S}^{\max}}\left(  \mathcal{L}\right)  \subseteq(\cap
_{k}\mathcal{L}_{kk})\cap(\cap_{k}\mathcal{L}_{k,k+1})=\mathcal{L}_{[2]}.
\]
Therefore%
\begin{align*}
r_{P_{\mathfrak{S}^{\max}}}^{\circ}\left(  \mathcal{L}\right)  \overset
{(\ref{e10.2})}{=}i_{s}(P_{\mathfrak{S}^{\text{max}}}(\mathcal{L}))+1  &  \leq
i_{s}(\mathcal{L}_{[2]})+1\text{ and }\\
r_{P_{\mathfrak{J}^{\max}}}^{\circ}\left(  \mathcal{L}\right)  \overset
{(\ref{e10.3})}{=}i_{s}(\mathcal{K}_{\mathcal{L}})+1  &  =i_{s}(\mathcal{L}%
_{[1]})+1,
\end{align*}
so that $r_{P_{\mathfrak{S}^{\max}}}^{\circ}\left(  \mathcal{L}\right)  +1\leq
r_{P_{\mathfrak{J}^{\max}}}^{\circ}\left(  \mathcal{L}\right)  .$ Thus, by
Proposition \ref{P9.1},
\[
r_{P_{\mathfrak{S}^{\max}}}^{\circ}\left(  \mathcal{L}\right)
+1=r_{P_{\mathfrak{J}^{\max}}}^{\circ}\left(  \mathcal{L}\right)
=i_{s}(\mathcal{K}_{\mathcal{L}})+1=i_{s}(\mathcal{L}_{[1]})+1=n.
\]
Then $\mathcal{L}\in\mathfrak{L}_{(n,n+1)}^{\text{f}}.$ Combining this and
Proposition \ref{P9.1} yields

\begin{corollary}
$\mathfrak{L}^{\mathrm{f}}=\cup_{n}\left(  \mathfrak{L}_{(n,n)}^{\mathrm{f}%
}\cup\mathfrak{L}_{(n,n+1)}^{\mathrm{f}}\right),$  all classes
$\mathfrak{L}_{(n,n)}^{\mathrm{f}}$ and $\mathfrak{L}_{(n,n+1)}^{\mathrm{f}}$
are non-empty.
\end{corollary}

Proposition \ref{P9.1} also gives us a proof of Corollary \ref{C10.2}.\medskip

\textit{Proof of Corollary }\ref{C10.2}. Let $\mathcal{L}\in$ $\mathbf{Sem}%
(P_{\mathfrak{J}^{\text{max}}})\cap\mathfrak{L}^{\mathrm{f}}.$ Then we have
$r_{P_{\mathfrak{J}^{\max}}}^{\circ}\left(  \mathcal{L}\right)  =1$ and, by
(\ref{10}), $i_{s}(\mathcal{K}_{\mathcal{L}})=0.$ Hence $\mathcal{K}%
_{\mathcal{L}}=\left[  \mathcal{L},\operatorname*{rad}\left(  \mathcal{L}%
\right)  \right]  =\{0\},$ so that rad$(\mathcal{L})$ is the center
$Z_{\mathcal{L}}$ of $\mathcal{L}.$ As $\mathcal{L}=N_{\mathcal{L}}%
\oplus^{\text{ad}}$rad$(\mathcal{L)}$ is the semidirect product of a
semisimple Lie algebra $N_{\mathcal{L}}$ and rad($\mathcal{L),}$ we have that
$\mathcal{L}=N_{\mathcal{L}}\oplus Z_{\mathcal{L}}$.

\end{document}